\documentclass[leqno,a4paper,11pt]{amsart}
\usepackage{fullpage}
\usepackage[english]{babel}

\usepackage[x11names]{xcolor} 

\usepackage[T1]{fontenc}

\renewcommand{\arraystretch}{1.3}

\usepackage{amsfonts,amsmath,amssymb}
\usepackage[fleqn]{mathtools}
\usepackage{float}
\usepackage{bm}
\usepackage{upgreek}
\usepackage{slashed}
\usepackage{accents}
\usepackage{youngtab}

\usepackage[ddmmyyyy]{datetime}

\usepackage[%
	backend=bibtex8,
	style=numeric,
	giveninits=true,
	doi=false,
	url=false,
	isbn=false
	]{biblatex} 
\renewbibmacro{in:}{} 
\bibliography{biblio}

\usepackage{tikz}
\usepackage{graphicx}
\usepackage[all]{xy}
\usepackage{array}
\usepackage{scalerel}
\usepackage{multirow}
\usepackage[shortlabels]{enumitem}
\usepackage{caption}
\usepackage{hyperref}

\usepackage{amsthm}
\theoremstyle{plain}
\newtheorem{thm}{Theorem}[section]
\newtheorem{lem}[thm]{Lemma}
\newtheorem{prop}[thm]{Proposition}
\newtheorem{cor}[thm]{Corollary}
\theoremstyle{definition}
\newtheorem{defn}[thm]{Definition}

\newtheorem{exa}[thm]{Example}

\newtheorem{rem}[thm]{Remark} 

\numberwithin{equation}{section}
\newcommand{\parderv}[2] {\frac{\partial #1}{\partial #2}}

\renewcommand{\d} {\mathrm{d}}

\renewcommand{\div} {\mathrm{div}}

\newcommand{\ol} [1] {\overline{#1}}
\newcommand{\wh} [1]{\widehat{#1}}
\newcommand{\wt} [1]{\widetilde{#1}}
\newcommand{\mr}[1] {\mathring{#1}}
\newcommand{\ul} [1]{\underline{#1}}
\newcommand{\ut} [1]{\undertilde{#1}}

\newcommand{\mbf} [1]{\mathbf{#1}}
\newcommand{\mbb} [1]{\mathbb{#1}}
\newcommand{\mc} [1]{\mathcal{#1}}
\newcommand{\mf} [1]{\mathfrak{#1}}
\newcommand{\mrm} [1]{\mathrm{#1}}
\newcommand{\msf} [1]{\mathsf{#1}}

\renewcommand{\i} {\mathrm{i}}
\newcommand{\e} {\mathrm{e}}

\newcommand{\dbl} {[\![}
\newcommand{\dbr} {]\!]}

\renewcommand{\bigwedge}{\scaleobj{1.2}{\wedge}}

\renewcommand{\bigodot}{\scaleobj{1.2}{\odot}}

\newcommand{\hook}{\makebox[7pt]{\rule{6pt}{.3pt}\rule{.3pt}{5pt}}\,}

\newcommand{\Rho} {\mathsf{P}}
\newcommand{\Nh} {\mathsf{N}}

\newcommand{\Lv} {\mathsf{L}}

\newcommand{\Hom} {\mathrm{Hom}}
\newcommand{\Ann} {\mathrm{Ann}}

\newcommand{\gr} {\mathrm{gr}}

\renewcommand{\u} {\mathfrak{u}}
\newcommand{\su} {\mathfrak{su}}

\newcommand{\so} {\mathfrak{so}}

\newcommand{\g} {\mathfrak{g}}

\newcommand{\p} {\mathfrak{p}}

\newcommand{\R} {\mathbf{R}}
\newcommand{\C} {\mathbf{C}}

\newcommand{\U}{\mathbf{U}}
\newcommand{\SU}{\mathbf{SU}}

\newcommand{\SO}{\mathbf{SO}}
\newcommand{\Spin} {\mathbf{Spin}}

\newcommand{\CO}{\mathbf{CO}}
\newcommand{\Sim}{\mathbf{Sim}}

\newcommand{\CP} {\mathbf{C}\mathbb{P}}
\newcommand{\RP} {\mathbf{R}\mathbb{P}}

\newcommand{\V}{{\mathbb V}}
\newcommand{\W}{{\mathbb W}}
\newcommand{\K}{{\mathbb K}}

\def\ydhook{\raisebox{-.4ex}{\scaleobj{0.35}{\yng(2,1)}}}

\def\ydskew{\raisebox{-.4ex}{\scaleobj{0.35}{\yng(1,1,1)}}}

\begin{document}

\title
{Almost Robinson geometries}

\begin{abstract}
We investigate the geometry of almost Robinson manifolds, Lorentzian analogues of almost Hermitian manifolds, defined by Nurowski and Trautman as Lorentzian manifolds of even dimension equipped with a totally null complex distribution of maximal rank. Associated to such a structure, there is a congruence of null curves, which, in dimension four, is geodesic and non-shearing if and only if the complex distribution is involutive. Under suitable conditions, the distribution gives rise to an  almost Cauchy--Riemann structure on the leaf space of the congruence.

We give a comprehensive classification of such manifolds on the basis of their intrinsic torsion. This includes an investigation of the relation between an almost Robinson structure and the geometric properties of the leaf space of its congruence. We also obtain conformally invariant properties of such a structure, and we finally study an analogue of so-called generalised optical geometries as introduced by Robinson and Trautman.
\end{abstract}

\date{\today\ at \xxivtime}

\author{Anna Fino}\address{Universit\`{a} di Torino, Dipartimento di Matematica ``G. Peano", Via Carlo Alberto, 10 - 10123, Torino, Italy $\&$  Department of Mathematics and Statistics\\
Florida International University\\
Miami Florida, 33199, USA}
 \email{annamaria.fino@unito.it, afino@fiu.edu}
\author{Thomas Leistner}\address{School of Mathematical Sciences, University of Adelaide, SA 5005, Australia}\email{thomas.leistner@adelaide.edu.au}
\author{Arman Taghavi-Chabert${}^\ast$}\address{Faculty of Physics, University of Warsaw, ul. Pasteura 5, 02-093 Warszawa, Poland}\email[Corresponding author]{arman.taghavi-chabert@fuw.edu.pl}

\thanks{\textit{Funding.} AF was supported by GNSAGA of INdAM, by PRIN 2017 \lq \lq Real and Complex Manifolds: Topology, Geometry and Holomorphic Dynamics" and by a grant from the Simons Foundation (\#944448).
TL was supported by
 the Australian Research
Council (Discovery Program DP190102360). ATC \& TL declare that this work was partially supported by the grant 346300 for IMPAN from the Simons Foundation and the matching 2015-2019 Polish MNiSW fund.  The research of ATC leading to these results has received funding from the Norwegian Financial Mechanism 2014-2021 UMO-2020/37/K/ST1/02788. ATC was also supported by a long-term faculty development grant from the American University of Beirut for his visit to IMPAN, Warsaw, in the summer 2018. He received funding from the GA\v{C}R (Czech Science Foundation) grant 20-11473S in 2020.
\\
\indent \textit{Conflict of interest.}
On behalf of all authors, the corresponding author states that there is no conflict of interest.
\\
\indent \textit{Data Availability Statement.}	
	Data sharing is not applicable to this article as no datasets were generated or analyzed during the current study.
}
\subjclass[2010]{Primary 53C50, 53C10; Secondary 53B30, 53C18}
\keywords{Lorentzian manifolds, almost Robinson structures, G-structure, intrinsic torsion, congruences of null geodesics, conformal geometry, almost CR structures}

\maketitle
\tableofcontents

\section{Introduction}
In a recent article \cite{Fino2020}, the authors give a comprehensive review of the notion of \emph{optical structure} on a Lorentzian manifold $(\mc{M},g)$, simply understood as a null line distribution $K$ on $\mc{M}$. Many of the geometric properties of this distribution and its orthogonal complement are encoded in terms of its screen bundle $H_K = K^\perp/K$, which is naturally equipped with a bundle metric $h$ inherited from $g$. One may naturally wish to endow $H_K$ with further bundle structures. In the present article, where we assume $\mc{M}$ to have dimension $2m+2$, we equip $H_K$ with a bundle complex structure $J$ compatible with $h$. Such a structure was introduced by Nurowski and Trautman in \cite{Nurowski2002,Trautman2002,Trautman2002a}, where it is equivalently described in terms of a totally null complex $(m+1)$-plane distribution $N$. The real span of the intersection $N \cap \overline{N}$ then determines the line distribution $K$. Following their terminology, we shall refer to the pair $(N,K)$ as an \emph{(almost) Robinson structure}. The structure group of the frame bundle is reduced to $(\R_{>0}\times \U(m))\ltimes (\R^{2m})^*$, which is a subgroup of the group $\Sim(2m)$, which characterises optical structures, and as in \cite{Fino2020}, we shall describe the geometric properties of an almost Robinson structure in terms of its \emph{intrinsic torsion}. Our approach is analogous to that of Gray and Hervella in the almost Hermitian setting \cite{Gray1980}. In our case, however, it is the decomposition of the screen bundle with its complex structure, rather than the tangent bundle, that encodes the geometric properties of the almost Robinson structure. To this end, we exploit the interaction with the optical structure and use results already obtained in \cite{Fino2020}. The main results, contained in Theorems \ref{thm-main-Rob} and \ref{thm:intors-rob}, give an invariant description of the module of intrinsic torsions of an almost Robinson structure. On the basis of this description, we proceed to examine the implications of the torsion classes in terms of geometric properties.

Such geometries have already been studied, notably in dimension four, and we shall briefly review some existing results below. As is well-known \cite{Penrose1986,Nurowski2002,Trautman2002,Trautman2002a,Fino2020}, an almost Robinson structure $(N,K)$ in dimension four is essentially equivalent to an optical structure. The key point, here, is that the involutivity of the totally null complex $2$-plane distribution $N$ is equivalent to the congruence $\mc{K}$ of null curves tangent to $K$ being geodesic and \emph{non-shearing}, that is, the conformal class of the bundle metric $h$ is preserved along the geodesic curves of $\mc{K}$. What is more,  the rank-one complex vector bundle $N/{}^{\C}K$ descends to the leaf space $\ul{\mc{M}}$ of $\mc{K}$, thereby endowing it with a \emph{Cauchy--Riemann (CR) structure}. This CR geometrical aspect of Robinson manifolds was particularly emphasised by Robinson, Trautman and the `Warsaw' school \cite{Tafel1985,Tafel1986,Robinson1986,nurowski93-phd,Lewandowski1990,Nurowski1997}, and in parallel, by the twistor school \cites{Mason1985,Mason1998}. This property is useful when seeking solutions to the Einstein field equations \cite{Lewandowski1991a,Tafel1991,Nurowski1992}, a problem that is in turn linked to analytic questions regarding the  embeddability of CR manifolds \cite{Tafel1985,Tafel1986,Lewandowski1990a,Hill2008,Schmalz2019}.

There are also three important theorems worthy of mention in the development of mathematical relativity in the present context:
\begin{itemize}
\item The \emph{Mariot--Robinson theorem} \cite{Mariot1954,Robinson1961} gives a correspondence between analytic non-shearing congruences of null geodesics and \emph{null} or \emph{algebraically special} electromagnetic fields in vacuum.
\item The \emph{Goldberg--Sachs theorem} \cite{Goldberg1962,Goldberg2009} relates the existence of non-shearing congruences of null geodesics to the algebraic degeneracy of the Weyl tensor for Einstein spacetimes.
\item The \emph{Kerr theorem}, as formulated in \cite{Penrose1967}, tells us how such congruences arise in Minkowski space from complex submanifolds of three-dimensional complex projective space.
\end{itemize}

In higher (even) dimensions, the congruence of null curves of an involutive almost Robinson structure $(N,K)$ is always geodesic, but shearing in general \cite{Trautman2002a}. The leaf space of $\mc{K}$ nevertheless still acquires a CR structure \cite{Nurowski2002,Trautman2002}. In addition, (almost) Robinson structures are Lorentzian analogues of (almost) Hermitian structures, to borrow the expression from Nurowski and Trautman \cite{Nurowski2002}. In both cases, the underlying geometric object is that of an \emph{almost null structure}, that is, a totally null complex $(m+1)$-plane distribution. This perspective allows one to have a unified approach to pseudo-Riemannian geometry in any signature. In dimension four, the analogies between Lorentzian and Hermitian geometries were already pointed out in \cite{nurowski93-phd,nurowski96} especially in connection with the aforementioned theorems of mathematical relativity. For instance, the Kerr theorem finds an articulation in Riemannian signature as follows: any local Hermitian structure on four-dimensional Euclidean space corresponds to a holomorphic section of its twistor bundle \cite{Eells1985,Salamon2009}. A Riemannian counterpart of the Goldberg--Sachs theorem is given in \cite{Przanowski1983,nurowski93-phd,Apostolov1997,GoverHillNurowski11}. In split signature, one obtains analogous results -- see e.g.\ \cite{Guillemin1986,GoverHillNurowski11}.

Almost null structures are also intimately connected with the notion of \emph{pure spinors}, and thus hark back to \'{E}lie Cartan's seminal work \cite{Cartan1967}, which was subsequently developed in \cite{Budinich1988,Budinich1989,Kopczy'nski1992,Kopczy'nski1997} among others. It is then no surprise that in dimension four, the spinorial approach to general relativity promoted by Penrose and his school \cite{Witten1959,Penrose1960,Penrose1984,Penrose1986}  shed much light on the \emph{complex} aspect of congruences of null geodesics, and was influential in the development of twistor theory \cite{Penrose1967}. These ideas were later developed in higher even dimensions in \cite{Hughston1990b,Hughston1990a,Hughston1990,Hughston1995,Jeffryes1995}, and most notably in the article \cite{Hughston1988} by Hughston and Mason, where the Kerr and Robinson theorems are generalised in the context of involutive almost null structures. These results were expanded by the third author of the present article in \cite{Taghavi-Chabert2016,Taghavi-Chabert2017a,TaghaviChabert2017}, where a comprehensive study of almost null structures  according to their intrinsic torsion is given in both even and odd dimensions. The recent articles \cite{Papadopoulos2019,FigueroaOFarrill2020,Petit2019} also touch on related topics on pseudo-Riemannian geometry.

Non-shearing congruences of null geodesics are ubiquitous in four-dimensional mathematical relativity, see e.g.\ \cite{Stephani2003} and references therein. One question that arises is, which of non-shearing congruences and (almost) Robinson structures have most relevance in higher dimensions? On the one hand, Robinson--Trautman and Kundt spacetimes, which are by definition characterised by the existence of a \emph{non-twisting} non-shearing congruence of null geodesics, have been well studied in arbitrary dimensions, see e.g. \cite{Podolsky2008,Podolsky2015}. On the other hand, the Kerr metric and its variants admit a pair of \emph{twisting} congruences of null geodesics, which are non-shearing in dimension four, but fail to be so in higher dimensions \cite{Pravda2007}.
Nonetheless, as was first brought to light in \cite{Mason2010}, these metrics admit several Robinson structures in any dimensions\footnote{These are not explicitly referred to as Robinson structures there, but may be interpreted as such.}. Almost Robinson structures can also be defined in terms of a maximal totally null complex distribution on odd-dimensional Lorentzian manifolds: the black ring in dimension five is equipped with a pair of Robinson structures,\footnote{In \cite{Taghavi-Chabert2011}, these Robinson structures are referred to as optical structures in a sense similar to \cite{nurowski96}. This terminology is now obsolete by virtue of \cite{Fino2020}.} but does not admit any non-shearing congruences of null geodesics \cite{Taghavi-Chabert2011} --- see also \cite{Taghavi-Chabert2012}. In dimension three, one can similarly obtain analogous results -- see e.g.\ \cite{Nurowski2015}.

Considering the length of this article and the technicalities involved, the following section includes a detailed summary of our main results, section by section.

\section{Summary of results}
Our journey starts in Section \ref{sec:algebra}, where we introduce the algebraic notion of a \emph{Robinson structure} on  a   $(2m+2)$-dimensional Minkowski space  $(\V,g)$, as a pair $(\mbb{N},\K)$, where $\mbb{N}$ is a totally null complex $(m+1)$-plane distribution and $\K$ the real null distribution whose complexification is given by $\mbb{N} \cap \ol{\mbb{N}}$. Proposition \ref{prop:char-Robinson} gives various algebraic characterisations of $(\mbb{N},\K)$ as
	\begin{enumerate}
		\item a totally null complex $(m+1)$-form;
		\item an optical structure $\K$ whose screenspace $ \mbb{H}_\K= \K^\perp/\K$ is endowed with a complex structure $J_i{}^j$ compatible with the induced metric $h_{i j}$;\label{item:opt}
		\item   a $1$-form $\kappa_a$ and  a $3$-form $\rho_{abc}$ satisfying 
		$\rho _{ab} \, {}^e \rho _{cde}  = - 4 \kappa_{[a} g_{b][c} \kappa_{d]}$;
		\item  a pure spinor of real index $1$.
	\end{enumerate}
	Using  the characterisation \eqref{item:opt} above, we  determine  the  stabiliser  of a Robinson structure  as a closed Lie subgroup $Q$ of the stabiliser $P$ of $\K$ in $G =\SO^0(2m+1,1)$. We can thus apply the findings of  \cite{Fino2020} to describe in Sections \ref{sec:alg_intors} and \ref{sec:alg_intors_iso} the space $\mbb{G}$ of algebraic intrinsic torsions for $m >1$: the basic idea is that the group $Q$ induces a $Q$-invariant filtration on $\mbb{G}$, and the associated graded $Q$-modules split into further irreducibles linearly isomorphic to $\U(m)$-modules. The main results are collected in Theorems \ref{thm-main-Rob} and \ref{thm:intors-rob}, and while comprehensive, they are also rather technical.

Having all the algebraic machinery at disposal, we proceed to apply it to the geometric setting in Section \ref{sec:geometry}: thus, an \emph{almost Robinson structure} on  an oriented and time-oriented Lorentzian manifold  $(\mc{M},g)$ of dimension $2m+2$  is defined as  a pair $(N,K)$ where $N$ is a complex distribution of rank $m+1$ totally null with respect to  the complexfication ${}^\C g$, and $K$ a real line distribution such that ${}^\C K = N \cap \overline{N}$. The quadruple $(\mc{M},g,N,K)$ is then referred to as an \emph{almost Robinson manifold} or \emph{geometry}. Considering the large number of classes of intrinsic torsions for almost Robinson geometries, we shall split almost Robinson geometries into a number of broad types, and we will focus on the cases most amenable to geometric interpretations.

An  almost Robinson structure $(N,K)$ induces an optical structure on $(\mc{M},g)$ in the sense of \cite{Fino2020}, namely a filtration of vector bundles $K \subset K^\perp \subset T \mc{M}$. The orientation and time-orientation on $\mc{M}$ induce an orientation on $K$, and the screen bundle  $H_K := K^\perp / K$ of $K$ inherits a positive-definite bundle metric $h$ from $g$.  An \emph{optical vector field}, that is, a non-vanishing section of $K$, generates a \emph{congruences of null curves}, and one of the main points is to investigate its geometric properties together with those of its associated leaf space. Dually, we may also consider any \emph{optical $1$-form}, i.e.\ a section of $\Ann(K^\perp)$, which, by virtue of the almost Robinson structure, has an associated \emph{Robinson $3$-form} $\rho$, which encodes an hermitian structure on $H_K$. There are other natural objects that can be used as specified in Proposition \ref {prop:char-Robinson-mfld}, notably a pure spinor field of real index one, defined up to scale.

As it will emerge, the leaf space of the congruence generated by an optical vector field in many cases turns out to be an \emph{almost CR manifold}, that is, a triple $(\ul{\mc{M}},\ul{H},\ul{J})$, where $\ul{\mc{M}}$ is a smooth manifold of dimension $2m+1$, $\ul{H}$ a rank-$2m$ distribution and $\ul{J}$ a bundle complex structure on $\ul{H}$. When the $\pm \i$-eigenbundles of $\ul{J}$ are involutive, we refer to $(\ul{\mc{M}},\ul{H},\ul{J})$ as a \emph{CR manifold}. Section \ref{sec:CR} is devoted to the subject, which plays an important part in this article, and one of the aims of the subsequent sections is to relate the classes of intrinsic torsions to the geometric property of the underlying (almost) CR structure.

This is in fact the main focus of Section \ref{sec:lift} regarding so-called \emph{nearly Robinson geometries}, that is, almost Robinson geometries for which $[K,N] \subset N$. This condition alone tells us that $N$ induces an almost CR structure on the leaf space of the null geodesic congruence tangent to $K$. They include as a subclass the so-called \emph{Robinson geometries} for which $[N,N] \subset N$, and in fact, generalise the notion of \emph{non-shearing congruence of null geodesics}, central object of mathematical relativity: these are generated by a null vector field $k$ that satisfies $\mathsterling_k g (v,w) \propto g (v,w)$ for any vector field $v$ and $w$ orthogonal to $k$. Among the most striking results of this section are Propositions \ref{prop:lift} and \ref{prop:Rob->CR}, which state that any almost CR structure $(\ul{\mc{M}},\ul{H},\ul{J})$ can be `lifted' to a nearly Robinson manifold on the trivial line bundle $\ul{\mc{M}} \times \R$, and conversely, any nearly Robinson manifold arises in this way. A normal form for the Robinson metric is provided therein, and we discuss its various consequences. For instance, if the congruence is \emph{maximally twisting}, in the sense that any optical $1$-form $\kappa$ satisfies $\kappa \wedge (\d \kappa)^m \neq 0$, then the underlying almost CR structure is \emph{contact}. Proposition \ref{prop:Rob->pi_cont_CR} characterises the existence of a so-called \emph{partially integrable almost CR structure} $(\ul{\mc{M}},\ul{H},\ul{J})$ and an auxiliary subconformal structure on $\ul{H}$ in terms of the intrinsic torsion of its nearly Robinson lift.

Section \ref{sec:tw-ind_al_Rob} shifts the focus to another particularly interesting class, which consists of almost Robinson structures that are \emph{twist-induced}, meaning that if one starts merely from an optical geometry and choose any optical $1$-form $\kappa$, then $\kappa \wedge \d \kappa$ is proportional to a Robinson $3$-form. In other words, the optical geometry has a canonically associated almost Robinson structure determined by the \emph{twist} of its null geodesic congruence. The most remarkable aspect of such a configuration is that such an optical, or almost Robinson, geometry admits a unique distinguished optical vector field, as pointed out in Proposition \ref{prop:twist-ind_Rob}. Twist-induced \emph{nearly} Robinson geometries are also natural generalisation of \emph{twisting} non-shearing congruence of null geodesics from four to higher even dimensions, and one obtains further characterisations of their intrinsic torsion in Propositions \ref{prop:tw-ind_Rob_shear} and \ref{prop:tw-ind_Rob_no_shear}.

In dimension four, the use of spinors provides another potent approach to the study of geometric structures on Lorentzian manifolds, and Section \ref{sec:spinor_des} explores this theme further. Theorem \ref{thm:spinor_intors} notably characterises a number of classes of intrinsic torsions in terms of irreducible equations on a pure spinor field, which had already been in obtained in \cite{Taghavi-Chabert2016}. Proposition \ref{prop:fol_spinor} gives a description of the intrinsic torsion of any Robinson structure in terms of a non-linear spinorial differential equation, which generalises Penrose's well-known equation $\nu^{\mbf{A}'} \nu^{\mbf{B}'} \nabla_{\mbf{A} \mbf{A}'} \nu_{\mbf{B}'} = 0$.

After a brief review of the Gray--Hervella classication of almost Hermitian structures and its relation to the present article in Section \ref{sec:Rob-GH}, we move on to the study of almost (and in fact nearly) Robinson geometries for which the associated congruence of null geodesics is non-twisting and non-shearing. They fall into two classes: the  \emph{Kundt type} in Section \ref{sec:alRob_Kundt}, where the congruence is also \emph{non-expanding}, and the \emph{Robinson--Trautman type} in Section \ref{sec:alRob_RT} for which the congruence is \emph{expanding}. In both cases, the idea here is that, since the congruence is non-twisting, its leaf space admits a Riemannian foliation, each leaf of which is in fact an almost Hermitian manifold. We can then associate the class of intrinsic torsion of the nearly Robinson manifold to the Gray--Hervella class of this almost Hermitian foliation. as described in Table \ref{tab-intors-GH}.

The very brief Section \ref{sec:lin_conn} concludes our exploration of almost Robinson geometries in the metric setting by considering compatible linear connections, described by Proposition \ref{prop:lin_conn}.

In Section \ref{sec:conf-aRstr}, our definition of almost Robinson manifold is extended to the conformal setting in the obvious way by simply replacing a Lorentzian metric structure by an equivalence class of conformally related Lorentzian metrics. Many of the properties investigated in Section \ref{sec:geometry} carry over, and the only aspect that really need to be taken care of is which classes of intrinsic torsion are conformally invariant, and the answer is given by Theorem \ref{thm:conf-Rob}. Just as in the metric case, one can `lift' a given almost CR structure as a \emph{conformal} nearly Robinson manifolds on the line bundle. This construction is particularly interesting when the almost CR structure is contact and partially integrable, in which case Proposition \ref{prop-adapted_frame_max_tw-nsh-nexp_conf} show that changes of contact forms induce conformal changes of the nearly Robinson lift, not unlike the classical Fefferman construction --- see Example \ref{exa:Feff}.

Sections \ref{sec:Robinson_theorem} and \ref{sec:Kerr_theorem} review two theorems of importance stemming from mathematical relativity, namely the Mariot--Robinson theorem and the Kerr theorem respectively, and how they generalise to higher dimensions. The former is concerned with solutions to an appropriate generalisation of the vacuum Maxwell field equations, while the latter provides a geometric construction of Robinson structures in twistor space.

Finally, in  Section \ref{sec:gen_Rob}   we consider    \emph{generalised almost Robinson structures}, which can be viewed as  an extension to higher dimensions of the  notion   of optical structure    in dimension four  presented in \cite{Trautman1984,Trautman1985,Robinson1985,Robinson1986,Robinson1989,Musso1992,Trautman1999}).  A  generalised almost Robinson structure on a smooth manifold  $(\mc{M},g)$ of dimension $2m +2$,  is defined as a  triple $(N,K,\mbf{o})$,  where $N$ is a complex $(m+1)$-plane distribution, $K := N \cap T \mc{M}$ is a real line distribution on $\mc{M}$, and $\mbf{o}$ is an  equivalence class of Lorentzian metrics such that,  for every $g \in \mbf{o}$, $N$ is null with respect to the complex linear extension of $g$, and  any two metrics $g, \wh{g}\in \mbf{o}$ are related by the relation
$ \wh{g}  = \e^{2 \varphi} \left( g + 2 \, \kappa \, \alpha \right)$,  for some smooth function $\varphi$, a $1$-form $\alpha$ on $\mc{M}$, and $\kappa = g(k ,\cdot)$, with $k$ some non-vanishing section of $K$. In particular, for each choice of metric $g$, $(N,K)$ is an almost Robinson structure in the sense of Section \ref{sec:geometry}. In Theorem \ref{thm:gen_Rob_prop} we  determine which subbundles of the bundle of intrinsic torsions do not depend on the choice of metric $g$ in $\mbf{o}$. We can therefore start from a given almost CR structure, and construct a family of nearly Robinson metrics on a trivial line bundle parametrised by a $1$-form and a conformal factor. This is particularly useful in application to general relativity since we may wish to add the requirement that one of these metrics has prescribed Ricci tensor.  In Theorem  \ref{thm:integrable_Rob}  we extend the results  in \cite{Robinson1985,Fino2020}, obtaining  a characterisation of   the integrability  of   generalised optical structures  as $G$-structures.

Section \ref{sec:other_sign} discusses the possible generalisations to other metric signatures. We have relegated to Appendices \ref{app:proj} and \ref{app:gen_al_Rob} a number of technical formulae that are used in the main text.

As pointed out earlier, the notion of nearly Robinson structure provides a generalisation of non-shearing congruences of null geodesics from four to higher even dimensions. These are intrinsically connected to algebraically special Einstein four-manifolds, and one of the current and future applications of nearly Robinson structures
is the construction of higher-dimensional solutions to Einstein's field equations --- see e.g.\ \cite{Mason2010,Alekseevsky2018,Alekseevsky2021,Taghavi-Chabert2021}. We have scattered a number of relevant examples throughout the article to illustrate the point: the Kerr--NUT--(A)dS metrics, the Taub--NUT--(A)dS metrics in Examples \ref{exa:KerrNUTAdS} and \ref{exa:F-E_Taub-NUT} respectively, Kundt and Robinson--Trautman metrics as in Example \ref{exa:RT}, and the Myers-Perry metric in Example \ref{exa:MP}, to name but a few.

We also provide examples to illustrate some of the algebraic conditions that the intrinsic torsion of an almost Robinson structure can satisfy, focussing essentially on dimensions greater than four. Considering the rich range of classes of almost Robinson structures, this article does not aim to cover every possible case, but it leaves the construction of almost Robinson structures with prescribed intrinsic torsion as open problems --- see for instance Remark \ref{rem:non-geod}. We do not touch on questions related to the curvature of almost Robinson manifolds, these being dealt with in \cite{Taghavi-Chabert2014}.

\section{Algebraic description}\label{sec:algebra}
\subsection{Notation and conventions}\label{sec:Notation}

We set up the notation and conventions used throughout this article by recalling some basic notions of algebra -- see e.g.\ \cite{Salamon1989,Chern1995} for further details. The fields of real numbers and complex numbers will be denoted $\R$ and $\C$ respectively, the imaginary unit  by $\i$, i.e.\ $\i^2 = -1$.

Let $\V$ and $\mbb{W}$ be two real or complex vector spaces with respective duals $\V^*$ and $\W^*$. The annihilator of a vector subspace $\mbb{U}$ of $\V$ will be abbreviated to $\mathrm{Ann}(\mbb{U})$. The tensor product of  $\V$ and $\mbb{W}$ will be denoted $\V \otimes \mbb{W}$, the $p$-th exterior power of $\V$ by $\bigwedge^p \V$, its $p$-th symmetric power by $\bigodot^p \V$.

If $g$ is a non-degenerate symmetric bilinear form on $\V$,  the orthogonal complement of a subspace $\mbb{U}$ of $\mbb{V}$ with respect to $g$ will be denoted $\mbb{U}^\perp$. The subspace of $\bigodot^p \V$ consisting of elements that are tracefree with respect to $g$ will be denoted by $\bigodot^p_\circ \V$.

Let us assume that $\V$ is complex and of dimension $2m$. Under the Hodge duality operator $\star : \bigwedge^p \V^* \rightarrow \bigwedge^{2m-p} \V^*$, for $p=0,\ldots,2m$, the space $\bigwedge^{m} \V^*$ splits into the space of self-dual $m$-forms $\bigwedge^{m}_+  \V^*$ and the space of anti-self-dual $m$-forms $\bigwedge^{m}_-  \V^*$, i.e.\
\begin{align*}
\bigwedge^{m} \V^* = \bigwedge^{m}_+  \V^* \oplus \bigwedge^{m}_-  \V^* \, ,
\end{align*}
where $\star \alpha = \pm (\i) \alpha$ for any $\alpha \in \bigwedge^{m}_{\pm}  \V^*$.

Suppose now that $\V$ is real. The complexification $\C \otimes \V \cong \V \oplus \i \V$ of $\V$ will be denoted ${}^\C \V$. There is an induced reality structure, $\bar{}: {}^\C \V \rightarrow {}^\C \V$ on ${}^\C \V$, which preserves the elements of $\V$, i.e.\ for $v \in {}^\C \V$, we have that $v \in \V$ if and only if $\bar{v} = v$. If $\mbb{A}$ is a vector subspace of ${}^\C \V$, its complex conjugate is defined by $\overline{\mbb{A}} := \{ v \in {}^\C \V : \overline{v} \in \mbb{A} \}$. We say that $\mbb{A}$ is (totally) real if $\overline{\mbb{A}} = \mbb{A}$.

Suppose now that $\V$ has dimension $2m$ and is equipped with a complex structure $J$, that is an endomorphism of $\V$ that squares to minus the identity on $\V$, i.e.\ $J \circ J = - \mathrm{Id}$. Then
\begin{align*}
{}^\C \V & = \V^{(1,0)} \oplus \V^{(0,1)} \, , 
\end{align*}
where $\V^{(1,0)}$ and $\V^{(0,1)}$ are the $+\i$- and $-\i$-eigenspaces of $J$ respectively. These $m$-dimensional complex vector subspaces are complex conjugate to each other, i.e. $\V^{(1,0)} \cong \overline{\V^{(0,1)}}$. Similarly, we have a splitting of the dual space
\begin{align*}
{}^\C \V^* & = (\V^{(1,0)})^* \oplus (\V^{(0,1)})^* \, ,
\end{align*}
and $(\V^{(1,0)})^* = \Ann(\V^{(0,1)})$ and $(\V^{(0,1)})^* = \Ann(\V^{(1,0)})$.  For any non-negative integer $p$, $q$, the space of all $(p,q)$-forms on $\V$ is defined to be
\begin{align*}
\bigwedge^{(p,q)} \V^*  & :=  \bigwedge^p ( \V^{(1,0)})^* \otimes  \bigwedge^q ( \V^{(0,1)})^* \, .
\end{align*}
Similarly, we define the spaces
\begin{align}
\bigodot{}^{(p,q)} \V^*  & :=  \bigodot{}^p ( \V^{(1,0)})^* \otimes  \bigodot{}^q ( \V^{(0,1)})^* \, , \nonumber \\
\ydhook(  \V^* ) & := \left\{ \alpha \in \bigwedge^{(1,0)} \V^*   \otimes  \bigwedge^{(2,0)}  \V^* : \pi^{(3,0)} (\alpha) = 0 \right\} \, . \label{eq:ydhook}
\end{align}
where $\pi^{(3,0)}$ is the natural projection from $\bigwedge^{(1,0)} \V^*   \otimes  \bigwedge^{(2,0)}  \V^*$ to $\bigwedge^{(3,0)} \V^*$. This notation reflects the Young diagram symmetries of this irreducible $\mbf{GL}(m,\C)$-module, where $\mbf{GL}(m,\C)$ is the complex general linear group acting on $\V^{(1,0)} \cong \C^m$.

Since we are interested in real vector spaces, we also define, following the notation in \cite{Salamon1989},
\begin{align}
\begin{aligned}\label{eq:Salamonotation}
 \dbl \bigwedge^{(p,q)} \V^* \dbr \otimes_\R \C & := \bigwedge^{(p,q)} \V^* \oplus \bigwedge^{(q,p)} \V^* \, , & & p \neq q  \, , \\
    [ \bigwedge^{(p,p)} \V^* ] \otimes_\R \C & := \bigwedge^{(p,p)} \V^* \, , \\
\end{aligned}
\end{align}
This notation will be extended in the obvious way to $\bigodot{}^{(p,q)} \V^*$ and $\ydhook( \V^* )$.

Finally, we shall consider a Hermitian vector space $(\V, J, h)$ where $J$ is a complex structure compatible with a positive-definite symmetric bilinear form $h$, i.e.\ $J \circ h = - h \circ J$. Then $\V^{(1,0)} \cong (\V^{(0,1)})^*$ and $\V^{(0,1)} \cong (\V^{(1,0)})^*$ so that $\V^{(1,0)}$ and $\V^{(0,1)}$ are totally null with respect to $h$.  The Hermitian $2$-form on $\V$ is defined by $\omega = h \circ J$. For $pq \neq 0$, the subspace of $\bigwedge^{(p,q)} \V^*$ and $\bigodot{}^{(p,q)} \V^*$ consisting of all $(p,q)$-forms that are tracefree with respect to $\omega^{-1}$ or $h^{-1}$ will be denoted $\bigwedge^{(p,q)}_\circ \V^*$ and $\bigodot{}^{(p,q)}_\circ \V^*$ respectively. Note that $\bigwedge^{(1,1)} \V^*  \cong \u(m)$ and $\bigwedge^{(1,1)}_\circ \V^* \cong \su(m)$, where $\u(m)$ and $\su(m)$ are the Lie algebras of the unitary group $\U(m)$ and special unitary group $\SU(m)$ respectively.

\subsection{Linear algebra}
\subsubsection{Null structures}
Let $\wt{\V}$ be a $(2m+2)$-dimensional oriented complex vector space equipped with a non-degenerate symmetric bilinear form $\wt{g}$.  We introduce abstract indices following the convention of \cite{Fino2020}: minuscule Roman indices starting with the beginnning of the alphabet $a, b, c,\ldots$ will refer to elements of $\wt{\V}$ and its dual, and tensor products thereof, e.g.\ $v^a \in \wt{\V}$ and $\alpha_a{}^b \in \wt{\V}^* \otimes \wt{\V}$. Round brackets and squared brackets enclosing a group of indices will denote symmetrisation and skew-symmetrisation respectively, e.g.\
\begin{align*}
T^{(a b)} & = \frac{1}{2} \left( T^{a b} + T^{b a} \right) \, , & \beta_{[a b c]} & = \frac{1}{3!} \left( \beta_{a b c} - \beta_{a c b} + \beta_{b c a}  - \beta_{b a c} + \beta_{c a b} - \beta_{c b a} \right) \, .
\end{align*}
In particular, the symmetric bilinear form satisfies $\wt{g}_{a b} = \wt{g}_{(a b)}$, and together with its inverse $\wt{g}^{a b}$, will be used to raise and lower indices. The tracefree (symmetric) part of a tensor with respect to $\wt{g}$ will be adorned with a small circle, e.g.\ either as $T_{(a b)_\circ}$ or $\left( T_{a b} \right)_\circ$.

\begin{defn}\cite{Budinich1989,Kopczy'nski1992,Nurowski2002,Taghavi-Chabert2016}
A \emph{null structure} on $(\wt{\V}, \wt{g})$ is a maximal totally null (MTN) vector subspace of $\wt{\V}$, i.e.\ $\mbb{N} = \mbb{N}^\perp$. In other words,  $\wt{g}(v,w)=0$ for any $v, w \in \mbb{N}$, and $\mbb{N}$ has dimension $m+1$.
\end{defn}

A null structure $\mbb{N}$ on $(\wt{\V}, \wt{g})$ singles out the one-dimensional vector subspace $\bigwedge^{m+1} \Ann(\mbb{N})$ of $\bigwedge^{m+1} \wt{\V}^*$. Any element $\nu$ of $\bigwedge^{m+1} \Ann(\mbb{N})$ is then totally null, i.e.\ $\nu_{a a_1 \ldots a_m} \nu^{a}{}_{b_1 \ldots b_m}  = 0$. In particular, $\nu$ satisfies the following properties:
\begin{enumerate}
\item $\nu$ is simple, i.e.\ $\nu_{a_1 \ldots a_m [a_{m+1}} \nu_{b_1 \ldots b_{m+1}]}  = 0$.
\item $\nu$ is either self-dual or anti-self-dual, i.e.\ either $\star \nu = (\i) \nu$ or $\star \nu = -(\i) \nu$.
\end{enumerate}
Conversely, any self-dual or anti-self-dual simple $(m+1)$-form $\nu$ must be totally null, and thus defines the MTN vector subspace
\begin{align*}
\mbb{N} & = \left\{ v \in \wt{\V} : v \hook \nu = 0 \right\} \, .
\end{align*}
We shall therefore refer to a null structure $\mbb{N}$ as either self-dual or anti-self-dual depending on whether $\bigwedge^{m+1} \Ann(\mbb{N}) \subset \bigwedge^{m+1}_+  \wt{\V}^*$ or $\bigwedge^{m+1} \Ann(\mbb{N}) \subset \bigwedge^{m+1}_-  \wt{\V}^*$.

The space of all MTN vector subspaces of $\wt{\V}$, i.e.\ null structures, is a complex homogeneous space of complex dimension $\frac{1}{2}m(m+1)$, referred to as the \emph{isotropic Grassmannian} $\mathrm{Gr}_{m+1}(\wt{\V}, \wt{g})$ of $(\wt{\V}, \wt{g})$. This space splits into two disconnected components $\mathrm{Gr}^+_{m+1}(\wt{\V}, \wt{g})$ and $\mathrm{Gr}^-_{m+1}(\wt{\V}, \wt{g})$ according to whether their elements are self-dual or anti-self-dual. The complex Lie group $\SO(2m+2,\C)$ acts transitively on each of these components. 

\begin{rem}
Any complement of $\mbb{N}$ in $\wt{\V}$ must be dual to $\mbb{N}$ and totally null with respect to $\wt{g}$, and we shall write $\wt{\V} = \mbb{N}^* \oplus \mbb{N}$, 
bearing in mind that in general such a splitting is not canonical. For consistency with the notation introduced subsequently, we shall assume with no loss of generality that $\mbb{N}$ is self-dual.  In abstract index notation, elements of $\Ann(\mbb{N})$ will be adorned with lower Roman majuscule indices, and elements of $\mbb{N}^*$ with upper Roman majuscule indices, e.g.\ $\alpha_A \in \Ann(\mbb{N}) \cong \mbb{N}$ and $v^A \in \mbb{N}^* \cong \Ann(\mbb{N}^*)$. These indices will be immovable. We also introduce splitting operators $(\delta^a_A ,\delta^{a A})$, that is, projections $\delta^{a A} : \wt{\V}^* \rightarrow \mbb{N}^*$ and $\delta^a_A : \wt{\V}^* \rightarrow \mbb{N}$ that satisfy $\delta^a_A \delta_a^B = \delta_A^B$. These will also be used to inject elements of $\mbb{N}$ and $\mbb{N}^*$ into $\wt{\V}$.
\end{rem}

\subsubsection{Robinson structures}
Let $\V$ be a $(2m+2)$-dimensional real vector space equipped with a non-degenerate symmetric bilinear form $g$ of (Lorentzian) signature $(2m+1,1)$, i.e.\ $(+,+, \ldots, +, -)$. As is customary, we call $(\V, g)$ Minkowski space. The  abstract index notation introduced in the previous section will equally apply to $(\V,g)$.

Denote by ${}^\C \V$ the complexification of $\V$, and extend $g$ to a non-degenerate complex-valued symmetric bilinear form ${}^\C g$ on ${}^\C \V$. By abuse of notation, we shall often denote ${}^\C g$ by $g$. The complexification $({}^\C \V,{}^\C g)$ of $(\V,g)$ thus gives rise to the complex space $(\wt{\V},\wt{g})$ considered in the previous section, together with a reality condition. By extension, there is a well-defined notion of null structure on $(\V,g)$ via $({}^\C \V,{}^\C g)$. To make this idea more precise, we note that the complex conjugate $\overline{\mbb{N}}$ of a MTN vector subspace $\mbb{N}$ on $({}^\C \V, {}^\C g)$ is also MTN.
\begin{defn}[\cite{Kopczy'nski1992}]\label{def:real_index}
The \emph{real index} of a null structure $\mbb{N}$ on $({}^\C \V,{}^\C g)$ is the complex dimension of the intersection of $\mbb{N}$ and $\overline{\mbb{N}}$.
\end{defn}

\begin{defn}
A \emph{Robinson structure} on Minkowski space $(\V,g)$ of dimension $2m+2$ is a null structure $\mbb{N}$ of real index one on $({}^\C \V , {}^\C g)$. We shall denote it by the pair $(\mbb{N}, \K)$ where
\begin{enumerate}
\item $\mbb{N}$ is an MTN vector subspace of ${}^\C \V$ of real index one,
\item $\K$ is the real null line $\mbb{N} \cap \V$.
\end{enumerate}
\end{defn}
With this second condition, we have that  ${}^\C \K = \mbb{N} \cap \overline{\mbb{N}}$ and $\mbb{N} + \overline{\mbb{N}} = {}^\C \K^\perp$.

It turns out that the real Lie group $\SO(2m+1,1)$ also acts transitively on each of the connected spaces of MTN vector spaces of $({}^\C \V, {}^\C g)$. In other words:
\begin{lem}[\cite{Kopczy'nski1992}]\label{lem:real_index}
Let $(\V,g)$ be Minkowski space of dimension $2m+2$. A null structure on $({}^\C \V, {}^\C g)$ always has real index one, and hence is a Robinson structure.
\end{lem}

\begin{rem}\label{rem:Hodge_duality}
Note that $(\mbb{N}, \K)$ and $(\overline{\mbb{N}}, \K)$ define the same Robinson structure. Their Hodge duality is the same when $m$ is even, but opposite when $m$ is odd. We shall say that a Robinson structure $(\mbb{N},\K)$ is (anti-)self-dual if $\mbb{N}$ is (anti-)self-dual. This entails of course a preference of $\mbb{N}$ over $\overline{\mbb{N}}$ when $m$ is odd, but there is no ambiguity when $m$ is even.

Any element $\nu$ of $\bigwedge^{m+1} \Ann(\mbb{N})$ will be referred to as a \emph{complex Robinson $(m+1)$-form}. When $m$ is even, if $\nu$ is a self-dual, so are its complex conjugate $\bar{\nu}$ and the real $(m+1)$-form $\nu + \overline{\nu}$.
\end{rem}

\subsubsection{Robinson structures and optical structures}\label{sec:Robinson-optical}
It is clear that a Robinson structure $(\mbb{N},\K)$ on $(\V, g)$ determines in particular an optical structure, namely $\K$, in the sense of \cite{Fino2020}. We therefore have a filtration of vector subspaces
\begin{align}\label{eq:sim-filtration}
 \{ 0 \}  \subset \K \subset \K^\perp \subset \V \, ,
\end{align}
and the screen space $\mbb{H}_\K = \K^\perp/\K$  inherits a positive-definite symmetric bilinear form $h$ given by
\begin{align*}
h( v + \K , w + \K ) & := g( v,w ) \, , & \mbox{for any $v, w \in \Gamma(\K^\perp)$.}
\end{align*}
Any element of $\K$ will be referred to as an \emph{optical vector}, and any element of $\Ann(\K)$ as an \emph{optical $1$-form}.

Now, define an endomorphism $J$ of $\mbb{H}_\K$ and its complexification ${}^\C \mbb{H}_\K$ by
\begin{align*}
J(v + {}^\C \K)&  = - \i v + {}^\C \K \, , & \mbox{for any $v \in \mbb{N}$,} \\
J(v + {}^\C \K) & =  \i v + {}^\C \K \, , & \mbox{for any $v \in \overline{\mbb{N}}$.} 
\end{align*}
Then $J$ is a complex structure on $\mbb{H}_\K$, and ${}^\C \mbb{H}_\K$ splits into the eigenspaces of $J$, i.e.\
\begin{align}\label{eq:splitH_K}
{}^\C \mbb{H}_\K = \mbb{H}_\K^{(1,0)} \oplus \mbb{H}_\K^{(0,1)} \, ,
\end{align}
where $\mbb{H}_\K^{(1,0)} := \overline{\mbb{N}} / {}^\C \K$ and $\mbb{H}_\K^{(0,1)} := \mbb{N} / {}^\C \K$. These can be shown to be maximal totally null with respect to the bilinear form $h$ on ${}^\C \mbb{H}_\K = (\mbb{N} + \overline{\mbb{N}})/(\mbb{N} \cap \overline{\mbb{N}})$, and thus $J$ is compatible with $h$, i.e.\ $J$ is Hermitian \cite{Nurowski2002}.

Conversely, suppose that $(\V,g)$ is equipped with an optical structure $\K$ together with a complex structure $J$ on the screen space $\mbb{H}_\K$ compatible with $h$. Define
\begin{align*}
\mbb{N} = \left\{ v \in {}^\C \K^\perp : J(v + {}^\C \K) = - \i v + {}^\C \K  \right\} \, .
\end{align*}
Then $\mbb{N}$ has dimension $m+1$, and is totally null. Indeed, for any $v,w \in \mbb{N}$, we have
\begin{align*}
 g(v,w)	& = h(v + {}^\C \K, w + {}^\C \K) \\
		& = h( \i J(v + {}^\C \K) , \i J( w + {}^\C \K)) \\
		&= - h ( v + {}^\C \K , w + {}^\C \K) \\
		& = - g (v, w) \, ,
\end{align*}
since $J \circ h = - h \circ J$, and thus $g(v,w)=0$. The complex conjugate $\overline{\mbb{N}}$ is defined analogously.

In abstract index notation, elements of $\mbb{H}_\K$ and of its dual, and tensor product thereof, will be adorned with minuscule Roman indices starting from the middle of the alphabet $i,j,k, \ldots$. In particular, the screen space inner product and its inverse will be expressed as $h_{i j}$ and $h^{i j}$ respectively, and will be used to lower and raise this type of indices. The complex structure and the Hermitian $2$-form will take the form $J_i{}^j$ and $\omega_{i j} := J_i{}^k h_{k j}$ respectively. As before, symmetrisation and skew-symmetrisation will be denoted by round and squared brackets around indices respectively, and the tracefree part of a tensor with respect to $h_{i j}$ will be adorned by a small circle.

We shall use upper and lower Greek indices to denote elements of $\mbb{H}_{\K}^{(1,0)}$ and $(\mbb{H}_{\K}^{(1,0)})^*$ respectively, and upper and lower overlined Greek indices to denote elements of $\mbb{H}_{\K}^{(0,1)}$ and $(\mbb{H}_{\K}^{(0,1)})^*$ respectively, e.g.\ $v^{\alpha} \in \mbb{H}_{\K}^{(1,0)}$ and $\alpha_{\bar{\beta}} \in (\mbb{H}_{\K}^{(0,1)})^*$. As usual, symmetrisation and skew-symmetrisation will be denoted by round brackets and squared brackets respectively. The Hermitian form on ${}^\C \mbb{H}_{\K}$ will then be expressed as $h_{\alpha \bar{\beta}}$ and its inverse by $h^{\alpha \bar{\beta}}$, which will be used to convert indices, i.e.\ $v_{\bar{\beta}} = v^{\alpha} h_{\alpha \bar{\beta}}$ for any $v^{\alpha} \in \mbb{H}_{\K}^{(1,0)}$. The totally tracefree part of a mixed tensor $T_{\alpha \beta \bar{\gamma}}$, say, with respect to $h^{\alpha \bar{\beta}}$ will be denoted $\left( T_{\alpha \beta \bar{\gamma}} \right)_\circ$.

We shall also introduce for convenience splitting operators $(\delta^i_\alpha, \delta^i_{\bar{\alpha}})$ on ${}^\C \mbb{H}_{\K}$, that is projections $\delta^i_\alpha : {}^\C \mbb{H}_{\K}^* \rightarrow (\mbb{H}_{\K}^{(1,0)})^*$ and $\delta^i_{\bar{\alpha}} :  {}^\C \mbb{H}_{\K}^* \rightarrow (\mbb{H}_{\K}^{(0,1)})^*$ that satisfy
\begin{align*}
\delta^i_\alpha J_i{}^j & = \i \delta^j_\alpha \, , & \delta^i_{\bar{\alpha}} J_i{}^j & = - \i \delta^j_{\bar{\alpha}} \, , \\
h_{i j} \delta^i_{\alpha} \delta^i_{\bar{\beta}} & = h_{\alpha \bar{\beta}} \, , & h_{i j} \delta^i_{\alpha} \delta^i_{\beta} & = 0 \, .
\end{align*}
These will also be used to inject $\mbb{H}_{\K}^{(1,0)}$ into ${}^\C \mbb{H}_{\K}$, and so on. Their dual versions $(\delta_i^\alpha, \delta_i^{\bar{\alpha}})$ can be obtained by raising and lowering the indices with $h_{i j}$ and $h^{\alpha \bar{\beta}}$. Thus, in particular, we can express $h_{i j}$ and $\omega_{i j} = J_{i}{}^{k} h_{k j}$ as
\begin{align*}
\omega_{i j} & = 2 \i h_{\alpha \bar{\beta}} \delta_{[i}^{\alpha} \delta_{j]}^{\bar{\beta}} \, , & h_{i j} & = 2 h_{\alpha \bar{\beta}} \delta_{(i}^{\alpha} \delta_{j)}^{\bar{\beta}} \, .
\end{align*}

\begin{rem}
From the discussion above, it is also conceptually useful to start with an optical structure $\K$ on $(\V, g)$, and declare a null or Robinson structure on $(\V,g)$ be \emph{compatible} with $\K$ if $\K = \V \cap \mbb{N}$. In dimension four, there is a single Robinson structure (up to complex conjugation) compatible with an optical structure, but this is not true in higher dimensions -- see Remark \ref{rem:sp_Rob}.
\end{rem}

\subsubsection{Splitting}
As before let $(\mbb{N}, \K)$ be a Robinson structure on Minkowski space $(\V,g)$. Any choice of splitting
\begin{align}\label{eq:MTN_split}
{}^\C \V & = \mbb{N}^* \oplus \mbb{N} \, ,
\end{align}
for some choice of dual $\mbb{N}^*$ induces a splitting of the filtration \eqref{eq:sim-filtration} as
\begin{align}\label{eq:V_opt_split}
\V & = \mbb{L} \oplus \mbb{H}_{\K,\mbb{L}} \oplus \K \, ,
\end{align}
where
\begin{align*}
\K & = \mbb{N} \cap \V = \overline{\mbb{N}} \cap \V \, , & \mbb{L} & := \mbb{N}^* \cap \V = \overline{\mbb{N}}^* \cap \V \, , \\
\mbb{H}_{\K,\mbb{L}} & := \K^\perp \cap \mbb{L}^\perp \, .
\end{align*}
In particular, $\mbb{L}$ is a null line dual to $\K$. Note that $\mbb{H}_{\K,\mbb{L}}$ is isomorphic to the screen space $\mbb{H}_\K$, but this isomorphism depends on the choice of $\mbb{L}$. Further, in complete analogy with \eqref{eq:splitH_K}, we obtain the splitting
\begin{align}\label{eq:splitH_KL}
{}^\C \mbb{H}_{\K , \mbb{L}} = \mbb{H}_{\K , \mbb{L}}^{(1,0)} \oplus \mbb{H}_{\K , \mbb{L}}^{(0,1)} \, ,
\end{align}
where
\begin{align*}
\mbb{H}_{\K , \mbb{L}}^{(1,0)} & :=  \overline{\mbb{N}} \cap {}^\C \mbb{H}_{\K , \mbb{L}}  = \overline{\mbb{N}} \cap \mbb{N}^* \, , & 
\mbb{H}_{\K , \mbb{L}}^{(0,1)} & := \mbb{N} \cap {}^\C \mbb{H}_{\K , \mbb{L}} = \mbb{N} \cap \overline{\mbb{N}}^* \, .
\end{align*}
We note that $\mbb{H}_\K^{(1,0)} \cong \mbb{H}_{\K , \mbb{L}}^{(1,0)}$ and $\mbb{H}_\K^{(0,1)} \cong \mbb{H}_{\K , \mbb{L}}^{(0,1)}$, but these isomorphisms depend on the choice of $\mbb{N}^*$

We introduce splitting operators $(\ell^a, \delta^a_i, k^a)$ with dual $(\kappa_a, \delta_a^i, \lambda_a)$  adapted to \eqref{eq:V_opt_split}, where $k^a$ and $\ell^a$ are elements  in $\K$ and in $\mbb{L}$ respectively such that $g_{a b} k^a \ell^b = 1$, and $\delta^a_i$ projects from $\V^*$ to $\mbb{H}^*_{\K,\mbb{L}}$ and satisfies $g_{a b} \delta^a_i \delta^b_j = h_{i j}$. Any change of splitting which preserves $k^a$ induces the transformations
\begin{align}\label{eq:basis-change}
 k^a & \mapsto k^a \, , & \delta^a_i & \mapsto \delta^a_i + \phi_i k^a \, , & \ell^a & \mapsto \ell^a - \phi^i \delta^a_i - \frac{1}{2} \phi^i \phi_i k^a \, ,
\end{align}
for some $\phi_i$ in $(\R^{2m})^*$, and similarly for their duals.

Finally, combining the splitting operators $(\ell^a, \delta^a_i, k^a)$ and $(\delta^i_\alpha , \delta^i_{\bar{\alpha}})$ yield new ones $(\ell^a, \delta^a_\alpha, \delta^a_{\bar{\alpha}}, k^a)$ where $\delta^a_{\alpha} := \delta^a_i \delta^i_\alpha$ projects from ${}^\C \V^*$ to $(\mbb{H}_{\K, \mbb{L}}^{(1,0)})^*$. We naturally obtain dual splitting operators $(\kappa_a, \delta_a^\alpha, \delta_a^{\bar{\alpha}}, \lambda_a)$. In terms of the splitting operators $(\delta^a_A, \delta_a^A)$, we have $\lambda_a \delta^a_A = \delta_a^{\bar{\alpha}} \delta^a_A = 0$ and $k^a \delta_a^A = \delta^a_{\bar{\alpha}} \delta_a^A = 0$, and we may define further splitting operators $\delta^A_{\alpha} := \delta^A_a \delta^a_{\alpha}$ projecting from $\mbb{N}$ to $(\mbb{H}_{\K, \mbb{L}}^{(1,0)})^*$, and $\delta_{A}^{\alpha}:= \delta_A^a \delta_{a}^{\alpha}$ from $\mbb{N}^*$ to $(\mbb{H}_{\K, \mbb{L}}^{(1,0)})$. One may similarly define $\delta^{A \bar{\alpha}}$ and $\delta_{A \bar{\alpha}}$.

\subsubsection{Robinson $3$-forms}\label{sec:Robinson-3}
For any choice of splitting operators $(\kappa_a, \delta^i_a, \ell_a)$, and recalling that $\omega_{i j}$ is the Hermitian form on $\mbb{H}_{\K}$, we set $\omega_{a b} = \omega_{i j}  \delta^i_a \delta^j_b$ and define
\begin{align}\label{eq:Rob3}
\rho_{a b c} & := 3 \, \kappa_{[a} \omega_{b c]}\, .
\end{align}
By Lemma 3.1 of \cite{Fino2020}, the definition of $\rho_{a b c}$ depends only on the choice of optical $1$-form $\kappa_a$, and not on the choice of $\delta^a_i$ and $\lambda_a$. One can check that the $3$-form $\rho_{abc}$ satisfies
\begin{align}\label{eq:3-form-sq2}
\rho _{ab} \, {}^e \rho _{cde} & = - 4 \, \kappa_{[a} g_{b][c} \kappa_{d]} \, .
\end{align}
We shall refer to such a $3$-form as a \emph{Robinson $3$-form (associated to the optical $1$-form $\kappa_a$)}.

Conversely, let $\rho_{a b c}$ be a $3$-form that satisfies the algebraic property \eqref{eq:3-form-sq2} for some null $1$-form $\kappa_a$. Then, one can check that $\kappa_{[a} \rho_{b c d]} =0$ and $k^a \rho_{a b c} =0$. To prove $\kappa_{[a} \rho_{b c d]} =0$, we skew \eqref{eq:3-form-sq2} with $\kappa_a$ to get $\kappa_{[a} \rho _{b c]} \, {}^f \rho _{d e f} =  0$. Now contracting with $\rho^{e}{}_{g h}$ and using $\eqref{eq:3-form-sq2}$ again yields $\kappa_{[a} \rho _{b c] [d} \kappa_{e]} =  0$, and the result follows by skewing over the first four indices. That $k^a \rho_{a b c} =0$ can be proved in a similar fashion. Hence, we can apply Lemma 3.1 of \cite{Fino2020}, and setting $\omega_{i j} = \ell^a \delta^b_i \delta^c_j \rho_{a b c}$ for any choice of splitting operators $(\ell^a, \delta^a_i, k^a)$ where $\kappa_a \ell^b = 1$, we see that $\omega_{i j}$ is the required Hermitian form on $\mbb{H}_\K$ up to sign. This sign can be fixed so that if $k^a =g^{a b} \kappa_b$,
\begin{align*}
v^c \rho_{c}{}^{a b} & = - 2 \i v^{[a} k^{b]} \, , & \mbox{for any $v^a \in \mbb{N}$,} \\
 w^c \rho_{c}{}^{a b} & = 2 \i w^{[a} k^{b]} \, , & \mbox{$w^a \in \overline{\mbb{N}}$.}
\end{align*}

\begin{rem}
In low dimensions, we note the following:
\begin{itemize}
\item In dimension four, $\rho {_{bcd}}$ is the Hodge dual of $\kappa_a$, i.e.\ $\kappa = \star \rho$. This reflects the fact that an optical structure is equivalent to a Robinson structure.
\item In dimension six, the $3$-form $\rho_{a b c}$ defined above can be either self-dual or anti-self-dual under Hodge duality, consistent with the fact that the complex conjugate $\overline{\mbb{N}}$ of a MTN vector space $\mbb{N}$ has the same Hodge duality as $\mbb{N}$ in this case -- see Remark \ref{rem:Hodge_duality}.
\end{itemize}
\end{rem}

\subsubsection{Robinson spinors}\label{sec-spinors}
We proceed to describe a Robinson structure in terms of spinors following the treatment of \cite{Cartan1967,Penrose1986,Budinich1988,Budinich1989,Kopczy'nski1992,Lawson1989,Taghavi-Chabert2016}.

We first consider a $(2m+2)$-dimensional complex vector space $(\wt{\V}, \wt{g})$ equipped with a non-degenerate symmetric bilinear form. The double cover $\Spin(2m+2,\C)$ of $\SO(2m+2,\C)$ allows us to define the spinor representation $\mbb{S}$ of $(\wt{\V}, \wt{g})$, which splits into a direct sum of two $2^{m}$-dimensional irreducible chiral spin spaces $\mbb{S}_+$ and $\mbb{S}_-$, the spaces of spinors of positive and negative chiralities respectively. Following the notation of \cite{TaghaviChabert2017}, elements of $\mbb{S}_+$ and $\mbb{S}_-$ will be adorned with bold primed and unprimed majuscule Roman indices respectively, e.g. $\alpha^{\mbf{A}'} \in \mbb{S}_+$ and $\beta^{\mbf{A}} \in \mbb{S}_-$, and similarly for the dual spin spaces $\mbb{S}_\pm^*$ with lower indices. The spin space $\mbb{S}$ is also equipped with $\Spin(2m+2,\C)$-invariant bilinear forms, which allow the following identifications:
\begin{align}\label{arr:spin-inner-prod}
\begin{array}{|c|c|}
\hline
\mbox{$m$ odd} &  \mbox{$m$ even} \\
\hline
\hline
\mbb{S}_\pm \cong \mbb{S}_\pm^*  &  \mbb{S}_\pm \cong \mbb{S}_\mp^* \\
\hline
\end{array}
\end{align}

The \emph{Clifford action} of $\wt{\V}$ on $\mbb{S}$ is effected by means of the \emph{van der Waerden symbols} $\gamma_{a \mbf{A}}{}^{\mbf{B}'}$ and $\gamma_{a \mbf{A}'}{}^{\mbf{B}}$, injective maps from $\V$ to the space of homomorphisms $\Hom(\mbb{S}_\pm , \mbb{S}_\mp)$. These satisfy the Clifford property
\begin{align}\label{eq:Clifford_id}
 \gamma_{(a \mbf{A}'}{}^{\mbf{C}} \gamma_{b) \mbf{C}}{}^{\mbf{B}'} & = g_{a b} \delta_{\mbf{A}'}^{\mbf{B}'} \, , & \gamma_{(a _\mbf{A}}{}^{\mbf{C}'} \gamma_{b) \mbf{C}'}{}^\mbf{B} & = g_{a b} \delta_{\mbf{A}}^{\mbf{B}} \, ,
\end{align}
where $\delta_{\mbf{A}'}^{\mbf{B}'}$ and $\delta_{\mbf{A}}^{\mbf{B}}$ denote the identity elements on $\mbb{S}_+$ and $\mbb{S}_-$ respectively. Let $\nu^{\mbf{A}'}$ be a spinor, and consider the linear map $\nu_a^{\mbf{A}} := \nu^{\mbf{B}'} \gamma_{a \mbf{B}'}{}^{\mbf{A}} : \wt{\V} \rightarrow \mbb{S}_-$. Denote by $\mbb{N}$ the kernel of $\nu_a^{\mbf{A}}$. By \eqref{eq:Clifford_id}, $\mbb{N}$ must be totally null. We say that $\nu^{\mbf{A}'}$ is \emph{pure} if $\mbb{N}$ has maximal dimension $m+1$, i.e.\ $\mbb{N}$ is a null structure on $(\wt{\V}, \wt{g})$. Any spinor proportional to $\nu^{\mbf{A}'}$ defines the same null structure. More generally, Cartan showed \cite{Cartan1967} that there is a one-to-one correspondence between  null structures on $(\wt{\V}, \wt{g})$ and pure spinors up to scale. Further, self-dual null structures correspond to pure spinors of positive chirality, and anti-self-dual null structures to pure spinors of negative chirality. The components of a pure spinor are algebraically constrained \cite{Cartan1967}. Indeed, a spinor $\nu^{\mbf{A}'}$ is pure if and only if it satisfies the \emph{purity condition} \cite{Penrose1986,Hughston1988,Taghavi-Chabert2012}
\begin{align}\label{eq:pure}
 \nu_a^{\mbf{A}}  \nu^{a \mbf{B}} = 0 \, .
\end{align}
When $m=0,1,2$, conditions \eqref{eq:pure} is vacuous, i.e.\ all spinors are pure. A self-dual null structure $\mbb{N}$ thus singles out a one-dimensional vector subspace $\mbb{S}_+^\mbb{N}$ of $\mbb{S}_+$, any element $\nu^{\mbf{A}'}$ of which satisfies \eqref{eq:pure}.

Note that the image of $\nu_a^{\mbf{A}}$ is isomorphic to $\wt{\V}/\mbb{N}$, and thus to any choice of complement $\mbb{N}^*$ of $\mbb{N}$ in $\wt{\V}$. More precisely, we have injective linear maps $\delta^{\mbf{A}}_{A} := \delta^{a}_{A} \nu^{\mbf{A}}_{a}$ from $\mbb{N}^*$ to $\mbb{S}_-$, and $\delta^{A}_{\mbf{A}}$ from $\mbb{N}$ to $\mbb{S}_-^*$ such that $\delta^{\mbf{C}}_{A} \delta_{\mbf{C}}^{B} = \delta_{A}^{B}$. Hence, by means of these, we can express a $1$-form $\alpha_a$ in $\Ann(\mbb{N})$ in the form $\alpha_{a} = \nu_{a}^{\mbf{A}} \alpha_{\mbf{A}}$ where $\alpha_{\mbf{A}} = \delta^{B}_{\mbf{A}} \delta^{c}_{B} \alpha_{c}$.

The van der Waerden symbols generate the \emph{Clifford algebra} $\mc{C}\ell(\wt{\V},  \wt{g})$ of $(\wt{\V},  \wt{g})$, which, by virtue of \eqref{eq:Clifford_id}, is isomorphic to the exterior algebra $\bigwedge^\bullet \wt{\V} \cong \bigwedge^\bullet \wt{\V}^*$ as a vector space. The Clifford algebra is also a matrix algebra isomorphic to the space of endomorphisms of $\mbb{S}$. These two properties allow us to construct invariant bilinear forms on $\mbb{S}$ with values in $\bigwedge^k \wt{\V}^*$ for $k=0, \ldots ,m+1$. The case $k=0$ yields the spin inner products implicitly used in \eqref{arr:spin-inner-prod}. Depending on the values of $m$ and $k$, these forms restrict to non-degenerate forms on either $\mbb{S}_\pm \times \mbb{S}_\mp$ or $\mbb{S}_\pm \times \mbb{S}_\pm$. Of relevance to the present article are the cases $k=1,3, m+1$. For $k = 1,3$, we have
\begin{align}\label{arr:bi-form}
\begin{array}{|c|c|}
\hline
\mbox{$m$ odd} &  \mbox{$m$ even} \\
\hline
\hline
\gamma_{a \mbf{A}' \mbf{B}} : \mbb{S}_+ \times \mbb{S}_- \rightarrow  \wt{\V}^*  &  \gamma_{a \mbf{A} \mbf{B}}  :  \mbb{S}_+ \times \mbb{S}_+ \rightarrow  \wt{\V}^* \\
\gamma_{a \mbf{A} \mbf{B}'} : \mbb{S}_- \times \mbb{S}_+ \rightarrow  \wt{\V}^*  &  \gamma_{a \mbf{A}' \mbf{B}'}  :  \mbb{S}_- \times \mbb{S}_- \rightarrow  \wt{\V}^* \\
\hline
\gamma_{a b c \mbf{A}' \mbf{B}} : \mbb{S}_+ \times \mbb{S}_- \rightarrow \bigwedge^3  \wt{\V}^* & \gamma_{a b c \mbf{A}' \mbf{B}'} : \mbb{S}_+ \times \mbb{S}_+ \rightarrow \bigwedge^3  \wt{\V}^* \\
\gamma_{a b c \mbf{A} \mbf{B}'} : \mbb{S}_+ \times \mbb{S}_- \rightarrow \bigwedge^3  \wt{\V}^* & \gamma_{a b c \mbf{A} \mbf{B}} : \mbb{S}_+ \times \mbb{S}_+ \rightarrow \bigwedge^3  \wt{\V}^* \\
\hline
\end{array}
\end{align}
For $k=m+1$, regardless of whether $m$ is odd or even, we have the following two bilinear forms
\begin{align}\label{eq:Clm+1HD}
\begin{aligned}
\gamma_{a_0 \ldots  a_{m} \mbf{A}' \mbf{B}'} & : \mbb{S}_+ \times \mbb{S}_+ \rightarrow \bigwedge^{m+1}_+  \wt{\V}^*  \, , \\
\gamma_{a_0 \ldots  a_{m} \mbf{A} \mbf{B}} & : \mbb{S}_- \times \mbb{S}_- \rightarrow \bigwedge^{m+1}_-  \wt{\V}^* \, ,
\end{aligned}
\end{align}
where we recall $\bigwedge^{m+1}_{\pm}  \wt{\V}^*$ are the self-dual and anti-self-dual parts of $\bigwedge^{m+1}  \wt{\V}^*$.

For specificity, we assume that $\mbb{N}$ is a self-dual null structure, the argument being the same for an anti-self-dual one. Then the restriction of the first display of \eqref{eq:Clm+1HD} to $\mbb{S}_+^\mbb{N}$ yields the isomorphism $\mbb{S}_+^\mbb{N} \otimes \mbb{S}_+^\mbb{N} \cong \bigwedge^{m+1} \Ann(\mbb{N})$. We can thus think of $\mbb{S}_+^\mbb{N}$ as a square root of $\bigwedge^{m+1} \Ann(\mbb{N})$.

At this stage, we return to the real picture by consider $(2m+2)$-dimensional Minkowski space $(\V, g)$. We can then apply all the facts outlined above to the complexification $({}^\C \V , {}^\C  g)$ of $(\V, g)$. In addition,  the real structure $\bar{}$ on ${}^\C \V$ preserving $\V$ induces an antilinear map on $\mbb{S}$, which interchanges the chiralities of spinors when $m$ is odd, and preserve them when $m$ is even. The image of a spinor under this antilinear map is referred to as the \emph{charge conjugate} of that spinor. Thus, the charge conjugate of a spinor $\nu^{\mbf{A}'} \in \mbb{S}_+$ will be denoted  $\overline{\nu}^{\mbf{A}}$ when $m$ is odd, and  $\overline{\nu}^{\mbf{A}'} \in \mbb{S}_+$ when $m$ is even.\footnote{Using the isomorphisms \eqref{arr:spin-inner-prod}, the charge conjugate of a spinor $\nu^{\mbf{A}'} \in \mbb{S}_+$ can always be identified with $\overline{\nu}_{\mbf{A}} \in \mbb{S}^*_-$ regardless of the parity of $m$.} Moreover, if $\nu^{\mbf{A}'}$ is pure so is its charge conjugate. We define the \emph{real index} of a pure spinor to be the real index of its associated null structure \cite{Kopczy'nski1992}. In particular, if $g$ has Lorentzian signature, all pure spinors have real index one. A Robinson structure $(\mbb{N},\K)$ can thus be defined by a pure spinor up to scale. We shall refer to any such spinor as a \emph{Robinson spinor}.

The pure spinor $\nu^{\mbf{A}'}$ and its charge conjugate can be paired using the spinor bilinear forms to obtain invariants of the Robinson structure as shown in \cite{Cartan1967,Kopczy'nski1992}. These are listed below:
\begin{align}\label{arr:rob_spin_inv}
\begin{array}{|c|c|}
\hline
\mbox{$m$ odd} &  \mbox{$m$ even} \\
\hline
\hline
\kappa_a = \gamma_{a \mbf{A}' \mbf{B}} \nu^{\mbf{A}'} \overline{\nu}^{\mbf{B}} & \kappa_a = \gamma_{a \mbf{A}' \mbf{B}'}\nu^{\mbf{A}'} \overline{\nu}^{\mbf{B}'} \\
 \rho_{a b c} =  \i \gamma_{a b c \mbf{A}' \mbf{B}} \nu^{\mbf{A}'} \overline{\nu}^{\mbf{B}} & \rho_{a b c} = \i \gamma_{a b c \mbf{A}' \mbf{B}'} \nu^{\mbf{A}'} \overline{\nu}^{\mbf{B}'} \\
 \hline
\nu_{a_0 \ldots a_m} = \gamma_{a_0 \ldots a_m \mbf{A}' \mbf{B}'} \nu^{\mbf{A}'} \nu^{\mbf{B}'} & \nu_{a_0 \ldots a_m} = \gamma_{a_0 \ldots a_m \mbf{A}' \mbf{B}'} \nu^{\mbf{A}'} \nu^{\mbf{B}'} \\
\overline{\nu}_{a_0 \ldots a_m}= \gamma_{a_0 \ldots a_m \mbf{A} \mbf{B}} \overline{\nu}^{\mbf{A}} \overline{\nu}^{\mbf{B}} & \overline{\nu}_{a_0 \ldots a_m} = \gamma_{a_0 \ldots a_m \mbf{A}' \mbf{B}'} \overline{\nu}^{\mbf{A}'} \overline{\nu}^{\mbf{B}'} \\
\hline
\end{array}
\end{align}
With these definitions, $\rho_{a b c}$ is the Robinson $3$-form associated to the optical $1$-form $\kappa_a$, i.e.\ these satisfy \eqref{eq:3-form-sq2}, and $\nu_{a_0 \ldots a_m}$ is a complex Robinson $(m+1)$-form. Forms of odd degrees can be constructed in a similar way. Details can be found in \cite{Kopczy'nski1992}.

\begin{rem}\label{rem:Robinson_scale}
Clearly, any given optical $1$-form or Robinson $3$-form is defined by a Robinson spinor up to a phase, and a given complex Robinson $(m+1)$-form by a Robinson spinor up to a sign. To see this we note that, under the transformation $\nu^{\mbf{A}'} \mapsto z \nu^{\mbf{A}'}$, for any non-zero complex number $z$, the charge conjugate of $\nu^{\mbf{A}'}$ gets multiplied by $\bar{z}$, and the forms defined in \eqref{arr:rob_spin_inv} transform as
\begin{align*}
\kappa_a & \mapsto r \kappa_a \, , &
\rho_{a b c} & \mapsto r \rho_{a b c} \, , &
\nu_{a_0 \ldots a_m} & \mapsto z^2 \nu_{a_0 \ldots a_m} \, , &
\end{align*}
where $|z| = r \in \R_{>0}$.
\end{rem}

\begin{rem}
Any choice of splitting \eqref{eq:MTN_split} of ${}^\C \V$ is equivalent to choosing a one-dimensional subspace of $\mbb{S}_+^*$ consisting of pure spinors dual to $\mbb{S}_+^\mbb{N}$. Elements thereof annihilate $\mbb{N}^*$.
\end{rem}

\subsubsection{Characterisations of Robinson structures}
We summarise the findings of the previous section in the following proposition.
\begin{prop}\label{prop:char-Robinson}
Let $(\V,g)$ be Minkowski space of dimension $2m+2$. The following statements are equivalent.
\begin{enumerate}
\item $(\V,g)$ is equipped with a Robinson structure $(\mbb{N},\K)$.
\item $(\V,g)$ is equipped with a totally null complex $(m+1)$-form.
\item $(\V,g)$ is equipped with an optical structure $\K$ whose screenspace $ \mbb{H}_\K= \K^\perp/\K$ is endowed with a complex structure $J_i{}^j$ compatible with the induced metric $h_{i j}$.\
\item $(\V,g)$ admits a $1$-form $\kappa_a$ and $3$-form $\rho_{abc}$ satisfying
\begin{align*}
\rho _{ab} \, {}^e \rho _{cde} & = - 4 \kappa_{[a} g_{b][c} \kappa_{d]} \, .
\end{align*}
\item $(\V,g)$ admits a pure spinor of real index $1$.
\end{enumerate}
\end{prop}

\begin{rem}\label{rem:Robinson_scale2}
In Proposition \ref{prop:char-Robinson}, the $1$-form $\kappa_a$ and the $3$-form $\rho_{a b c}$ are defined up to an overall real factor, while the pure spinor is defined up to an overall complex factor as explained in Remark \ref{rem:Robinson_scale}.
\end{rem}

\subsection{The stabiliser of a Robinson structure}
There are two approaches to describe the stabiliser of a Robinson structure $(\mbb{N},\K)$:
\begin{enumerate}
\item From its definition, it suffices to consider the respective stabilisers $R$ and $\overline{R}$ of $\mbb{N}$ and $\overline{\mbb{N}}$ in $\SO(2m+2,\C)$. The stabiliser of a Robinson structure is then the intersection $R \cap \overline{R} \cap \SO^0(2m+1,1)$.
\item We characterise a Robinson structure as an optical structure together with a Hermitian structure on the screen space to derive its stabiliser $Q$ as a closed Lie subgroup of the stabiliser $P$ of $\K$ in $G =\SO^0(2m+1,1)$.
\end{enumerate}
For the present purpose, it will be more useful to use the second approach. We shall assume that $\K$ is oriented so that the stabiliser of $\K$ together with its orientation is 
\begin{align*}
P & =\Sim^0(2m) = \CO^0(2m)\ltimes (\R^{2m})^* \ =\  (\R_{>0} \times \SO(2m))\ltimes (\R^{2m})^* \, .
\end{align*}
Note that $\CO^0(2m)$ acts on the screen space $\mbb{H}_\K = \K^\perp/\K$ as $\SO(2m)$ does, that is, $\R_{>0}$ acts trivially on $\mbb{H}_\K$. The nilpotent part of $P$ will be denoted $P_+$. Choose a semi-null frame $(e_0 , e_1, \ldots , e_n, e_{n+1})$ and dual coframe $(\theta^0, \theta^1, \ldots, \theta^n, \theta^{n+1})$ where by convention $K = \mathrm{span}(e_{n+1})$ and $K^\perp = \mathrm{Ann}(\theta^0)$. Then, since $Q \subset P$ is required to stabilise in addition a Hermitian structure on $\mbb{H}_\K$, we obtain that
\begin{align*}
Q & = (\R_{>0}\times \U(m))\ltimes (\R^{2m})^* \, , \\
&=
\left\{(\e^\varphi,\psi,\phi)
:=
\begin{pmatrix}
\e^{\varphi} & 0 & 0 \\
 -\e^\varphi \iota(\psi)\phi^\top & \iota(\psi) & 0\\
-\tfrac{\e^\varphi}{2} \phi \phi^\top  & \phi  & \e^{-\varphi}
\end{pmatrix}\mid \begin{array}{l} \varphi \in \R \, ,\\
\psi \in \U(m) \, ,\\
\phi \in (\R^{2m})^*\end{array}
 \right\} \, ,
\end{align*}
where we have used the standard embedding $\iota : \U(m) \rightarrow \SO(2m)$. The reductive part $\R_{>0} \times \U(m)$ of $Q$ will be denoted $Q_0$. Clearly, $P_+$ is also the nilpotent part of $Q$.

To describe the Lie algebra $\mf{q}$ of $Q$, we shall refer to the notation already introduced in \cite{Fino2020}.
Setting $\V^1=\V_1=\K$, $\V^0=\K^\perp$, we have a filtration of $P$-modules
\begin{align}\label{eq:sim-filtration2}
 \{ 0 \} =: \V^2 \subset \V^1 \subset \V^0 \subset \V^{-1} := \V \, ,
\end{align}
which we shall conveniently split into a direct sum of $P_0$-modules,
\begin{align}\label{eq:sim-grading}
 \V & = \V_{-1} \oplus \V_0 \oplus \V_{1} \, ,
\end{align}
where $P_0=\CO^0(2m)$ is the reductive part of $P$. For each $i=-1,0 ,1$, we have $\V_i \cong \V^i / \V^{i+1}$ as vector spaces. In terms of our earlier notation $\V_{-1} = \mbb{L}$ and $\V_0 = \mbb{H}_{\K, \mbb{L}}$.

Recall from \cite{Fino2020} that the Lie algebra $\g \cong \bigwedge^2 \V^*$ of $G =\SO^0(n+1,1)$ can then be expressed as a direct sum of $P_0$-modules
\begin{align}\label{eq:g_G0split}
\g & = \g_{-1} \oplus \g_0 \oplus \g_1 \, , 
\end{align}
where $\g_{\pm1} \cong  \V_{\pm1}^* \otimes \V_0^*$ and $\g_0 \cong \left( \V_{-1}^* \otimes \V_1^* \right) \oplus \bigwedge^2 \V_0^*$. Note that $ \V_{-1}^* \otimes \V_1^*$ is the one-dimensional centre $\mf{z}_0$ of $\g_0$.

Now, the complex structure splits ${}^\C \V_0 $ and its dual as
\begin{align}\label{eq:V0-split-even}
{}^\C \V_0 & = \V^{(1,0)}_0 \oplus \V^{(0,1)}_0 \, , &  {}^\C \V_0^* & = (\V^{(1,0)}_0)^* \oplus (\V^{(0,1)}_0)^* \, ,
\end{align}
where $\V_0^{(1,0)} :=  \overline{\mbb{N}} \cap {}^\C \V_0$ and $\V_0^{(0,1)} :=\mbb{N} \cap {}^\C \V_0$. Writing
\begin{align}\label{eq:VC_split}
\V^* & = \V_{-1}^* \oplus  \dbl \bigwedge^{(1,0)} \V_0^* \dbr \oplus \V_1^* \, ,
\end{align}
we find that the summands in the $P_0$-invariant decomposition of $\g$ given in \eqref{eq:g_G0split} split further into irreducible $Q_0$-modules:
\begin{align}
\begin{aligned}\label{eq:so-grad-even}
 \g_{\pm 1} & = \V_{\pm 1}^* \otimes \dbl \bigwedge^{(1,0)} \V_0^* \dbr \, , \\
  \g_0 & =  \mf{z}_0 \oplus  \left( \mrm{span}(\omega) \oplus [  \bigwedge^{(1,1)}_\circ \V_0^*  ] \oplus \dbl \bigwedge^{(2,0)} \V_0^*  \dbr \right)  \, ,
  \end{aligned}
\end{align}
where we recall that $\omega$ is the Hermitian $2$-form on $\V_0$. We identify the Lie algebra $\mf{q}_0$ of $Q_0$ as the Lie subalgebra
\begin{align*}
\mf{q}_0 = \mf{z}_0 \oplus \mrm{span}( \omega) \oplus [  \bigwedge^{(1,1)}_\circ \V_0^*  ] \, ,
\end{align*}
which can be seen to be isomorphic to $\R \oplus \u(m)$. In addition, the Lie algebra $\mf{q}$ of the stabiliser $Q$ of the Robinson structure is given by
\begin{align}\label{eq:rob-grad-even}
\mf{q} & := \mf{z}_0 \oplus \left( \mrm{span}( \omega) \oplus [  \bigwedge^{(1,1)}_\circ \V_0^*  ] \right) \oplus \left( \V_1^* \otimes \dbl \bigwedge^{(1,0)} \V_0^* \dbr \right) \, .
\end{align}
As expected, $\mf{q}_0  = \mf{g}_0 \cap \mf{q}$.

\begin{rem}
In low dimensions, we note the following points:
\begin{itemize}
 \item In dimension four, we have that $\p \cong \mf{q}$, i.e.\ an optical structure is a Robinson structure.
 \item In dimension six, the semi-simple part $\bigwedge^2 \V_0^*$ of $\g_0$ splits into a self-dual part and an anti-self-dual part. More explicitly, $\su^+(2) = \mrm{span}(\omega) \oplus \dbl \bigwedge^{(2,0)} \V_0^*  \dbr $ and $\su^-(2) =[  \bigwedge^{(1,1)}_\circ \V_0^*  ]$, where $\su^\pm(2)$ are isomorphic to two copies of $\su(2)$. 
\end{itemize}
\end{rem}

\begin{rem}\label{rem:sp_Rob}
Clearly, the space of all (oriented) self-dual Robinson structures on $(\V, g)$ is isomorphic to $\SO^0(2m+1,1)/Q$. It corresponds to the isotropic Grassmannian $\mrm{Gr}_{m+1}^+(\mbb{V},g)$ of self-dual MTN planes in $(\V, g)$ and has real dimension $m(m+1)$.

In addition, for a given optical structure $\K$ on $(\V,g)$, the space of all (oriented) self-dual Robinson structures compatible with $\K$ is isomorphic to $P/Q \cong \SO(2m)/\U(m)$, which has real dimension $m(m-1)$. It will be denoted $\mrm{Gr}_{m+1}^+ (\V, g, \mbb{K})$.
\end{rem}

\subsection{One-dimensional representations of $Q$}\label{sec:one-dim_rep}
For any $w \in \R$, we define the one-dimensional representations $\R(w)$ and $\C(w,0)$ of $Q$ on $\R$ of weight $w$ and on $\C$ of weight $(w,0)$ by
\begin{align*}
(\e^\varphi , \psi, \phi) \cdot  r & = \e^{w \varphi} r \, , & \mbox{for any $r \in \R$,} \\
(\e^\varphi , \psi, \phi) \cdot  z & = (\e^{\varphi} \det A)^w z \, , & \mbox{for any $z \in \C$.} 
\end{align*}
We also define $\C(0,w) := \overline{\C(w,0)}$. One can check that
\begin{align*}
\R(-1) & \cong \K \, , &
\C(-1,0) & \cong \bigwedge^{m+1} \Ann( \mbb{N}) \, , &
\C(0,-1) & \cong \bigwedge^{m+1} \Ann(\overline{\mbb{N}}) \, .
\end{align*}
These leads to the one-dimensional representations $\C(w,w') := \C(w,0) \otimes \C(0,w')$ for any real $w$, $w'$. We also note $\C \otimes \R(w) \cong \C(w,w)$ for any real $w$, and $\R(w)\otimes\R(w') \cong \R(w+w')$ for any real $w,w'$, and similarly for the complex representations.

\subsection{Space $\mbb{G}$ of algebraic intrinsic torsions}\label{sec:alg_intors} 
Let us now consider the $Q$-module $\mbb{G}=\V^* \otimes \left( \g / \mf{q} \right)$. We treat only the case $m>1$ since for $m=1$, we have $Q=P$, which is already dealt with \cite{Fino2020}.
\begin{thm}\label{thm-main-Rob} 
Assume $m>1$. The $Q$-module $\mbb{G} = \V^* \otimes \left( \g / \mf{q} \right)$ admits a filtration
\begin{align}\label{eq:filt-q-G}
 \mbb{G}^1 \subset \mbb{G}^0 \subset \mbb{G}^{-1} \subset \mbb{G}^{-2}
\end{align}
of $Q$-modules $\mbb{G}^i := (\V^*)^{i+1} \otimes \left( \g / \mf{q} \right)$ for $i=-2,-1,0$. The summands of its associated graded $Q$-module
\begin{align*}
\gr(\mbb{G}) = \gr_{-2}(\mbb{G}) \oplus \gr_{-1}(\mbb{G}) \oplus \gr_0(\mbb{G}) \oplus \gr_1(\mbb{G}) \, ,
\end{align*}
where, for each $i=-2,-1,0$, $\gr_i (\mbb{G})  = \mbb{G}^i / \mbb{G}^{i+1}$, and $\gr_1(\mbb{G}) = \mbb{G}^1$, decompose into direct sums of irreducible $Q$-modules $\gr_i^{j,k}(\mbb{G})$ as follows:
\begin{align*}
\gr_{-2}(\mbb{G}) & = \gr_{-2}^{0,0} ( \mbb{G} )\, , \\
\gr_{-1}(\mbb{G}) & = \gr_{-1}^{0,0} ( \mbb{G} ) \oplus \left( \gr_{-1}^{1,0} (\mbb{G}) \oplus \gr_{-1}^{1,1} (\mbb{G})  \oplus \gr_{-1}^{1,2} (\mbb{G}) \right) \oplus \left( \gr_{-1}^{2,0} (\mbb{G}) \oplus \gr_{-1}^{2,1} (\mbb{G}) \right) \oplus \gr_{-1}^{3,0} \, , \\
\gr_0(\mbb{G}) & = \gr_0^{0,0} ( \mbb{G} ) \oplus \left( \gr_0^{1,0} ( \mbb{G} ) \oplus \gr_0^{1,1} ( \mbb{G} ) \oplus \gr_0^{1,2} ( \mbb{G} ) \oplus \gr_0^{1,3} ( \mbb{G} ) \right) \, , \\
\gr_1(\mbb{G})& = \gr_1^{0,0} (\mbb{G}) \, ,
\end{align*}
where, for each $i,j,k$, the $Q$-module $\gr_{i}^{j,k} ( \mbb{G} )$ is isomorphic to the $Q_0$-module $\mbb{G}_i^{j,k}$ as given in Table \ref{tab:table-irred-mod-G}. Note that $\gr_0^{1,1} ( \mbb{G} )$ and $\gr_0^{1,3} ( \mbb{G} )$ do not occur when $m=2$.
\begin{center}
\begin{table}
\begin{displaymath}
{\renewcommand{\arraystretch}{1.5}
\begin{array}{|c|c|c|c|}
\hline
\text{$Q_0$-module} & \text{Description} & \text{Dimension} \\
\hline
\hline
  \mbb{G}_{-2}^{0,0} & \V_0^* \otimes \R (2) & 2m \\
  \hline
  \hline
  \mbb{G}_{-1}^{0,0} & \R(1) & 1  \\
    \hline
  \mbb{G}_{-1}^{1,0} & \R(1) & 1  \\
  \mbb{G}_{-1}^{1,1} & \dbl \bigwedge^{(2,0)} \V_0^* \dbr \otimes \R  (1) & m(m-1) \\ 
  \mbb{G}_{-1}^{1,2} & [ \bigwedge^{(1,1)}_\circ \V_0^* ]  \otimes \R  (1)  & (m+1)(m-1) \\
  \hline 
    \mbb{G}_{-1}^{2,0} & [ \bigwedge^{(1,1)}_\circ \V_0^* ]  \otimes \R  (1)  & (m+1)(m-1) \\
  \mbb{G}_{-1}^{2,1} & \dbl \bigodot^{(2,0)} \V_0^* \dbr \otimes \R  (1) & m(m+1) \\
     \hline
  \mbb{G}_{-1}^{3,0} & \dbl \bigwedge^{(2,0)} \V_0^* \dbr \otimes \R  (1) & m(m-1) \\
  \hline
  \hline
  \mbb{G}_0^{0,0} &  \V_0^*  & 2m \\
  \hline
  \mbb{G}_0^{1,0} & \dbl \bigwedge^{(1,0)} \V_0^*   \dbr & 2m \\
  \mbb{G}_0^{1,1} & \dbl \bigwedge^{(3,0)}  \V_0^* \dbr & \scriptstyle{\frac{1}{3}m(m-1)(m-2)}  \\
  \mbb{G}_0^{1,2} & \dbl \ydhook \, \V_0^*  \dbr & \scriptstyle{\frac{2}{3}m(m+1)(m-1)} \\
  \mbb{G}_0^{1,3} & \dbl  \bigwedge^{(1,2)}_\circ \V_0^*     \dbr & \scriptstyle{m(m+1)(m-2)} \\
  \hline
  \hline
  \mbb{G}_1^{0,0} & \dbl \bigwedge^{(2,0)} \V_0^*  \dbr \otimes \R(-1) & m(m-1) \\
\hline
\end{array}}
\end{displaymath}
\caption{\label{tab:table-irred-mod-G}Irreducible $Q_0$-submodules of $\mbb{G}$}
\end{table}
\end{center}
\end{thm}

\begin{proof}
We use the same strategy as in Proposition 3.3 of \cite{Fino2020}: the filtration \eqref{eq:sim-filtration2} induces the filtration \eqref{eq:filt-q-G} of $Q$-modules on $\mbb{G} \subset \V^* \otimes \left( \g^{-1} / \g^1 \right)$. We proceed as before using the decompositions \eqref{eq:VC_split}, \eqref{eq:so-grad-even} and \eqref{eq:rob-grad-even} to find
\begin{align*}
 \mbb{G} & \cong \left( \V^*_{-1} \oplus \dbl \bigwedge^{(1,0)} \V_0^* \dbr \oplus \V^*_1 \right) \otimes \left(  \left( \V_{-1}^* \otimes \dbl \bigwedge^{(1,0)} \V_0^* \dbr \right) \oplus \dbl \bigwedge^{(2,0)} \V_0^*  \dbr \right) \, .
\end{align*}
The result follows by distributing this expression and splitting each summand into irreducibles:
\begin{align*}
 \V^*_{-1} \otimes   \V_{-1}^* \otimes \dbl \bigwedge^{(1,0)} \V_0^* \dbr  & = \mbb{G}_{-2}^{0,0} \, ,\\
 \V^*_{-1} \otimes \dbl \bigwedge^{(2,0)} \V_0^*  \dbr & = \mbb{G}_{-1}^{3,0}  \, , \\
 \dbl \bigwedge^{(1,0)} \V_0^* \dbr \otimes   \V_{-1}^* \otimes \dbl \bigwedge^{(1,0)} \V_0^* \dbr & = \mbb{G}_{-1}^{1,0} \oplus \left(\mbb{G}_{-1}^{1,0} \oplus \mbb{G}_{-1}^{1,1} \oplus \mbb{G}_{-1}^{1,2} \right) \oplus \left( \mbb{G}_{-1}^{2,0} \oplus \mbb{G}_{-1}^{2,1} \right) \, , \\
  \dbl \bigwedge^{(1,0)} \V_0^* \dbr \otimes \dbl \bigwedge^{(2,0)} \V_0^*  \dbr  & = \mbb{G}_0^{1,0} \oplus \mbb{G}_0^{1,1} \oplus \mbb{G}_0^{1,2} \oplus \mbb{G}_0^{1,3}  \, , \\
  \V^*_1  \otimes   \V_{-1}^* \otimes \dbl \bigwedge^{(1,0)} \V_0^* \dbr  & = \mbb{G}_0^{0,0}\, , \\
    \V^*_1  \otimes \dbl \bigwedge^{(2,0)} \V_0^*  \dbr & =  \mbb{G}_1^{0,0} \, ,
\end{align*}
where, recalling that $\V_{\pm 1}^* \cong \R(\mp1)$, the modules are described in Table \ref{tab:table-irred-mod-G}.
\end{proof}

\begin{rem}
Observe that the $P$-module $\V^* \otimes \left( \g / \mf{p} \right)$ represents the space of intrinsic torsions of the underlying optical structure. This can be viewed as a $Q$-submodule of $\mbb{G}$. In the next proposition, which is a direct consequence of Theorem \ref{thm-main-Rob}, we single out the irreducible $P$-submodules of $\gr( \mbb{G})$, thereby making contact with the intrinsic torsion of an optical structure described in \cite{Fino2020}.
\end{rem}

\begin{prop}\label{prop-Q2P}
Assume $m>1$. Let $\mbb{G} = \V^* \otimes \left( \g / \mf{q} \right)$ and consider the graded module $\gr(\mbb{G})$ given in Theorem \ref{thm-main-Rob}. Define
\begin{align*}
\gr_{-2}^0 ( \mbb{G} ) & := \gr_{-2}^{0,0}  ( \mbb{G} ) \, ,&
\gr_{-1}^0 ( \mbb{G} ) & := \gr_{-1}^{0,0} ( \mbb{G} ) \, , \\
\gr_{-1}^1 ( \mbb{G} ) & := \gr_{-1}^{1,0} ( \mbb{G} ) \oplus \gr_{-1}^{1,1} ( \mbb{G} ) \oplus \gr_{-1}^{1,2}  ( \mbb{G} ) \, , &
\gr_{-1}^2 ( \mbb{G} ) & := \gr_{-1}^{2,0} ( \mbb{G} ) \oplus \gr_{-1}^{2,1}( \mbb{G} ) \, , \\
\gr_0^0  ( \mbb{G} ) & := \gr_0^{0,0} ( \mbb{G} )\, .
\end{align*}
and in dimension six,
\begin{align*}
\gr_{-1}^{1,+} (\mbb{G}) & := \gr_{-1}^{1,0} ( \mbb{G} ) \oplus \gr_{-1}^{1,1} (\mbb{G} ) \, , &
\gr_{-1}^{1,-} (\mbb{G}) & := \gr_{-1}^{1,2} ( \mbb{G} ) \, .
\end{align*}
Then, for each $i,j$, $\gr_i^j (\mbb{G})$ is an irreducible $P$-module except in dimension six, where $\gr_{-1}^{1} (\mbb{G})$ is not irreducible, but $\gr_{-1}^{1,\pm} (\mbb{G})$ are.

Moreover, we define
\begin{align*}
\mbb{G}_{-2}^0 & := \mbb{G}_{-2}^{0,0} \, , &
\mbb{G}_{-1}^0 & := \mbb{G}_{-1}^{0,0} \, , \\
\mbb{G}_{-1}^1 & := \mbb{G}_{-1}^{1,0} \oplus \mbb{G}_{-1}^{1,1} \oplus \mbb{G}_{-1}^{1,2} \, , \\
\mbb{G}_{-1}^2 & := \mbb{G}_{-1}^{2,0} \oplus \mbb{G}_{-1}^{2,1}  \, , &
 \mbb{G}_0^0 & := \mbb{G}_0^{0,0} \, ,
\end{align*}
and in dimension six,
\begin{align*}
\mbb{G}_{-1}^{1,+} & := \mbb{G}_{-1}^{1,0} \oplus \mbb{G}_{-1}^{1,1} \, , &
\mbb{G}_{-1}^{1,-} & := \mbb{G}_{-1}^{1,2} \, .
\end{align*}
Then, for each $i,j$, $\mbb{G}_i^j$ is an irreducible $P_0$-module except in dimension six, where $\mbb{G}_{-1}^{1}$ is not irreducible, but $\mbb{G}_{-1}^{1,\pm}$ are.
\end{prop}

\subsection{Isotypic $Q_0$-submodules of $\mbb{G}$}\label{sec:alg_intors_iso} 
Let us fix a splitting of $\mbb{G}$ into $Q_0$-modules. Observe that the modules in each of the pairs $(\mbb{G}_{-1}^{0,0} ,  \mbb{G}_{-1}^{1,0})$, $(\mbb{G}_{-1}^{1,2} , \mbb{G}_{-1}^{2,0})$, $(\mbb{G}_{-1}^{1,1} ,  \mbb{G}_{-1}^{3,0})$ and $(\mbb{G}_{-1}^{1,2} ,  \mbb{G}_{-1}^{2,0})$ are isotypic, i.e.\ they have the same dimensions. This means that one can construct further irreducible $Q_0$-modules by assigning some algebraic relations among these in terms of parameters. To this end, we need to describe them by means of the projections
\begin{align}
\begin{aligned}\label{eq:extra_proj}
(\Pi_{-1}^{0 \times 1})_{[x:y]} &:  \V^* \otimes \g \rightarrow \mbb{G}_{-1}^{0,0} \oplus  \mbb{G}_{-1}^{1,0} \, , & [x:y] & \in \RP^1 \, ,\\
(\Pi_{-1}^{1 \times 2})_{[x:y]} &:  \V^* \otimes \g \rightarrow \mbb{G}_{-1}^{1,2} \oplus  \mbb{G}_{-1}^{2,0} \, , & [x:y] & \in \RP^1 \, ,\\
(\Pi_{-1}^{1 \times 3})_{[z:w]} &:  \V^* \otimes \g \rightarrow \mbb{G}_{-1}^{1,1} \oplus  \mbb{G}_{-1}^{3,0}  \, , & [z:w] & \in \CP^1 \, , \\
(\Pi_{0}^{0 \times 1})_{[z:w]} &:  \V^* \otimes \g \rightarrow \mbb{G}_{0}^{0,0} \oplus  \mbb{G}_{0}^{1,0}  \, , & [z:w] & \in \CP^1 \, ,
\end{aligned}
\end{align}
whose precise definitions have been relegated to Appendix \ref{app:proj} for convenience. Here, $\RP^1$ and $\CP^1$ are real and complex projective lines, respectively. We can then define the following additional $Q_0$-submodules of $\mbb{G}$:
\begin{align}
\begin{aligned}\label{eq:extra_irr_G}
(\mbb{G}_{-1}^{0 \times 1})_{[x:y]}  & := \mathrm{im} (\Pi_{-1}^{0 \times 1})_{[x:y]} \subset \mbb{G}_{-1}^{0,0} \oplus  \mbb{G}_{-1}^{1,0} \, , & \mbox{for $[x:y] \in \RP^1$.} \\
(\mbb{G}_{-1}^{1 \times 2})_{[x:y]} & := \mathrm{im} (\Pi_{-1}^{1 \times 2})_{[x:y]} \subset \mbb{G}_{-1}^{1,2} \oplus  \mbb{G}_{-1}^{2,0} \, , & \mbox{for $[x:y] \in \RP^1$.}\\
(\mbb{G}_{-1}^{1 \times 3})_{[z:w]} & := \mathrm{im} (\Pi_{-1}^{1 \times 3})_{[z:w]} \subset \mbb{G}_{-1}^{1,1} \oplus  \mbb{G}_{-1}^{3,0} \, , & \mbox{for $[z:w] \in \CP^1$.} \\
(\mbb{G}_{0}^{0 \times 1})_{[z:w]} & := \mathrm{im} (\Pi_{-1}^{1 \times 2})_{[z:w]} \subset \mbb{G}_{-1}^{1,2} \oplus  \mbb{G}_{-1}^{2,0} \, , & \mbox{for $[z:w] \in \CP^1$.}
\end{aligned}
\end{align}
Their descriptions and dimensions are given in Table \ref{tab:table-isotypic}.
\begin{center}
\begin{table}
\begin{displaymath}
{\renewcommand{\arraystretch}{1.5}
\begin{array}{|c|c|c|c|}
\hline
\text{$Q_0$-module} & \text{Description} & \text{Dimension} \\
\hline
\hline
(\mbb{G}_{-1}^{0 \times 1})_{[x:y]}  &  \R(1) & 1  \\
(\mbb{G}_{-1}^{1 \times 2})_{[x:y]} & [ \bigwedge^{(1,1)}_\circ \V_0^*  ]  \otimes \R(1) & (m-1)(m+1) \\ 
(\mbb{G}_{-1}^{1 \times 3})_{[z:w]} & \dbl \bigwedge^{(2,0)} \V_0^*  \dbr \otimes \R(1)  &  m(m-1) \\
  \hline 
  \hline
(\mbb{G}_{0}^{0 \times 1})_{[z:w]} & \V_0^*  & 2m \\
\hline
\end{array}}
\end{displaymath}
\caption{\label{tab:table-isotypic} $Q_0$-submodules of $\mbb{G}$ --- here $[x:y] \in \RP^1$ and $[z:w] \in \CP^1$}
\end{table}
\end{center}
Note that by definition,
\begin{align*}
(\mbb{G}_{-1}^{0 \times 1})_{[1,0]} & =  \mbb{G}_{-1}^{0,0} \, , &
(\mbb{G}_{-1}^{1 \times 2})_{[1,0]} & = \mbb{G}_{-1}^{1,2} \, ,  &
(\mbb{G}_{-1}^{1 \times 3})_{[1,0]} & =  \mbb{G}_{-1}^{1,1}  \, , & 
(\mbb{G}_{0}^{0 \times 1})_{[1,0]} & =  \mbb{G}_{-1}^{1,2} \, , \\
 (\mbb{G}_{-1}^{0 \times 1})_{[0,1]} & =   \mbb{G}_{-1}^{1,0} \, ,  &
 (\mbb{G}_{-1}^{1 \times 2})_{[0,1]} & = \mbb{G}_{-1}^{2,0} \, , &
(\mbb{G}_{-1}^{1 \times 3})_{[0,1]} & =  \mbb{G}_{-1}^{3,0} \, , &
 (\mbb{G}_{0}^{0 \times 1})_{[0,1]} & =\mbb{G}_{-1}^{2,0} \, .
\end{align*}

\subsection{The $Q$-submodules of $\mbb{G}$}
We are now in the position of determining all the $Q$-submodules of $\mbb{G}$. For this purpose, we shall appeal to the $Q_0$-module epimorphisms $\Pi_i^{j,k} : \V^* \otimes \g \rightarrow \mbb{G}_i^{j,k}$, $\Pi_i^{j} : \V^* \otimes \g \rightarrow \mbb{G}_i^{j}$ and \eqref{eq:extra_proj}, all of which are described in Appendix \ref{app:proj} with respect to some chosen splitting of $\mbb{G}$. Any $Q$-submodule of $\mbb{G}$ must be a sum of the irreducible $Q_0$-submodules $\mbb{G}_i^{j,k}$ given in Table \ref{tab:table-irred-mod-G}, together with the families of irreducibles $Q_0$-modules given in Table \ref{tab:table-isotypic}.  Not every such sum is a $Q$-module. To determine which $Q_0$-submodules of $\mbb{G}$ are also $Q$-submodules, we compute how a change of splitting \eqref{eq:basis-change} transforms the maps $\Pi_i^{j,k}$. This will tell us how the various modules $\mbb{G}_i^{j,k}$ are related under the action of $P_+$, the nilpotent part of $Q$. We will then be able to determine the $Q$-submodules of $\mbb{G}$ accordingly.

To facilitate the readability, we contract the projections \eqref{eq:proj_Qmod} with suitable combinations of $\delta^i_{\alpha}$, $\delta^i_{\bar{\alpha}}$ to define
\begin{align}\label{eq:S-H2}
\begin{aligned}
\gamma_\alpha & := \Pi_{-2}^{0,0}(\Gamma)_\alpha \, , \\
\epsilon & := \Pi_{-1}^{0,0}(\Gamma) \, , \\
\tau^\omega  & := \Pi_{-1}^{1,0}(\Gamma) \, , \\
\tau_{\alpha \beta} & := \Pi_{-1}^{1,1}(\Gamma)_{\alpha \beta} = {\Pi'}_{-1}^{1,1}(\Gamma)_{\alpha \beta} \, , &
\tau^\circ_{\alpha \bar{\beta}} & := \Pi_{-1}^{1,2}(\Gamma)_{\alpha \bar{\beta}} \, , \\
\sigma_{\alpha \bar{\beta}} & := \Pi_{-1}^{2,0}(\Gamma)_{\alpha \bar{\beta}} \, , &
\sigma_{\alpha \beta} & := \Pi_{-1}^{2,1}(\Gamma)_{\alpha \beta} \, , \\
\zeta_{\alpha \beta} & := \Pi_{-1}^{3,0} (\Gamma)_{\alpha \beta}  \, , \\  
E_\alpha & := \Pi_0^{0,0}(\Gamma)_{\alpha} = {\Pi'}_0^{0,0}(\Gamma)_{\alpha} \, ,& \\
G_{\alpha} & := \Pi_0^{1,0} (\Gamma)_{\alpha}  \, ,  &
G^{\tiny{\ydskew}}_{\alpha \beta \gamma} & := \Pi_0^{1,1} (\Gamma)_{\alpha \beta \gamma}  \, , \\
G^{\tiny{\ydhook}}_{\alpha \beta \gamma} & := \Pi_0^{1,2} (\Gamma)_{\alpha \beta \gamma}  \, , &
G^{\circ}_{\bar{\alpha} \beta \gamma} & := \Pi_0^{1,3} (\Gamma)_{\bar{\alpha} \beta \gamma}  \, , \\
B_{\alpha \beta} & := \Pi_1^{0,0} (\Gamma)_{\alpha \beta}  \, .
\end{aligned}
\end{align}
Their complex conjugates are defined similarly. Note that with these definitions,
\begin{align*}
\sigma_{\alpha \beta} & = \sigma_{(\alpha \beta)} \, , & \tau_{\alpha \beta} & =  \tau_{[\alpha \beta]} \, ,  \\
\sigma_{\alpha \bar{\beta}} h^{\alpha \bar{\beta}} & = 0 \, , & \tau^\circ_{\alpha \bar{\beta}} h^{\alpha \bar{\beta}} & = 0 \, , \\
\overline{\sigma_{\alpha \bar{\beta}}} & = \sigma_{\bar{\alpha} \beta} = \sigma_{\beta \bar{\alpha}} \, , & \overline{\tau^\circ_{\alpha \bar{\beta}}}& = \tau^\circ_{\bar{\alpha} \beta} = - \tau^\circ_{\beta \bar{\alpha} } \, , \\
\zeta_{\alpha \beta} & = - \zeta_{\beta \alpha} \, , & B_{\alpha \beta} & = - B_{\beta \alpha} \, , \\
G^{\tiny{\ydskew}}_{(\alpha \beta) \gamma} & = 0 \, , &
G^{\tiny{\ydhook}}_{[\alpha \beta \gamma]} & = 0  \, , &
G^{\circ}_{\bar{\beta} \alpha \gamma} h^{\alpha \bar{\beta}} & = 0\, .
\end{align*}

\begin{thm}\label{thm:intors-rob}
Define
\begin{flalign*}
\slashed{\mbb{G}}_{-2}^{0,0}  & := \{ \Gamma \in \V^* \otimes \g : \Pi_{-2}^{0,0}(\Gamma) = 0 \} / \left( \V^* \otimes \mf{q} \right)  \, , &
\end{flalign*}
\begin{flalign*}
\slashed{\mbb{G}}_{-1}^{0,0}  & := \{ \Gamma \in \V^* \otimes \g : \Pi_{-1}^{0,0}(\Gamma) = \Pi_{-2}^{0,0}(\Gamma) = 0 \} / \left( \V^* \otimes \mf{q} \right) \, , & \\
\slashed{\mbb{G}}_{-1}^{1,i} & := \{ \Gamma \in \V^* \otimes \g : \Pi_{-1}^{1,i}(\Gamma) = \Pi_{-2}^{0,0}(\Gamma) = 0 \} / \left( \V^* \otimes \mf{q} \right) \, ,  \qquad  i=0,1 & \\
\slashed{\mbb{G}}_{-1}^{2,i} & := \{ \Gamma \in \V^* \otimes \g : \Pi_{-1}^{2,i}(\Gamma) = \Pi_{-2}^{0,0}(\Gamma) = 0 \} / \left( \V^* \otimes \mf{q} \right) \, ,  \qquad i=0,1,2 &\\
\slashed{\mbb{G}}_{-1}^{3,0} & := \{ \Gamma \in \V^* \otimes \g : \Pi_{-1}^{3,0}(\Gamma) = \Pi_{-2}^{0,0}(\Gamma) = 0 \} / \left( \V^* \otimes \mf{q} \right)  \, , &
\end{flalign*}
\begin{flalign*}
\slashed{\mbb{G}}_0^{0,0} & := \{ \Gamma \in \V^* \otimes \g : \Pi_0^{0,0} (\Gamma) = \Pi_{-1}^0(\Gamma) = \Pi_{-1}^1 (\Gamma) = \Pi_{-1}^2 (\Gamma) = \Pi_{-2}^0 (\Gamma) =  0 \} / \left( \V^* \otimes \mf{q} \right)  \, , & \\
\slashed{\mbb{G}}_0^{1,0} & := \{ \Gamma \in \V^* \otimes \g : \Pi_0^{1,0}(\Gamma) = \Pi_{-1}^{0,0}(\Gamma) = \Pi_{-1}^{1,0}(\Gamma)  = \Pi_{-1}^{1,2}(\Gamma)  & \\
& \qquad \qquad \qquad \qquad   = \Pi_{-1}^{2,0}(\Gamma) = \Pi_{-1}^{3,0}(\Gamma) = \Pi_{-2}^{0,0}(\Gamma) =  0 \} /\left( \V^* \otimes \mf{q} \right) \, , & \\
\slashed{\mbb{G}}_0^{1,1} & := \{ \Gamma \in \V^* \otimes \g : \Pi_0^{1,1}(\Gamma) = (\Pi_{-1}^{1\times3})_{[-4\i,1]}(\Gamma)  =  0 \} / \left( \V^* \otimes \mf{q} \right) \, , & \\
\slashed{\mbb{G}}_0^{1,2} & := \{ \Gamma \in \V^* \otimes \g :  \Pi_0^{1,2}(\Gamma) = (\Pi_{-1}^{1\times3})_{[2\i,1]}(\Gamma)  = \Pi_{-1}^{2,1}(\Gamma) = \Pi_{-2}^{0,0}(\Gamma) =  0 \} / \left( \V^* \otimes \mf{q} \right)  \, , & \\
\slashed{\mbb{G}}_0^{1,3} & := \{ \Gamma \in \V^* \otimes \g :  \Pi_0^{1,3}(\Gamma) =  \Pi_{-1}^{1,2}(\Gamma) = \Pi_{-1}^{2,0}(\Gamma) = \Pi_{-1}^{3,0}(\Gamma) =\Pi_{-2}^{0,0}(\Gamma) = 0 \} / \left( \V^* \otimes \mf{q} \right) \, , &
\end{flalign*}
and,
\begin{flalign*}
\slashed{\mbb{G}}_1^{0,0} & := \{ \Gamma \in \V^* \otimes \g : \Pi_1^{0,0}(\Gamma) =   (\Pi_0^{0\times1})_{[2(m-1)\i,-1]}(\Gamma)  = \Pi_0^{1,1}(\Gamma) = \Pi_0^{1,2}(\Gamma) = \Pi_0^{1,3}(\Gamma) & \\
& \qquad \qquad  = \Pi_{-1}^{1,1}(\Gamma) = \Pi_{-1}^{1,2}(\Gamma) = \Pi_{-1}^2 (\Gamma) & \\
& \qquad \qquad  \qquad \qquad  = \Pi_{-1}^{3,0}(\Gamma) = \Pi_{-2}^{0,0}(\Gamma) = 0 \} / \left( \V^* \otimes \mf{q} \right) \, , \qquad \mbox{when $m>2$} \, , &
\shortintertext{while}
\slashed{\mbb{G}}_1^{0,0} & := \{ \Gamma \in \V^* \otimes \g : \Pi_1^{0,0}(\Gamma) =   (\Pi_0^{0\times1})_{[2\i,-1]}(\Gamma)  = \Pi_0^{1,1}(\Gamma) = \Pi_0^{1,2}(\Gamma) &\\
& \qquad \qquad  = (\Pi_{-1}^{1\times 3})_{[2\i, 1]}(\Gamma) = \Pi_{-1}^{1,2}(\Gamma) = \Pi_{-1}^2 (\Gamma) \\
& \qquad \qquad  \qquad \qquad  = \Pi_{-1}^{3,0}(\Gamma) = \Pi_{-2}^{0,0}(\Gamma) = 0 \} / \left( \V^* \otimes \mf{q} \right) \, , \qquad  \mbox{when $m=2$} \, . &
\end{flalign*}
In dimension six, i.e.\ $m=2$, we always have $\slashed{\mbb{G}}_0^{1,1} = \slashed{\mbb{G}}_0^{1,3} = \mbb{G}$.

Define further, for any $[x:y] \in \RP^1$, $[z:w] \in \CP^1$,
\begin{flalign*}
 (\slashed{\mbb{G}}_{-1}^{0 \times 1})_{[x:y]}  & := \{ \Gamma \in \V^* \otimes \g : (\Pi_{-1}^{0 \times 1})_{[x:y]} (\Gamma) = \Pi_{-2}^{0,0}(\Gamma)  = 0\} / \left( \V^* \otimes \mf{q} \right) \, , &  \\
 (\slashed{\mbb{G}}_{-1}^{1 \times 2})_{[x:y]} & := \{ \Gamma \in \V^* \otimes \g :  (\Pi_{-1}^{1 \times 2})_{[x:y]} (\Gamma) = \Pi_{-2}^{0,0}(\Gamma) =  0 \} / \left( \V^* \otimes \mf{q} \right)  \, , & \\
 (\slashed{\mbb{G}}_{-1}^{1 \times 3})_{[z:w]} & := \{ \Gamma \in \V^* \otimes \g : (\Pi_{-1}^{1 \times 3})_{[z:w]} (\Gamma) = (z + 4 \i w)  \Pi_{-2}^{0,0} (\Gamma)  = 0 \} / \left( \V^* \otimes \mf{q} \right)  \, , & \\
 (\slashed{\mbb{G}}_{0}^{0 \times 1})_{[z:w]} & := \{ \Gamma \in \V^* \otimes \g :  (\Pi_{0}^{0 \times 1})_{[z:w]} (\Gamma) = (2 (m-1) \i w + z) \Pi_{-1}^{0,0}(\Gamma)  & \\
& \qquad  \qquad \qquad =  (2 (m-1) \i w + z) \Pi_{-1}^{1,0}(\Gamma) = (\Pi_{-1}^{1 \times 3})_{[z:w]} (\Gamma) =  ( 2 w \,  \i - z) \Pi_{-1}^{1,2}(\Gamma) \\
& \qquad  \qquad \qquad \qquad =  ( 2 w \,  \i - z) \Pi_{-1}^{2,0}(\Gamma) = z \Pi_{-1}^{2,1}(\Gamma) = \Pi_{-2}^{0,0}(\Gamma) =  0 \} / \left( \V^* \otimes \mf{q} \right)  \, . &
\end{flalign*}

Then, for each $i,j,k$, $\slashed{\mbb{G}}_{i}^{j , k}$ is the largest $Q$-submodule of $\mbb{G}$ that does not contain $\mbb{G}_{i}^{j , k}$, and similarly for $(\slashed{\mbb{G}}_{-1}^{0 \times 1})_{[x:y]}$,  $(\slashed{\mbb{G}}_{-1}^{1 \times 2})_{[x:y]}$, $(\slashed{\mbb{G}}_{-1}^{1 \times 3})_{[z:w]}$ and $ (\slashed{\mbb{G}}_{0}^{0 \times 1})_{[z:w]}$.
In particular, any $Q$-submodule of $\mbb{G}$ arises as the intersection of any of the ones above.

In addition, there are inclusions of $Q$-submodules, which are denoted by arrows in the diagrams below.
\newpage
For $m>2$:
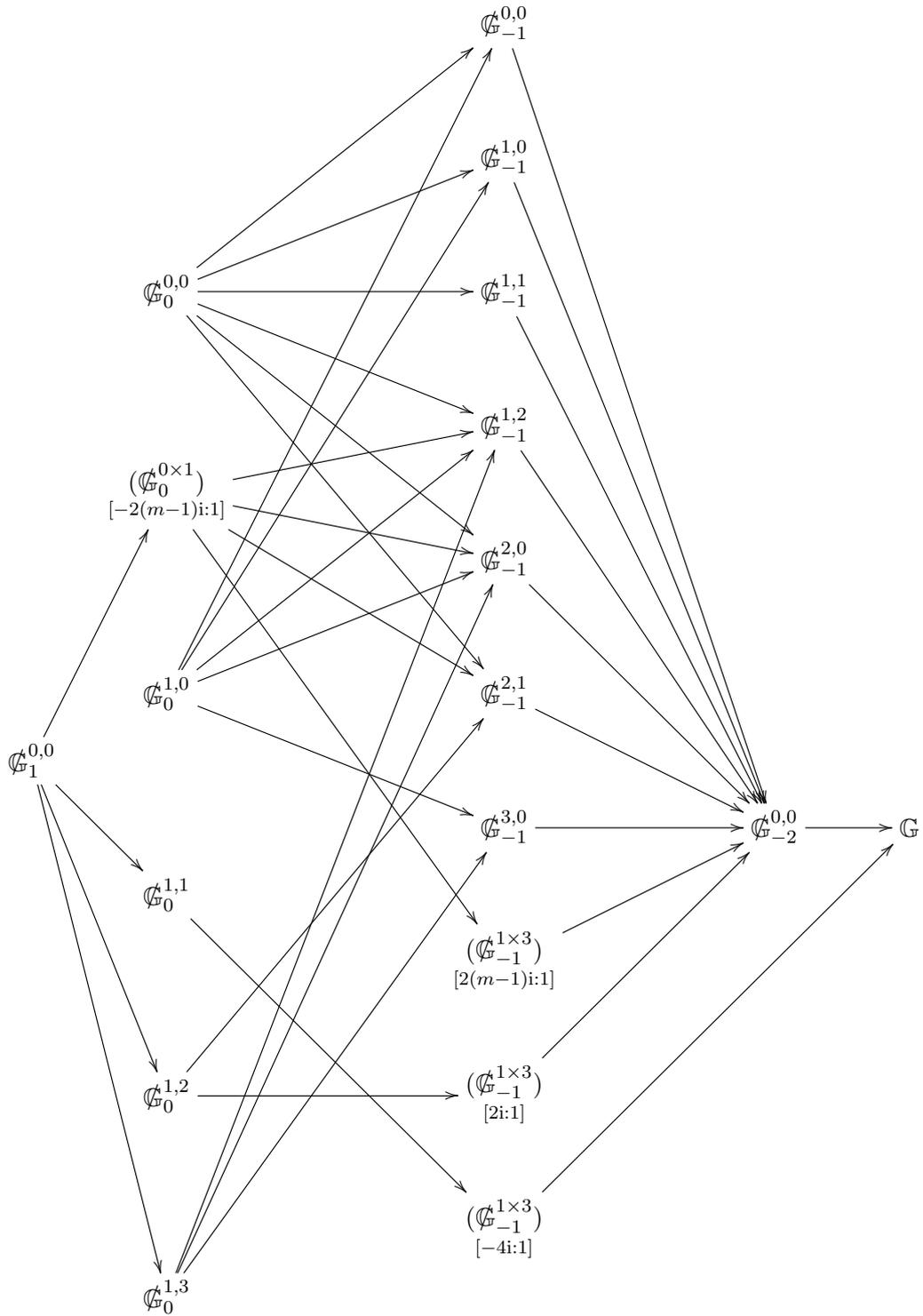
\begin{align*}
\xy
(-40,10)*+{\slashed{\mbb{G}}_1^{0,0}}="s0",
(-20,-70)*+{\slashed{\mbb{G}}_0^{1,3}}="s1",
(-20,-40)*+{\slashed{\mbb{G}}_0^{1,2}}="s2",
(-20,-10)*+{\slashed{\mbb{G}}_0^{1,1}}="s3",
(-20,20)*+{\slashed{\mbb{G}}_0^{1,0}}="s4",
(-20,50)*+{\underset{[-2(m-1)\i:1]}{(\slashed{\mbb{G}}_0^{0\times 1})}}="s5",
(-20,80)*+{\slashed{\mbb{G}}_0^{0,0}}="s6",
(30,-60)*+{\underset{[-4\i:1]}{(\slashed{\mbb{G}}_{-1}^{1\times 3})}}="s7",
(30,-40)*+{\underset{[2\i:1]}{(\slashed{\mbb{G}}_{-1}^{1\times 3})}}="s8",
(30,-20)*+{\underset{[2(m-1)\i:1]}{(\slashed{\mbb{G}}_{-1}^{1\times 3})}}="s17",
(30,0)*+{\slashed{\mbb{G}}_{-1}^{3,0}}="s9",
(30,20)*+{\slashed{\mbb{G}}_{-1}^{2,1}}="s10",
(30,40)*+{\slashed{\mbb{G}}_{-1}^{2,0}}="s11",
(30,60)*+{\slashed{\mbb{G}}_{-1}^{1,2}}="s12",
(30,80)*+{\slashed{\mbb{G}}_{-1}^{1,1}}="s13",
(30,100)*+{\slashed{\mbb{G}}_{-1}^{1,0}}="s14",
(30,120)*+{\slashed{\mbb{G}}_{-1}^{0,0}}="s15",
(70,0)*+{\slashed{\mbb{G}}_{-2}^{0,0}}="s16",
(90,0)*+{\mbb{G}}="s18",
"s1"; "s0" ; **@{-} ?>*\dir{>};
"s2"; "s0" ; **@{-} ?>*\dir{>};
"s3"; "s0" ; **@{-} ?>*\dir{>};
"s5"; "s0" ; **@{-} ?>*\dir{>};
"s9"; "s1" ; **@{-} ?>*\dir{>};
"s11"; "s1" ; **@{-} ?>*\dir{>};
"s12"; "s1" ; **@{-} ?>*\dir{>};
"s10"; "s2" ; **@{-} ?>*\dir{>};
"s8"; "s2" ; **@{-} ?>*\dir{>};
"s7"; "s3" ; **@{-} ?>*\dir{>};
"s9"; "s4" ; **@{-} ?>*\dir{>};
"s11"; "s4" ; **@{-} ?>*\dir{>};
"s12"; "s4" ; **@{-} ?>*\dir{>};
"s14"; "s4" ; **@{-} ?>*\dir{>};
"s15"; "s4" ; **@{-} ?>*\dir{>};
"s15"; "s6" ; **@{-} ?>*\dir{>};
"s14"; "s6" ; **@{-} ?>*\dir{>};
"s13"; "s6" ; **@{-} ?>*\dir{>};
"s12"; "s6" ; **@{-} ?>*\dir{>};
"s11"; "s6" ; **@{-} ?>*\dir{>};
"s10"; "s6" ; **@{-} ?>*\dir{>};
"s16"; "s8" ; **@{-} ?>*\dir{>};
"s16"; "s9" ; **@{-} ?>*\dir{>};
"s16"; "s10" ; **@{-} ?>*\dir{>};
"s16"; "s11" ; **@{-} ?>*\dir{>};
"s16"; "s12" ; **@{-} ?>*\dir{>};
"s16"; "s13" ; **@{-} ?>*\dir{>};
"s16"; "s14" ; **@{-} ?>*\dir{>};
"s16"; "s15" ; **@{-} ?>*\dir{>};
"s17"; "s5" ; **@{-} ?>*\dir{>};
"s10"; "s5" ; **@{-} ?>*\dir{>};
"s11"; "s5" ; **@{-} ?>*\dir{>};
"s12"; "s5" ; **@{-} ?>*\dir{>};
"s16"; "s17" ; **@{-} ?>*\dir{>};
"s18"; "s16" ; **@{-} ?>*\dir{>};
"s18"; "s7" ; **@{-} ?>*\dir{>};
\endxy
\end{align*}
\newpage
For $m=2$:
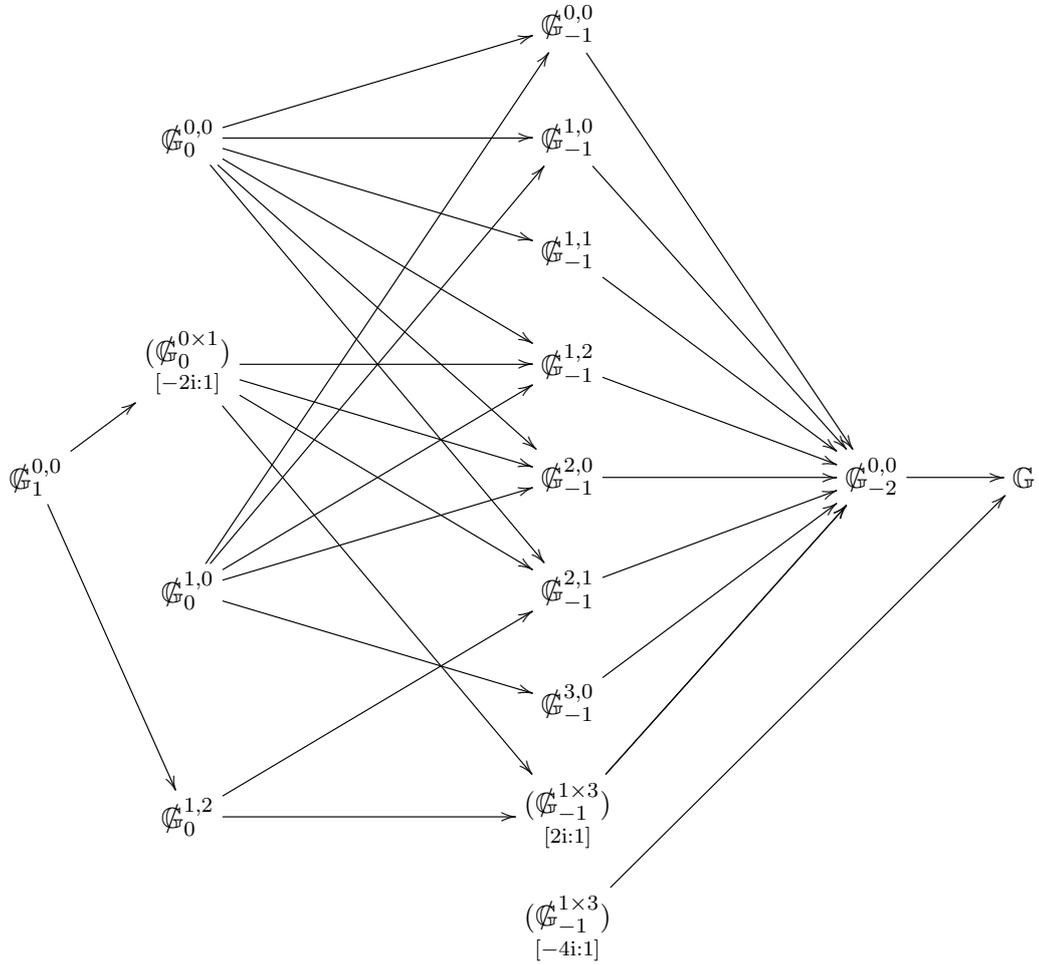
\begin{align*}
\xy
(-40,40)*+{\slashed{\mbb{G}}_1^{0,0}}="s0",
(-20,-5)*+{\slashed{\mbb{G}}_0^{1,2}}="s2",
(-20,25)*+{\slashed{\mbb{G}}_0^{1,0}}="s4",
(-20,55)*+{\underset{[-2\i:1]}{(\slashed{\mbb{G}}_0^{0\times 1})}}="s5",
(-20,85)*+{\slashed{\mbb{G}}_0^{0,0}}="s6",
(30,-20)*+{\underset{[-4\i:1]}{(\slashed{\mbb{G}}_{-1}^{1\times 3})}}="s7",
(30,-5)*+{\underset{[2\i:1]}{(\slashed{\mbb{G}}_{-1}^{1\times 3})}}="s17",
(30,10)*+{\slashed{\mbb{G}}_{-1}^{3,0}}="s9",
(30,25)*+{\slashed{\mbb{G}}_{-1}^{2,1}}="s10",
(30,40)*+{\slashed{\mbb{G}}_{-1}^{2,0}}="s11",
(30,55)*+{\slashed{\mbb{G}}_{-1}^{1,2}}="s12",
(30,70)*+{\slashed{\mbb{G}}_{-1}^{1,1}}="s13",
(30,85)*+{\slashed{\mbb{G}}_{-1}^{1,0}}="s14",
(30,100)*+{\slashed{\mbb{G}}_{-1}^{0,0}}="s15",
(70,40)*+{\slashed{\mbb{G}}_{-2}^{0,0}}="s16",
(90,40)*+{\mbb{G}}="s18",
"s2"; "s0" ; **@{-} ?>*\dir{>};
"s5"; "s0" ; **@{-} ?>*\dir{>};
"s10"; "s2" ; **@{-} ?>*\dir{>};
"s17"; "s2" ; **@{-} ?>*\dir{>};
"s9"; "s4" ; **@{-} ?>*\dir{>};
"s11"; "s4" ; **@{-} ?>*\dir{>};
"s12"; "s4" ; **@{-} ?>*\dir{>};
"s14"; "s4" ; **@{-} ?>*\dir{>};
"s15"; "s4" ; **@{-} ?>*\dir{>};
"s15"; "s6" ; **@{-} ?>*\dir{>};
"s14"; "s6" ; **@{-} ?>*\dir{>};
"s13"; "s6" ; **@{-} ?>*\dir{>};
"s12"; "s6" ; **@{-} ?>*\dir{>};
"s11"; "s6" ; **@{-} ?>*\dir{>};
"s10"; "s6" ; **@{-} ?>*\dir{>};
"s16"; "s17" ; **@{-} ?>*\dir{>};
"s16"; "s9" ; **@{-} ?>*\dir{>};
"s16"; "s10" ; **@{-} ?>*\dir{>};
"s16"; "s11" ; **@{-} ?>*\dir{>};
"s16"; "s12" ; **@{-} ?>*\dir{>};
"s16"; "s13" ; **@{-} ?>*\dir{>};
"s16"; "s14" ; **@{-} ?>*\dir{>};
"s16"; "s15" ; **@{-} ?>*\dir{>};
"s17"; "s5" ; **@{-} ?>*\dir{>};
"s10"; "s5" ; **@{-} ?>*\dir{>};
"s11"; "s5" ; **@{-} ?>*\dir{>};
"s12"; "s5" ; **@{-} ?>*\dir{>};
"s16"; "s17" ; **@{-} ?>*\dir{>};
"s18"; "s16" ; **@{-} ?>*\dir{>};
"s18"; "s7" ; **@{-} ?>*\dir{>};
\endxy
\end{align*}
For any $m>1$,
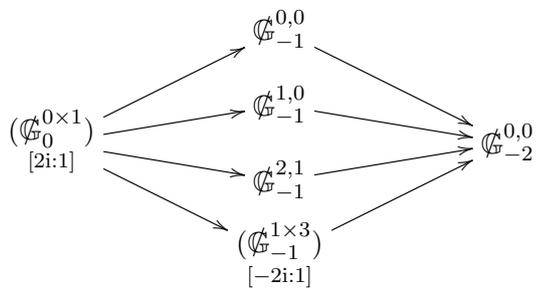
\begin{align*}
\xy
(-30,0)*+{\underset{[2 \i:1]}{(\slashed{\mbb{G}}_0^{0\times 1})}}="s5",
(0,-15)*+{\underset{[-2 \i:1]}{(\slashed{\mbb{G}}_{-1}^{1\times 3})}}="s7",
(0,-5)*+{\slashed{\mbb{G}}_{-1}^{2,1}}="s10",
(0,5)*+{\slashed{\mbb{G}}_{-1}^{1,0}}="s14",
(0,15)*+{\slashed{\mbb{G}}_{-1}^{0,0}}="s15",
(30,0)*+{\slashed{\mbb{G}}_{-2}^{0,0}}="s16",
"s16"; "s7" ; **@{-} ?>*\dir{>};
"s16"; "s10" ; **@{-} ?>*\dir{>};
"s16"; "s14" ; **@{-} ?>*\dir{>};
"s16"; "s15" ; **@{-} ?>*\dir{>};
"s10"; "s5" ; **@{-} ?>*\dir{>};
"s14"; "s5" ; **@{-} ?>*\dir{>};
"s15"; "s5" ; **@{-} ?>*\dir{>};
"s7"; "s5" ; **@{-} ?>*\dir{>};
\endxy
\end{align*}
For any $m>1$, any $[z:w] \in \CP^1 \setminus \{ [-2(m-1)\i,1], [2 \i, 1] \}$:
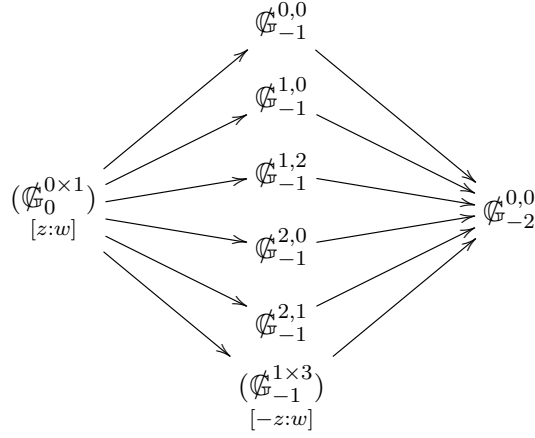
\begin{align*}
\xy
(0,0)*+{\underset{[z:w]}{(\slashed{\mbb{G}}_0^{0\times 1})}}="s5",
(30,-25)*+{\underset{[-z:w]}{(\slashed{\mbb{G}}_{-1}^{1\times 3})}}="s7",
(30,-15)*+{\slashed{\mbb{G}}_{-1}^{2,1}}="s10",
(30,-5)*+{\slashed{\mbb{G}}_{-1}^{2,0}}="s11",
(30,5)*+{\slashed{\mbb{G}}_{-1}^{1,2}}="s12",
(30,15)*+{\slashed{\mbb{G}}_{-1}^{1,0}}="s14",
(30,25)*+{\slashed{\mbb{G}}_{-1}^{0,0}}="s15",
(60,0)*+{\slashed{\mbb{G}}_{-2}^{0,0}}="s16",
"s16"; "s7" ; **@{-} ?>*\dir{>};
"s16"; "s10" ; **@{-} ?>*\dir{>};
"s16"; "s11" ; **@{-} ?>*\dir{>};
"s16"; "s12" ; **@{-} ?>*\dir{>};
"s16"; "s14" ; **@{-} ?>*\dir{>};
"s16"; "s15" ; **@{-} ?>*\dir{>};
"s10"; "s5" ; **@{-} ?>*\dir{>};
"s11"; "s5" ; **@{-} ?>*\dir{>};
"s12"; "s5" ; **@{-} ?>*\dir{>};
"s14"; "s5" ; **@{-} ?>*\dir{>};
"s15"; "s5" ; **@{-} ?>*\dir{>};
"s7"; "s5" ; **@{-} ?>*\dir{>};
\endxy
\end{align*}
For $m>1$, any $[-4\i:1] \neq [z:w] \in \CP^1$, and any $[x:y] \in \RP^1$:
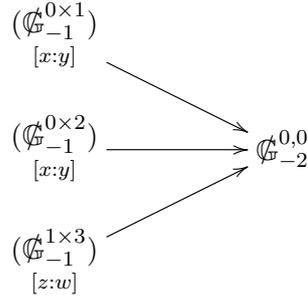
\begin{align*}
\xy
(30,15)*+{\underset{[x:y]}{(\slashed{\mbb{G}}_{-1}^{0\times 1})}}="s1",
(30,0)*+{\underset{[x:y]}{(\slashed{\mbb{G}}_{-1}^{0\times 2})}}="s2",
(30,-15)*+{\underset{[z:w]}{(\slashed{\mbb{G}}_{-1}^{1\times 3})}}="s3",
(60,0)*+{\slashed{\mbb{G}}_{-2}^{0,0}}="s4",
"s4"; "s1" ; **@{-} ?>*\dir{>};
"s4"; "s2" ; **@{-} ?>*\dir{>};
"s4"; "s3" ; **@{-} ?>*\dir{>};
\endxy
\end{align*}
\end{thm}

\begin{rem}
We note that
\begin{enumerate}
\item There is some redundancy in the inclusions: e.g.\
\begin{align*}
\slashed{\mbb{G}}_0^{1,1} \cap \slashed{\mbb{G}}_{-1}^{1,1} \cap  \slashed{\mbb{G}}_{-1}^{3,0} = \slashed{\mbb{G}}_0^{1,1} \cap \slashed{\mbb{G}}_{-1}^{1,1} = \slashed{\mbb{G}}_0^{1,1} \cap \slashed{\mbb{G}}_{-1}^{3,0}  \, ,
\end{align*}
and, for $m>2$,
\begin{align*}
\slashed{\mbb{G}}_1^{0,0} \subset (\slashed{\mbb{G}}_{0}^{0\times1})_{[-2(m-1) \i :1]} \cap  \slashed{\mbb{G}}_{0}^{1,2} \subset (\slashed{\mbb{G}}_{-1}^{1\times3})_{[2\i:1]} \cap  (\slashed{\mbb{G}}_{-1}^{1\times3})_{[2(m-1)\i:1]} = \slashed{\mbb{G}}_{-1}^{1,0} \cap \slashed{\mbb{G}}_{-1}^{3,0}  \, , 
\end{align*}
\item The only $Q$-submodules of $\mbb{G}$ that are not contained in $\slashed{\mbb{G}}_{-2}^{0,0}$ are $\slashed{\mbb{G}}_0^{1,1}  \subset  (\slashed{\mbb{G}}_{-1}^{1 \times 3})_{[4\i:-1]}$.
\end{enumerate}
\end{rem}

\begin{proof}
The theorem is a direct consequence of the following transformation rules of the quantities defined by \eqref{eq:S-H2} under the change \eqref{eq:basis-change}:
\begin{flalign*}
\gamma_\alpha & \mapsto \gamma_\alpha \, , & \\
\epsilon & \mapsto \epsilon +  \gamma_\alpha \phi^\alpha + \gamma^\alpha \phi_\alpha  \, , & \\
\tau^\omega  & \mapsto \tau^\omega  - \i \left( \gamma_{\alpha} \phi^{\alpha} - \gamma^{\alpha} \phi_{\alpha} \right) \, , & \\
\tau_{\alpha \beta} & \mapsto \tau_{\alpha \beta} - \gamma_{[\alpha} \phi_{\beta]} \, , & \\
 \tau^\circ_{\alpha \bar{\beta}} & \mapsto  \tau^\circ_{\alpha \bar{\beta}}  + \left( - \frac{1}{2} \gamma_{\alpha} \phi_{\bar{\beta}} + \frac{1}{2} \gamma_{\bar{\beta}} \phi_{\alpha} \right)_\circ
\, , & \\
\sigma_{\alpha \bar{\beta}} & \mapsto \sigma_{\alpha \bar{\beta}} + \left( \frac{1}{2} \gamma_{\alpha} \phi_{\bar{\beta}} + \frac{1}{2} \gamma_{\bar{\beta}} \phi_{\alpha} \right)_\circ 
\, , & \\
\sigma_{\alpha \beta} & \mapsto \sigma_{\alpha \beta} + \gamma_{(\alpha} \phi_{\beta)} \, , & \\
\zeta_{\alpha \beta} & \mapsto \zeta_{\alpha \beta} - 4 \, \i  \gamma_{[\alpha} \phi_{\beta]}  \, , &
\end{flalign*}
\begin{flalign*}
E_\alpha & \mapsto E_\alpha + \tau_{\alpha \beta} \phi^{\beta} -  \sigma_{\alpha \beta} \phi^\beta +  \tau^\circ_{\alpha \bar{\beta}}   \phi^{\bar{\beta}} - \sigma_{\alpha \bar{\beta}} \phi^{\bar{\beta}}  - \frac{\i}{2m} \tau^\omega  \phi_\alpha - \frac{1}{2m} \epsilon \phi_\alpha  -  \gamma_\alpha \phi^\beta \phi_\beta \, , & \\
G_{\alpha} 
& \mapsto G_{\alpha}  - 2 \,  \i  \tau^\circ_{\alpha \bar{\beta}} \phi^{\bar{\beta}}  + 2 \, \i \sigma_{\alpha \bar{\beta}} \phi^{\bar{\beta}}  + \frac{m-1}{m} \tau^\omega  \phi_{\alpha} - \frac{m-1}{m} \, \i \epsilon \, \phi_{\alpha} & \\
& \qquad \qquad + \phi^{\beta} \zeta_{\beta \alpha }  - 2 \, \i \phi^\beta \gamma_{\beta} \phi_{\alpha} + 2 \, \i \phi^\beta \phi_{\beta}  \gamma_\alpha \, , & \\
G^{\tiny{\ydskew}}_{\alpha \beta \gamma} & \mapsto G^{\tiny{\ydskew}}_{\alpha \beta \gamma} + \left( - 4 \,  \i \tau_{[\alpha \beta } + \zeta_{[\alpha \beta } \right) \phi_{\gamma]} \, , & \\
G^{\tiny{\ydhook}}_{(\alpha \beta) \gamma} & \mapsto G^{\tiny{\ydhook}}_{(\alpha \beta) \gamma } - \left( 2 \,  \i \tau_{\gamma (\alpha} 
 + \zeta_{\gamma (\alpha} \right) \phi_{\beta)}  
 + 2 \, \i \sigma_{\gamma (\alpha} \phi_{\beta)} 
- 2 \, \i \sigma_{\alpha \beta } \phi_{\gamma} 
 - 2 \, \i \phi_{(\alpha} \gamma_{\beta)} \phi_{\gamma} 
 + 2 \, \i \phi_{\alpha} \phi_{\beta} \gamma_{\gamma}  \, , & \\
 G^{\circ}_{\bar{\alpha} \beta \gamma }
& \mapsto  G^{\circ}_{\bar{\alpha} \beta \gamma } + \left( 4 \,  \i  \tau^\circ_{ [\beta|\bar{\alpha} }  \phi_{|\gamma ]}  - 4 \, \i  \sigma_{[\beta |\bar{\alpha} }  \phi_{|\gamma ]}  + \phi_{\bar{\alpha}} \zeta_{\beta \gamma }  - 4 \, \i \phi_{\bar{\alpha}} \gamma_{[\beta } \phi_{\gamma ]}  \right)_\circ \, , & \\
B_{\alpha \beta} & \mapsto B_{\alpha \beta} + \left( \frac{2}{m-1} G_{[\alpha}  - 4 \, \i E_{[\alpha} \right) \phi_{\beta]}    - \phi^\gamma G^{\tiny{\ydskew}}_{\gamma \alpha \beta} - \phi^\gamma G^{\tiny{\ydhook}}_{\gamma \alpha \beta} - \phi^{\bar{\gamma}} G^{\circ}_{\bar{\gamma} \alpha \beta}  \\
& \qquad \qquad  + 4 \i  \phi^\gamma \tau_{\gamma [\alpha} \phi_{\beta]} - \phi^\gamma \phi_\gamma \zeta_{\alpha \beta} 
+ 4 \, \i  \phi^\gamma \sigma_{\gamma [\alpha} \phi_{\beta]}  \\
& \qquad\qquad \qquad \qquad - 4 \i  \phi^{\bar{\gamma}}  \tau^\circ_{[\alpha|\bar{\gamma} }  \phi_{|\beta]} + 4 \, \i  \phi^{\bar{\gamma}} \sigma_{[\alpha|\bar{\gamma} } \phi_{|\beta]}   + 4 \i \phi^\gamma \phi_\gamma  \gamma_{[\alpha} \phi_{\beta]}  \, . &
\end{flalign*}
From these, we immediately deduce, for any $[x:y] \in \RP^1$, $[z:w] \in \CP^1$,
\begin{flalign*}
x \epsilon + y  \tau^\omega  & \mapsto x \epsilon + y  \tau^\omega  + (x - \i y ) \gamma_\alpha \phi^\alpha + (x + \i y )\gamma^\alpha \phi_\alpha   \, , & \\
 x \sigma_{\alpha \bar{\beta}} - y \i \tau^\circ_{\alpha \bar{\beta}} & \mapsto x \sigma_{\alpha \bar{\beta}}  - y \i \tau^\circ_{\alpha \bar{\beta}} + \left( - (x + y \i) \frac{1}{2} \gamma_{\alpha} \phi_{\bar{\beta}} +  (x - y \i ) \frac{1}{2} \gamma_{\bar{\beta}} \phi_{\alpha} \right)_\circ \, ,  &
\\
z \tau_{\alpha \beta} + w \zeta_{\alpha \beta} & \mapsto z \tau_{\alpha \beta} + w \zeta_{\alpha \beta} - ( z + 4 \i w) \gamma_{[\alpha} \phi_{\beta]} \, , & \\
z \,  E_\alpha + w \, G_{\alpha} & \mapsto
z \, E_\alpha + w \, G_{\alpha} -  z \, \phi^\beta \sigma_{\beta \alpha}  + \left( -  z \,  \tau_{\beta \alpha}  + w \, \zeta_{\beta \alpha } \right) \phi^{\beta} & \\
& \qquad - ( 2 w \,  \i - z)  \tau^\circ_{\alpha \bar{\beta}} \phi^{\bar{\beta}}  + (2 w \, \i - z) \sigma_{\alpha \bar{\beta}} \phi^{\bar{\beta}} & \\
& \qquad \qquad - \frac{\i}{2m}  (2 (m-1) \i w + z)  \tau^\omega  \phi_{\alpha}   -  \frac{1}{2m}  (2 (m-1) \i w + z) \epsilon \, \phi_{\alpha}  &  \\
& \qquad \qquad \qquad - 2 w \, \i \phi^\beta \gamma_{\beta} \phi_{\alpha} + ( 2 w \, \i - z) \phi^\beta \phi_{\beta}  \gamma_\alpha \, . &
\end{flalign*}
The result follows as in the proof of Proposition 3.4 of \cite{Fino2020}.

In dimension six, one has
\begin{align*}
\phi^\gamma \phi_\gamma \tau_{\alpha \beta} &= - 2 \phi^\gamma \tau_{\gamma [\alpha}\phi_{\beta]} \, , &
G^{\tiny{\ydskew}}_{\gamma \alpha \beta} & = 0 \, , &
G^{\circ}_{\bar{\gamma} \alpha \beta} & = 0 \, ,
\end{align*}
so that
\begin{align*}
B_{\alpha \beta} & \mapsto B_{\alpha \beta} + \left( 2 G_{[\alpha}  - 4 \, \i E_{[\alpha} \right) \phi_{\beta]}    - \phi^\gamma G^{\tiny{\ydhook}}_{\gamma \alpha \beta} - \left(  2 \i  \tau_{\alpha \beta}  +  \zeta_{\alpha \beta} \right) \phi^\gamma \phi_\gamma \\
& \qquad  
+ 4 \i  \phi^\gamma \sigma_{\gamma [\alpha} \phi_{\beta]}   - 4 \i  \phi^{\bar{\gamma}}  \tau^\circ_{[\alpha|\bar{\gamma} } \phi_{|\beta]} + 4 \i  \phi^{\bar{\gamma}} \sigma_{[\alpha|\bar{\gamma} } \phi_{|\beta]}   + 4 \i \phi^\gamma \phi_\gamma  \gamma_{[\alpha} \phi_{\beta]}  
\end{align*}
This completes the proof. Some of these computations were verified using the symbolic computer algebra system \texttt{cadabra} \cite{Peeters2007a,Peeters2018}.
\end{proof}

Finally, for future use, and to make contact with the intrinsic torsion of an optical structure given in \cite{Fino2020}, we also define
\begin{align}
\begin{aligned}\label{eq:G_sl_optical}
\slashed{\mbb{G}}_{-2}^0& := \slashed{\mbb{G}}_{-2}^{0,0} \, , \\
\slashed{\mbb{G}}_{-1}^0& := \slashed{\mbb{G}}_{-1}^{0,0} \, , &
\slashed{\mbb{G}}_{-1}^1& := \slashed{\mbb{G}}_{-1}^{1,0} \cap \slashed{\mbb{G}}_{-1}^{1,1} \cap \slashed{\mbb{G}}_{-1}^{1,2} \, , & 
\slashed{\mbb{G}}_{-1}^2 & := \slashed{\mbb{G}}_{-1}^{2,0} \cap \slashed{\mbb{G}}_{-1}^{2,1}  \, , \\
\slashed{\mbb{G}}_{0}^0& := \slashed{\mbb{G}}_{0}^{0,0} \, .
\end{aligned}
\end{align}




\section{Almost Robinson manifolds}\label{sec:geometry}
\subsection{Almost Robinson structures}
Throughout we shall follow the notation and conventions of Section \ref{sec:algebra} now translated into the bundle setting.

\begin{defn}\cite{Nurowski2002,Taghavi-Chabert2016}
Let $(\mc{M},g)$ be an oriented pseudo-Riemannian manifold of dimension $2m+2$. An \emph{almost null structure} on $(\mc{M},g)$ is a complex distribution $N$ of rank $(m+1)$ and totally null with respect to ${}^\C g$. When $N$ is involutive, i.e.\ $[N,N] \subset N$, we call $N$ a \emph{null structure}.
\end{defn}
In other words, an almost null structure is a smooth assignement of a null structure to the tangent space at each point, and according to Definition \ref{def:real_index}, one may talk of the real index of an almost null structure at a point. When $g$ is of Lorentzian signature, we make the following definition.

\begin{defn}
Let $(\mc{M},g)$ be an oriented and time-oriented Lorentzian manifold of dimension $2m+2$. An \emph{almost Robinson structure} on $(\mc{M},g)$ consists of a pair $(N,K)$ where $N$ is a complex distribution of rank $m+1$ totally null with respect to ${}^\C g$, and $K$ a real line distribution such that ${}^\C K = N \cap \overline{N}$.  We shall refer to the quadruple $(\mc{M},g,N, K)$ as an \emph{almost Robinson manifold} or \emph{almost Robinson geometry}.

In addition, we call $(N,K)$
\begin{itemize}
\item a \emph{nearly Robinson structure} when $[K,N] \subset N$, and
\item a \emph{Robinson structure} when $[N,N] \subset N$, i.e.\ $N$ is involutive.
\end{itemize}
We shall accordingly refer to $(\mc{M},g,N, K)$ as a \emph{nearly Robinson manifold} or as a \emph{Robinson manifold}.
\end{defn}
Clearly, a Robinson manifold is a nearly Robinson manifold. Definitions involving weaker assumptions on orientability are possible.

\begin{rem}
Equivalently put, an almost Robinson structure is an almost null structure of real index one. By Lemma \ref{lem:real_index}, any almost null structure on a Lorentzian manifold defines an almost Robinson manifold. The terminology `almost null structure' will nevertheless be preferred in the case when we wish to emphasize the geometric aspects of the almost Robinson structure not particularly tied to the geometry of the real null line distribution $K$, as will be done in Section \ref{sec:spinor_des}.
\end{rem}

An almost Robinson structure $(N,K)$ induces an optical structure on $(\mc{M},g)$ in the sense of \cite{Fino2020}, namely the filtration of vector bundles
\begin{align}\label{eq:K_Kperp_TM}
 K \subset K^\perp \subset T \mc{M} \, .
\end{align}
The orientation and time-orientation on $\mc{M}$ induce an orientation on $K$, and the screen bundle  $H_K := K^\perp / K$ of $K$ inherits a positive-definite bundle metric $h$ from $g$.  Any section of $K$ will be referred to as an \emph{optical vector field}, while any section of $\Ann(K^\perp)$ will be referred to as an \emph{optical $1$-form}.

In addition, there is a bundle complex structure $J$ on the screen bundle $H_K$ compatible with $h$, which induces a splitting of its complexification
\begin{align*}
{}^\C H_K & = H^{(1,0)}_K \oplus H^{(0,1)}_K \, , & {}^\C H_K^* & = (H^{(1,0)}_K)^* \oplus (H^{(0,1)}_K)^* \, , &
\end{align*}
where $H^{(1,0)}_K$ and $H^{(0,1)}_K$ denote the $+\i$- and $-\i$-eigenbundles of $J$ respectively. In abstract index notation, we shall denote the bundle complex structure and the bundle Hermitian structure on $H_K$  by $J_i{}^j$  and  $\omega_{i j}$ respectively, so that $\omega_{i j} = J_{i}{}^{k} h_{k j}$. Following the notation of Section \ref{sec:Notation}, we also define the bundles, for any non-negative integer $p$, $q$,
\begin{align*}
\bigwedge^{(p,q)} H_{K}^*  & :=  \bigwedge^p (H^{(1,0)}_{K})^* \otimes  \bigwedge^q (H^{(0,1)}_{K})^* \, , \\
\bigodot{}^{(p,q)} H_{K}^*  & :=  \bigodot{}^p (H^{(1,0)}_{K})^* \otimes  \bigodot{}^q (H^{(0,1)}_{K})^* \, .
\end{align*}
For $pq \neq0$, the subbundles of elements of $\bigwedge^{(p,q)} H_{K}^*$ and $\bigodot{}^{(p,q)} H_{K}^*$ that are tracefree with respect to the bundle Hermitian structure will be denoted by $\bigwedge^{(p,q)}_\circ H_{K}^*$ and $\bigodot{}^{(p,q)}_\circ H_{K}^*$ respectively. Similarly, we introduce the subbundle $\ydhook H_{K}^*$ as a bundle analogue of \eqref{eq:ydhook}. The corresponding real spans of these complex bundles will be enclosed between $\dbl \cdot \dbr$ or $[ \cdot ]$ as described by \eqref{eq:Salamonotation}.

As in Section \ref{sec:algebra}, we split the complexified tangent bundle as
\begin{align*}
{}^\C T \mc{M} = N^* \oplus N
\end{align*}
for some chosen complement $N^*$ of $N$ in ${}^\C T \mc{M}$, dual to $N$ via ${}^\C g$. This splitting is not canonical in general. This induces a splitting of the filtration \eqref{eq:K_Kperp_TM}
\begin{align}\label{eq:splitTM}
T \mc{M} & = L \oplus H_{K,L} \oplus K \, ,
\end{align}
where
\begin{align*}
 L & := N^* \cap T \mc{M} \, , &  H_{K,L} & := K^\perp \cap L^\perp \, .
\end{align*}
Note that $N^*$ defines the almost Robinson structure $(N^*, L)$ on $(\mc{M}, g)$, where $L$ is the real span of $N^* \cap \overline{N^*}$ and is dual to $K$.
In addition, 
\begin{align*}
{}^\C H_{K,L} & = H^{(1,0)}_{K,L} \oplus H^{(0,1)}_{K,L} \, , & {}^\C H^*_{K,L} & = (H^{(1,0)}_{K,L})^* \oplus (H^{(0,1)}_{K,L})^* \, .
\end{align*}
where
\begin{align*}
 H_{K,L}^{(1,0)} & =  \overline{N} \cap N^* \, , &   H_{K,L}^{(0,1)} & =  N \cap \overline{N^*}  \, .
\end{align*}
We also have isomorphisms of vector bundles $H_{K,L} \cong H_K$, $H^{(1,0)}_{K,L} =  H^{(1,0)}_{K}$, and so on, which  depend on the choice of $N^*$.

The splitting operators and their duals, introduced in Section \ref{sec:algebra}, that is,
\begin{align*}
(\delta^a_A, \delta^{a A}) \, , &&  (\ell^a, \delta^a_i , k^a) \, , && ( \delta^i_\alpha , \delta^i_{\bar{\alpha}} ) \, , && (\ell^a, \delta^a_\alpha , \delta^a_{\bar{\alpha}} , k^a) \, , \\
(\delta_a^A, \delta_{a A}) \, , && (\kappa_a, \delta_a^i , \lambda_a) \, , && ( \delta_i^\alpha , \delta_i^{\bar{\alpha}} ) \, , && (\kappa_a, \delta_a^\alpha , \delta_a^{\bar{\alpha}} , \lambda_a) \, ,
\end{align*}
will be used throughout the article to convert index types, with the convention that $k^a \alpha_a = \kappa_a \alpha^a = \alpha^0$ and $\ell^a \alpha_a = \lambda_a \alpha^a = \alpha_0$ for any $1$-form $\alpha_a$.

In order to avoid ambiguity when taking components of the covariant derivative of some tensor $\alpha_{b \ldots d}$, we shall often write
\begin{align*}
(\nabla \alpha)_{a b \ldots d} & := \nabla_a \alpha_{b \ldots d} \, .
\end{align*}
For instance,
\begin{align*}
(\nabla \alpha)^0{}_{\alpha 0 i} & = k^a  \delta^b_\alpha \ell^c \delta^d_i \left( \nabla_a \alpha_{b c d} \right) \, .
\end{align*}

The splitting operators will also be used as injectors. Thus, if $\omega_{i j}$ is the Hermitian $2$-form on $H_K$, we can set $\omega_{a b} = \omega_{i j} \delta^i_a \delta^j_b$ for some chosen splitting operators, and construct the \emph{Robinson $3$-form} $\rho_{a b c} := 3 \kappa_{[a} \omega_{b c]}$ associated to the optical $1$-form $\kappa_a$.

\subsection{Pure spinors}
Whenever $(\mc{M},g)$ is assumed to be spin, we introduce a spin bundle $S$ and translate the theory of spinors summarised in Section \ref{sec-spinors} to the language of bundles. We shall denote the spinor bundle by $S$ and its irreducible parts by $S_+$ and $S_-$. We shall not distinguish notationally between the Levi-Civita connection $\nabla$ and the induced spin connection. We now view the van der Waerden symbols $\gamma_{a \mbf{A}'}{}^{\mbf{B}}$ and $\gamma_{a \mbf{A}}{}^{\mbf{B}'}$, the bilinear forms \eqref{eq:Clm+1HD} as fields on $\mc{M}$ compatible with $\nabla$. There is also an antilinear map on each fiber of $S$, induced from the reality structure on $({}^\C T \mc{M} , {}^\C g)$, and thus, a notion of charge conjugate of a spinor field.

An almost Robinson structure $(N, K)$, where we assume for specificity that $N$ is self-dual, can therefore be expressed by a non-vanishing section $\nu^{\mbf{A}'}$ of $S_+$ that is pure at every point, i.e.\ the kernel of the map
\begin{align*}
\nu_a^{\mbf{A}}  & := \gamma_{a \mbf{B}'}{}^{\mbf{A}} \nu^{\mbf{B}'} : \Gamma(T \mc{M}) \rightarrow \Gamma(S^-)
\end{align*}
is precisely $N$. Since the kernel $N$ of the map $\nu_a^{\mbf{A}}$ is invariant under rescaling of the spinor $\nu^{\mbf{A}'}$, the almost Robinson structure is in fact equivalent to the existence of a complex line subbundle $S_+^N$ of $S_+$, which is spanned by $\nu^{\mbf{A}'}$. The bundle $S_+^N$ can be viewed as a square root of the line bundle $\bigwedge^{m+1} \Ann(N)$. Indeed, we have an isomorphism of bundles
\begin{align}\label{eq:sq_pure_spin}
\bigwedge^{m+1} \Ann(N) \cong S_+^N \otimes S_+^N \, .
\end{align}
Any spinor $\nu^{\mbf{A}'}$ annihilating the almost null structure $N$ will be referred to as a \emph{Robinson spinor}, and any section of $\bigwedge^{m+1} \Ann(N)$ a \emph{complex Robinson $(m+1)$-form}.
 
A completely parallel analysis can be carried out starting with the charge conjugate of $\nu^{\mbf{A}'}$, which spans the complex conjugate bundle $\overline{S_+^N}$.  The invariants $1$-form $\kappa_a$, $3$-form $\rho_{a b c}$, $(m+1)$-forms $\nu_{a_1 \ldots a_{m+1}}$ and $\overline{\nu}_{a_1 \ldots a_{m+1}}$ of the almost Robinson structure can then be recovered from $\nu^{\mbf{A}'}$ and its charge conjugate using \eqref{arr:rob_spin_inv}.

We are now in the position of stating the direct translation of Proposition \ref{prop:char-Robinson} into the language of manifolds:
\begin{prop}\label{prop:char-Robinson-mfld}
Let $(\mc{M},g)$ be an oriented and time-oriented smooth Lorentzian manifold of dimension $2m+2$. The following statements are equivalent.
\begin{enumerate}
\item $(\mc{M},g)$ is endowed with an almost Robinson structure $(N,K)$.
\item $(\mc{M},g)$ admits a simple totally null complex $(m+1)$-form.
\item $(\mc{M},g)$ is endowed with an optical structure $K$ whose screen bundle $(H_K , h)$ is equipped with a bundle complex structure compatible with the bundle metric $h$.
\item $(\mc{M},g)$ admits a null $1$-form $\kappa_a$ and a $3$-form $\rho_{a b c}$ such that
\begin{align}\label{eq:rho2->kap}
\rho _{ab} \, {}^e \rho _{cde} & = - 4 \kappa_{[a} g_{b][c} \kappa_{d]} \, .
\end{align}
\item when $(\mc{M},g)$ is spin, it admits a pure spinor  (of real index one).
\end{enumerate}
\end{prop}

\begin{rem}
In general, if one assumes that $(\mc{M},g)$ is spin, any geometric statement on a complex Robinson $(m+1)$-form can equivalently be expressed in terms of a Robinson spinor.
\end{rem}

\subsection{Almost Robinson structures as $G$-structures}
From the discussion of \cite{Fino2020}, we shall also view an almost Robinson structure as a reduction of the frame bundle to the $Q$-bundle $\mc{F}^Q$ where $Q = (\R_{>0} \times \U(m))\ltimes (\R^{2m})^*$, and given any $Q$-module $\mbb{A}$, we will construct associated vector bundles $\mc{F}^Q (\mbb{A}) :=\mc{F}^Q \times_Q \mbb{A}$. Similarly, a choice of splitting gives rise to the $Q_0$-invariant vector bundles where $Q_0 = \R_{>0} \times \U(m)$. It will often be more convenient to deal with the reduced coframe bundle ${\mc{F}^*}^Q$. A section of ${\mc{F}^*}^Q$ will then consist of a null complex coframe $(\kappa, \theta^{\alpha}, \overline{\theta}{}^{\bar{\alpha}}, \lambda)$ such that
\begin{enumerate}
\item $\kappa$ annihilates $K^\perp$, or equivalently, \ $\kappa = g(k, \cdot)$ for some section $k$ of $K$;
\item $(\kappa, \theta^\alpha )$  annililate $N$, or equivalently, $(\kappa, \overline{\theta}{}^{\bar{\alpha}})$ annililate $\overline{N}$;
\item $(\theta^{\alpha})$ are unitary with respect to the screen bundle metric;
\item the metric takes the form
\begin{align*}
g & = 2 \kappa \lambda + 2 h_{\alpha \bar{\beta}} \theta^{\alpha} \overline{\theta}{}^{\bar{\beta}} \, .
\end{align*}
\end{enumerate}
We shall refer to $(\kappa, \theta^{\alpha}, \overline{\theta}{}^{\bar{\alpha}}, \lambda)$ as a \emph{Robinson coframe}.

Any two Robinson coframes $(\kappa, \theta^{\alpha}, \overline{\theta}{}^{\bar{\alpha}}, \lambda)$ and $(\wh{\kappa}, \wh{\theta}^{\alpha}, \overline{\wh{\theta}}{}^{\bar{\alpha}}, \wh{\lambda})$ are related by a transformation of the form
\begin{align}\label{eq:Rob_cof_transf}
\wh{\kappa} & = \e^{\varphi} \kappa \, , & \wh{\theta}^{\alpha} & = \psi_{\beta}{}^{\alpha}\theta^{\beta} + \phi^\alpha \kappa  \, , &\wh{\lambda} & = \e^{-\varphi} \left(\lambda - \psi_{\alpha}{}^{\beta} \phi_{\beta}\theta^{\alpha} - \psi_{\bar{\alpha}}{}^{\bar{\beta}} \phi_{\bar{\beta}} \theta^{\bar{\alpha}} - \phi_{\alpha}  \phi^{\alpha} \kappa\right) \, , 
\end{align}
where $\varphi$ is a smooth real-valued function, and $\phi^{\alpha}$, $\psi_{\beta}{}^{\alpha}$ are smooth complex-valued functions on $\mc{M}$ with $\psi_{\alpha}{}^{\beta}$ being a $\U(m)$-transformation at any point, i.e.\ $h_{\alpha \bar{\beta}} = h_{\gamma \bar{\delta}} \psi_{\alpha}{}^{\gamma} \psi_{\bar{\beta}}{}^{\bar{\delta}}$, and $\phi_{\alpha} \phi^{\alpha} = h_{\alpha \bar{\beta}} \phi^{\alpha} \phi^{\bar{\beta}}$.

Associated to the representations $\R(w)$ and $\C(w,w')$ defined in Section \ref{sec:one-dim_rep}, where $w$ and $w'$ are real, we define the bundle $\mc{E}(w)$ of \emph{boost densities} of weight $w$, already introduced in \cite{Fino2020}, and the bundle $\mc{E}(w,w')$ of \emph{boost-spin densities} of weight $(w,w')$. In particular, we have the identifications
\begin{align}
K & \cong \mc{E}(-1) \, , & L & \cong \mc{E}(1) \, , \nonumber \\
\mc{E}(-1,0) & : = \bigwedge^{m+1} \Ann(N)  \, , &
 \mc{E}(0,-1) & := \overline{\mc{E}(-1,0)} = \bigwedge^{m+1} \Ann(\overline{N}) \, . \label{eq:complex_boost}
\end{align}

If $(\mc{M},g)$ is assumed to be spin, we define the smooth complex line bundles
\begin{align}
\begin{aligned}\label{eq:spin_complex_boost}
\mc{E}(\tfrac{1}{2},0) & := \left(S_+^N \right)^{-1} \, , & \mc{E}(0,\tfrac{1}{2}) & := \left(\overline{S_+^N}\right)^{-1} \, , \\
\mc{E}(-\tfrac{1}{2},0) & := \left(\mc{E}(\tfrac{1}{2},0)\right)^* \, , & \mc{E}(0,-\tfrac{1}{2}) & := \left(\mc{E}(0,\tfrac{1}{2})\right)^* \, .
\end{aligned}
\end{align}
In particular, we recover \eqref{eq:complex_boost} by virtue of \eqref{eq:sq_pure_spin}.
Our definitions are consistent with those of the real line bundles $\mc{E}(w)$ in the sense that
\begin{align*}
\mc{E}(w)& = \mc{E}(\tfrac{w}{2},\tfrac{w}{2}) / S^1 \, , & \mbox{for any real $w$.} 
\end{align*}
This can be readily be checked using \eqref{arr:rob_spin_inv}.

\subsection{Intrinsic torsion}\label{sec:Intrinsic_Torsion}
As explained in Section 2 of \cite{Fino2020}, the intrinsic torsion of an almost Robinson structure is given by a section 
$T^*\mc{M} \otimes \mc{F}^{Q}(\g/\mf{q})$, which we identify with the $Q$-invariant subbundle $\mc{G} := \mc{F}^Q \times_Q \mbb{G}$ where $\mbb{G} := \V^* \otimes \g / \mf{q}$. We shall accordingly call $\mc{G}$ the bundle of intrinsic torsions of $(N,K)$. Its $Q_0$-invariant subbundles and $Q$-invariant subbundles will presently be defined with reference to Section \ref{sec:algebra}, and in particular Theorem \ref{thm-main-Rob}.

The filtration \eqref{eq:filt-q-G} on $\mbb{G}$ induces a filtration of $Q$-invariant subbundles
\begin{align*}
\mc{G} & =: \mc{G}^{-2} \supset \mc{G}^{-1} \supset \mc{G}^0 \supset \mc{G}^{-1} \, ,
\end{align*}
where $\mc{G}^i := \mc{F}^Q \times_Q \mbb{G}^i$. Correspondingly, the associated graded vector bundle
\begin{align*}
\gr(\mc{G}) = \gr_{-2} (\mc{G}) \oplus  \gr_{-1} (\mc{G}) \oplus \gr_0 (\mc{G}) \oplus  \gr_1 (\mc{G}) \, ,
\end{align*}
where $\gr_i (\mc{G}) := \mc{F}^Q \times_Q \gr_i(\mbb{G})$, splits into irreducible $Q$-invariant subbundles $\gr_i^{j,k} (\mc{G}) := \mc{F}^Q \times_{Q_0} \gr_i^{j,k} (\mbb{G})$. For each choice of splitting, these are isomorphic to the $Q_0$-invariant subbundles
\begin{align*}
\mc{G}_i^{j,k} & := \mc{F}^Q \times_{Q_0} \mbb{G}_i^{j,k} \, ,
\end{align*}
and we introduce further
\begin{align*}
(\mc{G}_{-1}^{0 \times 1})_{[x:y]} & := \mc{F}^Q \times_{Q_0} (\mbb{G}_{-1}^{0 \times 1})_{[x:y]} \, , & \mbox{for each $[x : y] \in \RP^1$,} \\
(\mc{G}_{-1}^{1 \times 2})_{[x:y]}  & := \mc{F}^Q \times_{Q_0} (\mbb{G}_{-1}^{1 \times 2})_{[x:y]} \, , & \mbox{for each $[x : y] \in \RP^1$,} \\
(\mc{G}_{-1}^{1 \times 3})_{[z:w]}  & := \mc{F}^Q \times_{Q_0} (\mbb{G}_{-1}^{1 \times 3})_{[z:w]} \, , & \mbox{for each $[z:w] \in \CP^1$,}  \\
(\mc{G}_{0}^{0 \times 1})_{[z:w]} &  := \mc{F}^Q \times_{Q_0} (\mbb{G}_{0}^{0 \times 1})_{[z:w]} \, , & \mbox{for each $[z:w] \in \CP^1$,} 
\end{align*}
with reference to  Table \ref{tab:table-irred-mod-G} and  \eqref{eq:extra_irr_G}.

We correspondingly define the $Q$-invariant subbundles
\begin{align*}
\slashed{\mc{G}}_i^{j,k}& := \mc{F}^Q \times_Q \slashed{\mbb{G}}_i^{j,k} \, , \\
(\slashed{\mc{G}}_{-1}^{0 \times 1})_{[x:y]} & := \mc{F}^Q \times_Q (\slashed{\mbb{G}}_{-1}^{0 \times 1})_{[x:y]} \, , & \mbox{for each $[x : y] \in \RP^1$,} \\
(\slashed{\mc{G}}_{-1}^{1 \times 2})_{[x:y]}  & := \mc{F}^Q \times_Q (\slashed{\mbb{G}}_{-1}^{1 \times 2})_{[x:y]} \, , & \mbox{for each $[x : y] \in \RP^1$,} \\
(\slashed{\mc{G}}_{-1}^{1 \times 3})_{[z:w]}  & := \mc{F}^Q \times_Q (\slashed{\mbb{G}}_{-1}^{1 \times 3})_{[z:w]} \, , & \mbox{for each $[z:w] \in \CP^1$,}  \\
(\slashed{\mc{G}}_{0}^{0 \times 1})_{[z:w]} &  := \mc{F}^Q \times_Q (\slashed{\mbb{G}}_{0}^{0 \times 1})_{[z:w]} \, , & \mbox{for each $[z:w] \in \CP^1$,} 
\end{align*}
where the $Q$-modules are all defined in Theorem \ref{thm:intors-rob}.

Finally, in order to make contact with the intrinsic torsion of an optical geometry described in  \cite{Fino2020}, we define, for each $i, j$,
\begin{align*}
\mc{G}_i^{j} & := \mc{F}^P \times_{P_0} \mbb{G}_i^{j} \, , & 
\slashed{\mc{G}}_i^{j} & := \mc{F}^P \times_P \slashed{\mbb{G}}_i^{j} \, ,
\end{align*}
where $\mbb{G}_i^{j}$ are defined in Proposition \ref{prop-Q2P}, and $\slashed{\mbb{G}}_i^{j}$ are defined by \eqref{eq:G_sl_optical}.

To determine the algebraic properties of the intrinsic torsion of an almost Robinson structure, we choose an optical $1$-form $\kappa_a$ and its associated Robinson $3$-form $\rho_{a b c}$, compute their respective covariant derivatives $\nabla_a \kappa_b$ and $\nabla_a \rho_{b c d}$, and project these tensors onto the various $Q_0$-invariant subbundles of $\mc{G}$ by means of splitting operators $(\ell^a , \delta^a_{\alpha}, \delta^a_{\bar{\alpha}}, k^a)$. This is achieved by extending the projections $\Pi_i^{j,k}$ defined in Appendix \ref{app:proj} to bundle projections, where we identify the tensor $\Gamma_{a b}{}^c$ as a connection $1$-form of $\nabla_a$ adapted to the splitting operators. We shall make use of the notation already introduced in Section \ref{sec:algebra}, mirroring the definition \eqref{eq:S-H2}. Let us set
\begin{align}
\begin{aligned}\label{eq:S-H1_b}
&& \gamma_i & := (\nabla \kappa)^{0}{}_i \, , \\
\epsilon & :=  (\nabla \kappa )_{i j} h^{i j} \, , &
\tau_{ij} & := (\nabla \kappa )_{[i j]} \, , &
\sigma_{ij} & := (\nabla \kappa )_{(i j)_\circ} \, , \\
&& E_i & := ( \nabla \kappa )_{0 i} \, , 
\end{aligned}
\end{align}
These components split further into irreducible $Q_0$-components. In particular, the component $\tau_{\alpha \bar{\beta}}$ splits as
\begin{align*}
\tau_{\alpha \bar{\beta}} & =  \frac{\i }{2 m} \tau^\omega  h_{\alpha \bar{\beta}}  + \tau^\circ_{\alpha \bar{\beta}} \, ,
\end{align*}
where
\begin{align*}
\tau^\omega  & := \omega^{i j} \tau_{ij} = - 2 \i \tau_{\alpha \bar{\beta}} h^{\alpha \bar{\beta}} \, , &
\tau^\circ_{\alpha \bar{\beta}} & := \left( \tau_{\alpha \bar{\beta}} \right)_\circ\, , 
\end{align*}
In addition, the components of the covariant derivative of the Robinson $3$-form $\rho_{a b c}$ of interest are
\begin{align}\label{eq:S-H2_b}
\zeta_{\alpha \beta} & :=  ( \nabla \rho )^{0}{}_{\alpha \beta 0}  \, , \\
G_{i \beta \gamma} & :=  (\nabla \rho )_{i \beta \gamma 0}   \, , &
B_{\alpha \beta} & :=  (\nabla \rho )_{0 \alpha \beta 0}  \, .
\end{align}
the second of which splits further into the irreducible components
\begin{align*}
G^{\tiny{\ydskew}}_{\alpha \beta \gamma} & :=  G_{[\alpha \beta \gamma]} \, , &
G^{\tiny{\ydhook}}_{\alpha \beta \gamma} & :=  \frac{2}{3} \left( G_{(\alpha \beta) \gamma} - G_{(\alpha \gamma) \beta} \right) \, , \\
G^{\circ}_{\bar{\gamma} \alpha \beta}  & := \left( G_{\bar{\gamma} \alpha \beta}  \right)_\circ  \, , &
G_{\alpha} & := h^{\beta \bar{\gamma}} G_{\bar{\gamma} \beta \alpha}    \, , 
\end{align*}
so that
\begin{align*}
G_{\alpha \beta \gamma} & = G^{\tiny{\ydskew}}_{\alpha \beta \gamma} + G^{\tiny{\ydhook}}_{\alpha \beta \gamma} \, , &
G_{\bar{\alpha} \beta \gamma} & = G^{\circ}_{\bar{\alpha} \beta \gamma} - \frac{2}{m-1} G_{[\beta} h_{\gamma] \bar{\alpha} } \, , \\
G_{i \beta \gamma} & = \delta^{\alpha}_i G_{\alpha \beta \gamma} + \delta^{\bar{\alpha}}_i G_{\bar{\alpha} \beta \gamma} \, .
\end{align*}
The other relevant components of $\nabla_a \rho_{b c d}$ can be found below:
\begin{flalign*}
  (\nabla \rho)_{a}{}^{0}{}_{j k} & = 0 \, , & 
 (\nabla \rho)_{a \alpha \beta \gamma} & =  0 \, , &
\end{flalign*}
\begin{flalign*}
(\nabla \rho)^{0 0}{}_{0 \alpha} & = - \i \gamma_{\alpha} \, , & 
(\nabla \rho)^0{}_{\alpha \beta \bar{\gamma}} & =  2 \i \gamma_{[\alpha} h_{\beta] \bar{\gamma}} \, , & \\
(\nabla \rho)_{\alpha \beta}{}^{0}{}_{0} & = - \i \tau_{\alpha \beta}  - \i \sigma_{\alpha \beta}  \, , & 
(\nabla \rho)_{\alpha \beta \gamma \bar{\delta}} & =  2 \i \left( \tau_{\alpha [\beta} + \sigma_{\alpha [\beta} \right) h_{\gamma] \bar{\delta}}\, , & \\
(\nabla \rho)_{\bar{\alpha} \beta}{}^{0}{}_{0} & = - \i \frac{1}{2m} \epsilon h_{\beta \bar{\alpha}} + \i \tau_{\beta \bar{\alpha}} - \i \sigma_{\beta \bar{\alpha}}  \, , & 
(\nabla \rho)_{\bar{\alpha} \beta \gamma \bar{\delta}}  & =  2 \i \left( \frac{1}{2m} \epsilon h_{[\beta| \bar{\alpha} } - \tau_{[\beta| \bar{\alpha}} + \sigma_{[\beta| \bar{\alpha}} \right) h_{|\gamma] \bar{\delta}}\, , & \\
 (\nabla \rho)_{0 \alpha}{}^{0}{}_{0} & = - \i E_{\alpha} \, , &
 (\nabla \rho)_{0 \alpha \beta \bar{\gamma}} & =  2 \i E_{[\alpha} h_{\beta] \bar{\gamma}}  \, , &
 \end{flalign*}
\begin{flalign*}
(\nabla \rho)^0{}_{0 \beta \bar{\gamma}} & =  \i (\nabla \kappa)^{0}{}_{0}  h_{\beta \bar{\gamma}} \, , &
(\nabla \rho)_{i 0 \beta \bar{\gamma}}  & = \i (\nabla \kappa)_{i 0}  h_{\beta \bar{\gamma}} \, , & 
(\nabla \rho)_{0 0 \beta \bar{\gamma}}  & = \i (\nabla \kappa)_{0 0} h_{\beta \bar{\gamma}} \, . &
\end{flalign*}
The elements defined above should be regarded as trivialisations of sections of $\mc{G}_{i}^{j,k}$. Such sections generally carry a boost weight. These will be adorned with a breve accent. To be clear, we have collected them in Tables \ref{tab:table-irred-mod-P_bdl} and \ref{tab:table-irred-mod-G_bdl} --- the bracket notation $\dbl \cdot \dbr$ and $[ \cdot ]$ therein should be understood as taking the real span of the quantities enclosed therebetween. For instance,
\begin{align*}
\dbl \breve{G}^{\tiny{\ydskew}}_{\alpha \beta \gamma} \dbr & \sim \breve{G}^{\tiny{\ydskew}}_{\alpha \beta \gamma} \delta^{\alpha}_i \delta^{\beta}_j \delta^{\gamma}_k + \breve{G}^{\tiny{\ydskew}}_{\bar{\alpha} \bar{\beta} \bar{\gamma}} \delta^{\bar{\alpha}}_i \delta^{\bar{\beta}}_j \delta^{\bar{\gamma}}_k \, ,  
& \mbox{and} & & 
[ \breve{\tau}^\circ_{\alpha \bar{\beta}} ] & \sim 2 \breve{\tau}^\circ_{\alpha \bar{\beta}} \delta^{\alpha}_{[i} \delta^{\bar{\beta}}_{j]} \, , 
\end{align*}
and so on.

\begin{center}
\begin{table}
\begin{displaymath}
{\renewcommand{\arraystretch}{1.5}
\begin{array}{|c|c|c|c|}
\hline
\text{$P_0$-bundle} & \text{Description} & \text{Tensor} \\
\hline
\hline
  \mc{G}_{-2}^{0} & H^*_{K,L} \otimes \mc{E} (2) &\breve{\gamma}_i  \\
  \hline
  \hline
  \mc{G}_{-1}^{0} & \mc{E}(1) & \breve{\epsilon}  \\
    \hline
  \mc{G}_{-1}^{1} & \bigwedge^{2} H^*_{K,L} \otimes \mc{E} (1)  & \breve{\tau}_{ij} \\
  \hline 
    \mc{G}_{-1}^{2} & \bigodot^{2}_\circ H^*_{K,L} \otimes \mc{E}  (1) & \breve{\sigma}_{ij}  \\
  \hline
  \hline
  \mc{G}_0^{0} & H^*_{K,L}  & \breve{E}_i \\
\hline
\end{array}}
\end{displaymath}
\caption{\label{tab:table-irred-mod-P_bdl}Irreducible $P_0$-submodules of $\mc{G}$}
\end{table}
\end{center}

\begin{center}
\begin{table}[H]
\begin{displaymath}
{\renewcommand{\arraystretch}{1.5}
\begin{array}{|c|c|c|c|}
\hline
\text{$Q_0$-bundle} & \text{Description} & \text{Tensor} \\
\hline
\hline
  \mc{G}_{-2}^{0,0} &H^*_{K,L} \otimes \mc{E} (2) & \breve{\gamma}_i \\
  \hline
  \hline
  \mc{G}_{-1}^{0,0} & \mc{E}(1) & \breve{\epsilon}   \\
    \hline
  \mc{G}_{-1}^{1,0} & \mc{E}(1) & \breve{\tau}^\omega   \\
  \mc{G}_{-1}^{1,1} & \dbl \bigwedge^{(2,0)} H^*_{K,L}  \dbr \otimes \mc{E}  (1) & \dbl \breve{\tau}_{\alpha \beta} \dbr \\ 
  \mc{G}_{-1}^{1,2} & [ \bigwedge^{(1,1)}_\circ H^*_{K,L} ]  \otimes \mc{E}  (1)  & [ \breve{\tau}^\circ_{\alpha \bar{\beta}} ] \\
  \hline 
    \mc{G}_{-1}^{2,0} & [ \bigwedge^{(1,1)}_\circ H^*_{K,L} ]  \otimes \mc{E} (1)  & [ \breve{\sigma}_{\alpha \bar{\beta}} ]\\
  \mc{G}_{-1}^{2,1} & \dbl \bigodot^{(2,0)} H^*_{K,L} \dbr \otimes \mc{E} (1) & \dbl \breve{\sigma}_{\alpha \beta} \dbr \\
     \hline
  \mc{G}_{-1}^{3,0} & \dbl \bigwedge^{(2,0)} H^*_{K,L} \dbr \otimes \mc{E} (1) & \dbl \breve{\zeta}_{\alpha \beta} \dbr \\
  \hline
  \hline
  \mc{G}_0^{0,0} & H^*_{K,L}  & \breve{E}_i \\
  \hline
  \mc{G}_0^{1,0} & \dbl \bigwedge^{(1,0)} H^*_{K,L} \dbr & \dbl \breve{G}_{\alpha} \dbr \\
  \mc{G}_0^{1,1} & \dbl \bigwedge^{(3,0)} H^*_{K,L} \dbr & \dbl \breve{G}^{\tiny{\ydskew}}_{\alpha \beta \gamma} \dbr \\
  \mc{G}_0^{1,2} & \dbl \ydhook H^*_{K,L} \dbr &  \dbl \breve{G}^{\tiny{\ydhook}}_{\alpha \beta \gamma} \dbr \\
  \mc{G}_0^{1,3} & \dbl  \bigwedge^{(1,2)}_\circ H^*_{K,L}  \dbr & \dbl \breve{G}^{\circ}_{\bar{\gamma} \alpha \beta} \dbr \\
  \hline
  \hline
  \mc{G}_1^{0,0} & \dbl \bigwedge^{(2,0)} H^*_{K,L} \dbr \otimes \mc{E}(-1) & \dbl \breve{B}_{\alpha \beta} \dbr \\
\hline
\end{array}}
\end{displaymath}
\caption{\label{tab:table-irred-mod-G_bdl}Irreducible $Q_0$-submodules of $\mc{G}$}
\end{table}
\end{center}
We shall also introduce the quantity
\begin{align*}
\breve{\tau}_{\alpha \bar{\beta}} & =  \frac{\i}{2 m} \breve{\tau}^\omega h_{\alpha \bar{\beta}}  + \breve{\tau}^\circ_{\alpha \bar{\beta}} \, ,
\end{align*}
which we identify as a section of  $\mc{G}_{-1}^{1,0} \oplus  \mc{G}_{-1}^{1,2}$.

Finally, to characterise sections of the bundles $\slashed{\mc{G}}_i^{j,k}$, we simply apply the results of Theorem \ref{thm:intors-rob}. For instance, the intrinsic torsion is a section of $\slashed{\mc{G}}_0^{1,2}$ if and only if its weighted components in any splitting satisfy
\begin{align*}
\breve{\gamma}_i = \breve{\sigma}_{\alpha \beta} = 2 \i \breve{\tau}^\circ_{\alpha \beta} + \breve{\zeta}_{\alpha \beta} = \breve{G}^{\tiny{\ydhook}}_{\alpha \beta \gamma} = 0 \, .
\end{align*} 

\begin{rem}
We can also express some of the properties of the almost Robinson structure in terms of a Robinson spinor --- this will be done in Section \ref{sec:spinor_des}.
\end{rem}

\subsection{Congruences of null geodesics}
Since an almost Robinson manifold $(\mc{M},g,N, K)$ defines in particular an optical geometry, there is a congruence of null curves $\mc{K}$ associated to it. The algebraic properties of the intrinsic torsion  of the optical geometry $(\mc{M},g,K)$ can be related to the geometric properties of $\mc{K}$ as reviewed in \cite{Fino2020}. We summarise the correspondence in Table \ref{tab-geom-intors}.
\begin{center}
\begin{table}[!htbp]
\begin{displaymath}
{\renewcommand{\arraystretch}{1.5}
\begin{array}{|c|c|}
\hline
\text{Intrinsic torsion} & \text{Congruence} \\
\hline
\hline
\slashed{\mc{G}}_{-2}^0  = \mc{G}^{-1} & \text{geodesic} \\
\hline
\slashed{\mc{G}}_{-1}^0 & \text{non-expanding} \\
\slashed{\mc{G}}_{-1}^{1} & \text{non-twisting} \\
\slashed{\mc{G}}_{-1}^2 & \text{non-shearing} \\
\hline
\slashed{\mc{G}}_0^0 = \{ 0 \} & \text{parallel} \\
\hline
\end{array}}
\end{displaymath}
\caption{\label{tab-geom-intors} Geometric properties of $\mc{K}$ and intrinsic torsion $\mathring{T}$}
\end{table}
\end{center}
 We also record the following lemma, whose proof is straightforward.
 \begin{lem}\label{lem:opt_J}
 Let $(\mc{M},g,N,K)$ be an almost Robinson manifold with congruence of null curves $\mc{K}$.  Denote by $J_i{}^j$ the complex structure on the screen bundle $H_K$. Suppose that the intrinsic torsion $\mathring{T}$ of $(N,K)$ is a section of $\slashed{\mc{G}}_{-2}^0$ so that $\mc{K}$ is geodesic with twist $\breve{\tau}_{i j}$ and shear $\breve{\sigma}_{i j}$. Then  $\mathring{T}$ is a section of
 \begin{itemize}
 \item  $\slashed{\mc{G}}_{-1}^{1,1}$, i.e.\ $\breve{\tau}_{\alpha \beta} = 0$, if and only if $\breve{\tau}_{i j}$ and $J_i{}^j$ commute, i.e.\ $J_{[i}{}^k \breve{\tau}_{j] k} = 0$;
  \item $\slashed{\mc{G}}_{-1}^{2,1}$, i.e.\  $\breve{\sigma}_{\alpha \beta} = 0$, if and only if $\breve{\sigma}_{i j}$ and $J_i{}^j$ commute, i.e.\  $J_{(i}{}^k \breve{\sigma}_{j) k} = 0$;
   \item $\slashed{\mc{G}}_{-1}^{1,0} \cap \slashed{\mc{G}}_{-1}^{1,2}$, i.e.\  $\breve{\tau}_{\alpha \bar{\beta}} = 0$, if and only if $\breve{\tau}_{i j}$ and $J_i{}^j$ anti-commute,  i.e.\  $J_{(i}{}^k \breve{\tau}_{j) k} = 0$;
  \item $\slashed{\mc{G}}_{-1}^{2,0}$, i.e.\  $\breve{\sigma}_{\alpha \bar{\beta}} = 0$, if and only if $\breve{\sigma}_{i j}$ and $J_i{}^j$ anti-commute,  i.e.\  $J_{[i}{}^k \breve{\sigma}_{j] k} = 0$.
 \end{itemize}
 \end{lem}

Locally, we shall identify an almost Robinson manifold $(\mc{M},g, N, K)$ of dimension $2m+2$ as a surjective submersion over a $(2m+1)$-dimensional smooth manifold $\ul{\mc{M}}$, namely, the local leaf space of $\mc{K}$. This leaf space  will be endowed with various geometric structures depending on the geometric properties of $\mc{K}$. In the next sections, we shall examine the relation between the intrinsic torsion of $(N,K)$ with the induced geometric structures on the leaf space. As a notational rule followed in this article, tensor fields on $\ul{\mc{M}}$ will be underlined to distinguish them from tensor fields on $\mc{M}$.

\begin{rem}\label{rem:ind_str_opt}
We briefly recall the results of \cite{Fino2020} (see also \cite{Robinson1983}) that pertain to the optical structure $K$ associated to an almost Robinson manifold $(\mc{M}, g, N, K)$. In the following, we shall denote the conformal class of $g$ by $[g]$. If the congruence of null curves $\mc{K}$ tangent to $K$ is geodesic, there is a subclass $[g]_{n.e.}$ of metrics in $[g]$ for which $\mc{K}$ is also non-expanding. By extension, $[g]$ and $[g]_{n.e.}$ determine conformal subclasses of bundle metrics $[h]$ and $[h]_{n.e.}$ respectively on the screen bundle $H_K$. We then have the following relations between the geometric properties of $\mc{K}$ and its leaf space:
\begin{itemize}
\item If $\mc{K}$ is geodesic, $H_K$ descends to a rank-$2m$ distribution $\ul{H}$ on $\ul{\mc{M}}$. In this case, there is a one-to-one correspondence between optical vector fields $k$ such that $\mathsterling_k \kappa = 0$, where $\kappa = g(k,\cdot)$, and sections of $\Ann(\ul{H})$.
\item If $\mc{K}$ is geodesic and non-twisting, the induced distribution $\ul{H}$ on $\ul{\mc{M}}$ is involutive.
\item If $\mc{K}$ is geodesic and non-shearing,  the screen bundle metric $H_K$ induces a bundle conformal structure $\ul{\mbf{c}}_{\ul{H}}$ on $(\ul{\mc{M}}, \ul{H})$. There is a one-to-one correspondence between metrics in $[g]_{n.e.}$ (or equivalently, screen bundle metrics in $[h]_{n.e.}$) and metrics in $\ul{\mbf{c}}_{\ul{H}}$. It follows that if $\mc{K}$ is already non-expanding for $g$, then $h$ descends to a distinguished bundle metric $\ul{h}$ on $\ul{H}$.
\end{itemize}
\end{rem}

Dually, we thus have that the bracket condition $[K,K^\perp] \subset K^\perp$ is equivalent to $\mc{K}$ being geodesic, and the condition $[K^\perp, K^\perp] \subset K^\perp$ to $\mc{K}$ being geodesic and non-twisting. The latter can split into the two following obvious subcases.
\begin{prop}\label{prop:NNbKp}
	Let $(\mc{M}, g, N, K)$ be a $(2m+2)$-dimensional almost Robinson manifold with congruence of null curves $\mc{K}$ with leaf space $\ul{\mc{M}}$. The following statements are equivalent:
	\begin{enumerate}
		\item $[N, \overline{N}] \subset {}^\C K^\perp$; \label{item:NNbKp1}
		\item the intrinsic torsion is a section of $\slashed{\mc{G}}_{-1}^{1,0} \cap \slashed{\mc{G}}_{-1}^{1,2}$, i.e.\  \label{item:NNbKp2}
		\begin{align*}
			\breve{\gamma}_i =  \breve{\tau}^\omega = \breve{\tau}^\circ_{\alpha \bar{\beta}} = 0 \, ;
		\end{align*}
		\item $\mc{K}$ is geodesic and its twist anti-commutes with the screen bundle complex structure.  \label{item:NNbKp3}
	\end{enumerate}
\end{prop}

\begin{prop}\label{prop:NNKp}
	Let $(\mc{M}, g, N, K)$ be a $(2m+2)$-dimensional almost Robinson manifold with congruence of null curves $\mc{K}$ with leaf space $\ul{\mc{M}}$. The following statements are equivalent:
	\begin{enumerate}
		\item $[N, N] \subset {}^\C K^\perp$; \label{item:NNKp1}
		\item the intrinsic torsion is a section of $\slashed{\mc{G}}_{-1}^{1,1}$, i.e.\  \label{item:NNKp2}
		\begin{align*}
			\breve{\gamma}_i =  \breve{\tau}_{\alpha \beta} = 0 \, ;
		\end{align*}
		\item $\mc{K}$ is geodesic and its twist commutes with the screen bundle complex structure.  \label{item:NNKp3}
	\end{enumerate}
\end{prop}
We shall deal with the remaining bracket conditions, namely $[K,N] \subset N$ and $[N,N] \subset N$ in Propositions \ref{prop:N_pres} and \ref{prop:fol_spinor} respectively.

\begin{rem}
	In the context of almost Robinson geometry, there are many candidates generalising the notion of non-shearing congruence of null geodesics from four to higher even dimensions. All of them should be almost Robinson structures whose intrinsic torsion is a section of $\slashed{\mc{G}}_{-1}^{2,1}$, i.e.\ $\breve{\gamma}_i = \ \breve{\sigma}_{\alpha \beta} = 0$,
\end{rem}

\begin{rem}\label{rem:non-geod}
It is clear from Theorem \ref{thm:intors-rob} that the only proper subbundles of $\mc{G}$ that are not contained in $\mc{G}^{-1}$ are
\begin{align}\label{eq:non_geod}
 \slashed{\mc{G}}_0^{1,1}  \subset  (\slashed{\mc{G}}_{-1}^{1 \times 3})_{[4\i:-1]} \, ,
\end{align}
i.e.\ $- 4 \,  \i \tau_{\alpha \beta } + \zeta_{\alpha \beta }=0$ and $G^{\tiny{\ydskew}}_{\alpha \beta \gamma} =  - 4 \,  \i \tau_{\alpha \beta } + \zeta_{\alpha \beta } = 0$. This means that the congruence of null curves associated to an almost Robinson structure whose intrinsic torsion lies in any proper subbundles of  $\mc{G}$ except for \eqref{eq:non_geod} must be geodesic. We leave it as a conjecture whether one can construct almost Robinson manifolds whose intrinsic torsion lies in \eqref{eq:non_geod} but whose congruence is not geodesic.
\end{rem}

\subsection{Almost CR geometry}\label{sec:CR}
The geometric structure that an almost Robinson structure may induce on the leaf space of its associated congruence of null curves is an almost CR structure, of which we now recall some notions. For a friendly introduction, see \cite{Jacobowitz1990}.

\subsubsection{General definitions}
An \emph{almost CR structure} on a $(2m + 1)$-dimensional smooth manifold $\ul{\mc{M}}$ consists of a pair $(\ul{H},\ul{J})$, where $\ul{H}$ is a rank-$2m$ distribution equipped with a bundle complex structure $\ul{J}$.  A \emph{CR structure} is an almost CR structure for which the $-\i$-eigenbundle $\ul{H}^{(0,1)}$ of $\ul{J}$, or equivalently its $\i$-eigenbundle $\ul{H}^{(1,0)}$, is in involution. We call an almost CR structure $(\ul{H},\ul{J})$ together with a choice of $1$-form $\ul{\theta}{}^0$ annihilating $\ul{H}$ an \emph{almost pseudo-Hermitian structure}.

There are two notions of \emph{Levi forms} that we can associate to $(\ul{H},\ul{J}$):
\begin{itemize}
	\item Following \cite{Eastwood2016}, we define the \emph{Levi form of the distribution $\ul{H}$} to be the bundle homomorphism $\ul{\bm{\Lv}} : \Ann(\ul{H}) \rightarrow \bigwedge^2 \left(T^* \ul{\mc{M}}/\Ann(\ul{H})\right)$ given by the composition
	\begin{align*}
		\Ann(\ul{H}) \longrightarrow T^* \ul{\mc{M}} \stackrel{\d}{\longrightarrow} \bigwedge^2 T^* \ul{\mc{M}} \longrightarrow \bigwedge^2 \left( T^* \ul{\mc{M}}/\Ann(\ul{H}) \right) \, .
	\end{align*}
	Concretely, for any section $\ul{\theta}{}^0$ of $\Ann(\ul{H})$, $\ul{\theta}{}^0 \circ \ul{\bm{\Lv}} (v , w) = \d \ul{\theta}{}^0( v , w)$ for any sections $v$ and $w$ of $\ul{H}$. We shall refer to  $\ul{\Lv} := \ul{\theta}{}^0 \circ \ul{\bm{\Lv}}$ as the \emph{Levi form of $(\ul{H},\ul{\theta}^0)$}.
	\item Given a $1$-form $\ul{\theta}^0$ annihilating $\ul{H}$, the \emph{Levi form of  the almost pseudo-Hermitian structure} $(\ul{H}, \ul{J}, \ul{\theta}^0)$ is the Hermitian form $\ul{\msf{h}}$ on $\ul{H}^{(1,0)}$ defined by
	\begin{align*}
		\ul{\msf{h}} (v, w) & := - 2 \i \d \ul{\theta}^0 (v, \ol{w} ) \,  ,  & \mbox{for any $v, w \in \Gamma(\ul{H}^{(1,0)})$ .}
	\end{align*}
	By extension, the \emph{Levi form} $\ul{\bm{\msf{h}}}$ of $(\ul{H}, \ul{J})$ is the Hermitian form  taking values in $\C \otimes \left(T \ul{\mc{M}}/\ul{H}\right)$ defined by $\ul{\msf{h}} := \ul{\theta}{}^0 \circ \ul{\bm{\msf{h}}}$.
\end{itemize}
	Note that the definition of $\ul{\bm{\msf{h}}}$ depends on both $\ul{H}$ and $\ul{J}$, but that of $\ul{\bm{\Lv}}$ depends on $\ul{H}$ alone. One can identify $\ul{\bm{\msf{h}}}$ in a suitable way with the $(1,1)$-part of $\ul{\bm{\Lv}}$ with respect to $\ul{J}$ --- see below.

A coframe $(\ul{\theta}^0 , \ul{\theta}^{\alpha} ,  \overline{\ul{\theta}}{}^{\bar{\alpha}})$ adapted to the almost CR structure  $(\ul{H},\ul{J})$ on $\ul{\mc{M}}$ consists of a real $1$-form $\ul{\theta}{}^0$ and $m$ complex $1$-forms $\ul{\theta}{}^{\alpha}$ with $\overline{\ul{\theta}}{}^{\bar{\alpha}} = \overline{\ul{\theta}{}^{\alpha}}$, such that $\ul{H} = \Ann(\ul{\theta}^0)$, and $\ul{H}^{(0,1)} = \Ann( \ul{\theta}{}^0, \ul{\theta}{}^\alpha)$. For simplicity, we shall assume that $\Ann(\ul{H})$ is oriented, although this assumption can easily be dropped. Any other coframe $(\wh{\ul{\theta}}{}^0, \wh{\ul{\theta}}{}^\alpha, \overline{\wh{\ul{\theta}}}{}^{\bar{\alpha}})$ adapted to $(\ul{H},\ul{J})$ is related to $(\ul{\theta}{}^0, \ul{\theta}{}^\alpha , \overline{\ul{\theta}}{}^{\bar{\alpha}})$ by
\begin{align}\label{eq:CR_transf}
\wh{\ul{\theta}}{}^0 & = \e^{\ul{\varphi}} \ul{\theta}{}^0 \, , &
\wh{\ul{\theta}}{}^\alpha & = \ul{\psi}_\beta{}^\alpha \ul{\theta}{}^\beta + \ul{\phi}^\alpha \ul{\theta}{}^0 \, ,
\end{align}
where $\ul{\varphi}$, $\ul{\psi}_\beta{}^\alpha$, and  $\ul{\phi}^\alpha$ are smooth functions on $\ul{\mc{M}}$, with the requirement that the determinant of $\ul{\psi}_\beta{}^{\alpha}$ be non-vanishing.\footnote{One could also include negative rescalings of $\ul{\theta}^0$ if one drops the assumption of orientability of $\Ann(\ul{H})$.} Even with a choice of an almost pseudo-Hermitian structure $\ul{\theta}^0$ on $(\ul{\mc{M}},\ul{H},\ul{J})$, there is no canonical choice for $( \ul{\theta}{}^\alpha )$ in general. Any choice of vector field $\ul{e}_0$ dual to $\ul{\theta}^0$ splits $T \ul{\mc{M}}$ as
\begin{align}\label{eq:CR_splitting}
T \ul{\mc{M}} = \ul{H} \oplus \mathrm{span}( \ul{e}_0 ) \, .
\end{align}

The structure equations for a given CR coframe $(\ul{\theta}^0 , \ul{\theta}^{\alpha} ,  \overline{\ul{\theta}}{}^{\bar{\alpha}})$ can be expressed as
\begin{align}
\begin{aligned}\label{eq:al_CR_structure}
\d \ul{\theta}^0 & = \i \ul{\msf{h}}{}_{\alpha \bar{\beta}} \, \ul{\theta}^\alpha \wedge \overline{\ul{\theta}}{}^{\bar{\beta}} + \ul{\Lv}{}_{\alpha \beta} \, \ul{\theta}^\alpha \wedge \ul{\theta}^\beta +  \ul{\Lv}{}_{\bar{\alpha} \bar{\beta}} \, \overline{\ul{\theta}}{}^{\bar{\alpha}} \wedge \overline{\ul{\theta}}{}^{\bar{\beta}} + \ul{\alpha} \wedge \ul{\theta}^{0} \, ,  \\
\d \ul{\theta}^{\alpha} & = \ul{\theta}{}^{\beta} \wedge \ul{\Gamma}{}_{\beta}{}^{\alpha} + \ul{A}{}^{\alpha}{}_{\bar{\beta}} \, \ul{\theta}^0 \wedge \overline{\ul{\theta}}{}^{\bar{\beta}} - \frac{1}{2} \ul{\Nh}_{\bar{\beta} \bar{\gamma}}{}^{\alpha} \, \overline{\ul{\theta}}{}^{\bar{\beta}} \wedge \overline{\ul{\theta}}{}^{\bar{\gamma}} \, , \\
\d \overline{\ul{\theta}}{}^{\bar{\alpha}} & = \overline{\ul{\theta}}{}^{\bar{\beta}} \wedge \overline{\ul{\Gamma}}{}_{\bar{\beta}}{}^{\bar{\alpha}} + \ul{A}{}^{\bar{\alpha}}{}_{\beta} \, \ul{\theta}^0 \wedge \ul{\theta}{}^{\beta} - \frac{1}{2} \ul{\Nh}_{\beta \gamma}{}^{\bar{\alpha}} \, \ul{\theta}{}^{\beta} \wedge \ul{\theta}{}^{\gamma} \, ,
\end{aligned}
\end{align}
for some complex functions $\ul{\msf{h}}{}_{\alpha \bar{\beta}}$, $\ul{\Lv}{}_{\alpha \beta}$, $\ul{A}{}^{\bar{\alpha}}{}_{\beta}$, $\ul{\Nh}_{\beta \gamma}{}^{\bar{\alpha}}$ and $1$-forms $\ul{\Gamma}{}_{\beta}{}^{\alpha}$ and $\ul{\alpha}$ on $\ul{\mc{M}}$, where $\ul{\msf{h}}{}_{\alpha \bar{\beta}}$ is Hermitian, $\ul{\alpha}$ is real, and the remaining quantities are defined by complex conjugation. Here, we identify $\ul{\msf{h}}_{\alpha \bar{\beta}}$ as the components of the Levi form of the pseudo-Hermitian structure $(\ul{H},\ul{J},\ul{\theta}^0)$.

There are two important invariants of $(\ul{H}, \ul{J})$ at a point of $\ul{\mc{M}}$, namely
\begin{itemize}
\item the \emph{rank} of $\ul{\bm{\Lv}}$, that is, the largest integer $r$ for which $\ul{\theta}^0 \wedge \left( \d \ul{\theta}^0 \right)^r$ does not vanish;
\item the \emph{signature} of $\ul{\bm{\msf{h}}}$, i.e.\ the signature of the Hermitian matrix $\ul{\msf{h}}_{\alpha \bar{\beta}}$.
\end{itemize}
Clearly, these do not depend on the choice of CR coframe. We shall assume regularity of the rank and signature throughout the article, i.e.\ these will be constant everywhere. We say that the almost CR structure $(\ul{H}, \ul{J})$ is
\begin{itemize}
	\item \emph{contact} or \emph{non-degenerate} if $\ul{\bm{\Lv}}$ has maximal rank, i.e.\ $\ul{\theta}^0 \wedge \left( \d \ul{\theta}^0 \right)^m$ does not vanish at any point,  i.e.\ $\ul{H}$ is a contact distribution;
	\item \emph{totally degenerate} if $\ul{\bm{\Lv}}$ vanishes identically, i.e.\ $\ul{\Lv}{}_{\alpha \beta} = \ul{\msf{h}}{}_{\alpha \bar{\beta}} = 0$, i.e.\ $\ul{\theta}^0 \wedge  \d \ul{\theta}^0  = 0$ everywhere, i.e.\ $\ul{H}$ is involutive;
	\item \emph{partially integrable} if $\ul{\bm{\Lv}}$ is of type $(1,1)$, i.e.\ $\ul{\Lv}{}_{\alpha \beta} = 0$ --- in this case, $\ul{\bm{\Lv}}$ can be identified with the imaginary part of $\ul{\bm{\msf{h}}}$;
	\item \emph{integrable} or \emph{involutive} if $\ul{H}^{(0,1)}$ is involutive, i.e.\ $\ul{\Lv}{}_{\bar{\alpha} \bar{\beta}} = \ul{\Nh}_{\bar{\beta} \bar{\gamma}}{}^{\alpha} = 0$.
\end{itemize}
The first two properties pertain to $\ul{H}$ alone while the latter two depend on the pair $(\ul{H},\ul{J})$.

\begin{exa}
The model for a contact CR manifold is the CR sphere $S^{2m+1}$ viewed as a hypersurface in $\C^{2m+2}$. More generally, any real hypersurface in $\C^{2m+2}$ is a CR manifold.
\end{exa}

\begin{rem}\label{rem:CR_bundle}
As a special case of almost CR manifold, consider an almost complex manifold $(\ut{\mc{M}},\ut{J})$. Then, one can take $\ul{\mc{M}}$ to be a bundle over $\ut{\mc{M}}$ with one-dimensional fibers such as $\R \times \ut{\mc{M}}$, $\R_{>0} \times \ut{\mc{M}}$ and $S^1 \times \ut{\mc{M}}$, and choose $(\ut{\theta}^{\alpha})$ to be a frame of $(1,0)$-forms for $\ut{\mc{M}}$, extend it to a coframe $(\ul{\theta}^0, \ul{\theta}^\alpha, \ol{\ul{\theta}}{}^{\bar{\alpha}})$ on $\ul{\mc{M}} \rightarrow \ut{\mc{M}}$ by adjoining a vertical $1$-form $\ul{\theta}^0$. Then clearly $(\ul{\mc{M}}, \ul{H}, \ul{J})$ is an almost CR manifold, where $\ul{H} = \Ann(\ul{\theta}^0)$, and the $-\i$-eigenbundle of $\ul{J}$ is $\ul{H}^{(0,1)} = \Ann(\ul{\theta}^0, \ul{\theta}^\alpha)$.
\end{rem}

\subsubsection{Partially integrable contact almost CR manifolds}
Suppose that $(\ul{H}, \ul{J})$ is both contact and partially integrable. Then $(\ul{H}, \ul{J})$ is equipped with a subconformal contact structure $\ul{\mbf{c}}_{\ul{H},\ul{J}}$ compatible with the bundle complex structure $\ul{J}$. Indeed,  $\ul{\bm{\Lv}}$ is now a $(1,1)$-form, which we may identify with $\ul{\bm{\msf{h}}}$, and with a slight abuse of notation, $\ul{\bm{\msf{h}}} = 2 \, \ul{\bm{\Lv}} \circ \ul{J}$ is a \emph{subconformal metric} on $(\ul{\mc{M}},\ul{H},\ul{J})$. Note that  $\ul{H}^{(1,0)}$ and $\ul{H}^{(0,1)}$ are totally null with respect to $\ul{\mbf{c}}_{\ul{H},\ul{J}}$. In particular, we have a one-to-one correspondence between contact forms in $\Ann(\ul{H})$ and metrics in $\ul{\mbf{c}}_{\ul{H},\ul{J}}$, each metric being given by $\ul{\msf{h}} = 2 \, \ul{\Lv} \circ \ul{J}$ where $\ul{\Lv}= \ul{\theta}^0 \circ \ul{\bm{\Lv}}$ for some contact form $\ul{\theta}^0$.

Furthermore, each choice of contact $1$-form $\ul{\theta}^0$ determines a unique vector field $\ul{e}_0$, the \emph{Reeb} vector field, satisfying $\ul{\theta}^0 (\ul{e}_0) = 1$ and $\d \ul{\theta}^0 (\ul{e}_0 , \cdot) = 0$, which induces a canonical splitting \eqref{eq:CR_splitting}, and one can choose an \emph{adapted coframe} $(\ul{\theta}^{\alpha} , \overline{\ul{\theta}}{}^{\bar{\alpha}})$ for $\ul{H}$ such that
\begin{align*}
\d \ul{\theta}^0 & = \i \ul{\msf{h}}{}_{\alpha \bar{\beta}} \ul{\theta}^\alpha \wedge \overline{\ul{\theta}}{}^{\bar{\beta}} \, .
\end{align*}
Any two such coframes must be related by a change \eqref{eq:CR_transf} where $\ul{\phi}^\alpha = \i\ul{\msf{h}}^{\alpha \bar{\beta}} ( \d \ul{\varphi} )_{\bar{\beta}}$. With no loss, one can always choose $(\ul{\theta}^{\alpha} )$ to be unitary with respect to $\ul{\msf{h}}_{\alpha \bar{\beta}}$.

Finally, to each choice of contact $1$-form $\ul{\theta}^0$, there exists a unique connection $\ul{\nabla}$, namely the \emph{Webster--Tanaka connection}, that preserves $\ul{\theta}^0$ and $\d \ul{\theta}^0$, with prescribed torsion \cite{Tanaka1975,Webster1978}: with reference to the structure equations \eqref{eq:al_CR_structure}, we identity $\ul{\Gamma}_{\alpha}{}^{\beta}$ with the connection $1$-form of $\ul{\nabla}$, $\ul{A}_{\alpha \beta} = \ul{h}_{\alpha \bar{\gamma}} \ul{A}^{\bar{\gamma}}{}_{\beta}$ with the so-called \emph{pseudo-Hermitian torsion tensor}, and $\ul{\Nh}_{\beta \gamma \alpha} = \ul{\Nh}_{\beta \gamma}{}^{\bar{\delta}} \ul{h}_{\alpha \bar{\delta}}$ with the so-called \emph{Nijenhuis torsion tensor}. By virtue of the Bianchi identities, these satisfy $\ul{A}_{[\alpha \beta]} = \ul{\Nh}_{[\beta \gamma \alpha]} = 0$. While $\ul{\Nh}_{\beta \gamma \alpha}$ is a CR invariant -- it is the obstruction to the involutivity of $\ul{H}^{(1,0)}$ -- the torsion tensor $\ul{A}_{\alpha \beta}$ depends on the choice of contact form, and is an invariant of the almost pseudo-Hermitian structure only. It may be interpreted as the obstruction to the Reeb vector field being a transverse symmetry of the CR structure. If $\ul{A}_{\alpha \beta}=0$ vanishes, we call $(\ul{\mc{M}}, \ul{H}, \ul{J}, \ul{\theta}^0)$ a \emph{Sasaki} almost pseudo-Hermitian contact manifold.

Further invariants of the partially integrable contact structure $(\ul{H}, \ul{J})$ can be obtained from the curvature of $\ul{\nabla}$. Of interest are the \emph{Chern--Moser tensor} when $m>1$, and the \emph{Cartan tensor} when $m=1$. If $(\ul{H}, \ul{J})$ is integrable, then the vanishing of these tensors is equivalent to the CR manifold being locally diffeomorphic to the CR sphere  -- see e.g.\ \cite{Case2020}.

\begin{rem}\label{rem:CR-Kahler-Einstein}
There are close analogies between partially integrable contact almost CR geometry and conformal geometry by virtue of the existence of the subconformal structure $\ul{\mbf{c}}_{\ul{H},\ul{J}}$: here, an almost pseudo-Hermitian structure can be seen as a choice of a bundle metric in $\ul{\mbf{c}}_{\ul{H},\ul{J}}$, and one may define an almost pseudo-Hermitian analogue of the Einstein condition in terms of the $\ul{\theta}^0$-dependent \emph{Webster--Ricci tensor} and the pseudo-Hermitian torsion tensor $\ul{A}_{\alpha \beta}$. This condition was introduced in the integrable case in \cite{Leitner2007,Cap2008}, and generalised to the non-integrable case in \cite{Taghavi-Chabert2021} where it is referred to as an \emph{almost CR--Einstein structure}.\footnote{It is rather unfortunate that the terminology ``almost CR--Einstein structure'' is used in different ways in \cite{Cap2008} and \cite{Taghavi-Chabert2021}: in the former reference, ``almost" refers to the Einstein condition, meaning that the manifold is a CR manifold that is CR--Einstein off the zero set of some density. In the latter reference, ``almost" refers to an almost CR structure that is not necessarily integrable.} As shown in \cites{Leitner2007,Taghavi-Chabert2021}, such a structure can be constructed on the anti-canonical bundle of an almost K\"{a}hler--Einstein manifold. Conversely, any almost CR--Einstein manifold locally arises in this way.
\end{rem}

We shall leave aside analytical questions related to CR manifolds, especially in connection with embeddability, and we refer the reader to \cites{Jacobowitz1990,Trautman2002} and references therein for further details.

\subsection{Nearly Robinson structures and almost CR structures}\label{sec:lift}
In this section, we restrict our attention to nearly Robinson structures. These lie at the junction between Lorentzian geometry and almost CR geometry, as the next proposition makes clear.
\begin{prop}\label{prop:N_pres}
Let $(\mc{M}, g, N, K)$ be a $(2m+2)$-dimensional almost Robinson manifold with congruence of null curves $\mc{K}$ with leaf space $\ul{\mc{M}}$. The following statements are equivalent:
\begin{enumerate}
\item $(N,K)$ is a nearly Robinson structure, i.e\ $[K,N] \subset N$, i.e.\ for any optical vector field $k$, and any $v \in \Gamma(N)$, $\mathsterling_k v \in \Gamma(N)$; \label{item:N_pres1}
\item any complex Robinson $(m+1)$-form $\nu$ is preserved along $K$,  i.e.\  for any optical vector field $k$, $\mathsterling_k \nu = f \nu$ for some smooth function $f$; \label{item:N_pres2}
\item the intrinsic torsion is a section of $(\slashed{\mc{G}}_{-1}^{1 \times 3})_{[-2\i:1]} \cap \slashed{\mc{G}}_{-1}^{2,0}$, i.e.\  \label{item:N_pres3}
\begin{align*}
\breve{\gamma}_i = \breve{\sigma}_{\alpha \beta} & = 0 \, , & \breve{\zeta}_{\alpha \beta} & =  2\i \breve{\tau}_{\alpha \beta} \, ;
\end{align*}

\item $(N,K)$ induces an almost CR structure $(\ul{H}, \ul{J})$ on $\ul{\mc{M}}$.  \label{item:N_pres4}
\end{enumerate}
If any of these conditions holds, $\mc{K}$ is geodesic and its shear commutes with the screen bundle complex structure.
\end{prop}

\begin{rem}
Condition \eqref{item:N_pres1} tells us that the splitting of ${}^\C H_K$ into $H^{(1,0)}_K$ and $H^{(0,1)}_K$ is preserved along the integral curves of $\mc{K}$. In particular, the distribution $\ul{H}$ on $\ul{\mc{M}}$ inherits this splitting, which is equivalent to an almost CR structure, as claimed by \eqref{item:N_pres4}.
\end{rem}

\begin{proof}
The equivalence between \eqref{item:N_pres1} and \eqref{item:N_pres2} is clear. For the almost Robinson structure to be preserved along $K$, i.e.\ $\mathsterling_k v \in \Gamma(\ol{N})$ for any $k \in \Gamma(K)$, $v \in \Gamma(\ol{N})$, we must have
\begin{align*}
0 & = \kappa \left( \mathsterling_k e_{\alpha} \right)  =  - \nabla \kappa \left( k, e_{\alpha} \right)  =  - \gamma_\alpha \, ,\\
0 & = g \left( \mathsterling_k e_{\alpha} , e_{\beta} \right) = g \left( \nabla_k e_{\alpha} , e_{\beta} \right) -   \nabla \kappa \left( e_{\alpha}, e_{\beta} \right) = \frac{1}{2\i} \zeta_{\alpha \beta} - \tau_{\alpha \beta} - \sigma_{\alpha \beta} \, .
\end{align*}
for any adapted frame $(\ell, e_\alpha, e_{\bar{\alpha}}  , k)$ and $\kappa = g(k,\cdot)$. It now follows that $\gamma_{\alpha} = 0$, and taking the symmetric and skew-symmetric parts yields $- 2\i \tau_{\alpha \beta} + \zeta_{\alpha \beta} = \sigma_{\alpha \beta} = 0$. This computation shows that conditions \eqref{item:N_pres1} and \eqref{item:N_pres3} are equivalent.

The equivalence between \eqref{item:N_pres2} and \eqref{item:N_pres4} can be proved following \cite{Nurowski2002}. One can always find a complex Robinson $(m+1)$-form $\nu$ such that $\mathsterling_k \nu = 0$ for some optical vector field $k$. So $\nu$ is the pullback of a complex $(m+1)$-form $\ul{\nu}$ on $\ul{\mc{M}}$. This $\ul{\nu}$ clearly shares the same algebraic properties of $\nu$. In particular, it is simple, and $\ul{\theta}^0 \wedge \ul{\nu} = 0$, where $\ul{\theta}^0$ is a real $1$-form that pulls back to an optical $1$-form on $\mc{M}$, and annihilates a rank-$2m$ distribution $\ul{H}$ on $\ul{\mc{M}}$. This means that $\mathrm{span} ( \ul{\nu} ) = \bigwedge^{m+1} \Ann(\ul{H}^{(0,1)})$ for some complex rank-$m$ vector subbundle $\ul{H}^{(0,1)}$ of ${}^\C \ul{H}$. The story for the complex conjugate of $\nu$ is entirely analogous, and yields 
a complex rank-$m$ vector subbundle $\ul{H}^{(1,0)}$ of ${}^\C \ul{H}$. It is then straightforward to check $\ul{H}^{(1,0)} \cap \ul{H}^{(0,1)} = \{ 0 \}$ at any point, i.e.\ ${}^\C \ul{H} = \ul{H}^{(1,0)} \oplus \ul{H}^{(0,1)}$, which defines the almost CR structure on $\ul{\mc{M}}$ as required.

The last claim of the proposition follows from Lemma \ref{lem:opt_J}.
\end{proof}

In \cite{Nurowski2002}, the authors show how to construct Robinson manifolds as trivial lines bundles over CR manifolds. Here, we generalise the construction to nearly Robinson manifolds. The proof of the following result is self-evident.
\begin{prop}\label{prop:lift}
Let $(\ul{\mc{M}}, \ul{H}, \ul{J})$ be a $(2m+1)$-dimensional oriented almost CR manifold, and $\mc{M} := \R \times \ul{\mc{M}} \stackrel{\varpi}{\longrightarrow} \ul{\mc{M}}$ be a trivial line bundle over  $\ul{\mc{M}}$. Fix a triplet $\left( (\ul{\theta}^0, \ul{\theta}^\alpha, \overline{\ul{\theta}}{}^{\bar{\alpha}}) , h_{\alpha \bar{\beta}} , \lambda \right)$ where
\begin{itemize}
\item $(\ul{\theta}^0, \ul{\theta}^\alpha, \overline{\ul{\theta}}{}^{\bar{\alpha}})$ is a CR-coframe  on $(\ul{\mc{M}}, \ul{H}, \ul{J})$,
\item $h_{\alpha \bar{\beta}}$ is a positive-definite Hermitian matrix depending smoothly on $\mc{M}$, and
\item $\lambda$ is a $1$-form  on $\mc{M}$ such that $\lambda \wedge \varpi^*\ul{\varepsilon}$ does not vanish at any point for any non-vanishing $(2m+1)$-form $\ul{\varepsilon}$ on $\ul{\mc{M}}$.
\end{itemize}
Then $(\mc{M},g)$ is an oriented and time-oriented Lorentzian manifold with metric
\begin{align}\label{eq:al_Rob_met}
g & = 4 \,  \varpi^*\ul{\theta}^0 \, \lambda + 2 \, h_{\alpha \bar{\beta}}  \varpi^*\ul{\theta}^\alpha \,  \varpi^* \overline{\ul{\theta}}{}^{\bar{\beta}} \, ,
\end{align}
and $(N,K)$, where $N = \Ann (\varpi^* \ul{\theta}^0 , \varpi^* \ul{\theta}^\alpha)$ and $K = \Ann ( \varpi^* \ul{\theta}^0)^\perp$, defines a nearly Robinson structure on $(\mc{M},g)$. In particular, $\ul{\mc{M}}$ is the leaf space of the congruence of null geodesics tangent to $K$.

Any two such triplets $\left( (\ul{\theta}^0, \ul{\theta}^\alpha, \overline{\ul{\theta}}{}^{\bar{\alpha}}) , h_{\alpha \bar{\beta}} , \lambda \right)$ and $\left( (\wh{\ul{\theta}}{}^0, \wh{\ul{\theta}}{}^{\alpha}, \overline{\wh{\ul{\theta}}}{}^{\bar{\alpha}}) , \wh{h}_{\alpha \bar{\beta}} , \wh{\lambda} \right)$ where $(\ul{\theta}^0, \ul{\theta}{}^\alpha, \overline{\ul{\theta}}{}^{\bar{\alpha}})$ and $(\wh{\ul{\theta}}{}^0, \wh{\ul{\theta}}{}^\alpha, \overline{\wh{\ul{\theta}}}{}^{\bar{\alpha}})$ are related by \eqref{eq:CR_transf} define the same metric \eqref{eq:al_Rob_met} if and only if $\ul{\psi}{}_{\alpha}{}^{\gamma}$ is an element of $\U (m)$ at every point, i.e.\
\begin{align*}
\wh{h}_{\alpha \bar{\beta}} & = h_{\gamma \bar{\delta}} (\ul{\psi}^{-1}){}_{\alpha}{}^{\gamma}  (\ul{\psi}^{-1}){}_{\bar{\beta}}{}^{\bar{\delta}} \, , \\
\intertext{and $\lambda$ transforms as}
\wh{\lambda} & = \e^{-\ul{\varphi}} \lambda - \frac{1}{2} \ul{\psi}{}_{\alpha}{}^{\beta} \ul{\phi}{}_{\beta} \ul{\theta}^{\alpha} - \frac{1}{2} \ul{\psi}{}_{\bar{\alpha}}{}^{\bar{\beta}} \ul{\phi}{}_{\bar{\beta}} \overline{\ul{\theta}}^{\bar{\alpha}} - \frac{1}{2} \ul{\phi}{}_{\beta} \ul{\phi}{}^{\beta} \ul{\theta}^0 \, .
\end{align*}
\end{prop}

\begin{rem}
Variations of the above construction are possible by replacing the $\R$-factor of $\mc{M}$ by $\R_{>0}$ or $S^1$ for instance.
\end{rem}

\begin{defn}
We shall refer to any nearly Robinson structure constructed on $\mc{M} := \R \times \ul{\mc{M}} \longrightarrow \ul{\mc{M}}$ as in Proposition \ref{prop:lift}, as a \emph{lift} of the almost CR manifold $(\ul{\mc{M}}, \ul{H}, \ul{J})$.  The pullbacks of the $1$-forms $(\ul{\theta}^0, \ul{\theta}^{\alpha}, \ol{\ul{\theta}}{}^{\bar{\alpha}})$ will be referred to as \emph{horizontal}, and the $1$-form $\lambda$ as \emph{vertical} (with respect to the fibration $\mc{M} \rightarrow \ul{\mc{M}}$). 
\end{defn}

\begin{rem}
We emphasise that the metric constructed in Proposition \ref{prop:lift} is \emph{not} canonical in general. To do away with the choice of CR coframe and $1$-form $\lambda$, while fixing the conformal class $[ h_{\alpha \bar{\beta}}]$, one needs to introduce the notion of \emph{generalised almost Robinson geometry}, which is dealt with in Section \ref{sec:gen_Rob} -- see Proposition \ref{prop:lift_gen}.
\end{rem}

Before we proceed, we give the converse of Proposition \ref{prop:lift} -- see  \cite{Nurowski2002} for the involutive case.
\begin{prop}\label{prop:Rob->CR}
Let $(\mc{M}, g, N, K)$ be a nearly Robinson manifold with congruence of null geodesics $\mc{K}$. Then $\mc{M}$ is locally diffeomorphic to the trivial line bundle $\R \times \ul{\mc{M}} \stackrel{\varpi}{\longrightarrow} \ul{\mc{M}}$, where $\ul{\mc{M}}$ is  the local leaf space of $\mc{K}$ and is equipped with an almost CR structure $(\ul{H}, \ul{J})$. Further, locally, $g$ takes the form \eqref{eq:al_Rob_met} for some CR coframe $(\ul{\theta}^0, \ul{\theta}^\alpha, \overline{\ul{\theta}}{}^{\bar{\alpha}})$, Hermitian matrix $h_{\alpha \bar{\beta}}$ depending smoothly on $\mc{M}$, and  $1$-form $\lambda$  on $\mc{M}$ that never vanishes on $K$.
\end{prop}

\begin{proof}
Note that $\mc{M}$ is locally diffeomorphic to the line bundle $\R \times \ul{\mc{M}} \stackrel{\varpi}{\longrightarrow} \ul{\mc{M}}$, where each fiber is a null curve of $\mc{K}$. By Proposition \ref{prop:N_pres}, $(N,K)$ descends to an almost CR structure $(\ul{H},\ul{J})$ on the local leaf space $\ul{\mc{M}}$ of $\mc{K}$: following our convention, $N/{}^\C K$ and $\overline{N}/{}^\C K$ descend to the eigenbundles $\ul{H}^{(0,1)}$ and $\ul{H}^{(1,0)}$ of $\ul{J}$ respectively. Let $(\ul{\theta}^0 , \ul{\theta}^\alpha , \overline{\ul{\theta}}{}^{\bar{\alpha}})$ be a coframe on $\ul{\mc{M}}$ where $\mathrm{span} (\ul{\theta}^0) = \Ann(\ul{H})$ and $\mathrm{span} (\ul{\theta}^0 , \ul{\theta}^\alpha ) = \Ann(\ul{H}^{(0,1)})$. Then $\mathrm{span} (\varpi^* \ul{\theta}^0) = \Ann(\ul{H})$ and $\mathrm{span} (\varpi^* \ul{\theta}^0 , \varpi^* \ul{\theta}^\alpha) = \Ann(N)$ , and it follows immediate that $g$ must be of the form  \eqref{eq:al_Rob_met} where  $h_{\alpha \bar{\beta}}$ and $\lambda$ have the required properties.
\end{proof}

Propositions \ref{prop:lift} and \ref{prop:Rob->CR} thus allow us to relate the geometric properties of a nearly Robinson manifold $(\mc{M}, g, N, K)$ with those of the almost CR leaf space $(\ul{\mc{M}}, \ul{H}, \ul{J})$ of its associated congruence of null geodesics $\mc{K}$. This can be seen at four levels:
\begin{itemize}
\item the involutivity of $N$ is equivalent to that of $\ul{H}^{(0,1)}$ --- see Section \ref{sec:spinor_des};
\item the twist $\breve{\tau}_{i j}$ of $\mc{K}$ encodes the geometric and algebraic properties of the Levi form $\ul{\bm{\Lv}}$ of $\ul{H}$;
\item what remains of the shear, that is, its $(1,1)$-part $\breve{\sigma}_{\alpha \bar{\beta}}$, and the expansion $\breve{\epsilon}$ of $\mc{K}$ obstruct the existence of a conformal or metric structure on $\ul{H}$;
\item further denegeracy conditions of the intrinsic torsion of $(N, K)$ will also depend on the choice of the $1$-form $\lambda$ --- see Lemma \ref{lem:lambda} for instance.
\end{itemize}
Unlike $N$ and $\breve{\tau}_{i j}$, which are tied to the properties of $(\ul{H},\ul{J})$, the shear and expansion depend only on the screen bundle metric of \eqref{eq:al_Rob_met} --- see also Remark \ref{rem:lift-choice} below.

To delve into this matter further, we must bear in mind that the coframe $(\ul{\theta}^0 , \ul{\theta}^\alpha , \overline{\ul{\theta}}{}^{\bar{\alpha}})$ on $\ul{\mc{M}}$ does \emph{not} in general pull back to a Robinson coframe on $\mc{M}$ by simply adjoining the $1$-form $\lambda$, in the sense that $(\ul{\theta}^\alpha)$ do \emph{not} form a unitary coframe with respect to $h_{\alpha \bar{\beta}}$. This essentially depends on the choice of $h_{\alpha \bar{\beta}}$. To be precise,  a Robinson coframe $(\kappa , \theta^\alpha , \overline{\theta}{}^{\bar{\alpha}} , \lambda)$ for the metric \eqref{eq:al_Rob_met}, so that $(\theta^{\alpha})$ is unitary for $h_{\alpha \bar{\beta}}$, is related to $(\ul{\theta}^0 , \ul{\theta}^\alpha , \overline{\ul{\theta}}{}^{\bar{\alpha}})$ via
\begin{align}\label{eq:CR->Rob_cof}
\kappa & = 2 \varpi^* \ul{\theta}^0 \, , &
\theta^\alpha & = \psi{}_{\beta}{}^{\alpha} \left( \varpi^* \ul{\theta}^\beta \right) + \phi^{\alpha} \ul{\theta}^0 \, ,
\end{align}
for some smooth functions $\psi^{\alpha}$ and $\psi{}_{\beta}{}^{\alpha}$ on $\mc{M}$, where $\psi{}_{\beta}{}^{\alpha}$ takes values in $\mbf{GL}(m,\C)$. Note that one can always choose our Robinson coframe such that $\psi^{\alpha} = 0$ in \eqref{eq:CR->Rob_cof}. When $m=1$, we have the decomposition $\mbf{GL}(1,\C) \cong \C^* \cong \R_{>0} \cdot \U(1)$, in which case we always have $\psi{}_{\alpha}{}^{\beta} = r \e^{\i \phi} \delta_{\alpha}^{\beta}$ where $r, \phi$ are real with $r>0$. Now, the space of all Hermitian forms on $\C^m$ is isomorphic to the homogeneous space $\mbf{GL}(m,\C)/\U(m)$ of real dimension $m^2$. Thus, the failure of $(\ul{\theta}^\alpha)$ to be unitary with respect to some scalar multiple of $h_{\alpha \bar{\beta}}$ at any point is measured by an element of $\mbf{GL}(m,\C) \mod \U(m)$.

At present, let us relate the twist of $\mc{K}$ to the Levi forms of $\ul{H}$ and $(\ul{H}, \ul{J})$. Using \eqref{eq:CR->Rob_cof}, we find
\begin{align*}
\ul{\msf{h}}{}_{\alpha \bar{\beta}} & =  \tau{}_{\gamma \bar{\delta}} \psi{}_{\alpha}{}^{\gamma} \overline{\psi}{}_{\bar{\beta}}{}^{\bar{\delta}} \, , &
2 \, \ul{\Lv}{}_{\alpha \beta} & =  \tau{}_{\gamma \delta} \psi{}_{\alpha}{}^{\gamma} \psi{}_{\beta}{}^{\delta} \, .
\end{align*}
From the first of these equations, we conclude that the $\mc{G}_{-1}^{1,0} \oplus \mc{G}_{-1}^{1,2}$-component $\breve{\tau}_{\alpha \bar{\beta}}$ of the intrinsic torsion encodes the signature of the Levi form of $(\ul{H}, \ul{J})$, while the $\mc{G}_{-1}^{1,1}$-component $\breve{\tau}_{\alpha \beta}$ of the intrinsic torsion encodes the partial integrability of $(\ul{H}, \ul{J})$. As a direct consequence, we obtain the following three propositions:

\begin{prop}\label{prop:twist_al_CR4}
Let $(\mc{M}, g, N, K)$ be a nearly Robinson manifold with congruence of null geodesics $\mc{K}$ and almost CR leaf space $(\ul{\mc{M}},\ul{H}, \ul{J})$. The following statements are equivalent:
\begin{enumerate}
\item the intrinsic torsion has non-degenerate $\slashed{\mc{G}}_{-1}^{1}$-component  $\breve{\tau}_{i j}$;
\item $\mc{K}$ is maximally twisting;
\item $(\ul{H},\ul{J})$ is contact.
\end{enumerate}
\end{prop}

\begin{prop}\label{prop:twist_al_CR1}
Let $(\mc{M}, g, N, K)$ be a nearly Robinson manifold with congruence of null geodesics $\mc{K}$ and almost CR leaf space $(\ul{\mc{M}},\ul{H}, \ul{J})$. The following statements are equivalent:
\begin{enumerate}
\item the intrinsic torsion is also a section of $\slashed{\mc{G}}_{-1}^{1,1}$, i.e.\
\begin{align*}
\breve{\gamma}_i = \breve{\tau}_{\alpha \beta} = \breve{\sigma}_{\alpha \beta} = \breve{\zeta}_{\alpha \beta} & = 0 \, ;
\end{align*}
\item the twist of $\mc{K}$ commutes with the screen bundle complex structure;
\item $(\ul{H},\ul{J})$ is partially integrable.
\end{enumerate}
\end{prop}

\begin{prop}\label{prop:nRob_pi_ctct}
Let $(\mc{M}, g, N, K)$ be a nearly Robinson manifold with congruence of null geodesics $\mc{K}$ and almost CR leaf space $(\ul{\mc{M}},\ul{H}, \ul{J})$. The following statements are equivalent:
\begin{enumerate}
\item the intrinsic torsion is also a section of $\slashed{\mc{G}}_{-1}^{1,1} \cap \slashed{\mc{G}}_{-1}^{3,0}$ with non-degenerate $\mc{G}_{-1}^{1,0} \oplus \mc{G}_{-1}^{1,2}$-component, i.e.\
\begin{align*}
\breve{\gamma}_i = \breve{\tau}_{\alpha \beta} = \breve{\sigma}_{\alpha \beta} = \breve{\zeta}_{\alpha \beta} & = 0 \, , & &  \mbox{with non-degenerate $\breve{\tau}_{\alpha \bar{\beta}}$} \, ;
\end{align*}
\item $\mc{K}$ is maximally twisting, and the twist commutes with the screen bundle complex structure;
\item $(\ul{H},\ul{J})$ is partially integrable and contact, and is thus endowed with a subconformal contact structure $\mbf{c}_{\ul{H},\ul{J}}$.
\end{enumerate}
\end{prop}

\begin{rem}
 The metrics in the subconformal structure $\ul{\mbf{c}}_{\ul{H},\ul{J}}$ of Proposition \ref{prop:nRob_pi_ctct} are in one-to-one correspondence with optical vector fields $k$ such that $\mathsterling_k \kappa = 0$ where $\kappa = g( k, \cdot)$.
\end{rem}

\begin{rem}
Special cases of the nearly Robinson manifolds given in Propositions \ref{prop:twist_al_CR4}, \ref{prop:twist_al_CR1} and \ref{prop:nRob_pi_ctct} are those for which $\breve{\tau}^\circ_{\alpha \bar{\beta}} = 0$, i.e.\ the twist determines the almost Robinson structure --- see Section \ref{sec:tw-ind_al_Rob}.
\end{rem}

As has already been treated in \cite{Fino2020}, the abscence of shear induces a subconformal structure $\mbf{c}_{\ul{H}}$ on $\ul{H}$, in which metrics are in one-to-one correspondence with metrics in $[g]_{n.e.}$ --- see Remark \ref{rem:ind_str_opt}. This subconformal structure is however \emph{not} compatible with the complex structure $\ul{J}$ in general. Combining this fact with Proposition \ref{prop:nRob_pi_ctct} yields the following result:
\begin{prop}\label{prop:Rob->pi_cont_CR}
Let $(\mc{M}, g, N, K)$ be a nearly Robinson manifold with congruence of null geodesics $\mc{K}$ and almost CR leaf space $(\ul{\mc{M}},\ul{H}, \ul{J})$. The following statements are equivalent:
\begin{enumerate}
\item the intrinsic torsion is also a section of $\slashed{\mc{G}}_{-1}^{1,1} \cap \slashed{\mc{G}}_{-1}^{2,0} \cap \slashed{\mc{G}}_{-1}^{3,0}$   with non-degenerate $\mc{G}_{-1}^{1,0} \oplus \mc{G}_{-1}^{1,2}$-component, i.e.\
\begin{align*}
\breve{\gamma}_i = \breve{\tau}_{\alpha \beta} = \breve{\sigma}_{i j} = \breve{\zeta}_{\alpha \beta} & = 0 \, , & &  \mbox{with non-degenerate $\breve{\tau}_{\alpha \bar{\beta}}$} \, ;
\end{align*}
\item $\mc{K}$ is maximally twisting and non-shearing, and the twist commutes with the screen bundle complex structure;
\item $(\ul{H},\ul{J})$ is partially integrable and contact, and thus equipped with a subconformal structure $\ul{\mbf{c}}_{\ul{H},\ul{J}}$, and also inherits a subconformal structure $\ul{\mbf{c}}_{\ul{H}}$ from $[g]_{n.e.}$.
\end{enumerate}
\end{prop}

\begin{rem}\label{rem:two_cf_str}
The subconformal structures $\ul{\mbf{c}}_{\ul{H},\ul{J}}$ and $\ul{\mbf{c}}_{\ul{H}}$ are distinct in general. They can however be related in the following way. Since $\mc{K}$ is maximally twisting and non-shearing, we know from \cite{Fino2020} that there exists a unique optical vector field $k^a$ such that the twist of $k^a$ is normalised to $\tau_{i j} \tau^{i j} = 2m$ for any metric $g$ in $[g]_{n.e.}$.  With this normalisation, for each choice of metric $g$ in $[g]_{n.e.}$, the Levi form of $(\ul{H},\ul{J})$ is related to the twist of $k$ by
\begin{align}\label{eq:Levi-conf}
 \ul{\msf{h}}_{\alpha \bar{\beta}} = \ul{h}_{\alpha \bar{\beta}} + \ul{\tau}^\circ_{\alpha \bar{\beta}} \, .
\end{align}
We thus see that $\ul{\msf{h}}_{\alpha \bar{\beta}}$ is a deformation of the metric $\ul{h}_{\alpha \bar{\beta}}$ by the tracefree part of the twist $\ul{\tau}^\circ_{\alpha \bar{\beta}}$. In particular, $\ul{\mbf{c}}_{\ul{H},\ul{J}}$ and $\ul{\mbf{c}}_{\ul{H}}$ coincide if and only if $\ul{\tau}^\circ_{\alpha \bar{\beta}} = 0$, i.e.\ the nearly Robinson structure is twist-induced. In Section \ref{sec:tw-ind_al_Rob}, we shall focus on a special case of the aforementioned results where the only non-vanishing of the twist lies in $\mc{G}_{-1}^{1,0}$.
\end{rem}

\begin{rem}\label{rem:lift-choice}
Proposition \ref{prop:lift} allows us to construct an almost Robinson manifold $(\mc{M}, g, N, K)$ with prescribed intrinsic torsion from a chosen almost CR manifold $(\ul{\mc{M}}, \ul{H}, \ul{J})$. While $(N,K)$, including the twist of the congruence $\mc{K}$, is determined by $(\ul{H}, \ul{J})$, there is more freedom as to the choice of Hermitian matrix $h_{\alpha \bar{\beta}}$, which will impact the $(1,1)$-part of the shear $\breve{\sigma}_{\alpha \bar{\beta}}$ and the expansion $\breve{\epsilon}$ of $\mc{K}$. In fact, using \eqref{eq:CR->Rob_cof}, we compute
\begin{align*}
\frac{\epsilon}{m} h_{\alpha \bar{\beta}} + \sigma_{\alpha \bar{\beta}} = h_{\gamma \bar{\beta}} \dot{\psi}_{\delta}{}^{\gamma} (\psi^{-1})_{\alpha}{}^{\delta} + h_{\alpha \bar{\delta}} \dot{\ol{\psi}}_{\bar{\gamma}}{}^{\bar{\delta}} (\ol{\psi}{}^{-1})_{\bar{\beta}}{}^{\bar{\gamma}} \, ,
\end{align*}
where $\dot{\psi}_{\delta}{}^{\gamma} := \mathsterling_k \psi_{\delta}{}^{\gamma}$ for the optical vector field $k = g^{-1} (\kappa ,\cdot)$. In particular, we interpret the $\mc{G}_{-1}^{2,0}$-component $\breve{\sigma}_{\alpha \bar{\beta}}$ of the intrinsic torsion as the infinitesimal obstruction to  $(\ul{\theta}^{\alpha})$ being unitary with respect to $h_{\alpha \bar{\beta}}$.

Various considerations may dictate the choice of screen bundle Hermitian form $h_{\alpha \bar{\beta}}$. For $\mc{K}$ to be non-shearing, we may set $h_{\alpha\bar{\beta}} = \e^{2 \varphi} \ul{h}_{\alpha\bar{\beta}}$ for some Hermitian form $\ul{h}_{\alpha \bar{\beta}}$ on $(\ul{\mc{M}}, \ul{H}, \ul{J})$ and smooth function $\varphi$ on $\mc{M}$. If $\varphi$ is a function on $\ul{\mc{M}}$, $\mc{K}$ is non-expanding too. There are two extreme cases to consider:
\begin{itemize}
\item If $\ul{\msf{h}}_{\alpha \bar{\beta}}$ is positive-definite, we \emph{may} take $\ul{h}_{\alpha\bar{\beta}} = \ul{\msf{h}}_{\alpha \bar{\beta}}$. This is the case if $(\ul{H}, \ul{J})$ is partially integrable and contact: the resulting $(N,K)$ is then said to be \emph{twist-induced}, of which we shall say more in Section \ref{sec:tw-ind_al_Rob}.
\item If the almost CR structure is totally degenerate, $\mc{K}$ is also non-twisting, and the resulting nearly Robinson manifold $(\mc{M},g,N,K)$ is either of \emph{Kundt type} or of \emph{Robinson--Trautman type} --- see Sections \ref{sec:alRob_Kundt} and \ref{sec:alRob_RT}.
\end{itemize}
There are further, intermediate, situations where the Levi form of $(\ul{H},\ul{J})$ is degenerate but not identically zero. This allows for screen bundle metrics to be constructed partly from the Levi form.
\end{rem}

\subsection{Conditions on the Robinson $3$-forms}\label{sec:other}
The purpose of this section is to highlight the fact that for a given almost Robinson manifold $(\mc{M}, g, N, K)$, conditions on Robinson $3$-forms do not necessarily entail that $(N, K)$ is nearly Robinson. The next two propositions illustrate the point.
\begin{prop}\label{prop-Rob-3-form-const}
Let $(\mc{M}, g, N, K)$ be an almost Robinson manifold with congruence of null curves $\mc{K}$  with leaf space $\ul{\mc{M}}$. The following statements are equivalent:
\begin{enumerate}
\item any Robinson $3$-form $\rho_{abc}$ is preserved along $K$, i.e.\ $\mathsterling_k \rho_{abc} = f \,  \rho_{abc}$ for some smooth function $f$, and any optical vector field $k$. \label{item:Rob-3-form-const1}
\item the intrinsic torsion is a section of
$(\slashed{\mc{G}}_{-1}^{1 \times 3})_{[2\i:1]} \cap \slashed{\mc{G}}_{-1}^{2,1}$, \label{item:Rob-3-form-const2} i.e.\
\begin{align*}
\breve{\gamma}_i = \breve{\sigma}_{\alpha \bar{\beta}} & = 0 \, , & \breve{\zeta}_{\alpha \beta} & =  - 2 \i \breve{\tau}_{\alpha \beta} \, .
\end{align*}
\item any Robinson $3$-form induces a $3$-form on the leaf space $(\ul{\mc{M}},\ul{H})$.\label{item:Rob-3-form-const3}
\end{enumerate}
If any of these conditions holds, $\mc{K}$ is geodesic and its shear anticommutes with the screen bundle complex structure.
\end{prop}

\begin{proof}
Choose splitting operators $(\ell^a, \delta^a_{\alpha},  \delta^a_{\bar{\alpha}} , k^a)$ and let $\rho_{a b c}$ be the Robinson $3$-form  associated to the optical $1$-form $\kappa_a = g_{a b} k^b$. The condition that $\rho_{a b c}$ be preserved along $K$ gives
\begin{align*}
 k^d \nabla_d \rho_{a b c}   - 3  k^d \nabla_{[a} \rho_{b c] d}  & = f \,  \rho_{a b c} \, .
\end{align*}
Now, contracting with $\delta^a_{\alpha}  \delta^b_{\beta} \delta^c_{\bar{\gamma}}$, $\delta^a_{\alpha}  \delta^b_{\beta} \ell^c$ and $\delta^a_{\alpha}  \delta^b_{\bar{\beta}} \ell^c$ yields $\gamma_\alpha = 0$,  $\zeta_{\alpha \beta} = - 2 \i \tau_{\alpha \beta}$ and $\sigma_{\alpha \bar{\beta}} = 0$ respectively. This proves the equivalence between conditions \eqref{item:Rob-3-form-const1} and \eqref{item:Rob-3-form-const2}. The equivalence with condition \eqref{item:Rob-3-form-const3} follows from the geometric interpretation of condition \eqref{item:Rob-3-form-const1}. An application of Lemma \ref{lem:opt_J} completes the proof.
\end{proof}

\begin{prop}\label{prop:Rob_rec}
Let $(\mc{M}, g, N, K)$ be an almost Robinson manifold with congruence of null curves $\mc{K}$. The following statements are equivalent:
\begin{enumerate}
\item the intrinsic torsion is a section of $\slashed{\mc{G}}_{-1}^{3,0}$,  i.e.\ \label{item:Rob_rec1}
\begin{align*}
\breve{\gamma}_i = \breve{\zeta}_{\alpha \beta} = 0 \, .
\end{align*}
\item any Robinson $3$-form is recurrent along $K$.\label{item:Rob_rec2}
\item $N$ is parallel along $K$.\label{item:Rob_rec3}
\end{enumerate}
If any of these conditions holds, $\mc{K}$ is geodesic.
\end{prop}

\begin{proof}
The equivalence between \eqref{item:Rob_rec1} and \eqref{item:Rob_rec2} is tautological, while the equivalence between \eqref{item:Rob_rec2} and \eqref{item:Rob_rec3} follows from the definition of $\breve{\zeta}_{\alpha \beta}$. That $\mc{K}$ is geodesic in this case follows from $\breve{\gamma}_i = 0$.
\end{proof}

Returning to nearly Robinson manifolds, we have the following lemma:
\begin{lem}\label{lem:twist_par}
Let $(\mc{M}, g, N, K)$ be a nearly Robinson manifold with congruence of null geodesics $\mc{K}$ and almost CR leaf space $(\ul{\mc{M}},\ul{H}, \ul{J})$. The following statements are equivalent:
\begin{enumerate}
\item any Robinson $3$-form is parallel along $K$;
\item $N$ is parallel along $K$;
\item the twist of $\mc{K}$ commutes with the screen bundle complex structure;
\item $(\ul{H},\ul{J})$ is partially integrable.
\end{enumerate}
\end{lem}

\begin{proof}
For a nearly Robinson manifold, the intrinsic torsion is a section of $\slashed{\mc{G}}_{-1}^2$, and so in particular, $\breve{\zeta}_{\alpha \beta} =  2\i \breve{\tau}_{\alpha \beta}$. Hence $\breve{\zeta}_{\alpha \beta}$ and $\breve{\tau}_{\alpha \beta} = 0$, and the result follows immediately by Lemma \ref{lem:opt_J}, Proposition \ref{prop:twist_al_CR1} and Proposition \ref{prop:Rob_rec}.
\end{proof}

It can readily be checked that the nearly Robinson manifolds of Propositions \ref{prop:twist_al_CR1}, \ref{prop:nRob_pi_ctct} and \ref{prop:Rob->pi_cont_CR} all satisfy this property.

\subsection{Twist-induced almost Robinson structures}\label{sec:tw-ind_al_Rob}
We shall now present a special case of almost Robinson structures, which arise from the twist of an optical geometry.
\begin{prop}\label{prop:twist-ind_Rob}
Let $(\mc{M},g,K)$ be an optical geometry of dimension $2m+2$ with congruence of null curves $\mc{K}$. Let $\kappa_a$ be an optical $1$-form and set $\tau_{a b c} := 3 \, \kappa_{[a} \nabla_{b} \kappa_{c]}$. The following two conditions are equivalent:
\begin{enumerate}
\item the $3$-form $\tau_{a b c}$ satisfies
\begin{align}\label{eq:Rob3sq}
\tau_{a b}{}^{e} \tau_{e c d} & = - \frac{2}{m} \tau_{[a}{}^{e f} g_{b][c} \tau_{d] e f} \neq 0 \, ;
\end{align}
\item $\mc{K}$ is geodesic and twisting, and there exists a unique optical vector field $k$ whose twist endormorphism $h^{-1} \circ \tau$ is a bundle complex structure $J$ compatible with $h$ on the screen bundle $(H_K,h)$, i.e.\ $J=h^{-1} \circ \tau$. In particular, the twist of $\mc{K}$ induces an almost Robinson structure $(N,K)$ on  $(\mc{M},g,K)$, and $\kappa = g(k, \cdot)$ determines a unique Robinson $3$-form given by
\begin{align}\label{eq:tw-Rob}
\rho_{a b c} & = \kappa_{[a} \nabla_{b} \kappa_{c]} \, .
\end{align}
\end{enumerate}
\end{prop}

\begin{proof}
Choose splitting operators $(\ell^a , \delta^a_i , k^a)$. Then we can write $\nabla_{[a} \kappa_{b]} = \lambda_{[a} \gamma_{b]} + \tau_{a b} + \alpha_{[a} \kappa_{b]}$ for some $\gamma_a$, $\tau_{a b}$ and $\alpha_a$, where $\ell^a \gamma_a =  \ell^a \tau_{a b} = k^a \gamma_a =  k^a \tau_{a b} = 0$. Contracting \eqref{eq:Rob3sq}  with $k^a \ell^b k^c \ell^d$ and $\delta^a_i \ell^b \delta^c_j \ell^d$ yields
\begin{align}
\gamma_i \gamma^i  & = 0 \, , \label{eq:cplx_geod} \\
\tau_{i}{}^{k} \tau_{k}{}^{j} & = - \frac{1}{2m} \tau_{k \ell} \tau^{k \ell} \delta_i^j \neq 0\, , \label{eq:cplx_tw}
\end{align}
respectively, where $\gamma_i = \gamma_a \delta^a_i$ and  $\tau_{i j} = \tau_{a b} \delta^a_i \delta^b_j$. Since $h_{i j}$ is positive-definite, equation \eqref{eq:cplx_geod} tells us that $\gamma_i = 0$, i.e.\ $\mc{K}$ is geodesic. Equation \eqref{eq:cplx_tw} tells us that we can rescale $k$ by $\tfrac{\sqrt{2m}}{ \| \tau \|}$ so that the twist of the rescaled optical vector field satisfies
\begin{align*}
\tau_{i}{}^{k} \tau_{k}{}^{j} & = -  \delta_i^j \, ,
\end{align*}
i.e.\ $h^{-1} \circ \tau$ is a bundle complex structure on $H_K$. The uniqueness of $k$ follows from the assumption that $K$ is oriented.
\end{proof}

\begin{rem}
That equation \eqref{eq:tw-Rob} singles out an optical $1$-form also follows from the fact that LHS has boost weight $2$ and the RHS boost weight $1$.
\end{rem}

\begin{defn}
We shall refer to the almost Robinson structure given in Proposition \ref{prop:twist-ind_Rob} as a \emph{twist-induced almost Robinson structure}.
\end{defn}

\begin{rem}
Let us re-emphasise that by Proposition \ref{prop:twist-ind_Rob}, the congruence associated to a twist-induced almost Robinson structure is always geodesic and maximally twisting.
\end{rem}

It is clear that an almost Robinson structure $(N,K)$ is twist-induced if and only if its intrinsic torsion is a section of $\slashed{\mc{G}}_{-1}^{1,1} \cap \slashed{\mc{G}}_{-1}^{1,2}$ with non-vanishing $\mc{G}_{-1}^{1,0}$-component, i.e.\
\begin{align*}
\breve{\gamma}_i = \breve{\tau}_{\alpha \beta} =  \breve{\tau}^\circ_{\alpha \bar{\beta}} & = 0 \, , & \breve{\tau}^\omega  & \neq 0 \, .
\end{align*}
However, the following proposition tells us that the intrinsic torsion must in fact be a section of a subbundle thereof.

\begin{prop}\label{prop:deg+}
Let $(\mc{M},g,N, K)$ be an almost Robinson manifold with congruence of null curves $\mc{K}$. The following statements are equivalent:
\begin{enumerate}
\item $(N,K)$ is a twist-induced almost Robinson structure;\label{item:deg+1}
\item the intrinsic torsion of $(N,K)$ is a section of $\slashed{\mc{G}}_{-1}^{1,1} \cap \slashed{\mc{G}}_{-1}^{1,2}  \cap \slashed{\mc{G}}_{-1}^{3,0} \cap \slashed{\mc{G}}_0^{1,1}$ with non-vanishing $\mc{G}_{-1}^{1,0}$-component, i.e.\ \label{item:deg+2}
\begin{align*}
\breve{\gamma}_i = \breve{\tau}_{\alpha \beta} =  \breve{\tau}^\circ_{\alpha \bar{\beta}} = \breve{\zeta}_{\alpha \beta} = \breve{G}^{\tiny{\ydskew}}_{\alpha \beta \gamma} & = 0 \, , & \breve{\tau}^\omega  & \neq 0 \, .
\end{align*}
\end{enumerate}
If any of these conditions holds, $N$ is parallel along $K$.
\end{prop}

\begin{proof}
That \eqref{item:deg+2} implies \eqref{item:deg+1} is immediate since $\slashed{\mc{G}}_{-1}^{1,1} \cap \slashed{\mc{G}}_{-1}^{1,2}  \cap \slashed{\mc{G}}_{-1}^{3,0} \cap \slashed{\mc{G}}_0^{1,1}$ is a subbundle of $\slashed{\mc{G}}_{-1}^{1,1} \cap \slashed{\mc{G}}_{-1}^{1,2}$. For the converse, we note that a twist-induced almost Robinson structure singles out a preferred optical $1$-form $\kappa$ such that $(\d \kappa)_{a b} = \omega_{a b} + \kappa_{[a} \alpha_{b]}$ for some $1$-form $\alpha_a$, where $\omega_{a b}$ is a representative of the screen bundle Hermitian structure. Note that the associated Robinson $3$-form is given by $\rho_{a b c} = 3 \kappa_{[a} \omega_{b c]}$. Taking the exterior derivative of $\d \kappa$ yields
\begin{align}\label{eq:dom}
0 & = (\d \omega)_{a b c} + \omega_{[a b} \alpha_{c]}  - \kappa_{[a} (\d \alpha)_{b c]} \, .
\end{align}
Choose a splitting $(\ell^a , \delta^a_{\alpha} , \delta^a_{\bar{\alpha}} , k^a)$. Contracting \eqref{eq:dom} with $k^a \delta^b_\alpha \delta^c_\beta$ leads to
\begin{align*}
0 & = 3 (\d \omega)^{0}{}_{\alpha \beta} \\
& = ( \nabla \omega )^{0}{}_{\alpha \beta}  - 2\i  \omega_{a b}  \delta^a_\alpha \delta^b_\beta \\
& = ( \nabla \rho )^{0}{}_{\alpha \beta 0} = \zeta_{\alpha \beta} \, ,
\end{align*}
i.e.\ $\mathring{T}$ is a section of $\slashed{\mc{G}}_{-1}^{3,0}$. Now, contract \eqref{eq:dom} with $\delta^a_{\alpha} \delta^b_{\beta} \delta^c_{\gamma}$ yields \begin{align*}
0 & = ( \d \omega )_{\alpha \beta \gamma} \\
& = ( \nabla \rho )_{[\alpha \beta \gamma] 0} = G^{\tiny{\ydskew}}_{\alpha \beta \gamma} \, ,
\end{align*}
i.e.\ $\mathring{T}$ is a section of $\slashed{\mc{G}}_0^{1,1}$, which completes the proof.
\end{proof}

\begin{rem}
Proposition \ref{prop:deg+} tells us that if the intrinsic torsion of a given almost Robinson manifold is a section of $\slashed{\mc{G}}_{-1}^{1,1} \cap \slashed{\mc{G}}_{-1}^{1,2}$ but not of $\slashed{\mc{G}}_{-1}^{1,0}$, then it must be a section of $\slashed{\mc{G}}_{-1}^{1,1} \cap \slashed{\mc{G}}_{-1}^{1,2}  \cap \slashed{\mc{G}}_{-1}^{3,0} \cap \slashed{\mc{G}}_0^{1,1}$. This should be contrasted with the situation regarding the Gray--Hervella classification of almost Hermitian manifolds \cite{Gray1980}, which is briefly reviewed in Section \ref{sec:Rob-GH}: the sixteen classes of almost Hermitian manifolds can be naturally arranged in terms of inclusions, which are shown to be strict in the sense that each class contains an almost Hermitian metric that does not belong to any of the other fifteen classes.
\end{rem}

The next result is a direct consequence of Propositions \ref{prop:nRob_pi_ctct} and \ref{prop:deg+}.
\begin{prop}\label{prop:tw-ind_Rob_shear}
Let $(\mc{M},g,N, K)$ be an almost Robinson manifold with congruence of null curves $\mc{K}$. The following statements are equivalent:
\begin{enumerate}
\item $(N,K)$ is a twist-induced almost Robinson structure and the shear of $\mc{K}$ commutes with the screen bundle complex structure;
\item the intrinsic torsion of $(N,K)$ is a section of $\slashed{\mc{G}}_{-1}^{1,1} \cap \slashed{\mc{G}}_{-1}^{1,2}  \cap \slashed{\mc{G}}_{-1}^{2,1}  \cap \slashed{\mc{G}}_{-1}^{3,0} \cap \slashed{\mc{G}}_0^{1,1} $ with non-vanishing $\mc{G}_{-1}^{1,0}$-component, i.e.\
\begin{align*}
\breve{\gamma}_i = \breve{\tau}_{\alpha \beta} =  \breve{\tau}^\circ_{\alpha \bar{\beta}} = \breve{\sigma}_{\alpha \beta} = \breve{\zeta}_{\alpha \beta} = \breve{G}^{\tiny{\ydskew}}_{\alpha \beta \gamma} & = 0 \, , & \breve{\tau}^\omega  & \neq 0 \, .
\end{align*}
\end{enumerate}
If any of these conditions holds, $(N,K)$ is a nearly Robinson structure and the induced almost CR leaf space $(\ul{\mc{M}},\ul{H},\ul{J})$ of $\mc{K}$ is partially integrable and contact, and is thus equipped with a subconformal structure $\ul{\mbf{c}}_{\ul{H},\ul{J}}$ compatible with $\ul{J}$.
\end{prop}

The next proposition, which collects some of the results found in  \cite{Taghavi-Chabert2021}, is a special case of Proposition \ref{prop:Rob->pi_cont_CR} applied to a twist-induced almost Robinson structure --- see also Remark \ref{rem:two_cf_str} where  equation \eqref{eq:Levi-conf} reduces to $\ul{\msf{h}}_{\alpha \bar{\beta}} = \ul{h}_{\alpha \bar{\beta}}$.
\begin{prop}[\cite{Taghavi-Chabert2021}]\label{prop:tw-ind_Rob_no_shear}
Let $(\mc{M},g,N, K)$ be an almost Robinson manifold with congruence of null curves $\mc{K}$. The following statements are equivalent:
\begin{enumerate}
\item $(N,K)$ is a twist-induced almost Robinson structure and $\mc{K}$ is a non-shearing congruence of null geodesics;
\item the intrinsic torsion of $(N,K)$ is a section of $\slashed{\mc{G}}_{-1}^{1,1} \cap \slashed{\mc{G}}_{-1}^{1,2}  \cap \slashed{\mc{G}}_{-1}^{2}  \cap \slashed{\mc{G}}_{-1}^{3,0} \cap \slashed{\mc{G}}_0^{1,1} \cap \slashed{\mc{G}}_0^{1,3}$ with non-zero $\mc{G}_{-1}^{1,0}$-component, i.e.\
\begin{align*}
\breve{\gamma}_i = \breve{\tau}_{\alpha \beta} =  \breve{\tau}^\circ_{\alpha \bar{\beta}} = \breve{\sigma}_{i j} = \breve{\zeta}_{\alpha \beta} = \breve{G}^{\tiny{\ydskew}}_{\alpha \beta \gamma} = \breve{G}^{\circ}_{\bar{\alpha} \beta \gamma} & = 0 \, , & \breve{\tau}^\omega  & \neq 0 \, .
\end{align*}
\end{enumerate}
If any of these conditions holds, $(N,K)$ is a nearly Robinson structure and the almost CR leaf space $(\ul{\mc{M}},\ul{H},\ul{J})$ of $\mc{K}$ is partially integrable and contact.

In particular, $(\ul{\mc{M}},\ul{H},\ul{J})$ is equipped with a positive-definite subconformal structure $\ul{\mbf{c}}_{\ul{H},\ul{J}}$ compatible with $\ul{J}$, which is also induced from $[g]_{n.e.}$.
\end{prop}

\begin{rem}
In both Propositions \ref{prop:tw-ind_Rob_shear} and \ref{prop:tw-ind_Rob_no_shear}, the $\mc{G}_{0}^{1,2}$-component $\breve{G}^{\tiny{\ydhook}}_{\alpha \beta \gamma}$ of the intrinsic torsion of $(N,K)$ can be identified with the Nijenhuis  tensor of $(\ul{H},\ul{J})$.
\end{rem}

\begin{rem}
Let $(\mc{M},g,N, K)$ be a twist-induced almost Robinson manifold with non-shearing congruence of null geodesics $\mc{K}$.  Applying Proposition 4.35 of \cite{Fino2020}, one can show that for each metric $\wh{g}$ in $[g]_{n.e}$, there exists a unique generator $k$ of $\mc{K}$ and a null vector field $\ell$ such that $\wh{g}(k,\ell)=1$ and $\kappa = \wh{g}(k,\cdot)$ satisfies
\begin{align*}
\d \kappa (k, \cdot) & = 0 \, , & \d \kappa (\ell, \cdot) & = 0 \, .
\end{align*}
We shall elaborate on this result in the conformal setting in Section \ref{sec:conf-aRstr}.
\end{rem}

\subsection{Almost Robinson structures as almost null structures}\label{sec:spinor_des}
To obtain further geometric interpretations of the subbundles of $\mc{G}$, we shall presently regard the almost Robinson structure $(N,K)$ as an almost null structure $N$ on $(\mc{M},{}^\C g)$ in its own right, i.e.\ without any reference to the complex conjugate $\overline{N}$. This perspective is in part motivated by the potential involutivity of $N$. The structure group of $N$ is the stabiliser $R$ of an MTN vector subspace of $\C^{2m+2}$ in $\SO(2m+2,\C)$. As we shall be using a spinorial approach, we shall assume with no loss of generality, at least locally, that $R$ is a subgroup of $\Spin(2m+2,\C)$.  We shall then identify $N$ as the kernel of the map $\nu_a^{\mbf{A}} := \gamma_{a \mbf{B}'}{}^{\mbf{A}} \nu^{\mbf{B}'} : \Gamma(T \mc{M }) \rightarrow \Gamma(S_-)$ for some Robinson spinor $\nu^{\mbf{A}'}$. In particular, with a choice of dual $N^*$, the image of $\nu_a^{\mbf{A}}$ is isomorphic to $N^* \cong T \mc{M}/N$ --- see the discussion in Section \ref{sec-spinors}. The intrinsic torsion of such a structure is already investigated in \cite{Taghavi-Chabert2016}, and we shall appeal to the results contained therein for the subsequent analysis.
\begin{thm}\label{thm:spinor_intors}
Let $(\mc{M}, g, K, N)$ be a $(2m+2)$-dimensional almost Robinson manifold with intrinsic torsion $\mathring{T}$. Let $\nu^{\mbf{A}'}$ be a Robinson spinor and set $\nu_a^{\mbf{A}} := \gamma_{a \mbf{B}'}{}^{\mbf{A}} \nu^{\mbf{B}'}$. Then, for $m>2$,
\begin{align}
\mathring{T} & \in \Gamma ( \slashed{\mc{G}}_0^{1,1} )  & \Longleftrightarrow & &\left( \nu^{a [\mbf{A}} \nabla_a \nu^{b \mbf{B}} \right) \nu_{b}^{\mbf{C}]} & = 0 \, , \label{eq:pure_spinor1} \\
\mathring{T} & \in \Gamma ( \slashed{\mc{G}}_0^{1,2} ) & \Longleftrightarrow & & \left(\nu^{a (\mbf{A}} \nabla_a \nu^{b \mbf{B})} \right) \nu_{b}^{\mbf{C}} & = 0 \, , \label{eq:pure_spinor2}  \\
 \mathring{T} & \in \Gamma ( \slashed{\mc{G}}_1^{0,0} ) & \Longleftrightarrow & & \left( \nabla_a \nu^{b \mbf{B}} \right) \nu_{b}^{\mbf{C}} + \frac{2}{m} \left( \nu_a^{[\mbf{B}} \nabla_b \nu^{b \mbf{C}]} + \nu^{b [\mbf{B}} \nabla_b \nu_{a}^{\mbf{C}]} \right) & = 0  \, , \label{eq:pure_spinor3}\\
  \mathring{T} & \in \Gamma \left( ({\slashed{\mc{G}}_0^{0 \times 1}})_{[2 \i:1]} \right)  & \Longleftrightarrow & & \nu^{\mbf{A}'} \nabla_a \nu^{a \mbf{B}} - \nu^{a \mbf{B}} \nabla_b \nu^{\mbf{A}'} & = 0 \, . \label{eq:pure_spinor0} 
\end{align}
In dimension six, i.e.\ $m=2$, equivalences  \eqref{eq:pure_spinor2} and  \eqref{eq:pure_spinor0} hold, but equivalences \eqref{eq:pure_spinor1} and \eqref{eq:pure_spinor3} are replaced by
\begin{align}
\mathring{T} & \in \Gamma ((\slashed{\mc{G}}_{-1}^{1 \times 3})_{[4\i:-1]} )  & \Longleftrightarrow & &\left(\nu^{a [\mbf{A}} \nabla_a \nu^{b \mbf{B}} \right) \nu_{b}^{\mbf{C}]} = 0 \, , \label{eq:pure_spinor1d6}
\end{align}
and
\begin{multline}
 \mathring{T} \in \Gamma ( \slashed{\mc{G}}_1^{0,0} )  \Longleftrightarrow \\
 \left( \nabla_a \nu^{b \mbf{B}} \right) \nu_{b}^{\mbf{C}} + \left( \nu_a^{[\mbf{B}} \nabla_b \nu^{b \mbf{C}]} + \nu^{b [\mbf{B}} \nabla_b \nu_{a}^{\mbf{C}]} \right) - \frac{2}{3} \nu^{b \mbf{A}} \nabla_b \nu_{\mbf{A}} \gamma_a{}^{\mbf{B} \mbf{C}} = 0  \, , \label{eq:pure_spinor3d6}
\end{multline}
respectively. For the last equation, we have used the bundle isomorphisms $S_\pm \cong S_\mp^*$ and ${}^\C T \mc{M} \cong \bigwedge^2 S_+ \cong \bigwedge^2 S_-$.
\end{thm}

\begin{proof}
We choose a splitting ${}^{C} T \mc{M} = N^* \oplus N$. With reference to the notation of Section \ref{sec-spinors}, we may choose a connection $1$-form $\Gamma_{a b}{}^{c}$ with values in $\so(2m+2,\C)$ for $\nabla$, so that $\left( \nabla_a \nu^{b \mbf{B}} \right) \nu_{b}^{\mbf{C}} = \Gamma_{a}{}^{B C} \delta_{B}^{\mbf{B}} \delta_{C}^{\mbf{C}}$ and in particular, $\left(\nu^{a \mbf{A}}  \nabla_a \nu^{b \mbf{B}} \right) \nu_{b}^{\mbf{C}} = \Gamma^{A B C} \delta_{A}^{\mbf{A}} \delta_{B}^{\mbf{B}} \delta_{C}^{\mbf{C}}$.
Then, one can show \cite{Taghavi-Chabert2016}
\begin{align*}
\mbox{RHS of \eqref{eq:pure_spinor1} and \eqref{eq:pure_spinor1d6}} & & \Longleftrightarrow & & \Gamma^{[A B C]} & = 0 \, ,  \\
\mbox{RHS of \eqref{eq:pure_spinor2}} & & \Longleftrightarrow & & \Gamma^{(A B) C} & = 0 \, ,  \\
\mbox{RHS of \eqref{eq:pure_spinor3}} & & \Longleftrightarrow & & \left( \Gamma_{A}{}^{B C} \right)_\circ = \Gamma^{A B C} & = 0  \, , \\
\mbox{RHS of \eqref{eq:pure_spinor0}} & & \Longleftrightarrow & & \Gamma_{B}{}^{B A} = \Gamma^{A B C}& = 0 \, ,
\end{align*}
and in dimension six
\begin{align*}
\mbox{RHS of \eqref{eq:pure_spinor3d6}} & & \Longleftrightarrow & & \left( \Gamma_{A}{}^{B C} \right)_\circ = \Gamma^{(A B) C} & = 0  \, .
\end{align*}
Here $\left( \Gamma_{A}{}^{B C} \right)_\circ = \Gamma_{A}{}^{B C} - \frac{2}{m} \delta_A^{[B|} \Gamma_{D}{}^{D|C]}$. We can immediately deduce that
\begin{align*}
\Gamma^{[A B C]} & = 0   & \Longleftrightarrow & & G^{\tiny{\ydskew}}_{\bar{\alpha} \bar{\beta} \bar{\gamma}} = \zeta_{\bar{\alpha} \bar{\beta}} + 4 \i \tau_{\bar{\alpha} \bar{\beta}} & = 0 \, , \\
\Gamma^{(A B) C} & = 0  & \Longleftrightarrow & & G^{\tiny{\ydhook}}_{\bar{\alpha} \bar{\beta} \bar{\gamma}} = \sigma_{\bar{\alpha} \bar{\beta}} = 2 \i \tau_{\bar{\alpha} \bar{\beta}} - \zeta_{\bar{\alpha} \bar{\beta}} = \gamma_{\bar{\alpha}} & = 0 \, , \\
\left( \Gamma_{A}{}^{B C} \right)_\circ & = 0 & \Longleftrightarrow & & B_{\bar{\alpha} \bar{\beta}} = G^{\circ}_{\alpha \bar{\beta} \bar{\gamma}} = 2 \i (m-1) E_{\bar{\alpha}} +  G_{\bar{\alpha}} =  \tau^\circ_{\alpha \bar{\beta}} =  \sigma_{\alpha \bar{\beta}} & = 0 \, , \\
\Gamma_{B}{}^{B A} & = 0 & \Longleftrightarrow & & - 2 \i E_{\bar{\alpha}} + G_{\bar{\alpha}} = \tau^\omega  = \epsilon & = 0 \, ,
\end{align*}
where $B_{\bar{\alpha} \bar{\beta}}$, $G^{\tiny{\ydskew}}_{\bar{\alpha} \bar{\beta} \bar{\gamma}}$, $G^{\tiny{\ydhook}}_{\bar{\alpha} \bar{\beta} \bar{\gamma}}$, $G^{\circ}_{\alpha \bar{\beta} \bar{\gamma}}$, $\zeta_{\bar{\alpha} \bar{\beta}}$, $\tau_{\bar{\alpha} \bar{\beta}}$, $E_{\bar{\alpha}}$ and $\gamma_{\bar{\alpha}}$ are the complex conjugates of the tensors defined in Section \ref{sec:Intrinsic_Torsion}. The result now follows immediately from the definition of the $Q$-invariant bundles and Theorem \ref{thm:intors-rob} --- this is clearly independent from the choice of dual $N^*$.
\end{proof}

As a consequence, we obtain:
\begin{prop}\label{prop:fol_spinor}
Let $(\mc{M}, g, N, K)$ be an almost Robinson manifold with congruence of null curves $\mc{K}$. The following statements are equivalent:
\begin{enumerate}
\item any Robinson spinor $\nu^{\mbf{A}'}$ satisfies \label{item:fol_spinor1}
\begin{align}\label{eq:foliating_spinor}
\left( \nu^{a \mbf{A}} \nabla_a \nu^{b \mbf{B}} \right) \nu_{b}^{\mbf{C}} & = 0 \, ;
\end{align}
\item any Robinson spinor $\nu^{\mbf{A}'}$ is recurrent along $N$, i.e.\ \label{item:fol_spinor2}
\begin{align}\label{eq:foliating_rec_spinor}
\left( \nu^{a \mbf{A}} \nabla_a \nu^{[\mbf{B}'} \right) \nu^{\mbf{C}']} & = 0 \, ;
\end{align}
\item the intrinsic torsion of $(N,K)$ is a section of $\slashed{\mc{G}}_{0}^{1,1} \cap \slashed{\mc{G}}_{0}^{1,2}$, i.e.\ \label{item:fol_spinor3}
\begin{align*}
\breve{\gamma}_i = \breve{\tau}_{\alpha \beta} = \breve{\sigma}_{\alpha \beta} = \breve{\zeta}_{\alpha \beta} = \breve{G}^{\tiny{\ydskew}}_{\alpha \beta \gamma} = \breve{G}^{\tiny{\ydhook}}_{\alpha \beta \gamma} = 0 \, ;
\end{align*}
\item $N$ is in involution, i.e.\ for any $v, w \in \Gamma(N)$, $[v,w] \in \Gamma(N)$; \label{item:fol_spinor4}
\item $N$ is parallel along itself, i.e.\ for any $v, w \in \Gamma(N)$, $\nabla_v w \in \Gamma(N)$; \label{item:fol_spinor5}
\item $(N,K)$ induces a CR structure on the leaf space of $\mc{K}$. \label{item:fol_spinor6}
\end{enumerate}
If any of these conditions holds, $\mc{K}$ is geodesic, its twist and shear commute  with the screen bundle complex structure, and $N$ is parallel along $K$.
\end{prop}

\begin{proof}
One can show \cite{Taghavi-Chabert2016} that equations \eqref{eq:foliating_spinor} and \eqref{eq:foliating_rec_spinor} are equivalent, from which follows the equivalence \eqref{item:fol_spinor1} and \eqref{item:fol_spinor2}. The equivalence between \eqref{item:fol_spinor1} and \eqref{item:fol_spinor3} is a direct consequence of Theorem \ref{thm:spinor_intors}. The equivalence between \eqref{item:fol_spinor1} and \eqref{item:fol_spinor6} is already given in \cite{Hughston1988,Taghavi-Chabert2016}, while the equivalence between \eqref{item:fol_spinor4} and \eqref{item:fol_spinor5} is established in \cite{Taghavi-Chabert2012}, and that between \eqref{item:fol_spinor4} and \eqref{item:fol_spinor6} in \cite{Nurowski2002}.
\end{proof}

\begin{rem}
It is important to note that the involutivity of the almost Robinson structure $(N,K)$ does \emph{not} imply that the Robinson $3$-form is preserved along $K$, and thus, may not descend to the leaf space $\ul{\mc{M}}$ --- see Proposition \ref{prop-Rob-3-form-const}. If it did, it would imply that the congruence of null geodesics is non-shearing, which is not true in general except in dimension four.
\end{rem}

\begin{rem}\label{rem:complex_ext}
In the analytic category, one may complexify $(\mc{M}, g)$ to a complex Riemanian manifold $(\wt{\mc{M}}, \wt{g})$, and extend $N$ analytically to $(\wt{\mc{M}}, \wt{g})$ \cite{Whitney1959,Woodhouse1977,Eastwood1984}. By the Frobenius theorem, condition \eqref{item:fol_spinor4} of Proposition \ref{prop:fol_spinor} is equivalent to the local existence of a complex foliation of $(\wt{\mc{M}}, \wt{g})$ by $(m+1)$-dimensional complex submanifolds on which $\wt{g}$ is totally degenerate. Condition \eqref{item:fol_spinor5} then tells us that these leaves are totally geodesic with respect to the Levi-Civita connection of $\wt{g}$. See also \cite{Penrose1986,Hughston1988,TaghaviChabert2017} for further details.
\end{rem}

The following proposition gives a characterisation of a certain class of almost Robinson manifolds, not necessarily nearly Robinson.
\begin{prop}\label{prop:al_K-Rob}
Let $(\mc{M}, g, N, K)$ be an almost Robinson manifold with congruence of null curves $\mc{K}$. The following statements are equivalent:
\begin{enumerate}
\item any Robinson $3$-form $\rho_{a b c}$ satisfies \label{item:al_K-Rob1}
\begin{align}\label{eq:closed-Rob3f}
\d \rho & = \alpha \wedge \rho \, , & \mbox{for some $1$-form $\alpha$} \, ;
\end{align}
\item the intrinsic torsion is a section of $\slashed{\mc{G}}_{0}^{1,1} \cap \slashed{\mc{G}}_{0}^{1,3}$ (or $\slashed{\mc{G}}_{0}^{1,3}$ in dimension six), i.e.\ \label{item:al_K-Rob2}
\begin{align*}
\breve{\gamma}_i = \breve{\tau}_{\alpha \beta} =  \breve{\tau}^\circ_{\alpha \bar{\beta}} = \breve{\sigma}_{\alpha \bar{\beta}} = \breve{\zeta}_{\alpha \beta} = \breve{G}^{\tiny{\ydskew}}_{\alpha \beta \gamma}  = \breve{G}^{\circ}_{\bar{\alpha} \beta \gamma} = 0 \, .
\end{align*}
\end{enumerate}
Further, assuming that any of these conditions holds:
\begin{itemize}
\item $\mc{K}$ is geodesic, its shear anti-commutes with the screen bundle complex structure.
\item $(N,K)$ is twist-induced if and only if $\mc{K}$ is twisting.
\item $(N,K)$ is in addition nearly Robinson if and only if $\mc{K}$ is also non-shearing.
\end{itemize}
\end{prop}

\begin{proof}
Choose an optical $1$-form $\kappa_a$ with Robinson $3$-form $\rho_{a b c}$. Then, using $(\d \rho)_{a b c d} = \nabla_{[a} \rho_{b c d]}$, we compute the various components of \eqref{eq:closed-Rob3f} to find
\begin{align*}
\left(\d \rho - \alpha \wedge \rho \right)_{\alpha \beta \bar{\gamma}}{}^{0} & = 0 & \Longleftrightarrow & & \gamma_\alpha = 0 \, ,  \\
\left(\d \rho - \alpha \wedge \rho \right)_{\alpha \beta 0}{}^{0} & = 0 & \Longleftrightarrow & & 2 \i \tau_{\alpha \beta} + \zeta_{\alpha \beta} = 0 \, , \\
\left(\d \rho - \alpha \wedge \rho \right)_{\alpha \bar{\beta} 0}{}^{0} & = 0 & \Longleftrightarrow &  & \sigma_{\alpha \bar{\beta}} = 0 \, , \\
\left(\d \rho - \alpha \wedge \rho \right)_{\alpha \beta \gamma 0}  & = 0 & \Longleftrightarrow & & G^{\tiny{\ydskew}}_{\alpha \beta \gamma} = 0 \, , \\
\left(\d \rho - \alpha \wedge \rho \right)_{\bar{\alpha} \beta \gamma 0} & = 0 & \Longleftrightarrow & & G^{\circ}_{\bar{\alpha} \beta \gamma} = 0 \, ,\\
\left(\d \rho - \alpha \wedge \rho \right)_{\alpha \beta \gamma \bar{\delta}} & = 0 & \Longleftrightarrow & & \tau_{\alpha \beta} = 0 \, , \\
\left(\d \rho - \alpha \wedge \rho \right)_{\alpha \bar{\beta} 0}{}^{0} & = 0 & \Longleftrightarrow &&  \tau^\circ_{\alpha \bar{\beta}} = 0 \, .
\end{align*}
The remaining contractions are vacuous. The equivalence between \eqref{item:al_K-Rob1} and \eqref{item:al_K-Rob2} now follows.

The properties of $\mc{K}$ and its leaf space $\ul{\mc{M}}$ follow from $\breve{\gamma}_i = \breve{\tau}_{\alpha \beta} =  \breve{\tau}^\circ_{\alpha \bar{\beta}} = \breve{\sigma}_{\alpha \bar{\beta}} = \breve{\zeta}_{\alpha \beta}  =0$. In particular, the only non-vanishing component of the twist is $\breve{\tau}^\omega$, which tells us that $(N,K)$ is induced from the twist if and only if $\breve{\tau}^\omega \neq 0$. Finally, with reference to Proposition \ref{prop:N_pres}, we see that the only obstruction to $(N,K)$ being nearly Robinson is given by the $(2,0)$-part of the shear $\breve{\sigma}_{\alpha \beta}$.
\end{proof}

\begin{rem}
Proposition \ref{prop:al_K-Rob} reduces to Proposition \ref{prop:tw-ind_Rob_no_shear} when $\mc{K}$ is (maximally) twisting and non-shearing.
\end{rem}

We now present a couple of examples illustrating notably the results of this section and that of Section \ref{sec:lift}. 

\begin{exa}[The Kerr--NUT--(A)dS metric]\label{exa:KerrNUTAdS}
In \cite{Chen2006}, the authors present the \emph{Kerr--NUT--(A)dS metric} in arbitrary dimensions, which partly generalises the Pleba\'{n}ski--Demia\'{n}ski metric \cite{Plebanski1976}, and admits a Euclidean analogue under a Wick rotation. In dimension $2m+2$, in coordinates $(r, x_\alpha, t, \psi_i)$, where $\alpha, i=1, \ldots m$, this Einstein metric takes the form
\begin{multline*}
g = \frac{U}{X} \d r^2 - \frac{X}{U}\left( \d t + \sum_{k=1}^{m} A^{(k)} \d \psi_k \right)^2 \\
+ \sum_{\alpha=1}^m \left( \frac{U_\alpha}{X_\alpha} \d x_\alpha^2 
+ \frac{X_\alpha}{U_\alpha} \left( \d t + \sum_{k=1}^m A^{(k)}_{\alpha}  \d \psi_k \right)^2 \right) \, ,
\end{multline*}
where
\begin{align*}
U & = \prod_{\beta=1}^m (r^2 + x_\beta^2) \, ,  \qquad 
X =  \sum_{k=1}^{m+1} (-1)^k c_k r^{2k} + M r \, ,  &
A^{(k)} & = \sum_{\mathclap{\nu_1 < \ldots < \nu_k}} x_{\nu_1}^2 \ldots x_{\nu_k}^2 \, ,
\intertext{and, for $\alpha=1, \ldots , m$,}
 U_\alpha & = (r^2 + x_\alpha^2) \prod_{\alpha\neq\beta} (x_\beta^2 - x_\alpha^2) \, ,  &
X_\alpha & = \sum_{k=1}^{m+1} c_k x_{\alpha}^{2k} + L_{\alpha} r \, ,  \\
A^{(k)}_{\alpha} & = \sum_{\mathclap{\substack{\beta_1 < \ldots < \beta_k\\ \beta_i \neq \alpha}}} x_{\beta_1}^2 \ldots x_{\beta_k}^2 - r^2 \sum_{\mathclap{\substack{\beta_1 < \ldots < \beta_{k-1}\\ \beta_i \neq \alpha}}} x_{\beta_1}^2 \ldots x_{\beta_{k-1}}^2 \, .
\end{align*}
Here, $M$, $L_{\alpha}$, $\alpha=1,\ldots m$ and $c_{\alpha}$ are constants related to the mass, NUT parameters, the cosmological constant and rotation parameters of the black hole.

Define 
\begin{align*}
\kappa & =  \sqrt{\frac{U}{2X}} \left( \d r + \frac{X}{U}\left( \d t + \sum_{k=1}^{m} A^{(k)} \d \psi_k \right) \right)  \, , \\
\lambda & = \sqrt{\frac{U}{2X}} \left( \d r - \frac{X}{U}\left( \d t + \sum_{k=1}^{m} A^{(k)} \d \psi_k \right) \right)  \, , \\
\theta^\alpha & = \sqrt{\frac{U_\alpha}{2X_\alpha}} \left(  \d x_\alpha 
+ \i \frac{X_\alpha}{U_\alpha} \left( \d t + \sum_{k=1}^m A^{(k)}_{\alpha}  \d \psi_k \right)\right) \, , \quad (\alpha=1, \ldots , m) \, .
\end{align*}
The set of null $1$-forms $(\kappa, \theta^\alpha)$, where $\kappa$ is real and $\theta^\alpha$ are complex, defines an almost Robinson structure $(N,K)$.  We compute the intrinsic torsion of $(N,K)$, and find, suspending the summation convention,
\begin{flalign*}
\gamma_i & = \sigma_{\alpha \beta} = \tau_{\alpha \beta} = \zeta_{\alpha \beta} = G_{\alpha \beta \gamma} = 0 \, , & \\
\frac{\epsilon}{2m} h_{\alpha \bar{\beta}} + \sigma_{\alpha \bar{\beta}} + \tau_{\alpha \bar{\beta}} & = \sqrt{\frac{X}{2U}} \frac{r - \i x_{\alpha}}{r^2 + x_{\alpha}^2} \delta_{\alpha \bar{\beta}} \, , & \\
E_{\alpha} & = \i \sqrt{\frac{X_{\alpha}}{2U_{\alpha}}} \frac{r + \i x_{\alpha}}{r^2 + x_{\alpha}^2} \delta_{\alpha \bar{\beta}} \, ,
& G_{\bar{\gamma} \alpha \beta} = - \i \sqrt{\frac{2 X_{\beta}}{U_{\beta}}} \frac{1}{x_{\alpha} + x_{\beta}} \delta_{\alpha \bar{\gamma}} \, , & \\
B_{\alpha \beta} & = 0 \, . &
\end{flalign*}
We see at once that the conditions of Proposition \ref{prop:fol_spinor} are met. In particular, $(N,K)$ is involutive, i.e.\ any of the associated spinors satisfies \eqref{eq:foliating_spinor}, but none of the stronger conditions \eqref{eq:pure_spinor3}, \eqref{eq:pure_spinor3d6} and \eqref{eq:pure_spinor0}. Thus, the intrinsic torsion of $(N,K)$ is a section of $\slashed{\mc{G}}_{0}^{1,1} \cap \slashed{\mc{G}}_{0}^{1,2}$, and does not degenerate further.\footnote{This can be seen by inspection of the above computation bearing in mind that here, the vanishing of $B_{\alpha \beta}$ is not invariant under the structure group of $(N,K)$.} In particular, when $m>1$, the Robinson structure is not twist-induced.

In addition, the congruence $\mc{K}$ tangent to $K$ is geodesic, expanding, maximally twisting and shearing, when $m>1$  --- see also \cite{Pravda2007} --- and non-shearing when $m=1$ as is well-known.

Hence, by Propositions \ref{prop:twist_al_CR4} and \ref{prop:fol_spinor}, the Robinson structure $(N,K)$ descends to a contact CR structure on the leaf space of $\mc{K}$.

Similarly, the set of null $1$-forms $(\lambda, \theta^\alpha)$ also defines an almost Robinson structure whose intrinsic torsion share the same properties as $(N,K)$. More generally, it is shown in \cite{Mason2010} that this metric admits $2^{m+1}$ almost null structures, which yield $2^{m-1}$ Robinson structures associated to each of the optical structures $K$ and $L$, all sharing the same properties.

These findings also apply to other related metrics such as the Myers--Perry metric \cite{Myers1986} that may be viewed as a special case of the Kerr--NUT--(A)dS metric in the limit where the NUT parameters and cosmological constant tend to zero.
\end{exa}

\begin{exa}[Taub--NUT--(A)dS and Fefferman--Einstein metrics]\label{exa:F-E_Taub-NUT}
It is shown in \cite{Taghavi-Chabert2021} that any conformal optical geometry $(\mc{M}, \mbf{c}, K)$ of dimension $2m+2$ greater than four, whose associated congruence is  geodesic, twisting and non-shearing, and whose Weyl tensor satisfies
\begin{align}
	W (k, v, k, v) & = 0 \, , &  \mbox{for any sections $k$ of $K$, $v$ of $K^\perp$,} \label{eq:NS-IC}
\end{align}
admits a twist-induced almost Robinson structure $(N,K)$. Here, we recall that $\mbf{c}$ denotes an equivalence class of conformally related Lorentzian metrics. The description of the intrinsic torsion of $(N,K)$ is then described by Proposition \ref{prop:tw-ind_Rob_no_shear}. In particular, the leaf space $\ul{\mc{M}}$ of $\mc{K}$ is endowed with a partially integrable contact almost CR structure $(\ul{H},\ul{J})$ whose associated subconformal structure coincides with that induced by $\accentset{n.e.}{\mbf{c}}$. The involutivity of $(N,K)$ (or equivalently, the vanishing of the Nijenhuis tensor $\ul{\Nh}_{\alpha \beta \gamma}$ of $(\ul{H},\ul{J})$) is equivalent to  the Weyl tensor satisfying
\begin{align*}
	W  (k, u , v, w) & = 0 \, , & \mbox{for any sections $k$ of $K$, $u,v,w$ of $N$.}
\end{align*}

Let us now assume further that the Weyl tensor satisfies
\begin{align*}
	W (k, v, k, \cdot) & = 0 \, , &  \mbox{for any sections $k$ of $K$, $v$ of $K^\perp$.}
\end{align*}
Then locally any Einstein metric $\wh{g}$ in $\mbf{c}$ determines a contact form $\ul{\theta}^0$ on $(\ul{H},\ul{J})$ such that its corresponding almost pseudo-Hermitian structure is almost CR--Einstein. Such a structure is briefly discussed in Remark \ref{rem:CR-Kahler-Einstein} --- see references therein for more precise definitions. Suffices to say that locally $(\ul{\mc{M}},\ul{H},\ul{J},\ul{\theta}^0)$ is fibered over an almost K\"{a}hler--Einstein manifold $(\ut{\mc{M}},\ut{h},\ut{J})$ of dimension $2m$, and the Levi form $\ul{h}$ of $(\ul{H},\ul{J},\ul{\theta}^0)$ can be identified as the pullback of the almost K\"{a}hler--Einstein metric $\ut{h}$.

Further, one can choose a coordinate $\phi$ on the fibers $\mc{M} \accentset{\varpi}{\longrightarrow} \ul{\mc{M}}$ so that
\begin{align*}
		\wh{g} & = \sec^2 \phi \, g \, , & \mbox{for $- \frac{\pi}{2} < \phi < \frac{\pi}{2}$,} 
		\end{align*}
where
\begin{align*}
		g & = 4 \, \varpi^* \ul{\theta}^0 \left( \d \phi + \lambda_0 \, \varpi^* \ul{\theta}^0 \right) + \varpi^* \ul{h}  \, ,
		\end{align*}
		with
		\begin{multline*}
		\lambda_0 = \frac{\ul{\Lambda}}{2m+2}   +   \left( \frac{\Lambda}{2m+1}  - \frac{\ul{\Lambda} }{2m+2} \right)\left(  \sum_{j=0}^{m} a_{j} \cos^{2 j} \phi   -2  a_{m}  \cos^{2m+2} \phi \right) \\
		+ \ul{c} \cos^{2m+1} \phi \sin \phi \, ,
		\end{multline*}
		\begin{align*}
		a_{0} = 1  \, , & & a_{j} =  \frac{2m-2j+4}{2m-2j+1 }  a_{j-1} \, , & & j = 1, \ldots, m \, ,
		\end{align*}
and $\Lambda$, $\ul{\Lambda}$ and $\ul{c}$ are constant. Here, the Ricci scalars of $(\mc{M}, \wh{g}, K)$ and $(\ut{\mc{M}},\ut{h},\ut{J})$ are proportional to  $\Lambda$ and $\ul{\Lambda}$ respectively.

We can compute the intrinsic torsion of $(\mc{M},\wh{g},N,K)$ explicitely. We find
\begin{align*}
\gamma_i = \sigma_{i j} = \tau_{\alpha \beta} = \tau^\circ_{\alpha \bar{\beta}} & = 0 \, , & \tau^\omega & = \sec^2 \phi \, , & \epsilon & = 2m \tan \phi \, ,\\
E_{\alpha} = G_{\alpha} = G^\circ_{\bar{\alpha} \beta \gamma} = G^{\tiny{\ydskew}}_{\alpha \beta \gamma} & = 0 \, , & G^{\tiny{\ydhook}}_{\alpha \beta \gamma} & = - 2  \i \sec^4 \phi \, \ul{\Nh}_{\beta \gamma \alpha} \, , &
B_{\alpha \beta} & = 0 \, .
\end{align*}
Therefore, the intrinsic torsion of $(N,K)$ is a section of 
\begin{itemize}
\item $(\slashed{\mc{G}}_{0}^{0 \times 1})_{[2(m-1)\i:-1]} \cap \slashed{\mc{G}}_0^{1,3} \cap \slashed{\mc{G}}_0^{1,1}$, and
\item $\slashed{\mc{G}}_1^{0,0}$ when $(N,K)$ is involutive.
\end{itemize}
In particular, any of the Robinson spinors associated to $(N,K)$ satisfies \eqref{eq:pure_spinor1}, and if $(N,K)$ is involutive, \eqref{eq:pure_spinor3}. We shall see in Section \ref{sec:conf-aRstr} that these bundles do not depend on the metric $\wh{g}$ but solely on the conformal structure $\mbf{c}$. In fact, the intrinsic torsion of $(N,K)$ with respect to the metric $g$ is similar except that $\tau^{\omega} =1$, $\epsilon = 0$ and $G^{\tiny{\ydhook}}_{\alpha \beta \gamma}  = - 2 \i \ul{\Nh}_{\beta \gamma \alpha}$.

Furthermore we remark that for certain values of $\Lambda$, $\ul{\Lambda}$ and $\ul{c}$, and when $(N,K)$ is involutive, the metric $\wh{g}$ is locally isometric to the Taub--NUT--(A)dS metric of \cite{Bais1985,Awad2002,Alekseevsky2021}  generalising the four-dimensional one of \cite{Taub1951,Newman1963}, or the Fefferman--Einstein metric given in \cite{Leitner2007}.

Finally, there is a secondary almost Robinson structure $(N^*,L)$, dual to $(N,K)$, with non-shearing congruence of null geodesics $\mc{L}$. If the function $\lambda_0$ is non-vanishing, then $\mc{L}$ is twisting and $(N^*,L)$ is twist-induced. Otherwise, $\mc{L}$ is non-twisting.

Details can be found in \cite{Taghavi-Chabert2021}.
\end{exa}

\begin{rem}
One reason why non-shearing congruences of null geodesics in higher dimensions are not as common as in dimension four has to do with curvature. Let us review the various geometric interpretations of the algebraic condition on the Weyl tensor \eqref{eq:NS-IC}. These are as follows:
\begin{itemize}
	\item In dimension \emph{four}, if $\mc{K}$ is geodesic and non-shearing \emph{regardless of whether $\mc{K}$ is non-twisting or not}, then condition \eqref{eq:NS-IC} holds --- see e.g.\ \cite{Stephani2003,Penrose1986}.
	\item In dimensions \emph{greater} than four, if $\mc{K}$ is non-shearing \emph{and non-twisting}, then condition \eqref{eq:NS-IC} holds --- in fact, the Weyl tensor satisfies an even stronger condition \cite{Ortaggio2007}.
	\item In \emph{odd} dimensions greater than four, if $\mc{K}$ is twisting and condition \eqref{eq:NS-IC} holds, then $\mc{K}$ must also be shearing \cite{Ortaggio2007}.
	\item In \emph{even} dimensions greater than four, if $\mc{K}$ is twisting and non-shearing, and condition \eqref{eq:NS-IC} holds, then the twist induces an almost Robinson structure on $(\mc{M},g)$ \cite{Taghavi-Chabert2021}.
\end{itemize}
\end{rem}
\subsection{Analogies between almost Robinson geometry and almost Hermitian geometry}\label{sec:Rob-GH}
Recall that an almost Hermitian manifold consists of a triple $(\mc{M}, g, J)$, where $(\mc{M},g)$ is a $(2m+2)$-dimensional smooth Riemannian manifold, and $J$ is an almost complex structure compatible with $g$, i.e.\ $J \circ g = - g \circ J$. An equivalent definition of an almost Hermitian structure on $(\mc{M},g)$ is as an almost null structure $N$ of real index zero, i.e.\ the complexified tangent bundle splits as ${}^\C T \mc{M} = N \oplus \overline{N}$ \cite{Nurowski2002} -- since in Riemannian signature, an almost null structure always has real index zero, we may dispense of this attribute. The equivalence between the two definitions is established by identifying $N$ and $\overline{N}$ as the eigenbundles of $J$. Dually, one can express the almost Hermitian structure in terms of a non-vanishing section of $\bigwedge^{m+1} \Ann(N)$, which can be normalised up to a phase against its complex conjugate. Locally, or globally if $(\mc{M},g)$ is spin, this section is the `square' of a pure spinor field (of real index 0). It annihilates $N$, while its charged conjugate annihilates $\overline{N}$. Their pairing yields the almost Hermitian $2$-form of $(\mc{M},g,J)$ and powers thereof \cite{Kopczy'nski1992}.

As emphazised in \cite{Nurowski2002}, the point of contact between almost Robinson geometry and almost Hermitian geometry is their underlying almost null structure, and the only distinguishing feature between them is its real index, which is itself determined by the metric signature.

The relation between almost Robinson structures and almost null structures was already investigated in the previous section, especially in Theorem \ref{thm:spinor_intors}, using the results of \cite{Taghavi-Chabert2016}. One can play the same game by studying the geometry of an almost Hermitian manifold $(\mc{M},g,J)$ in the light of its underlying almost null structure $(\mc{M}, {}^\C g, N)$. To this end, we recall  the Gray--Hervella classification of almost Hermitian manifolds given in \cite{Gray1980}. Following the notation of that reference, the bundle $\mc{W}$ of intrinsic torsions of $(\mc{M},g,J)$ splits into irreducible $\mathbf{U}(m+1)$-invariant subbundles as
\begin{align}\label{eq:GH_sum}
\mc{W} & = \mc{W}_1 \oplus \mc{W}_2 \oplus \mc{W}_3 \oplus \mc{W}_4 \, ,
\end{align}
where at any point
\begin{align}
\begin{aligned}\label{eq:GH_sp}
 \mc{W}_1 & \cong \dbl \bigwedge^{(3,0)} (\R^{2m+2})^* \dbr \, , & \mc{W}_2 & \cong  \dbl \ydhook (\R^{2m+2})^*  \dbr \, , \\
  \mc{W}_3 & \cong  \dbl  \bigwedge^{(1,2)}_\circ (\R^{2m+2})^*  \dbr \, , & \mc{W}_4 & \cong  \dbl  \bigwedge^{(1,0)} (\R^{2m+2})^*  \dbr \ \, .
\end{aligned}
\end{align}
There are various ways to characterise the intrinsic torsion $\mathring{T}$ of $(\mc{M}, g, J)$. For instance, $(\mc{M}, g, J)$ is \emph{almost K\"{a}hler}, i.e.\ $\mathring{T}$ is a section of $\mc{W}_2$, if and only if the almost Hermitian $2$-form $\omega = g \circ J$  is closed. It is Hermitian, i.e.\ $\mathring{T}$ is a section of  $\mc{W}_3 \oplus \mc{W}_4$, if and only if the Nijenhuis tensor of the complex structure vanishes.
But only a subset of the Gray--Hervella classes will be relevant to the present discussion, namely those that reflect the geometric properties of the underlying almost null structure.  For instance, one can characterise a Hermitian manifold in terms of a pure spinor field that is recurrent along the totally null distribution it defines \cite{Hughston1988,Lawson1989,Taghavi-Chabert2016}: this is equivalent to the eigenbundles of $J$ being in involution --- see Proposition \ref{prop:fol_spinor} for the Robinson analogue. Proceeding as in the proof of Theorem \ref{thm:spinor_intors}, one can easily prove the equivalence between the first and last columns of Table \ref{table:Hermite-Robinson}, which summarizes the correspondences between the various classes of almost null structures, almost Robinson structures and almost Hermitian structures. The equivalence between  the first and second columns follows directly from Theorem \ref{thm:spinor_intors}. We leave the details for the reader. 
\begin{center}
\begin{table}[!htbp]
\begin{displaymath}
{\renewcommand{\arraystretch}{1.5}
\begin{array}{|c|c|c|}
\hline
\text{Almost null structures} & \text{Almost Robinson structures} & \text{Almost Hermitian structures}  \\
\hline
\hline
\mbox{Eq. \eqref{eq:pure_spinor1}} & \slashed{\mc{G}}_0^{1,1} & \mc{W}_2 \oplus \mc{W}_3 \oplus \mc{W}_4 \\
\mbox{Eq. \eqref{eq:pure_spinor2}} &  \slashed{\mc{G}}_0^{1,2}  & \mc{W}_1 \oplus \mc{W}_3 \oplus \mc{W}_4 \\
\mbox{Eqs. \eqref{eq:pure_spinor1} and \eqref{eq:pure_spinor2}} &  \slashed{\mc{G}}_0^{1,1} \cap  \slashed{\mc{G}}_0^{1,2}  & \mc{W}_3 \oplus \mc{W}_4 \\
\mbox{Eq. \eqref{eq:pure_spinor3}  in dim$>6$} &  \slashed{\mc{G}}_1^{0,0} & \mc{W}_4 \\
\mbox{Eq. \eqref{eq:pure_spinor3d6} in dim $6$} &  \slashed{\mc{G}}_1^{0,0} &  \mc{W}_1 \oplus \mc{W}_4 \\
\mbox{Eq. \eqref{eq:pure_spinor0}} & ({\slashed{\mc{G}}_0^{0 \times 1}})_{[2 \i:1]} &\mc{W}_3 \\
\mbox{Eq. \eqref{eq:pure_spinor3} and \eqref{eq:pure_spinor0}} & \{ 0 \}  & \{ 0 \} \\
\hline
\end{array}}
\end{displaymath}
\caption{\label{table:Hermite-Robinson}Comparison of the intrinsic torsion of almost Robinson structures and the Gray--Hervella classification of almost Hermitian manifolds on the basis of the properties of their underlying null structure.}
\end{table}
\end{center}

\subsection{Almost Robinson manifolds of Kundt type}\label{sec:alRob_Kundt}
\begin{defn}\label{def:alRob_Kundt}
An almost Robinson manifold $(\mc{M}, g, N, K)$ is said to be \emph{of Kundt type} if $(\mc{M}, g, K)$ is a Kundt spacetime, i.e.\ $K$ is tangent to a non-expanding, non-shearing, non-twisting congruence of null geodesics.
\end{defn}
Equivalently, the intrinsic torsion of such manifolds is a section of $\slashed{\mc{G}}_{-1}^0 \cap \slashed{\mc{G}}_{-1}^1 \cap \slashed{\mc{G}}_{-1}^2$, i.e.\
\begin{align*}
\breve{\gamma}_i = \breve{\epsilon}  = \breve{\tau}_{i j} = \breve{\sigma}_{i j} & = 0 \, .
\end{align*}

For most of this section, however, we shall restrict ourselves to nearly Robinson manifolds of Kundt type, in which case the intrinsic torsion is a section of $\mc{G}^0 = \slashed{\mc{G}}_{-1}^0 \cap \slashed{\mc{G}}_{-1}^1 \cap \slashed{\mc{G}}_{-1}^2 \cap \slashed{\mc{G}}_{-1}^{3,0}$, i.e.\
\begin{align*}
\breve{\gamma}_i = \breve{\epsilon}  = \breve{\tau}_{i j} = \breve{\sigma}_{i j} = \breve{\zeta}_{\alpha \beta} & = 0 \, .
\end{align*}
Note that this implies that $N$ is parallel along $K$. From Section \ref{sec:lift}, $(N,K)$ induces an almost CR structure with totally degenerate Levi form on the leaf space $(\ul{\mc{M}},\ul{H})$, and the screen bundle metric $h$ on $H_K$ descends to a bundle metric $\ul{h}$ on $\ul{H}$. This tells us that $(\ul{H}, \ul{h}, \ul{J})$ is a Hermitian vector bundle. In addition, since $K^\perp$ is in involution, so is $\ul{H}$. Thus, $(\ul{\mc{M}},\ul{H},\ul{h})$ admits a local Riemannian foliation. Putting these two facts together allows us to characterise nearly Robinson manifolds of Kundt type in the following terms.
\begin{prop}
Let  $(\mc{M},g, N, K)$ be a $(2m+2)$-dimensional nearly Robinson manifold of Kundt type with congruence of null geodesics $\mc{K}$. Then the local leaf space $(\ul{\mc{M}}, \ul{H}, \ul{h}, \ul{J})$ of $\mc{K}$ is foliated by a smooth one-parameter family of $2m$-dimensional almost Hermitian manifolds, each tangent to $\ul{H}$.
\end{prop}

In the neighbourhood of every point of $\mc{M}$, we can apply Proposition \ref{prop:Rob->CR} to express the metric $g$ in terms of a Robinson coframe adapted to the Kundt geometry in terms of the leaf space $\ul{\mc{M}}$ as follows. Locally, we shall refer to a Robinson coframe $(\kappa, \theta^{\alpha}, \ol{\theta}{}^{\bar{\alpha}} , \lambda)$ as a \emph{(complex) Kundt coframe} if
\begin{align*}
	\kappa & := 2 \varpi^* \ul{\theta}^{0} \, , & \theta^{\alpha} & = \varpi^* \ul{\theta}^{\alpha} \, ,
\end{align*}
for some exact $1$-form $\ul{\theta}^{0}$ annihilating $\ul{H}$, and coframe $(\ul{\theta}^{\alpha}, \ol{\ul{\theta}}{}^{\bar{\alpha}})$ on $\ul{H}$ unitary with respect to $\ul{h}$. As before, $\varpi$ is the local projection from $\mc{M}$ to $\ul{\mc{M}}$. The $1$-form $\lambda$ must be vertical with respect to $\varpi$.

In addition, one can find local coordinates $u$, $v$ on $\mc{M} := \R \times \ul{\mc{M}} \stackrel{\varpi}{\longrightarrow} \ul{\mc{M}}$ such that
\begin{align*}
	\kappa & = 2 \varpi^* \ul{\theta}^0 =  \d u \, , & \mbox{and} & & k & = g^{-1}( \kappa, \cdot) = \parderv{}{v} \, ,
\end{align*}
i.e.\ $u$ is a smooth function on $\ul{\mc{M}}$ parametrising the leaves of the almost Hermitian foliation, and $v$ is an affine parametre along the geodesics of $\mc{K}$.  This allows us to write
$\lambda$ in the form
\begin{align}\label{eq:vertical_form}
	\lambda & = \d v + \lambda_\alpha \varpi^* \ul{\theta}^\alpha + \lambda_{\bar{\alpha}} \varpi^* \ul{\theta}^{\bar{\alpha}} + \lambda_0 \varpi^* \ul{\theta}^0 \, ,
\end{align}
where $\lambda_\alpha$, $\lambda_{\bar{\alpha}}$ and $\lambda_0$ are smooth functions on $\mc{M}$.

Any two Kundt coframes $(\kappa, \theta^{\alpha}, \overline{\theta}{}^{\bar{\alpha}}, \lambda)$ and $(\wh{\kappa}, \wh{\theta}{}^{\alpha}, \overline{\wh{\theta}}{}^{\bar{\alpha}}, \wh{\lambda})$ are related by the transformations \eqref{eq:Rob_cof_transf} where $\phi^\alpha$ and $\psi{}_{\beta}{}^{\alpha}$ are now required to be constant along $K$, and $\varphi$  constant along $K^\perp$. Such a transformation can be induced from a transformation of adapted coframes on $(\ul{\mc{M}},\ul{H},\ul{h})$, a change of parameter for the almost Hermitian foliation thereon, or a change of affine parameter along $\mc{K}$.

We shall streamline the notation by setting $(\kappa, \theta^i, \lambda ) = (\kappa, \theta^{\alpha}, \overline{\theta}{}^{\bar{\alpha}}, \lambda)$. In the following, we fix a $1$-form $\ul{\theta}^0$ annihilating $\ul{H}$ and a splitting $T \mc{M} = \mrm{span} (\ul{e}_0) \oplus \ul{H}$ where $\ul{e}_0$ is dual to $\ul{\theta}^{\alpha}(\ul{e}_0)=0$. Note that this fixes the freedom in choosing  $(\ul{\theta}{}^{\alpha} , \overline{\ul{\theta}}{}^{\bar{\alpha}})$, up to unitary transformations. We introduce a connection $\ul{\nabla}$ on $\ul{\mc{M}}$ that preserves $\ul{h}_{i j}$, $\ul{\theta}^0$ and $\ul{e}_0$ with torsion tensor $\ul{A}_{i j} = \ul{A}_{(i j)}$, i.e.\
\begin{align*}
\left( \ul{\nabla}_0 \ul{\nabla}_i - \ul{\nabla}_i \ul{\nabla}_0 \right) \ul{f} & = - \ul{A}_i{}^{j} \ul{\nabla}_j \ul{f} \, , & \mbox{for any smooth function $\ul{f}$ on $\ul{\mc{M}}$.}
\end{align*}
Note this connection depends on the choice of $\ul{\theta}^i$ up to orthogonal transformations, and thus on the choice of $\ul{e}_0$. Dropping the pullback maps, we can express the Levi-Civita connection $\nabla$ of $g$ in the following way:
\begin{align}
\begin{aligned}\label{eq:Kundt_conn}
\nabla \kappa & = 2 \,  \kappa \odot E  \, , \\
\nabla \theta^i & =  \ul{\nabla}  \ul{\theta}^i  - 2 \, B^{i}{}_{j}  \ul{\theta}^{j} \odot \kappa  -  C^{i} \kappa \otimes \kappa - \frac{1}{2} \ul{A}^{i}{}_{j} \ul{\theta}^{j} \otimes \kappa - 2 \, E^{i}  \lambda  \odot \kappa \, , \\
\nabla \lambda & = - 2  E \wedge  \lambda - E_0 \kappa \odot \lambda + \kappa \otimes   C  + \frac{1}{2} \ul{A} -  B   \, ,
\end{aligned}
\end{align}
where
\begin{align*}
\ul{A} & =  \ul{A}_{i j} \ul{\theta}^i \otimes \ul{\theta}^j \, , &
B & =  B_{i j} \ul{\theta}^i \otimes \ul{\theta}^j \, , & E & = E_i \ul{\theta}^i + E_0 \ul{\theta}^0 \, ,  & C & = C_i \ul{\theta}^i \, ,
\end{align*}
with  $B_{i j} = B_{[i j]}$, $E_i$, $E_0$ and $C_i$  being functions on $\mc{M}$ satisfying
\begin{align*}
B_{i j} & = - \ul{\nabla}_{[i} \lambda_{j]} + \lambda_{[i} \dot{\lambda}_{j]} \, , \\
C_{i} & = - \frac{1}{2}\left( \ul{\nabla}_i \lambda_0 - \ul{\nabla}_0 \lambda_i - \lambda_i \dot{\lambda}_0 +  \lambda_0 \dot{\lambda}_i - \ul{A}_{i j} \lambda^j \right)  \, , \\
E_{i} & = \frac{1}{2} \dot{\lambda}_i \, , \\
E_0 & = \frac{1}{2} \dot{\lambda}_0 \, .
\end{align*}
This choice of notation is of course not fortuitous since we can then identify the components $E_\alpha$ and $B_{\alpha \beta}$ of the intrinsic torsion.

\begin{lem}\label{lem:lambda}
Let $(\mc{M}, g, N, K)$ be a nearly Robinson manifold of Kundt type  with congruence of null geodesics $\mc{K}$ and leaf space $(\ul{\mc{M}}, \ul{H}, \ul{h}, \ul{J})$. Then, for any Kundt coframe $(\kappa, \theta^{\alpha}, \overline{\theta}{}^{\bar{\alpha}}, \lambda)$, the $1$-form $\lambda$ satisfies
\begin{align*}
E_i & = -(\nabla \lambda)_{i}{}^{0} = -( \d \lambda)_{i}{}^{0} \, ,
\end{align*}
and $\breve{E}_i$ is an invariant of $(N,K)$.

If the intrinsic torsion is in addition a section of $(\slashed{\mc{G}}_{0}^{0 \times 1})_{[-2(m-1)\i:1]} \cap\slashed{\mc{G}}_0^{1,1} \cap \slashed{\mc{G}}_0^{1,2} \cap \slashed{\mc{G}}_0^{1,3}$, then
\begin{align*}
B_{\alpha \beta} & = (\nabla \lambda)_{\alpha \beta} = -( \d \lambda)_{\alpha \beta} \, .
\end{align*}
In particular, $\breve{B}_{\alpha \beta}$ is an invariant of $(N,K)$.
\end{lem}

\begin{proof}
 The expressions for $E_i$ and $B_{\alpha \beta}$ follow from reading off the components of the Levi-Civita connections from equations \eqref{eq:Kundt_conn}. The statements on the invariance of the corresponding weighted quantities $\breve{E}_i$ and $\breve{B}_{\alpha \beta}$ can be checked from the transformations in the proof of Theorem \ref{thm:intors-rob}.
\end{proof}

Since both  $\kappa_a$ and $\rho_{a b c}$ are the pull-backs of some $1$-form and $3$-form from $\ul{\mc{M}}$ to $\mc{M}$, we can relate the present classification of almost Robinson structures with the Gray--Hervella classification \cite{Gray1980} of almost Hermitian manifolds. To this end, we simply note that for any splitting $(\ell^a, \delta^a_i, k^a)$, using \eqref{eq:Kundt_conn}, we have
\begin{align}\label{eq:GH-Rob3}
(\nabla \rho)_{i j k 0} & = \ul{\nabla}_i \ul{\omega}_{j k} \, ,
\end{align}
where $\ul{\omega}_{i j} = \ul{J}_i{}^k \ul{h}_{k j}$ is the smooth family of Hermitian $2$-form associated to the almost Hermitian foliation on $(\ul{\mc{M}}, \ul{H}, \ul{h}, \ul{J})$. One can readily check that this does not depend on the choice of coframe --- this essentially follows from equations \eqref{eq:G-1_alt} and the fact that the optical invariants $\breve{\gamma}_i$, $\breve{\epsilon}$, $\breve{\sigma}_{i j}$ and $\breve{\tau}_{i j}$ all vanish. The Gray--Hervella classes can be easily obtained by comparing the LHS and the RHS of equation \eqref{eq:GH-Rob3} with references to \eqref{eq:GH_sum} and \eqref{eq:GH_sp}. We have collected the findings in Table \ref{tab-intors-GH}.
\begin{center}
\begin{table}[!htbp]
\begin{displaymath}
{\renewcommand{\arraystretch}{1.5}
\begin{array}{|c|c|c|c|c|c|c|c|c|}
\hline
\text{Type of almost Hermitian structure} & \multicolumn{4}{c|}{\text{Intrinsic torsion $\ul{\mathring{T}}$}} & \multicolumn{4}{c|}{\text{Intrinsic torsion $\mathring{T}$}}  \\
\text{on $\ul{\mc{M}}$} & \mc{W}_1 & \mc{W}_2 & \mc{W}_3 & \mc{W}_4 & \slashed{\mc{G}}_{0}^{1,0} & \slashed{\mc{G}}_{0}^{1,1} & \slashed{\mc{G}}_{0}^{1,2} & \slashed{\mc{G}}_{0}^{1,3} \\
\hline
\hline
 \text{almost Hermitian} & \checkmark & \checkmark & \checkmark & \checkmark &  &  &  &  \\
\mc{G}_2 &  & \checkmark & \checkmark & \checkmark
&  & \checkmark &  &  \\
\mc{G}_1 & \checkmark &  & \checkmark & \checkmark
&  &  & \checkmark &  \\
& \checkmark & \checkmark &  & \checkmark
&  &  &  & \checkmark \\
\text{semi-K\"{a}hler} & \checkmark & \checkmark & \checkmark & 
& \checkmark &  &  &  \\
\text{Hermitian} &  &  & \checkmark & \checkmark
&  & \checkmark & \checkmark &  \\
\text{incl. locally conformally almost K\"{a}hler} &  & \checkmark &  & \checkmark
&  & \checkmark &  & \checkmark \\
&  & \checkmark & \checkmark & 
& \checkmark & \checkmark &  &  \\
& \checkmark &  &  & \checkmark
&  &  & \checkmark & \checkmark \\
& \checkmark &  & \checkmark & 
& \checkmark &  & \checkmark &  \\
\text{quasi-K\"{a}hler} & \checkmark & \checkmark &  & 
& \checkmark &  &  & \checkmark \\
\text{nearly K\"{a}hler} & \checkmark &  &  & 
& \checkmark &  & \checkmark & \checkmark \\
\text{almost K\"{a}hler} &  & \checkmark &  & 
& \checkmark & \checkmark &  & \checkmark \\
\text{special Hermitian} &  &  & \checkmark & 
& \checkmark & \checkmark & \checkmark &  \\
\text{incl. locally conformally K\"{a}hler} &  &  &  & \checkmark
&  & \checkmark & \checkmark & \checkmark \\
\text{K\"{a}hler} &  &  &  & 
& \checkmark & \checkmark & \checkmark & \checkmark \\
\hline
\end{array}}
\end{displaymath}
\caption{\label{tab-intors-GH} Relation between the intrinsic torsions $\mathring{T}$ and $\ul{\mathring{T}}$ for nearly Robinson manifolds of Kundt type.}
\end{table}
\end{center}

The following proposition gives a characterisation of an almost Robinson manifold of Kundt type in the case where the leaves of the Riemannian foliation on the leaf space are K\"{a}hler manifolds. We leave the proof to the reader.
\begin{prop}
Let $(\mc{M}, g, N, K)$ be a nearly Robinson manifold of Kundt type with congruence of null geodesics $\mc{K}$ and leaf space $(\ul{\mc{M}}, \ul{H}, \ul{h}, \ul{J})$. The following statements are equivalent.
\begin{enumerate}
\item the intrinsic torsion is a section of $ \slashed{\mc{G}}_{0}^{1,0} \cap \slashed{\mc{G}}_0^{1,1} \cap \slashed{\mc{G}}_0^{1,2} \cap \slashed{\mc{G}}_0^{1,3}$, i.e.\ 
\begin{align*}
\breve{\gamma}_i = \breve{\epsilon} = \breve{\tau}_{i j} = \breve{\sigma}_{i j} = \breve{\zeta}_{\alpha \beta} = \breve{G}_{\alpha} = \breve{G}^{\tiny{\ydskew}}_{\alpha \beta \gamma} = \breve{G}^{\tiny{\ydhook}}_{\alpha \beta \gamma} = \breve{G}^{\circ}_{\bar{\alpha} \beta \gamma} = 0 \, .
\end{align*}
\item any Robinson $3$-form $\rho_{a b c}$ is recurrent along $K^\perp$, i.e.\
\begin{align*}
\kappa_{[a} \nabla_{b]} \rho_{c d e} & = \kappa_{[a} \alpha_{b]} \rho_{c d e} \, , & \mbox{for some $1$-form $\alpha_a$;}
\end{align*}
\item $N$ is parallel along $K^\perp$.
\item Each leaf of the almost Hermitian foliation on $(\ul{\mc{M}}, \ul{H}, \ul{h}, \ul{J})$ is a K\"{a}hler manifold.
\end{enumerate}
\end{prop}

Besides the classes of nearly Robinson manifolds of Kundt type enumerated in Table \ref{tab-intors-GH}, other interesting degeneracy conditions on the intrinsic torsion are also possible. In fact, these can be partly characterised by the covariant derivative of the $1$-form $\lambda$ using Lemma \ref{lem:lambda}.

\begin{prop}\label{prop:Kundt_Rob_EG}
Let $(\mc{M}, g, N, K)$ be a nearly Robinson manifolds of Kundt type  with congruence of null geodesics $\mc{K}$ and leaf space $(\ul{\mc{M}}, \ul{H}, \ul{h}, \ul{J})$. Let $[z:w] \in \CP^1$ such that $z \neq 0$. The following two statements are equivalent.
\begin{enumerate}
\item the intrinsic torsion of $(N,K)$ is a section of $(\slashed{\mc{G}}_0^{0\times 1})_{[z:w]}$, i.e.\
\begin{align*}
\breve{\gamma}_i = \breve{\epsilon} = \breve{\tau}_{i j} = \breve{\sigma}_{i j} = \breve{\zeta}_{\alpha \beta} = 0 \, , & &
z \breve{E}_\alpha + w \, \breve{G}_{\alpha}  & = 0 \, .
\end{align*}
\item for any choice of Kundt coframe, the components $\lambda_{\alpha}$ of the vertical $1$-form $\lambda$ given by \eqref{eq:vertical_form} are determined by
\begin{align}\label{eq:lambda01}
\lambda_{\alpha} & = \ul{\lambda}_{\alpha}^{(0)}  - 2 \frac{w}{z}  v \ul{\lambda}^{(1)}_{\alpha} \, ,
\end{align}
for some smooth functions $\ul{\lambda}_\alpha^{(0)}$ and $\ul{\lambda}_\alpha^{(1)}$ on $\ul{\mc{M}}$ such that
\begin{align}\label{eq:lambda1}
 \ul{\lambda}_{\alpha}^{(1)} & =  G_{\alpha} =  \ul{h}^{i j} (\ul{\nabla}\ul{\omega} )_{i j \alpha} \, , & \mbox{where $\ul{\omega}_{i j} = \ul{J}_{i}{}^{k} \ul{h}_{k j}$.}
\end{align}
\end{enumerate}

Finally, for any optical vector field $k$, with $\kappa = g(k,\cdot)$, the Weyl tensor satisfies
\begin{align}\label{eq:W_sp}
k^a W_{a b c [d} \kappa_{e]} k^c & = 0 \, .
\end{align}
\end{prop}

\begin{proof}
This is a straightforward computation. By definition, the condition $z E_\alpha + w \, G_{\alpha}  = 0$ can be rewritten as $z  \dot{\lambda}_{\alpha}  + 2 w \, G_{\alpha}  = 0$, which has solution given precisely by \eqref{eq:lambda01} and \eqref{eq:lambda1}.

For the last part, it is shown in \cite{Podolsky2013} that condition \eqref{eq:W_sp} on Weyl tensor is equivalent to $\lambda_i$ being linear in $v$.
\end{proof}

\begin{rem}
Weaker conditions, where one takes $z$ and $w$ to be non-constant complex-valued smooth functions in Proposition \ref{prop:Kundt_Rob_EG}, are possible. In this case, it is no longer true that $\lambda_i$ is linear in $v$.
\end{rem}

The next proposition follows from the interpretation of the vanishing of the intrinsic torsion as the reduction of the holonomy of the Levi-Civita connection to $Q$, or equivalently, to the parallelism of the distribution $N$. The last item follows from Lemma \ref{lem:lambda}.

\begin{prop}\label{prop:no_intors}
Let $(\mc{M}, g, N, K)$ be an almost Robinson manifold with congruence of null curves $\mc{K}$. The following statements are equivalent.
\begin{enumerate}
\item the intrinsic torsion vanishes identically, i.e.\ 
\begin{align*}
\breve{\gamma}_i = \breve{\epsilon} = \breve{\tau}_{i j} = \breve{\sigma}_{i j} = \breve{\zeta}_{\alpha \beta} = \breve{E}_i = \breve{G}_{\alpha} = \breve{G}^{\tiny{\ydskew}}_{\alpha \beta \gamma} =  \breve{G}^{\tiny{\ydhook}}_{\alpha \beta \gamma} =  \breve{G}^{\circ}_{\bar{\alpha} \beta \gamma}= \breve{B}_{\alpha \beta} = 0 \, .
\end{align*}
\item the holonomy of the Levi-Civita connection is reduced to $Q = (\R_{>0} \times \U(m))\ltimes (\R^{2m})^*$, the structure group of $(N,K)$.
\item any Robinson $3$-form $\rho_{b c d}$ is recurrent, i.e.\
\begin{align*}
 \nabla_a \rho_{b c d} & = \alpha_a \rho_{b c d} \, , & \mbox{for some $1$-form $\alpha_a$.}
\end{align*}
\item $N$ is parallel, i.e.\ for any $v \in \Gamma(N)$, $\nabla v \in \Gamma(N)$.
\item any Robinson spinor $\nu$ is recurrent, i.e.\
\begin{align*}
\left( \nabla_a \nu^{b \mbf{A}} \right) \nu_{b}{}^{\mbf{B}} & = 0 \, , & \mbox{i.e.} & & \left( \nabla_a \nu^{[\mbf{A}'} \right) \nu^{\mbf{B}']} & = 0 \, .
\end{align*}
\item $(\mc{M}, g, N, K)$ is of Kundt type with leaf space $(\ul{\mc{M}}, \ul{H}, \ul{h}, \ul{J})$ where $\ul{h}$ is a smooth one-parameter family of K\"{a}hler metrics, and for any choice of Kundt coframe, the components $\lambda_{i}$ of the vertical $1$-form $\lambda$ given by \eqref{eq:vertical_form} are smooth functions $\ul{\lambda}_{i}^{(0)}$ on $\ul{\mc{M}}$, and, at any point, $\ul{\nabla}_{[i} \ul{\lambda}_{j]}^{(0)}$ is an element of $\u(m)$
\end{enumerate}
\end{prop}

\begin{rem}
We note that if a Robinson $3$-form is recurrent, so is any optical $1$-form.
\end{rem}

\begin{rem}
One can also weaken the assumptions of Proposition \ref{prop:no_intors}, by supposing that on each leaf tangent to $\ul{H}$, the metric $\ul{h}_{i j}$ is almost K\"{a}hler, i.e.\ the Hermitian $2$-form is closed, rather than K\"{a}hler. Then locally one may take $\ul{\lambda}_i^{(0)}$ to be any potential $1$-form for the Hermitian $2$-form. Then $\ul{\nabla}_{[i} \ul{\lambda}_{j]}^{(0)}$ is an element of $\u(m)$.  This implies that $B_{\alpha \beta} =0$, i.e.\ the $\mc{G}_1^{0,0}$-component of the intrinsic torsion vanishes.
\end{rem}

Finally, the next two results, stated without proofs, are concerned with further holonomy reduction.
\begin{prop}
Let $(\mc{M}, g, N, K)$ be an almost Robinson manifold with congruence of null curves $\mc{K}$. The following statements are equivalent:
\begin{enumerate}
\item the holonomy of the Levi-Civita connection is reduced to a subgroup of $\U(m)\ltimes (\R^{2m})^*$;
\item $(\mc{M}, g, N, K)$ admits a parallel Robinson $3$-form $\rho_{b c d}$, i.e.\ $\nabla_a \rho_{b c d} = 0$;
\item $(\mc{M}, g, N, K)$ is of Kundt type with leaf space $(\ul{\mc{M}}, \ul{H}, \ul{h})$ where $\ul{h}$ is a smooth one-parameter family of K\"{a}hler metrics, and for any choice of Kundt coframe, the components $\lambda_{i}$ and $\lambda_{0}$ of the vertical $1$-form $\lambda$ given by \eqref{eq:vertical_form} are smooth functions $\ul{\lambda}_{i}^{(0)}$ and $\ul{\lambda}_{0}^{(0)}$  on $\ul{\mc{M}}$, and, at any point, $\ul{\nabla}_{[i} \ul{\lambda}_{j]}^{(0)}$ is an element of $\u(m)$.
\end{enumerate}
If any of these conditions holds, $(\mc{M}, g, N, K)$ admits a parallel optical vector field.
\end{prop}

\begin{prop}
Let $(\mc{M}, g, N, K)$ be an almost Robinson spin manifold with congruence of null curves $\mc{K}$. The following statements are equivalent:
\begin{enumerate}
\item the holonomy of the Levi-Civita connection is reduced to a subgroup of $\SU(m)\ltimes (\R^{2m})^*$;
\item $(\mc{M}, g, N, K)$ admits a parallel Robinson spinor $\nu$, i.e.\ $\nabla_a \nu = 0$;
\item $(\mc{M}, g, N, K)$ admits a parallel optical vector field and a parallel Robinson $3$-form;
\item $(\mc{M}, g, N, K)$ is of Kundt type with leaf space $(\ul{\mc{M}}, \ul{H}, \ul{h})$ where $\ul{h}$ is a smooth one-parameter family of Ricci-flat K\"{a}hler metrics, and for any choice of Kundt coframe, the components $\lambda_{i}$ and $\lambda_{0}$ of the vertical $1$-form $\lambda$ given by \eqref{eq:vertical_form} are smooth functions $\ul{\lambda}_{i}^{(0)}$ and $\ul{\lambda}_{0}^{(0)}$  on $\ul{\mc{M}}$, and, at any point, $\ul{\nabla}_{[i} \ul{\lambda}_{j]}^{(0)}$ is an element of $\su(m)$.
\end{enumerate}
\end{prop}

\begin{exa}[Metrics of supergravity]\label{exa:SUGRA}
Parallel Robinson structures are relevant to the study of solutions to the supergravity equations. These equations are rather restrictive. For instance, it is shown in \cite{Chamseddine2003} (see also \cite{Gutowski2003}) that solutions known as 
$(1,0)$-vacua (up to local isometry) are six-dimensional Lie groups admitting a bi-invariant Lorentzian metric and an anti-self-dual $3$-form induced from the Lie bracket. In the same reference, it is proved that these must be either Minkowski space, a one-parameter family of so-called Freud--Rubin vacua on $AdS_3 \times S^3$ with equal radii, or a six-dimensional Nappi--Witten vacuum. The latter is locally isometric to a certain Cahan--Wallach space. In coordinates $(u, v, x^1, x^2, x^3, x^4)$, the metric is given by
\begin{align}\label{eq:CW_metric}
g & = \d u \left( \d v - \frac{1}{2} \ut{h}_{i j} x^i x^j \d u \right) + \ut{h}_{i j} \d x^i \d x^j \, ,
\end{align}
where $\ut{h}_{i j}$ is the standard Euclidean metric on $\R^4$. The anti-self-dual $3$-form induced from the Lie bracket takes the form
\begin{align}\label{eq:CW_Rob3f}
\rho & = 2 \d u \wedge \left( \d x^1 \wedge \d x^2 + \d x^3 \wedge \d x^4 \right) \, .
\end{align}
It can be checked that $\rho$ defines a Robinson $3$-form with optical $1$-form $\kappa = \d u$.
Both $\kappa$ and $\rho$ are parallel with respect to the Levi-Civita connection.
\end{exa}

One may start with a Kundt geometry $(\mc{M},g,K)$ with congruence of null geodesics $\mc{K}$ whose leaf space $\ul{\mc{M}}$ of $\mc{K}$ is a fiber bundle over a smooth $2m$-dimensional Riemannian manifold $(\ut{\mc{M}}, \ut{h})$. Then any almost complex structure on $\ut{\mc{M}}$ compatible with $\ut{h}$ lifts to an almost Robinson structure on $(\mc{M}, g)$ compatible with $K$. Depending on its topology, $(\ut{\mc{M}}, \ut{h})$ may admit many almost Hermitian structures, each with a specific intrinsic torsion, or even families thereof. For instance, let us take $(\ut{\mc{M}}, \ut{h})$ to be Euclidean space. Then locally there are infinitely many Hermitian structures (see e.g.\ \cite{Eells1985}) that can be lifted to an almost Robinson structure $(N,K)$ on $(\mc{M},g)$.

A less trivial example follows.

\begin{exa}
Let $(\ut{\mc{M}},\ut{h})$ be the Iwasawa manifold, that is, the quotient of the three-dimensional complex Heisenberg group by a discrete subgroup. In \cite{Abbena2001}, the authors construct almost Hermitian structures in the following Gray--Hervella classes
\begin{align*}
\mc{W}_2 \oplus \mc{W}_3 \oplus \mc{W}_4 \, , && 
\mc{W}_1 \oplus \mc{W}_3 \oplus \mc{W}_4 \, , &&
\mc{W}_1 \oplus \mc{W}_2 \oplus \mc{W}_4 \, , &&
\mc{W}_1 \oplus \mc{W}_2 \oplus \mc{W}_3 \, .
\end{align*}
Of these, the set of all invariant Hermitian structures on $(\ut{\mc{M}},\ut{h})$ is known to consist of the union of a point (its bi-invariant Hermitian structure) and a 2-sphere \cite{Abbena1997,Ketsetzis2004}. These are in fact special Hermitian, i.e.\ their intrinsic torsion is of class $\mc{W}_3$. For topological reasons, $\ut{\mc{M}}$ cannot admit any K\"{a}hler structure. It is also conjectured that $\ut{\mc{M}}$ cannot admit almost Hermitian structures in the classes $\mc{W}_1$ and $\mc{W}_4$.

The Kundt geometry $(\mc{M}, g, K)$ associated to $(\ut{\mc{M}}, \ut{h})$ admits almost Robinson structures corresponding to the almost Hermitian structures on $(\ut{\mc{M}}, \ut{h})$, and their classes of intrinsic torsion can be read off from Table \ref{tab-intors-GH}.
\end{exa}

As illustrated in the following example, not every almost Robinson structure on a Kundt spacetime is a nearly Robinson structure.
\begin{exa}
Take $\mc{M} = \R \times \R \times \R^{2m}  \cong \R \times \R \times \C^m  = (u, v, z^{\alpha}, \ol{z}{}^{\bar{\alpha}})$, and let $g$ be the metric on $\mc{M}$ given by
\begin{align*}
g & = 2 \d u \left( \d v + \lambda_{\alpha} \d z^\alpha + \lambda_{\bar{\alpha}} \d \ol{z}^{\bar{\alpha}}   +  \lambda_0  \d u \right) + 2 \, \ut{h}_{\alpha \bar{\beta}}  \d z^\alpha  \d \overline{z}^{\bar{\beta}} \, ,
\end{align*}
where $\lambda_{\alpha}$, $\lambda_{\bar{\alpha}}$ and $\lambda_{0}$ are arbitrary smooth functions on $\mc{M}$, and $\ut{h}_{\alpha \bar{\beta}}$ is the standard Hermitian form on $\C^m$.  Let $k = \frac{\partial}{\partial v}$ and set $K := \mathrm{span}(k)$. Then $(\mc{M}, g, K)$ is a Kundt spacetime, and any almost Robinson structure $(N,K)$ on $\mc{M}$ compatible with $K$ is annihilated by the set of $1$-forms
\begin{align*}
\kappa & = g(k, \cdot) = \d u \, , & \theta^\alpha = \d z^\alpha + \phi^{\alpha}{}_{\bar{\beta}} \d \overline{z}^{\bar{\beta}} \, ,
\end{align*}
for some complex-valued functions $\phi^{\alpha \beta}$ on $\mc{M}$ with $\phi^{\alpha \beta} = \phi^{[\alpha \beta]}$ -- the functions $\phi^{\alpha \beta}$ are essentially the components of a section of the bundle $\mathrm{Gr}_{m+1}^+(\mc{M},g,K)$ of (self-dual) almost Robinson structures compatible with $K$ --- see Remark \ref{rem:sp_Rob}. Note that $(\theta^{\alpha})$ does \emph{not} constitute a unitary coframe for $\ut{h}_{\alpha \bar{\beta}}$ in general. Choosing the $\phi^{\alpha \beta}$ such that $\mathsterling_k \phi^{\alpha \beta} \neq 0$ anywhere yields a coframe that does not descend to the leaf space $\ul{\mc{M}}$ of $\mc{K}$. Since the manifold is Kundt, it is clear that the remaining obstruction to $(N,K)$ being nearly Robinson is the $\mc{G}_{-1}^{3,0}$-component $\breve{\zeta}_{\alpha \beta}$ of the intrinsic torsion, which, here, can be identified with $\mathsterling_k \phi_{\alpha \beta}$. Take for instance, $\phi^{\alpha \beta} = f(v) \ul{\phi}^{\alpha \beta}$ for some smooth function $f$ of $v$ and smooth functions $\ul{\phi}^{\alpha \beta}$ on $\ul{\mc{M}}$. For definiteness, let us assume $m=2$. The bundle of all almost Hermitian structures on $\R^4$ has fibers isomorphic to $\CP^1$. Any unitary frame $(\theta^1, \theta^2)$ for $\ul{h}$ takes the form
\begin{align*}
\theta^1 & = \left( a \bar{a} + b \bar{b} \right)^{-\frac{1}{2}}  \left( a \,  \d z^1 + b \, \d \overline{z}^2 \right) \, , &
\theta^2 & = \left( a \bar{a} + b \bar{b} \right)^{-\frac{1}{2}} \left( a \, \d z^2 - b \, \d \overline{z}^1 \right) \, .
\end{align*}
for some smooth complex-valued functions $a$ and $b$ on $\mc{M}$ with $a, b$ not both vanishing. Note that this expression is invariant under non-zero rescaling of $(a,b)$, so at any point, $[a:b]$ defines an element of $\CP^1$ as expected. Take $a=1$ and $b$ any smooth complex-valued function depending on $v$, i.e.\ $\mathsterling_k b \neq 0$. Then $(N,K)$ is an almost Robinson structure on $(\mc{M}, g, K)$ that does not descend to $(\ul{\mc{M}}, \ul{H}, \ul{h})$, i.e.\ it is not nearly Robinson.
\end{exa}

\subsection{Almost Robinson manifolds of Robinson--Trautman type}\label{sec:alRob_RT}
In complete analogy with Definition \ref{def:alRob_Kundt}, we make the following definition.
\begin{defn}\label{def:alRob_RT}
An almost Robinson manifold $(\mc{M}, g, N, K)$ is said to be \emph{of Robinson--Trautman type} if $(\mc{M}, g, K)$ is a Robinson--Trautman spacetime, i.e.\ the congruence of curves tangent to $K$ is geodesic, expanding, non-shearing and non-twisting.
\end{defn}
The intrinsic torsion of such a manifold is a section of $\slashed{\mc{G}}_{-1}^1 \cap \slashed{\mc{G}}_{-1}^2$, with non-degenerate $\mc{G}_{-1}^0$-component i.e.\
\begin{align*}
\breve{\gamma}_i = \breve{\sigma}_{i j} = \breve{\tau}_{i j} = 0 \, , &&  \breve{\epsilon} & \neq 0 \, .
\end{align*}
Again, it is natural to consider nearly Robinson manifolds of Robinson--Trautman type. These manifolds enjoy properties similar to those of their Kundt counterparts. In particular, the leaf space $\ul{\mc{M}}$ of the congruence $\mc{K}$ tangent to $K$ is foliated by $2m$-dimensional almost Hermitian manifolds. The intrinsic torsion of $(\mc{M}, g, N, K)$ can also be related to the Gray--Hervella class of the almost Hermitian foliation as in Table \ref{tab-intors-GH} \emph{except} in cases where $\slashed{\mc{G}}_{0}^{1,0}$ is involved --- this would contradict the fact that $\mc{K}$ is expanding. This means that the only Gray--Hervella classes allowed would be those containing $\mc{W}_4$.

However, since a Robinson--Trautmann spacetime is conformal to a Kundt spacetime \cite{Pravda2017,Fino2020}, one may still consider the full Gray--Hervella classification applied to the almost Hermitian foliation on $\ul{\mc{M}}$. For instance, a nearly Robinson manifold of Robinson--Trautman type with intrinsic torsion in $\slashed{\mc{G}}_0^{1,1} \cap \slashed{\mc{G}}_0^{1,3}$ will arise from an almost Hermitian foliation on $\ul{\mc{M}}$ of either class $\mc{W}_2 \oplus \mc{W}_4$ or class $\mc{W}_2$.

\begin{exa}[The Tangherlini--Schwarzschild metric]
The smooth manifold $\mc{M} = \R \times \R_{>0} \times S^{2m}$ admits the Tangherlini--Schwarzschild metric -- see e.g.\ \cite{Podolsky2006a}. Then $(\mc{M},g,K)$ is a Ricci-flat Robinson--Trautman spacetime which does not admit any global Robinson structure except in the case $m=1$. This follows from the fact that the $2m$-sphere admits a Hermitian structure if and only if $m=1$. However, since $S^{2m}$ is conformally flat, it locally admits infinitely many Hermitian structures -- these correspond precisely to holomorphic sections of the twistor bundles over $S^{2m}$, a Riemannian articulation of the Kerr theorem -- see e.g.\ \cite{Eells1985}.
\end{exa}

In the above example, there is no distinguished (almost) Robinson structure on the Robinson--Trautman manifold. However, as the following example due to \cite{Ortaggio2008} shows, the Einstein--Maxwell equations may single out an almost Robinson structure on a Robinson--Trautman spacetime.

\begin{exa}\label{exa:RT}
Let $(\mc{M}, g, K)$ be a Robinson--Trautman optical geometry of dimension $2m+2$. Suppose $g$ satisfies the Einstein--Maxwell equations with an electromagnetic field $F_{a b}=F_{[a b]}$, that is, $F_{a b}$ is closed and co-closed, and the Einstein field equations take the form
\begin{align*}
\mathrm{Ric}_{a b} & = \frac{1}{m} \Lambda g_{a b} + 8 \pi T_{a b} + \frac{m-1}{2m} g_{a b} T_{c d} T^{c d} \, ,
\end{align*}
where $\Lambda$ is the cosmological constant, and the energy-momentum tensor is given by
\begin{align*}
T_{a b} & = \frac{1}{4 \pi} \left( F_{a c} F_b{}^c - \frac{1}{4} g_{a b} F_{c d} F^{c d} \right) \, .
\end{align*}
We assume further that $F_{a b}$ satisfying $k^a F_a{}^{[b} k^{c]} = 0$ for any section $k$ of $K$. For definiteness, assume $m>2$. Then \cite{Ortaggio2008}, there exist coordinates $(r, u, x^i)$ such that the metric takes the form
\begin{align*}
g & = -2 \d u \d r - 2 H(r) (\d u)^2 + r^2 \ut{h}_{i j}(x) \d x^i \d x^j  \, ,
\end{align*}
where $\ut{h}_{i j}$ is a metric on each slice of constant $(r,u)$,
\begin{align*}
2 H (r) & = K - \frac{2 \Lambda}{2m(2m+1)} r^2 - \frac{\mu}{r^{2m-1}} + \frac{2 Q^2}{2m(2m-1)} \frac{1}{r^{2(2m-1)}} - \frac{\|\ut{F}\|^2}{2m(2m-3)} \frac{1}{r^2} \, ,
\end{align*}
and the electromagnetic field is given by
\begin{align*}
F & = \frac{Q}{r^{2m}} \d r \wedge \d u  + \frac{1}{2} \ut{F}_{i j}(x) \d x^i \wedge \d x^j \, .
\end{align*}
Here, $K \in \{ -1, 0, 1 \}$, $\mu$, $Q$ and $\|\ut{F}\|^2 = \ut{F}_{i j} \ut{F}^{i j}$ are constants, and $k = \frac{\partial}{\partial r}$ is a null vector field tangent to $K$ whose congruence is also geodetsic, non-twisting and non-shearing. The vector field $\ell = \frac{\partial}{\partial u} - H (r) \frac{\partial}{\partial r}$ defines a optical structure $L$ dual to $K$. Set $\kappa = g(k,\cdot) = \d u $ and $\lambda = g(\ell, \cdot) = - \d r - H(r) \d u$. Assuming $\| F \|^2 \neq 0$, the electromagnetic $2$-form $F_{a b}$ determines two almost Robinson structures $(N_K,K)$ and $(N_L,L)$: their associated $3$-forms are proportional to $\kappa \wedge F$ and $\lambda \wedge F$ respectively. By virtue of the Maxwell equations, it can be shown \cite{Ortaggio2008} that the metric $\ut{h}_{i j}$ is almost K\"{a}hler--Einstein. It follows from Table \ref{tab-intors-GH} that the intrinsic torsion of each of these almost Robinson structures must be a section of $\slashed{\mc{G}}_{0}^{1,1} \cap \slashed{\mc{G}}_{0}^{1,3}$.

The six-dimensional case is similar, and further generalisations of these results can be found in \cite{Kokoska2021}.
\end{exa}

As for Kundt spacetimes, one can associate to any smooth $2m$-dimensional Riemannian manifold $(\ut{\mc{M}}, \ut{h})$ a Robinson--Trautman geometry $(\mc{M},g,K)$ with congruence of null geodesics $\mc{K}$ such that the leaf space $\ul{\mc{M}}$ of $\mc{K}$ is the trivial line bundle $\R \times \ut{\mc{M}}$. Any almost complex structure on $\ut{\mc{M}}$ compatible with $\ut{h}$ lifts to an almost Robinson structure on $(\mc{M}, g)$ compatible with $K$.

\subsection{Compatible linear connections}\label{sec:lin_conn}
We end Section \ref{sec:geometry} with a brief consideration of linear connections compatible with a given almost Robinson structure.

\begin{prop}\label{prop:lin_conn}
Let $(\mc{M},g,N,K)$ be an almost Robinson manifold. Fix a splitting $(\ell^a , \delta^a_i , k^a)$. Define a linear connection $\nabla'$ with
\begin{align}\label{eq:comp_conn}
\nabla'_a \xi_b & = \nabla_a \xi_b - Q_{a b}{}^c \xi_c \, , & \mbox{for any $1$-form $\xi_a$,}
\end{align}
where $Q_{a b c}$ is a tensor such that
\begin{align*}
 Q^{0 0 0} =   Q_{i}{}^{0 0} & = Q_{0}{}^{0 0} =  Q^{0}{}_{0}{}^0  = 0 \, , \\
Q^0{}_{j}{}^0 = -  Q^{0 0}{}_{j}  & =  \gamma_j \, , \\
Q_{i j}{}^0 = -   Q_{i}{}^0{}_{j} & =  \frac{1}{2m} \epsilon \, h_{i j} + \tau_{i j} + \sigma_{i j}  \, , &
Q^0{}_{\alpha \bar{\beta}} & = -\frac{1}{2m} \epsilon \, h_{\alpha \bar{\beta}}  - \tau_{\alpha \bar{\beta}} \, , \\
Q_{i (j k)} & = -\frac{1}{\i(m-1)} G_i h_{j k}  \, , \\
Q_{0 j}{}^0 = Q_{j 0}{}^0 & = - Q_{0}{}^0{}_{j} = - Q^{0}{}_{0 k}   = E_ j  \, , \\
Q_{0 (j k)} & = -\frac{1}{2m} f_0 h_{j k}  \, ,
\end{align*}
and
\begin{align*}
Q^0{}_{\beta \gamma}  & =  - \tau_{\beta \gamma} + \frac{1}{2 \i} \zeta_{\beta \gamma}  \, , \\
Q_{\alpha \beta \gamma}  & = \frac{1}{2 \i}  G_{\alpha \beta \gamma} \, , &
Q_{\bar{\alpha} \beta \gamma}  & =  \frac{1}{2 \i} G_{\bar{\alpha} \beta \gamma} \, , \\
Q_{\alpha \beta \bar{\gamma}}  & = - \frac{2}{\i(m-1)} G_{(\alpha} h_{\beta) \bar{\gamma}}  \, , &
Q_{\alpha \bar{\gamma} \beta}  & =   - \frac{1}{\i(m-1)} G_{\alpha} h_{\beta \bar{\gamma}}   \, , \\
Q_{0 \beta \gamma}  & = \frac{1}{2 \i} B_{\alpha \beta} \, .
\end{align*}

Then $\nabla'$ is a connection compatible with $K$ and $[h]$, i.e.\
\begin{align*}
\nabla'_u \kappa (v) & = 0 \, , & \mbox{for any $u \in \Gamma(T \mc{M})$, $v \in \Gamma(K^\perp)$,} \\
\nabla'_u g (v ,w) & = 0 \, , & \mbox{for any $u \in \Gamma(T \mc{M})$, $v,w \in \Gamma(K^\perp)$,}
\end{align*}
with torsion tensor satisfying
\begin{align*}
T^{0}{}_{j}{}^{0} & = - \gamma_{j} \, , \\
T_{i j}{}^{0} & = - 2 \, \tau_{i j} \, , &  T^{0}{}_{j}{}_{k} & =  - \sigma_{j k}  \, ,  \\
T^{0}{}_{0}{}^{0} = T^{0}{}_{0 k} & = T_{0 j}{}^{0} = 0 \, , 
\end{align*}
and
\begin{align*}
T^{0}{}_{[\beta \gamma]} & = - \frac{1}{2\i} \zeta_{\beta \gamma}  - \tau_{\beta \gamma}  \, , &
T^{0}{}_{(\beta \gamma)} & =  - \sigma_{\beta \gamma}  \, , \\
T_{[\alpha \beta \gamma]} & = - \frac{1}{2\i} G^{\tiny{\ydskew}}_{\alpha \beta \gamma}   \, , &
T_{\alpha (\beta \gamma)} & = \frac{1}{4\i} G^{\tiny{\ydhook}}_{(\beta \gamma) \alpha}   \, , \\
T_{\bar{\alpha} [\beta \gamma]}   & = - \frac{1}{2\i} G^\circ_{\bar{\alpha} \beta \gamma}  \, , \\
 T_{\alpha \beta \bar{\gamma}} = T_{\bar{\alpha} (\beta \gamma)}  & = T_{0 [\beta \gamma]}  = 0 \, .
\end{align*}
\end{prop}

\begin{proof}
This is a straightforward computation using \eqref{eq:comp_conn} and the fact that the torsion tensor is given by $T_{a b c} = - 2 \, Q_{[a b] c}$.
\end{proof}

\begin{rem}
The linear connection defined in the proposition above depends in general on the choice of splitting. Note however that even with fixed $k$ and $\lambda$, $\nabla'$ as defined in the proof of the proposition is not unique, the undefined components of $Q_{a b c}$ in the proposition above being entirely arbitrary and not affecting the property of the torsion.
\end{rem}

We may set all the remaining components of $Q_{[a b] c}$ to zero, and obtain the following corollary.
\begin{cor}\label{cor:lin_conn}
Let $(\mc{M},g,N,K)$ be an almost Robinson manifold. Suppose that the intrinsic torsion of $(N,K)$ is a section of $\slashed{\mc{G}}_{-1}^1 \cap \slashed{\mc{G}}_0^1 \cap \slashed{\mc{G}}_0^2 \cap \slashed{\mc{G}}_0^3$. Then $(\mc{M},g,N,K)$ admits a torsion-free connection that preserves $(N,K)$ and (the conformal class of) the screen bundle metric.
\end{cor}

\section{Conformal almost Robinson structures}\label{sec:conf-aRstr}
As for optical geometry and almost Hermitian geometry, the notion of almost Robinson structure, or indeed almost null structure, finds a very natural setting in conformal geometry. We shall follow the conventions set up in \cite{Bailey1994,Fino2020}, and denote a conformal structure on a smooth manifold $\mc{M}$ by $\mbf{c}$. For each $w \in \R$, the bundle of conformal densities of weight $w$ is denoted by $\mc{E}[w]$. In particular, a choice of ray subundle $\mc{E}_+[1]$ of $\mc{E}[1]$ is referred to as the bundle of conformal scales. The conformal structure can be encoded by means of the \emph{conformal metric} $\bm{g}_{a b}$, that is a non-degenerate global section of $\bigodot^2 T^* \mc{M} \otimes \mc{E}[2]$. For each $g$ in $\mbf{c}$, we extend the Levi-Civita connection $\nabla$ of $g$ to a linear connection on $\mc{E}[w]$ for each $w \in \R$. The exterior covariant derivative will be denoted by $\d^\nabla$. Further details can be found in the aforementioned references.

\begin{defn}
Let $(\mc{M},\mbf{c})$ be an oriented and time-oriented Lorentzian conformal manifold of dimension $2m+2$. An \emph{almost Robinson structure} on $(\mc{M},\mbf{c})$ consists of a pair $(N,K)$ where $N$ is a complex distribution of rank $m+1$ totally null with respect to $g$, and $K$ a real line distribution such that ${}^\C K = N \cap \overline{N}$. We shall call $(N,K)$ a \emph{nearly Robinson structure} when $[K,N] \subset N$, and a \emph{Robinson structure} when $[N,N] \subset N$.

We shall accordingly refer to the quadruple $(\mc{M},\mbf{c},N, K)$ as an \emph{almost (conformal) Robinson manifold}, as a \emph{nearly (conformal) Robinson manifold} or as a \emph{(conformal) Robinson manifold}.
\end{defn}

\begin{rem}
As in the metric case, one can describe an almost conformal Robinson manifold as an almost null structure (of real index one).
\end{rem}

The conformal metric $\bm{g}_{a b}$ induces a conformal bundle metric $\bm{h}_{i j}$ on $H_K$. Proposition \ref{prop:char-Robinson-mfld} can immediately be translated into the conformal setting as follows:

\begin{prop}\label{prop:equiv_Rob_conf}
Let $(\mc{M},\mbf{c})$ be an oriented and time-oriented Lorentzian manifold of dimension $2m+2$. The following are equivalent:
\begin{enumerate}
\item $(\mc{M},\mbf{c})$ is endowed with an almost Robinson structure $(N,K)$;
\item $(\mc{M},\mbf{c})$ admits a totally null complex $(m+1)$-form of conformal weight $m+2$;
\item $(\mc{M},\mbf{c})$ is endowed with an optical structure $K$ whose screen bundle $H_K = K^\perp/K$ is equipped with a bundle complex structure compatible with the induced conformal structure;
\item $(\mc{M},\mbf{c})$ admits a $1$-form $\bm{\kappa}_a$ of conformal weight $2$, and a $3$-form $\bm{\rho}_{a b c}$ of conformal weight $4$ such that
\begin{align*}
\bm{\rho}_{ab} \, {}^e \bm{\rho} _{cde} = - 4 \bm{\kappa}_{[a} \bm{g}_{b][c} \bm{\kappa}_{d]} \, ;
\end{align*}
\item $(\mc{M},\mbf{c})$ admits a pure spinor of real index one.
\end{enumerate}
\end{prop}
We shall follow the terminology already introduced in Section \ref{sec:geometry}: thus the $1$-form $\bm{\kappa}_a$ and $3$-form $\bm{\rho}_{a b c}$ given in Proposition \ref{prop:equiv_Rob_conf} will be referred to as \emph{optical $1$-form} and \emph{Robinson $3$-form} respectively, and so on.

With reference to the proposition above, the bundle complex structure $J_i{}^j$ yields a bundle Hermitian structure $\bm{\omega}_{i j} = J_{i}{}^k \bm{h}_{k j}$ of conformal weight $2$. For a given optical $1$-form $\bm{\kappa}_a$ we obtain a Robinson $3$-form $\bm{\rho}_{a b c} = 3 \bm{\kappa}_{[a} \bm{\omega}_{b c]}$ of conformal weight $4$, where $\bm{\omega}_{a b}$ is such that $k^c \bm{\omega}_{c [a} \bm{\kappa}_{b]} = 0$, $\bm{\omega}_{i j} = \bm{\omega}_{a b} \delta^a_i \delta_j^b$. The complex $(m+2)$-form $\bm{\nu}_{a_0 a_1 \ldots a_m}$ is required to have conformal $m+1$ since
\begin{align*}
\bm{\nu}_{a a_1 \ldots a_m}  \overline{\bm{\nu}}_{b}{}^{a_1 \ldots a_m} \,  \propto \, \bm{\kappa}_a \bm{\kappa}_b \, .
\end{align*}
The relation with pure spinor fields is analogous to the Lorentzian case. Here, neither a Robinson spinor nor its charge conjugate are conformally weighted, but the van der Waerden symbols carry some conformal weight. This means that for each $k=0,\ldots , m+1$, the spinor bilinear form with values in (complex) $k$-forms has conformal weight $k+1$. We omit the details, which will play no r\^{o}le in the subsequent discussion.

\subsection{Conformal invariants of an almost Robinson structure}
An optical geometry with a congruence of null geodesics has two conformal invariants, the shear and the twist, which we may view as fields of conformal weight two \cite{Robinson1985,Fino2020}. To determine the conformal invariants of an almost Robinson manifold $(\mc{M}, \mbf{c}, N, K)$, we examine how the covariant derivative of the Robinson $3$-form changes under a change of metrics $\wh{g} = \e^{2 \varphi} g$ for some smooth function $\varphi$ as given below:
\begin{align*}
\wh{\nabla}_a \bm{\kappa}_b & = \nabla_a \bm{\kappa}_b + 2 \Upsilon_{[a} \bm{\kappa}_{b]} + \bm{g}_{a b} \Upsilon^c \bm{\kappa}_c \, , \\
\wh{\nabla}_a \bm{\rho}_{b c d} & = \nabla_a \bm{\rho}_{b c d} + 3 \Upsilon_a \bm{\rho}_{b c d} - 3  \Upsilon_{[b} \bm{\rho}_{c d] a} + 3 \, \bm{g}_{a [b} \bm{\rho}_{c d] e} \Upsilon^e \, ,
\end{align*}
where $\Upsilon = \d \varphi$. Projecting these tensors into their components with splitting operators, we find that
\begin{align*}
\wh{\gamma}_{i} & = \e^{2 \varphi} \gamma_i   \, , \\
\wh{\tau}_{i j} & = \e^{2 \varphi} \tau_{i j}  \, , &
\wh{\sigma}_{i j} & = \e^{2 \varphi} \sigma_{i j} \, , &
\wh{\epsilon} & = \epsilon + 2m \Upsilon^{c} k_{c} \, , \\
\wh{\zeta}_{\beta \gamma} & = \e^{4 \varphi} \zeta_{\beta \gamma} \, , \\
\wh{E}_{i} & = \e^{2 \varphi}  \left( E_{i}  - \Upsilon_{i} \right) \, , \\
\wh{G}_{\gamma}  & = \e^{2 \varphi} \left(G_{\gamma}  - 2 (m-1) \i \Upsilon_{\gamma} \right) \, , \\
\wh{G}^{\tiny{\ydskew}}_{\alpha \beta \gamma} & = \e^{4 \varphi} G^{\tiny{\ydskew}}_{\alpha \beta \gamma} \, , &
\wh{G}^{\tiny{\ydhook}}_{\alpha \beta \gamma} & = \e^{4 \varphi} G^{\tiny{\ydhook}}_{\alpha \beta \gamma} \, , &
 \wh{G}^{\circ}_{\bar{\alpha} \beta \gamma}  & = \e^{4 \varphi} G^{\circ}_{\bar{\alpha} \beta \gamma}  \, , \\
\wh{B} _{\beta \gamma} & =  \e^{4 \varphi} B_{\beta \gamma} \, .
\end{align*}
Note that
\begin{align*}
2(m-1) \i \wh{E} _{\gamma} - \wh{G}_{\gamma}  & = \e^{2 \varphi} \left( 2(m-1) \i E _{\gamma} - G_{\gamma} \right) \, .
\end{align*}
We also find that for any $[z:w] \in \CP^1$,
\begin{align*}
z \wh{\tau}_{\alpha \beta} + w \wh{\zeta}_{\beta \gamma} & = \e^{4 \varphi} \left( z \tau_{\alpha \beta} + w \zeta_{\beta \gamma} \right) \, .
\end{align*}
From these computations, we immediately conclude:

\begin{thm}\label{thm:conf-Rob}
Let $(\mc{M}, \mbf{c},N, K)$ be an almost conformal Robinson manifold. Let $g$ be a metric in $\mbf{c}$ so that $(\mc{M}, g , K)$ is an optical geometry with bundle of intrinsic torsions $\mc{G}$. Any conformally invariant subbundle of $\mc{G}$ must be an intersection of the following:
\begin{align*}
& \slashed{\mc{G}}_{-2}^{0,0} \, , \\
& \slashed{\mc{G}}_{-1}^{1,0} \, , 
&& \slashed{\mc{G}}_{-1}^{1,1} \, , 
&& \slashed{\mc{G}}_{-1}^{1,2} \, , \\
& \slashed{\mc{G}}_{-1}^{2,0} \, , 
&& \slashed{\mc{G}}_{-1}^{2,1} \, , 
&& \slashed{\mc{G}}_{-1}^{3,0} \, , \\
& \slashed{\mc{G}}_0^{1,1} \, ,
&& \slashed{\mc{G}}_0^{1,2} \, ,
&& \slashed{\mc{G}}_0^{1,3} \, ,
&& (\slashed{\mc{G}}_{0}^{0 \times 1})_{[2(m-1)\i:-1]} \, , \\
& \slashed{\mc{G}}_1^{0,0} \, , \\
& (\slashed{\mc{G}}_{-1}^{1 \times 2})_{[x:y]} \, , && \mbox{\scriptsize{$[x:y] \in \RP^1$}} \, , \\
& (\slashed{\mc{G}}_{-1}^{1 \times 3})_{[z:w]} \, , && \mbox{\scriptsize{$[z:w] \in \CP^1$}} \, .
\end{align*}
\end{thm}

As for conformal optical geometries, there is a subclass $\accentset{n.e.}{\mbf{c}}$ of metrics in $\mbf{c}$ with the property that whenever $g$ is in $\accentset{n.e.}{\mbf{c}}$, the congruence $\mc{K}$ is non-expanding, i.e.\ for any $k \in \Gamma(K)$ with $\kappa = g(k,\cdot)$, $\kappa \, \div  k -  \nabla_k \kappa = 0$.

We also know from \cite{Fino2020} that there exists a family of optical vector fields $k$ such that $\mathsterling_k \bm{\kappa} = 0$ where $\kappa = \bm{g}(k,\cdot)$.  If the corresponding Robinson $3$-form $\bm{\rho}_{a b c}$ is preserved along $K$, then $\mathsterling_k  \bm{\rho}_{a b c} = 0$, where the Lie derivative is given by
\begin{align*}
\mathsterling_k  \bm{\rho}_{a b c} & = k^d \nabla_d  \bm{\rho}_{a b c} + 3  \bm{\rho}_{d [a b} \nabla_{c]} k^d - \frac{4}{n+2}  \bm{\rho}_{a b c} \nabla_d k^d \, .
\end{align*}
Here $\nabla$ is the Levi-Civita connection of any metric in $\mbf{c}$. The details are left to the reader.

\subsection{Conformally parallel Robinson structures}
In \cite{Fino2020}, we saw that under certain conditions, one can find metrics in $\mbf{c}$ for which the optical structure is parallel. We extend this result to the Robinson setting.

\begin{prop} \label{PropprallelalmostRob}
Let $(\mc{M}, \mbf{c} ,N, K)$ be an almost Robinson manifold with congruence of null curves $\mc{K}$. Suppose that the intrinsic torsion of $(N,K)$ for some (and thus any) metric $g$ in $\mbf{c}$ is a section of $\slashed{\mc{G}}_{-1}^{1} \cap \slashed{\mc{G}}_1^{0,0}$, i.e.\
\begin{align}
\breve{\gamma}_ i = \breve{\tau}_ {i j} = \breve{\sigma}_{i j} = \breve{\zeta}_{\alpha \beta} = \breve{G}^{\tiny{\ydskew}}_{\alpha \beta \gamma} = \breve{G}^{\tiny{\ydhook}}_{\alpha \beta \gamma} = \breve{G}^{\circ}_{\bar{\alpha} \beta \gamma} = \breve{B}_{\alpha \beta} & = 0 \, , \nonumber \\
2(m-1) \i \breve{E} _{\gamma} - \breve{G}_{\gamma} & = 0 \, . \label{eq:remains}
\end{align}
Equivalently, any Robinson spinor satisfies \eqref{eq:pure_spinor3} and any optical $1$-form $\bm{\kappa}$ satisfies $\bm{\kappa} \wedge \d^{\nabla} \bm{\kappa} = 0$. Suppose further that the Weyl tensor $W_{a b c d}$ satisfies
\begin{align*}
k^a W_{a b [c d} \bm{\kappa}_{e]} & = 0 \, .
\end{align*}
Then locally, there is a subclass $\accentset{par.}{\mbf{c}}$ of metrics in $\mbf{c}$ with the property that whenever $g$ is in $\accentset{par.}{\mbf{c}}$,  the almost Robinson structure is parallel, i.e.\ any Robinson spinor, and thus any optical $1$-form and Robinson $3$-form, are recurrent. In particular, the holonomy of the Levi-Civita connection of any metric in $\accentset{par.}{\mbf{c}}$ is contained in $Q=(\R_{>0} \times \U(m))\ltimes (\R^{2m})^*$. Any two metrics in $\accentset{par.}{\mbf{c}}$ differ by a factor constant along $K^\perp$.
\end{prop}

\begin{proof}
The hypothesis can be expressed by saying that the only three possibly non-vanishing components of the intrinsic torsion are $\breve{\epsilon}$, $\breve{E}_i$ and $\breve{G}_{\alpha}$, the latter two being related by \eqref{eq:remains}. It is already shown in \cite{Fino2020} that the curvature prescription yields the existence of the subclass $\accentset{par.}{\mbf{c}}$ for which we have $\breve{E}_i=0$ and $\breve{\epsilon}=0$. But then this implies that $\breve{G}_\alpha = 0$. Hence, the intrinsic torsion vanishes and the result follows.
\end{proof}

As a direct consequence of Proposition \ref{prop:Kundt_Rob_EG}, we obtain:
\begin{prop}
Let $(\mc{M}, \mbf{c} ,N, K)$ be an almost Robinson manifold with congruence of null curves $\mc{K}$. Suppose that the intrinsic torsion of $(N,K)$ for some (and thus any) metric $g$ in $\mbf{c}$ is a section of $\slashed{\mc{G}}_{-1}^{1} \cap \slashed{\mc{G}}_0^{1,1} \cap \slashed{\mc{G}}_0^{1,2} \cap \slashed{\mc{G}}_0^{1,3} \cap (\slashed{\mc{G}}_{0}^{0 \times 1})_{[2(m-1)\i:-1]}$, i.e.\
\begin{align*}
\breve{\gamma}_ i = \breve{\tau}_ {i j} = \breve{\sigma}_{i j} = \breve{\zeta}_{\alpha \beta} = \breve{G}^{\tiny{\ydskew}}_{\alpha \beta \gamma} = \breve{G}^{\tiny{\ydhook}}_{\alpha \beta \gamma} = \breve{G}^{\circ}_{\bar{\alpha} \beta \gamma} & = 0 \, , \\
2(m-1) \i \breve{E} _{\gamma} - \breve{G}_{\gamma} & = 0 \, .
\end{align*}
Then for every metric $g$ in $\accentset{n.e.}{\mbf{c}}$, $(\mc{M}, g, N, K)$ is a nearly Robinson manifold of Kundt type with congruence of null geodesics $\mc{K}$.

Denote by $(\ul{\mc{M}}, \ul{h}, \ul{J})$ denote the leaf space of $\mc{K}$, and set $\ul{\omega}_{i j} = \ul{J}_{i}{}^{k} \ul{h}_{k j}$. Then there exists a Kundt frame $(\kappa, \theta^i , \lambda)$ such that the vertical $1$-form $\lambda$ has components $\lambda_i$ given by
\begin{align*}
\lambda_\ell & = \ul{\lambda}_\ell^{(0)} - v \frac{1}{m-1}  \ul{h}^{i j} \left( \ul{\nabla}_i  \ul{\omega}_{j k} \right) \ul{J}_\ell{}^{k} \, ,
\end{align*}
where $\ul{\lambda}_i^{(0)}$ are smooth functions on $\ul{\mc{M}}$, and $v$ an affine parameter along the geodesics of $\mc{K}$.
\end{prop}

\subsection{Conformal lift of almost CR structures}
In this short section, we revisit the lift of almost CR structures considered in Section \ref{sec:lift}. A conformal version of Proposition \ref{prop:lift} can be formulated in the following terms:
\begin{prop}\label{prop:conf_lift}
Let $(\ul{\mc{M}}, \ul{H}, \ul{J})$ be a $(2m+1)$-dimensional oriented almost CR manifold, and $\mc{M} := \R \times \ul{\mc{M}} \stackrel{\varpi}{\longrightarrow} \ul{\mc{M}}$ be a trivial line bundle over  $\ul{\mc{M}}$.

Let $\left( (\ul{\theta}^0, \ul{\theta}^\alpha, \overline{\ul{\theta}}{}^{\bar{\alpha}}) , h_{\alpha \bar{\beta}} , \lambda \right)$ and $\left( (\wh{\ul{\theta}}{}^0, \wh{\ul{\theta}}{}^{\alpha}, \overline{\wh{\ul{\theta}}}{}^{\bar{\alpha}}) , \wh{h}_{\alpha \bar{\beta}} , \wh{\lambda} \right)$ be two triplets that give rise to two almost Robinson geometries $(\mc{M}, g, N, K)$ and $(\mc{M}, \wh{g}, N, K)$ as in Proposition \ref{prop:lift}. Suppose $(\ul{\theta}^0, \ul{\theta}{}^\alpha, \overline{\ul{\theta}}{}^{\bar{\alpha}})$ and $(\wh{\ul{\theta}}{}^0, \wh{\ul{\theta}}{}^\alpha, \overline{\wh{\ul{\theta}}}{}^{\bar{\alpha}})$ are related by \eqref{eq:CR_transf}. Then
\begin{align*}
\wh{g} & =  \e^{\ul{\varphi}} g \, ,
\end{align*}
if and only if $\ul{\psi}{}_{\alpha}{}^{\gamma}$ is an element of $\U (m)$ at every point, i.e.\
\begin{align*}
\wh{h}_{\alpha \bar{\beta}} & =  \e^{\ul{\varphi}} h_{\gamma \bar{\delta}} (\ul{\psi}^{-1}){}_{\alpha}{}^{\gamma}  (\ul{\psi}^{-1}){}_{\bar{\beta}}{}^{\bar{\delta}} \, , \\
\intertext{and $\lambda$ transforms as}
\wh{\lambda} & = \lambda - \frac{1}{2} \ul{\psi}{}_{\alpha}{}^{\beta} \ul{\phi}_{\beta} \ul{\theta}^{\alpha} - \frac{1}{2} \ul{\psi}{}_{\bar{\alpha}}{}^{\bar{\beta}} \ul{\phi}_{\bar{\beta}} \overline{\ul{\theta}}^{\bar{\alpha}} - \frac{1}{2} \ul{\phi}_{\beta} \ul{\phi}^{\beta} \ul{\theta}^0 \, .
\end{align*}
\end{prop}
Again, the proof is pretty much tautological.

\subsection{Non-shearing twist-induced almost Robinson structures}
Most of the results in Sections \ref{sec:lift}, \ref{sec:other}, \ref{sec:tw-ind_al_Rob} and \ref{sec:spinor_des} can be safely formulated in the conformal setting. In particular, we give the conformal version of Proposition \ref{prop:tw-ind_Rob_no_shear} below, which is a direct consequence of Proposition 5.16 of \cite{Fino2020}:
\begin{prop}\label{prop-adapted_frame_max_tw-nsh-nexp_conf}
Let $(\mc{M}, \mbf{c}, N, K)$ be a $(2m+2)$-dimensional conformal twist-induced almost Robinson manifold with non-shearing congruence of null geodesics $\mc{K}$. We then have the following properties:
\begin{enumerate}
\item For each $g \in \accentset{n.e.}{\mbf{c}}$, there exists a unique pair $(k, \ell)$ where $k$ is a generator of $\mc{K}$ and $\ell$ a null vector field such that $g(k, \ell) = 1$ and  $\kappa = g( k , \cdot)$ satisfies
\begin{align*}
\d \kappa (k , \cdot) & = 0 \, , & \d \kappa (\ell , \cdot) & = 0 \, .
\end{align*}
In particular, the twist $\tau$ of $k$ is represented by $\d \kappa$ and determines the screen bundle Hermitian form of $(N,K)$ with respect to $g$, i.e.\ we have $(\d \kappa)_{i j} = \tau_{i j} = \omega_{i j} = J_{i}{}^{k} h_{j k}$.

If $( k, \ell )$ and $( \wh{k}, \wh{\ell} )$ be any two such pairs corresponding to metrics $g$ and $\wh{g}$ in $\accentset{n.e.}{\mbf{c}}$, with $\wh{g} = \e^{\ul{\varphi}} g$ for some smooth function $\ul{\varphi}$ constant along $K$ then,
\begin{align}\label{eq:kl}
\wh{k} & = k \, , &
\wh{\ell} & = \e^{-\ul{\varphi}} \left( \ell + \frac{1}{2} \omega^{-1}( \cdot, \ul{\Upsilon} ) - \frac{1}{4} \| \ul{\Upsilon} \|^2_g k \right)\, ,
\end{align}
where $\ul{\Upsilon} = \d \ul{\varphi}$.

\item $(N,K)$ induces a partially integrable contact almost CR structure $(\ul{H},\ul{J})$ on the leaf space $\ul{\mc{M}}$ of $\mc{K}$. In particular, $(\ul{H},\ul{J})$ is equipped with a subconformal structure $\ul{\mbf{c}}_{\ul{H},\ul{J}}$ compatible with $\ul{J}$ and also induced from $\mbf{c}$, and there is a one-to-one correspondence between metrics in $\accentset{n.e.}{\mbf{c}}$ and contact $1$-forms for $(\ul{H},\ul{J})$. More specifically, for any two adapted frames $(\ul{\theta}^0, \ul{\theta}^{\alpha} , \ol{\ul{\theta}}{}^{\bar{\alpha}})$ and $(\wh{\ul{\theta}}{}^0, \wh{\ul{\theta}}{}^{\alpha} , \ol{\wh{\ul{\theta}}}{}^{\bar{\alpha}})$ for $(\ul{H},\ul{J})$ related by
\begin{align*}
	\wh{\ul{\theta}}{}^0 & =  \e^{\ul{\varphi}} \ul{\theta}^0 \, , &
	\widehat{\ul{\theta}}{}^{\alpha} & =  \ul{\theta}^\alpha + \i \ul{\Upsilon}^{\alpha} \ul{\theta}^0 \, , 
\end{align*}
up to $\U(m)$-transformations, and with Levi forms related by $\wh{\ul{\msf{h}}}_{\alpha \bar{\beta}} =  \e^{2 \varphi} \ul{\msf{h}}_{\alpha \bar{\beta}}$,  the corresponding lifts in $\accentset{n.e.}{\mbf{c}}$ are given by
\begin{align*}
g & = 4 \,  \varpi^*\ul{\theta}{}^0 \, \lambda + 2 \, \varpi^*\left( \ul{\msf{h}}_{\alpha \bar{\beta}} \ul{\theta}{}^\alpha \,  \overline{\ul{\theta}}{}^{\bar{\beta}} \right) \, , & 
\wh{g} & = 4 \,  \varpi^*\wh{\ul{\theta}}{}^0 \, \wh{\lambda} + 2 \, \varpi^*\left( \wh{\ul{\msf{h}}}_{\alpha \bar{\beta}}  \wh{\ul{\theta}}{}^\alpha \,  \ol{\wh{\ul{\theta}}}{}^{\bar{\beta}} \right) \, ,
\end{align*}
where
\begin{align*}
 \widehat{\lambda} & =  \lambda + \frac{1}{2} \i \ul{\Upsilon}_\alpha \theta^{\alpha} - \frac{1}{2} \i \ul{\Upsilon}_{\bar{\alpha}} \theta^{\bar{\alpha}} - \frac{1}{2} \ul{\Upsilon}_\alpha \ul{\Upsilon}^\alpha \ul{\theta}^0\, ,
\end{align*}
and $\wh{g} = \e^{\ul{\varphi}} g$ .

In addition, $\varpi^*\ul{\theta}{}^0 = g(k, \cdot)$, $\varpi^*\wh{\ul{\theta}}{}^0 = \wh{g}(k, \cdot)$,  $\lambda = g(\ell, \cdot)$, and $\widehat{\lambda} = \wh{g}(\wh{\ell}, \cdot)$, where $(k,\ell)$ and $(\wh{k},\wh{\ell})$ are related by \eqref{eq:kl}.
\end{enumerate}
\end{prop}

\begin{rem}
It is shown in \cite{Taghavi-Chabert2021} how the Levi-Civita connection of a metric in $\accentset{n.e.}{\mbf{c}}$ relates to the Webster--Tanaka connection of its corresponding almost pseudo-Hermitian structure. In this way, one can identify $\ul{\Upsilon}_{\alpha}$ in Proposition \ref{prop-adapted_frame_max_tw-nsh-nexp_conf} as the $(1,0)$-part of the difference between two Webster--Tanaka connections.
\end{rem}

\begin{exa}[Fefferman construction]\label{exa:Feff}
There is a well-known canonical construction, originally due to Fefferman \cite{Fefferman1976,Fefferman1976a} and later characterised by Sparling, Graham \cite{Graham1987} (see also \cite{Cap2008,Cap2010}), which associates to any contact CR structure $(\ul{\mc{M}}, \ul{H}, \ul{J})$ of dimension $2m+1$  a conformal structure $\mbf{c}$ of Lorentzian signature and dimension $2m+2$ on the total space of a circle bundle $\mc{M} \stackrel{\varpi}{\longrightarrow} \ul{\mc{M}}$, namely the quotient of $\bigwedge^{m+1}(\Ann(\ul{H}^{(1,0)}))$ with the zero section removed, by a $\R_{>0}$-action. More explicitly, let $\phi$ be a fiber coordinate on $\mc{M}$. Choose of contact form $\ul{\theta}^0$ of $(\ul{H}, \ul{J})$ with Levi form $\msf{h}_{\alpha \bar{\beta}}$ and corresponding Webster--Tanaka connection $1$-form $\ul{\Gamma}_{\alpha}{}^{\beta}$, \emph{Webster--Schouten scalar $\ul{\Rho}$}, and define the $1$-form
\begin{align*}
 \lambda & = \d \phi+ \frac{1}{m+2} \left( \i \ul{\Gamma}_{\alpha}{}^{\alpha} - \i \frac{1}{2} \ul{\msf{h}}^{\alpha \bar{\beta}} \d \ul{\msf{h}}_{\alpha \bar{\beta}} - \ul{\Rho} \ul{\theta}^0 \right)\, .
\end{align*}
We may view $\lambda$ as a Weyl connection on the fiber bundle $\mc{M}$. Then
\begin{align*}
g & = 4 \,  \varpi^*\ul{\theta}{}^0 \, \lambda + 2 \, \varpi^*\left( \ul{\msf{h}}_{\alpha \bar{\beta}} \ul{\theta}{}^\alpha \,  \overline{\ul{\theta}}{}^{\bar{\beta}} \right) 
\end{align*}
is a metric in the Fefferman conformal class $\mbf{c}$. One can check that any change of contact forms induces a change of metrics in $\mbf{c}$ as described in Proposition \ref{prop-adapted_frame_max_tw-nsh-nexp_conf}.

Note that in this particular case, the fibers of $\mc{M} \stackrel{\varpi}{\longrightarrow} \ul{\mc{M}}$ are generated by a null conformal Killing field $k$ that is Killing for any metric $g$ in $\accentset{n.e.}{\mbf{c}}$ --- in this case, we have $\kappa = g(k,\cdot) = 2 \varpi^* \ul{\theta}^0$ for some corresponding contact $1$-form $\ul{\theta}^0$.

A generalisation of this construction to the partially integrable case is given in \cite{Leitner2010}, and characterised in \cite{Taghavi-Chabert}. An appropriate modification of this construction yields Taub-NUT metrics as shown in \cite{Alekseevsky2021} and \cite{Taghavi-Chabert2021}.
\end{exa}

We end this section with the following proposition regarding the most degenerate conformally invariant condition on the intrinsic torsion. Its proof can easily be obtained from Theorem \ref{thm:intors-rob}, and the geometric interpretations are a direct consequence of the various results of Section \ref{sec:geometry}.
\begin{prop}\label{prop:cplx_intors}
Let $(\mc{M}, \mbf{c}, N, K)$ be an almost Robinson manifold with congruence of null curves $\mc{K}$. Let $g$ be any metric in $\mbf{c}$, and suppose its intrinsic torsion is a section of $\slashed{\mc{G}}_1^{0,0}$, i.e.\
\begin{align*}
	\breve{\gamma}_i = \breve{\tau}_{\alpha \beta} = \breve{\tau}^\circ_{\alpha \bar{\beta}} = \breve{\sigma}_{i j} = \breve{\zeta}_{\alpha \beta} = 
	\breve{G}^{\tiny{\ydskew}}_{\alpha \beta \gamma} =
	\breve{G}^{\tiny{\ydhook}}_{\alpha \beta \gamma} =
	\breve{G}^{\circ}_{\bar{\alpha} \beta \gamma} = 2(m-1) \i \breve{E} _{\gamma} - \breve{G}_{\gamma} =
	\breve{B}_{\beta \gamma} = 0 \, .
\end{align*}
 Then $\mc{K}$ is a non-shearing congruence of null geodesics, $(N,K)$ is involutive, and:
\begin{itemize}
\item If $\mc{K}$ is non-twisting, i.e.\ $\breve{\tau}^{\omega} = 0$, the leaf space is foliated by Hermitian manifolds of Gray--Hervella class $\mc{W}_4$, and $(\mc{M}, g, N, K)$ is a Robinson manifold of Kundt type for any metric $g$ in $\accentset{n.e.}{\mbf{c}}$, and of Robinson--Trautman type otherwise.
\item If $\mc{K}$ is twisting, i.e.\ $\breve{\tau}^{\omega} \neq 0$, the almost Robinson structure is induced by the twist of $\mc{K}$, and descends to a contact CR structure $(\ul{H}, \ul{J})$ on the leaf space $\ul{\mc{M}}$ of $\mc{K}$. There is a one-to-one correspondence between metrics in $\accentset{n.e.}{\mbf{c}}$ and contact forms of $(\ul{H}, \ul{J})$.
\end{itemize}
\end{prop}

\subsection{The Mariot--Robinson theorem}\label{sec:Robinson_theorem}
A solution to the vacuum Maxwell equations on a four-dimensional Lorentzian manifold $(\mc{M}, g)$ is a $2$ form $F$ that is both closed and co-closed, i.e.\ $\d F = \d \star F = 0$ -- here $\star$ is the Hodge duality operator. Note that these equations are conformally invariant, which justifies the inclusion of this section at this point. Such a solution is said to be \emph{null} or \emph{algebraically special} if $F$ satisfies the algebraic constraint $k \hook F =0$ for some null vector field $k$, and $\kappa \wedge F = 0$ where $\kappa = g(k,\cdot)$.

The \emph{Mariot theorem} \cite{Mariot1954}  states that any solution to the vacuum Maxwell equations gives rise to a non-shearing congruence of null geodesics. The congruence is generated by $k = g^{-1}(\kappa , \cdot)$. The converse, known as the \emph{Robinson theorem} \cite{Robinson1961}, is also true, provided that we work in the analytic category: one can construct an analytic null solution to the vacuum Maxwell equations from any analytic non-shearing congruence of null geodesics. We can clearly substitute non-shearing congruence of null geodesics by Robinson structure here, and this move allows for generalisation of the theorem to irreducible spinor fields \cite{Penrose1986}. In fact, it makes its understanding more transparent as we shall now explain. We note that the assumption of analyticity is crucial for the implication part of the theorem as references \cite{Tafel1985,Tafel1986} make clear.

Following \cite{Eastwood1995,Mason1995}, we first note that any null $2$-form $F$ must be the sum of a self-dual totally null simple complex $2$-form $\nu$ and its (anti-self-dual) complex conjugate $\overline{\nu}$.  In the language of the present paper, $\nu$ is called a complex Robinson $2$-form and annihilates an almost null structure $N$. The condition that $F$ be both closed and co-closed is equivalent to $\nu$ being closed. But if $\nu$ is closed, $N$ must be integrable, i.e.\ $N$ is a null structure (or equivalently, a Robinson structure). Conversely, if $N$ is an analytic self-dual null structure, it gives rise to a foliation $\mc{N}$ by two-dimensional totally null complex leaves on the complexification $(\wt{\mc{M}}, \wt{g})$ of $(\mc{M}, g)$ -- see Remark \ref{rem:complex_ext}. Take any $2$-form $\ul{\nu}$ on the two-dimensional local leaf space $\ul{\wt{\mc{M}}}$ of $\mc{N}$. Then $\ul{\nu}$ is necessarily closed, and so is its pullback from $\ul{\wt{\mc{M}}}$ to $\wt{\mc{M}}$. A completely parallel argument applies to the complex conjugate of $\ul{\nu}$, and their sum gives rise to an analytic null solution to the vacuum Maxwell equation on restriction to $(\mc{M}, g)$.

The complex `part' of Mariot--Robinson theorem was later generalised to even dimensions in \cite{Hughston1988} and to odd dimensions in \cite{Taghavi-Chabert2017a}. Its proof hinges on the same reasoning. We work in the analytic category with a complex Riemannian manifold $(\wt{\mc{M}}, \wt{g})$ of dimension $2m+2$: a totally null (and thus simple) $(m+1)$-form $\nu$ defines an almost null structure $N$, and if $\nu$ is closed, $N$ is integrable. Conversely, any null structure $N$ gives rise to a foliation $\mc{N}$ by $(m+1)$-dimensional totally null complex leaves on $\wt{\mc{M}}$. The pullback of any form of top degree on the $(m+1)$-dimensional local  leaf space of $\mc{N}$ is a totally null $(m+1)$-form that is necessarily closed, and also co-closed since it is either self-dual or anti-self-dual.

If we now start with an analytic Lorentzian manifold $(\mc{M}, g)$ of dimension $2m+2$, we can apply the above result to analytic Robinson structures by extending them to the complexification of $(\mc{M}, g)$, which eventually leads to a suitable Lorentzian articulation of the Mariot--Robinson theorem.

\subsection{The Kerr theorem}\label{sec:Kerr_theorem}
We now describe all local analytic Robinson structures on even-dimensional Minkowski space $\mbb{M}$. This problem is conformally invariant, and as such, is most elegantly formulated in the language of twistor geometry. The result, now known as the \emph{Kerr theorem}, was initially motivated by the search for Kerr--Schild solutions to the Einstein field equations \cite{Kerr2009}, but came to play a seminal r\^{o}le in Penrose's then-nascent twistor theory \cite{Penrose1967}.

We first review the story in dimension four, where analytic Robinson structures are identified with analytic non-shearing congruences of null geodesics. The appropriate framework is the so-called \emph{twistor correspondence} (also referred to as the \emph{Klein correspondence}), which we shall presently describe -- see \cite{Penrose1967,Ward1990,Huggett1994} for details. We consider  a four-dimensional complex vector space $\mbb{T}$. The Grassmannian of two-planes in $\mbb{T}$ is a smooth four-dimensional complex projective quadric $\mc{Q}$ in $\mbb{P} \left( \bigwedge^2 \mbb{T} \right) \cong \CP^5$, and as such, is naturally equipped with a complex holomorphic conformal structure. There are two disjoint families of two-dimensional linear subspaces of $\mc{Q}$, elements of which are called \emph{$\alpha$-planes} and \emph{$\beta$-planes}, according to whether these planes are self-dual or anti-self-dual. The $\alpha$-planes of $\mc{Q}$ are parametrised by the points of the projective space $\mbb{PT} \cong \CP^3$ , known as the \emph{twistor space} of $\mc{Q}$, and the $\beta$-planes are parametrised by the points of \emph{dual twistor space} $\mbb{PT}^*$. Twistor space contains an analytic family $\mc{F}$ of complex lines parametrised by the points of $\mc{Q}$. We thus have a geometric correspondence between $\mc{Q}$ and $\mbb{PT}$. The Kerr theorem in this context extends this correspondence to one between null foliations in $\mc{Q}$ and hypersurfaces in $\mbb{PT}$. To be precise, locally, a null structure on $\mc{Q}$ is simply a foliation by $\alpha$-planes, and it immediately follows that its leaf space can be viewed as a hypersurface in $\mbb{PT}$ intersecting the lines of $\mc{F}$ transversely. Conversely, any null structure on $\mc{Q}$ arises in this way.

To consider real Minkowski space, we introduce a Hermitian inner product $\langle \cdot , \cdot \rangle$ on $\mbb{T}$ of signature $(2,2)$. Under the action of the stabiliser $\mbf{SU}(2,2)$ of $\langle \cdot , \cdot \rangle$, $\mc{Q}$ decomposes into six orbits, one of which we identify as \emph{compactified Minkowski space} $\mbb{M}^{c} := S^3 \times S^1$. In other words, $\mc{Q}$ is the complexification of $\mbb{M}^{c}$. This comes as no surprise considering that $\mbf{SU}(2,2)$  is the double cover of the conformal group $\SO(4,2)$.  Similarly, $\mbb{PT}$ admits the orbit decomposition
\begin{align}\label{eq:PT_decomp}
\mbb{PT} & = \mbb{PT}_+ \sqcup \mbb{PN} \sqcup \mbb{PT}_- \, ,
\end{align}
where
\begin{align}
\begin{aligned}\label{eq:orbit_PT}
\mbb{PT}_+ & := \left\{ [Z] \in \mbb{PT} : \langle Z, Z \rangle > 0 \right\} \, , \\
\mbb{PN}  & := \left\{ [Z] \in \mbb{PT} : \langle Z, Z \rangle = 0 \right\} \, , \\
\mbb{PT}_- & := \left\{ [Z] \in \mbb{PT} : \langle Z, Z \rangle < 0 \right\} \, .
\end{aligned}
\end{align}
Here, $Z$ can be viewed as complex coordinates on $\mbb{T}\cong \mbf{C}^4$. As a real hypersurface in $\CP^3$, $\mbb{PN}$ is the five-dimensional CR hypersphere, i.e.\ $\mbb{PN}$ has topology $S^3 \times S^2$ and is equipped with a contact CR structure of signature $(1,1)$. It also turns out that $\mbb{PN}$ is the space of null lines in $\mbb{M}^{c}$.

Now consider an analytic Robinson structure $(N,K)$ on some subset of $\mbb{M}^{c}$. Let $\mc{K}$ be the non-shearing congruence of null geodesics tangent to $K$ in $\mbb{M}^{c}$, and $\mc{N}$ the complex foliation by (self-dual) totally null $2$-planes tangent to $N$ in $\mc{Q}$. Denote their respective leaf spaces by $\ul{\mc{M}}_{\mc{K}}$ and $\ul{\mc{M}}_{\mc{N}}$. Then, identifying $\ul{\mc{M}}_{\mc{N}}$ as a complex hypersurface in $\mbb{PT}$, the Kerr theorem asserts that $\ul{\mc{M}}_{\mc{K}}$ is a three-dimensional CR submanifold of $\mbb{PN}$ that arises as the intersection of $\ul{\mc{M}}_{\mc{N}}$ and $\mbb{PN}$. Non-analytic Robinson structures on Minkowski space can also be dealt with as a limiting case \cite{Tafel1985,Tafel1986,Penrose1986}.

The generalisation of the Kerr theorem to higher dimensions in the complex case was carried out in \cite{Hughston1988,TaghaviChabert2017} and is analogous. We consider a smooth projective complex quadric $\mc{Q}$ in $\CP^{2m+3}$, which we may identify as the space of null lines in $\C^{2m+4}$ --- see \cite{Harnad1992,Harnad1995}. An $\alpha$-plane in $\mc{Q}$ is now a self-dual linear subspace of $\mc{Q}$ of dimension $m+1$ and a $\beta$-plane its anti-self-dual counterpart. As before, we define the twistor space $\mbb{PT}$ of $\mc{Q}$ to be the space of all $\alpha$-planes of $\mc{Q}$, and the \emph{primed} twistor space $\mbb{PT}'$ of $\mc{Q}$ to be the space of all $\beta$-planes of $\mc{Q}$. When $m$ is odd, $\mbb{PT}' \cong \mbb{PT}^*$, while when $m$ is even $\mbb{PT} \cong \mbb{PT}^*$ and $\mbb{PT}' \cong (\mbb{PT}')^*$. From an algebraic viewpoint, it is convenient to realise $\mbb{PT}$ and $\mbb{PT}'$ as the spaces of pure spinors, up to scale, for the double cover $\mbf{Spin}(2m+4,\C)$ of the complex conformal group $\mbf{SO}(2m+4,\C)$.

Twistor space is a complex manifold of dimension $\frac{1}{2}(m+1)(m+2)$ and contains an analytic family $\mc{F}$ of $\frac{1}{2}m(m+1)$ complex submanifolds parametrised by the points of $\mc{Q}$. In this complex setting, the Kerr theorem states \cite{Hughston1988,TaghaviChabert2017} that any local analytic null structure $N$ on $\mc{Q}$ locally gives rise to a complex submanifold $\ul{\mc{N}}$ of dimension $m+1$ meeting $\mc{F}$ transversely, and every null structure arises in this way. In effect, the submanifold $\ul{\mc{N}}$ is none other than the leaf space of the foliation tangent to $N$.

From this complex description, it is only a small step to obtain the Lorentzian version of the Kerr theorem. Just as in dimension four, under the action of the real form $\SO(2m+2,2)$ of $\SO(2m+4,\C)$ or its spin analogue, $\mc{Q}$ decomposes into six orbits, which includes compactified Minkowski space $\mbb{M}^{c} := S^{2m+1} \times S^1$. To deal with $\mbb{PT}$, we note that the spin representation for the conformal group $\SO(2m+2,2)$ is equipped with a Hermitian inner product $\langle \cdot , \cdot \rangle$ of split signature \cite{Budinich1988}, which restricts to a Hermitian form on $\mbb{PT}$. Using the terminology and results of  \cite{Kopczy'nski1992}, we find that $\mbb{PT}$ admits the decomposition \eqref{eq:PT_decomp} where its orbits are defined  just as in \eqref{eq:orbit_PT}. Their interpretation is as follows:
\begin{itemize}
\item $\mbb{PN}$ is of real dimension $(m+1)(m+2)-1$, and consists of self-dual $(m+2)$-dimensional linear subspaces of $\mc{Q}$ of \emph{real index $2$}: these are the pure spinors (up to scale) that are null with respect to the Hermitian inner product on $\mbb{PT}$;
\item $\mbb{PT}_+$ and $\mbb{PT}_-$ are of real dimension $(m+1)(m+2)$, and consist of self-dual $(m+2)$-dimensional linear subspaces of $\mc{Q}$ of \emph{real index $0$}: these are the pure spinors (up to scale) that are spacelike, respectively timelike, with respect to the Hermitian inner product on $\mbb{PT}$.
\end{itemize}
In particular, by virtue of being a real hypersurface defined by the vanishing of the Hermitian form on $\mbb{PT}$, the orbit $\mbb{PN}$ is a \emph{CR manifold}, whose Levi form, one can check, has signature $(m,m)$.

At present, let us take $(N,K)$ to be an analytic Robinson structure  on some subset of $\mbb{M}^{c}$ with congruence of null geodesics $\mc{K}$ and complex foliation by (self-dual) totally null $(m+1)$-planes $\mc{N}$ in $\mc{Q}$. View the leaf space $\ul{\mc{M}}_{\mc{N}}$ of $\mc{N}$ as a complex submanifold in $\mbb{PT}$. Then the intersection of $\ul{\mc{M}}_{\mc{N}}$ with $\mbb{PN}$ is a $(2m+1)$-dimensional CR submanifold of $\mbb{PN}$, which is precisely the leaf space $\ul{\mc{M}}_{\mc{K}}$ of $\mc{K}$.

\begin{rem}
It is crucial to note that for $m>1$, $\mbb{PN}$ \emph{cannot} be identified with the $(4m+1)$-dimensional space of null lines $\mbb{NL}$ in $\mbb{M}^c$. In general $\mbb{NL}$ is a homogeneous space equipped with a \emph{Lie contact structure} \cite{Sato1989}, and does not admit any distinguished CR structure unless $m=1$.

There are thus two ways of embedding the leaf space of the null geodesic congruence $\mc{K}$ associated to a Robinson structure: one as a submanifold of $\mbb{NL}$ and the other as a \emph{CR} submanifold of $\mbb{PN}$. But only the latter can encode the CR structure of the leaf space of $\mc{K}$.
\end{rem}

\begin{rem}
As stated, the Kerr theorem is only concerned with the involutivity of an almost null structure be it in the complex case or in Lorentzian signature. Further degeneracy conditions on the intrinsic torsion of the null structure will in general impact the way the leaf space of a null complex foliation sits in twistor space. These features have only been marginally investigated so far.\footnote{In fact, this was already demonstrated in the odd-dimensional analogue of the Kerr theorem in \cite{TaghaviChabert2017}. In even dimensions, complex case, this was established in unpublished work by Jan Gutt (Private communication with the third author.)}
\end{rem}

\section{Generalised almost Robinson geometries}\label{sec:gen_Rob}

\subsection{Generalised Robinson structures}
We now present a variant of the notion of almost Robinson structure, which in dimension four corresponds to the notion of optical geometry presented in \cite{Trautman1984,Trautman1985,Robinson1985,Robinson1986,Robinson1989,Musso1992,Trautman1999}, and which was referred to as \emph{generalised optical geometry} in \cite{Fino2020}.
\begin{defn}\label{def:gen_Rob}
Let $\mc{M}$ be a smooth manifold of dimension $2m+2$. A \emph{generalised almost Robinson structure} consists of a triple $(N,K,\mbf{o})$,  where $N$ is a complex $(m+1)$-plane distribution, $K := N \cap T \mc{M}$ is a real line distribution on $\mc{M}$, and $\mbf{o}$ an equivalence class of Lorentzian metrics such that
\begin{enumerate}
\item for each $g$ in $\mbf{o}$, $N$ is null with respect to the complex linear extension of $g$;
\item any two metrics $g$ and $\wh{g}$ in $\mbf{o}$ are related by
\begin{align}\label{eq:optical-metric}
 \wh{g} & = \e^{2 \varphi} \left( g + 2 \, \kappa \, \alpha \right) \, ,
\end{align}
for some smooth function $\varphi$ and $1$-form $\alpha$ on $\mc{M}$, and $\kappa = g(k ,\cdot)$ for some non-vanishing section $k$ of $K$.
\end{enumerate}
We shall say that the generalised almost Robinson structure is
\begin{itemize}
\item \emph{restricted} if any of the $1$-forms $\alpha$ in \eqref{eq:optical-metric} satisfies $\alpha(k)=0$.
\item \emph{of Kerr--Schild type} if any of the $1$-forms $\alpha$ in \eqref{eq:optical-metric} satisfies $\alpha \wedge \kappa=0$.
\end{itemize}
We shall refer to $(\mc{M},N,K,\mbf{o})$ as a \emph{generalised almost Robinson geometry}. In addition, we shall call $(\mc{M},N,K,\mbf{o})$ a \emph{generalised nearly Robinson geometry} if $[K,N] \subset N$, and a \emph{generalised Robinson geometry} if $[N,N] \subset N$.
\end{defn}
We see in particular that a generalised almost Robinson structure determines a generalised optical geometry $(K,\mbf{o})$ in the sense of \cite{Fino2020}. It is straightforward to check that the two conditions above are well-defined. In particular, the property of $K$ and $N$ being totally null does not depend on the choice of metric in $\mbf{o}$, and neither does the notion of orthogonal complement $K^\perp$ of $K$. A generalised almost Robinson geometry $(\mc{M}, N, K, \mbf{o})$ also has an associated congruence of null curves $\mc{K}$ tangent to $K$.

The following lemma is immediate.
\begin{lem}\label{lem:gen_al_Rob}
Let $(\mc{M},N,K,\mbf{o})$ be a generalised almost Robinson geometry. For each metric $g$ in $\mbf{o}$, $(N,K)$ is an almost Robinson structure on $(\mc{M},g)$.
\end{lem}

We shall therefore re-employ the terminology used in the previous sections. In particular, any non-vanishing section of $K$ will be called an optical vector field, and any $1$-form annihilating $K^\perp$ an optical $1$-form. The definition of a complex Robinson $(m+1)$-form as a section of $\bigwedge^{m+1} \Ann(N)$ does not depend on the choice of metric in $\mbf{o}$, and neither does the notion of Robinson $3$-form. To see this, we take, for specificity, two metrics $g$ and $\wh{g}$ in $\mbf{o}$ related via
\begin{align*}
\wh{g}_{a b} & = g_{a b} + 2 \kappa_{(a} \alpha_{b)} \, ,
\end{align*}
for some $1$-form $\alpha_a$. We can choose splitting operators for $(N, K)$ for each of the metrics: $(\kappa_a, \delta_a^\alpha, \delta_a^{\bar{\alpha}} , \lambda_a)$ and $(\ell^a, \delta^a_\alpha, \delta^a_{\bar{\alpha}} , k^a)$ for $g_{a b}$, and $(\wh{\kappa}_a, \wh{\delta}_a^\alpha, \wh{\delta}_a^{\bar{\alpha}} , \wh{\lambda}_a)$ and $(\wh{\ell}^a, \wh{\delta}^a_\alpha, \wh{\delta}^a_{\bar{\alpha}} , \wh{k}^a)$ for $\wh{g}_{a b}$ such that
\begin{align}
\begin{aligned}\label{eq:spli_change}
\wh{\kappa}_a & = \kappa_a \, , & \wh{\delta}_a^\alpha & = \delta_a^\alpha  \, , & \wh{\lambda}_a & = \lambda_a + \alpha_a \, , \\
\wh{k}^a & = \beta k^a \, , & \wh{\delta}^a_\alpha & = \delta^a_\alpha - \beta \alpha_\alpha k^a \, , & \wh{\ell}^a & = \ell^a - \beta \alpha_0 k^a \, .
\end{aligned}
\end{align}
Here, $\beta = (1 + \alpha^0)^{-1}$. Let $\omega_{i j}$ be the bundle Hermitian structure for $(N, K)$. Then its associated Robinson $3$-form $\rho_{a b c} = 3 \kappa_{[a} \omega_{b c]}$ where $\omega_{a b} = \omega_{i j} \delta^i_a \delta^j_b$ remains invariant under the change \eqref{eq:spli_change}.

Generalised almost Robinson geometry arises naturally in the context of lifts of almost CR manifolds as described in Section \ref{sec:lift}.

\begin{prop}\label{prop:lift_gen}
Let $(\ul{\mc{M}}, \ul{H}, \ul{J})$ be a $(2m+1)$-dimensional almost CR manifold, and $\mc{M} := \R \times \ul{\mc{M}}$ be a trivial line bundle over  $\ul{\mc{M}}$. Then $\mc{M}$ is naturally equipped with a generalised almost Robinson structure $(N, K, \mbf{o})$ such that for any metric $g$ in $\mbf{o}$, $(\mc{M}, g, N, K)$ is a nearly Robinson structure, i.e.\ the intrinsic torsion of the corresponding almost Robinson structure $(N, K)$ on $(\mc{M},g)$ is a section of $(\slashed{\mc{G}}_{-1}^{1 \times 3})_{[-2\i:1]} \cap \slashed{\mc{G}}_{-1}^{2,0}$.
\end{prop}

\begin{proof}
This is a direct consequence of Proposition \ref{prop:lift}: the equivalence class $\mbf{o}$ of metrics on $\mc{M}$ related via \eqref{eq:optical-metric} is simply the set of all lifts of $(\ul{\mc{M}}, \ul{H}, \ul{J})$ to $\mc{M}$ for a \emph{fixed} choice of a conformal class of Hermitian forms $[h_{\alpha \bar{\beta}}]$. By construction, $(N,K)$ is clearly a nearly Robinson structure. The distributions $N$ and $K$ do not depend on the lift, and by Lemma \ref{lem:gen_al_Rob}, $(N,K)$ is an almost Robinson structure on $(\mc{M},g)$.
\end{proof}

The key idea of the next theorem is that it allows us to construct families of Lorentzian metrics equipped with almost Robinson structures sharing the same geometric properties. To be precise, choosing a metric $g$ in $\mbf{o}$ determines an almost Robinson structure on $(\mc{M},g)$, and one may ask which subbundles of the bundle of intrinsic torsions $\mc{G}$ do not depend on the choice of metric $g$ in $\mbf{o}$. We shall call those that remain invariant under such a change \emph{$\mbf{o}$-invariant} subbundles.

\begin{thm}\label{thm:gen_Rob_prop}
Let $(\mc{M}, N, K, \mbf{o})$ be a generalised almost Robinson geometry. Let $g $ be a metric in $\mbf{o}$ so that $(N, K)$ defines an almost Robinson structure for $(\mc{M},g)$ with bundle of intrinsic torsions $\mc{G}$.
\begin{enumerate}
\item The following subbundles of $\mc{G}$ are $\mbf{o}$-invariant:\\
\begin{tabular}{m{2.2cm} m{2.2cm} m{2.2cm} m{2.2cm}}
$\slashed{\mc{G}}_{-2}^{0,0}$, \\
 $\slashed{\mc{G}}_{-1}^{1,0}$, & $\slashed{\mc{G}}_{-1}^{1,1}$, & $\slashed{\mc{G}}_{-1}^{1,2}$, \\
  $\slashed{\mc{G}}_{-1}^{2,0}$, & $\slashed{\mc{G}}_{-1}^{2,1}$, &
  $(\slashed{\mc{G}}_{-1}^{1 \times 2})_{[x:y]}$, & \scriptsize{$[x:y] \in \RP^1$}, \\
  $\slashed{\mc{G}}_{-1}^{1,1} \cap \slashed{\mc{G}}_{-1}^{3,0}$, & $(\slashed{\mc{G}}_{-1}^{1 \times 3})_{[-2\i:1]}$, \\
  $\slashed{\mc{G}}_{-1}^{1,1}  \cap  \slashed{\mc{G}}_0^{1,1}$, &
  $\slashed{\mc{G}}_{-1}^{1,1}  \cap \slashed{\mc{G}}_0^{1,2}$, &
  $\slashed{\mc{G}}_{-1}^{1,1}  \cap \slashed{\mc{G}}_0^{1,3}$.
\end{tabular}\\
Any $\mbf{o}$-invariant subbundle of a generalised almost Robinson geometry that is not restricted must be an intersection of these.

\item Assuming that the generalised almost Robinson geometry is restricted, in addition the following subbundles of $\mc{G}$ are $\mbf{o}$-invariant:\\
\begin{tabular}{m{2.2cm} m{2.2cm}}
 $\slashed{\mc{G}}_{-1}^{3,0}$, \\
  $(\slashed{\mc{G}}_{-1}^{1 \times 3})_{[z:w]}$, &  \scriptsize{$[z:w] \in \CP^1$}, \\
 $\slashed{\mc{G}}_0^{1,1}$.
 \end{tabular}\\
Any $\mbf{o}$-invariant subbundle of a restricted generalised almost Robinson geometry that is not of Kerr--Schild type must be an intersection of these and the ones in (1).

\item Assuming that the generalised almost Robinson geometry is of Kerr--Schild type, in addition the following subbundles of $\mc{G}$ are $\mbf{o}$-invariant:\\
\begin{tabular}{m{2.2cm} m{2.2cm} m{2.2cm}}
$\slashed{\mc{G}}_0^{1,2}$, &  $\slashed{\mc{G}}_0^{1,3}$, & $(\slashed{\mc{G}}_{0}^{0 \times 1})_{[-2(m-1)\i:1]}$,
\end{tabular}\\
and\\
\begin{tabular}{m{2.2cm} m{2.2cm}}
$\slashed{\mc{G}}_1^{0,0}$, & when $m>2$, \\
$\slashed{\mc{G}}_{-1}^{1,1} \cap \slashed{\mc{G}}_1^{0,0}$, & when $m=2$.
 \end{tabular}\\
Any $\mbf{o}$-invariant subbundles of a generalised almost Robinson geometry of Kerr--Schild type must be an intersection of these and the ones in (1) and (2).
\end{enumerate}
\end{thm}

\begin{proof}
Let $g$ be a metric in $\mbf{o}$. Any subbundle of $\mc{G}$ that is invariant under changes of metrics in $\mbf{o}$ must also be conformally invariant. Thus, it is enough to consider the subbundles given in Theorem \ref{thm:conf-Rob}, and a metric $\wh{g}$ in $\mbf{o}$ related to $g$ by
\begin{align*}
\wh{g}_{a b} & = g_{a b} + 2 \kappa_{(a} \alpha_{b)} \, ,
\end{align*}
for some $1$-form $\alpha_a$. Denote by $\wh{\nabla}$ and $\nabla$ their corresponding Levi-Civita connections. Then for any $1$-form $\xi_a$, we have
\begin{align*}
\wh{\nabla}_a \xi_b = \nabla_a \xi_b - Q_{a b}{}^{c} \xi_c \, ,
\end{align*}
where $Q_{a b c} = Q_{a b}{}^{d} g_{d c}$ is given explicitly in Appendix \ref{app:gen_al_Rob}. From \eqref{eq_DD1}, \eqref{eq:Q1}, \eqref{eq:Q2} and \eqref{eq:Q3}, we immediately find that the subbundles $\slashed{\mc{G}}_{-2}^{0,0}$, $\slashed{\mc{G}}_{-1}^{1,0}$, $\slashed{\mc{G}}_{-1}^{1,1}$, $\slashed{\mc{G}}_{-1}^{1,2}$, $\slashed{\mc{G}}_{-1}^{2,0}$, $\slashed{\mc{G}}_{-1}^{2,1}$ do not depend on the choice of metric in $\mbf{o}$. The same clearly applies to the subbundles $(\slashed{\mc{G}}_{-1}^{1 \times 2})_{[x:y]}$ for any $[x:y] \in \RP^1$.

We now proceed with the remaining subbundles given in Theorem \ref{thm:conf-Rob}. Throughout, $\mr{T}$ will denote the intrinsic torsion of $(\mc{M},g,N,K)$.
\begin{itemize}
\item  Suppose $\mathring{T}$ is a section of $\slashed{\mc{G}}_{-1}^{3,0}$. Then by \eqref{eq_DD2} and \eqref{eq:Q4}, we have
\begin{align*}
(\wh{\nabla}  \rho)^{0}{}_{\beta \gamma 0}
& =  2 \i \alpha^{0} \tau_{\beta \gamma} \, .
\end{align*}
The LHS is zero if and only if either $\tau_{\alpha \beta} = 0$ or $\alpha(k) =0$. The former is equivalent to $\mathring{T}$ being a section of $\slashed{\mc{G}}_{-1}^{1,1} \cap \slashed{\mc{G}}_{-1}^{3,0}$.

\item  Suppose $\mathring{T}$ is a section of $(\slashed{\mc{G}}_{-1}^{1 \times 3})_{[z:w]}$ where $[z:w] \in \CP^1$. Then by \eqref{eq_DD2}, \eqref{eq_DD3} and \eqref{eq:Q4}, we have
\begin{align*}
 \i z ( \wh{\nabla} \rho )_{[\alpha \beta]}{}^{0}{}_{0} + w (\wh{\nabla} \rho )^{0}{}_{\alpha \beta 0}
& = ( z + 2 w \i )  \alpha^{0}  \tau_{\alpha \beta}  \, .
\end{align*}
We immediately conclude that $(\slashed{\mc{G}}_{-1}^{1 \times 3})_{[-2 \i:1]}$ is $\mbf{o}$-invariant. Suppose now $[z:w] \neq [-2\i:1]$. Then, we must have either $\tau_{\alpha \beta} = 0$ or $\alpha(k)=0$.

\item  Suppose $\mathring{T}$ is a section of $\slashed{\mc{G}}_0^{1,1}$. This subbundle is contained in $(\slashed{\mc{G}}_{-1}^{1 \times 3})_{[4\i:-1]}$, which we know is $\mbf{o}$-invariant provided either $\alpha(k)=0$ or $\tau_{\alpha \beta}=\zeta_{\alpha \beta}=0$. In addition, by skew-symmetry we find that
\begin{align*}
(\wh{\nabla} \rho )_{[\alpha \beta \gamma] 0} 
& = 0 \, ,
\end{align*}
which does not yield any further conditions. Hence, $\slashed{\mc{G}}_0^{1,1}$ is $\mbf{o}$-invariant.

\item Suppose $\mathring{T}$ is a section of $\slashed{\mc{G}}_0^{1,2}$.   This is a subbundle of $\slashed{\mc{G}}_{-1}^{1,2}$, i.e.\ $\sigma_{\alpha \beta} = 0$.  It is also contained in $(\slashed{\mc{G}}_{-1}^{1 \times 3})_{[2\i:-1]}$. So, for $\mbf{o}$-invariance, we must have
\begin{itemize}
\item[$\diamond$] either $\tau_{\alpha \beta}=\zeta_{\alpha \beta}=0$. In this case, by \eqref{eq_DD4}, \eqref{eq:Q5} and \eqref{eq:Q2}, we find that
\begin{align*}
(\wh{\nabla} \rho )_{(\alpha \beta) \gamma 0}
& = 0 \, .
\end{align*}
\item[$\diamond$] or $\alpha(k)=0$, from which we find
\begin{align*}
(\wh{\nabla} \rho )_{(\alpha \beta) \gamma 0}
& = \alpha_{(\alpha} \tau_{\beta) \gamma} \, ,
\end{align*}
which tells us that one must impose in addition $\alpha(v)=0$ for all $v \in \Gamma(K^\perp)$ for invariance.
\end{itemize}

\item Suppose $\mathring{T}$ is a section of $\slashed{\mc{G}}_0^{1,3}$. This is also a subbundle of $\slashed{\mc{G}}_{-1}^{1,2}$ and $\slashed{\mc{G}}_{-1}^{2,0}$, i.e.\ $\tau^\circ_{\alpha \bar{\beta}} = \sigma_{\alpha \bar{\beta}} = 0$. It is also contained in $\slashed{\mc{G}}_{-1}^{3,0}$, so for $\mbf{o}$-invariance, we must have
\begin{itemize}
\item[$\diamond$] either $\tau_{\alpha \beta} = 0$ in which case, by \eqref{eq_DD5}, \eqref{eq:Q6} and \eqref{eq:Q3},
\begin{align*}
\left( (\wh{\nabla} \rho )_{\bar{\alpha} \beta \gamma 0} \right)_\circ & = 0 \, .
\end{align*}
If we assume $\tau_{\alpha \beta} = 0$, then no further conditions are necessary. 
\item[$\diamond$]  or $\alpha(k)=0$. Then, again,  by \eqref{eq_DD5}, \eqref{eq:Q6} and \eqref{eq:Q3}, we have
\begin{align*}
\left( (\wh{\nabla} \rho )_{\bar{\alpha} \beta \gamma 0} \right)_\circ
& = 2 \i \left( \alpha_{\bar{\alpha}} \tau_{\beta \gamma} \right)_\circ \, ,
\end{align*}
from which we immediately conclude that $\alpha(v)=0$ for all $v \in \Gamma(K^\perp)$ for $\mbf{o}$-invariance.
\end{itemize}

\item Suppose $\mathring{T}$ is a section of $(\slashed{\mc{G}}_{0}^{0 \times 1})_{[-2(m-1)\i:1]}$. This is contained in  $(\slashed{\mc{G}}_{-1}^{1 \times 3})_{[2(m-1)\i:1]}$, which is $\mbf{o}$-invariant if $\alpha(k) = 0$ or $\tau_{\alpha \beta} = \zeta_{\alpha \beta} = 0$. In particular, by \eqref{eq_DD8}, \eqref{eq_DD9}, \eqref{eq:Q4} and \eqref{eq:Q6}, we have
\begin{align*}
2 h^{\gamma \bar{\alpha}} (\wh{\nabla} \rho )_{0 \beta \gamma \bar{\alpha}} + h^{\gamma \bar{\alpha}} (\wh{\nabla} \rho )_{\bar{\alpha} \beta \gamma 0}   & =
 - 4 (m-1)  \i  Q_{0 \beta}{}^{0} -  \i h^{\gamma \bar{\alpha}} \left( Q_{\bar{\alpha} \beta \gamma} -  Q_{\bar{\alpha} \gamma \beta } \right) \, .
\end{align*}
Comparing the two terms on the RHS shows that for invariance to hold, one needs $\alpha \wedge \kappa = 0$.

\item Suppose $\mathring{T}$ is a section of $\slashed{\mc{G}}_1^{0,0}$. In dimension greater than six, this is a subbundle of $\slashed{\mc{G}}_{-1}^{1,2}$, i.e.\ $\tau_{\alpha \beta} = 0$. Then, by \eqref{eq_DD7}, \eqref{eq:Q7}, \eqref{eq:Q4}, we have
\begin{align*}
(\wh{\nabla} \rho )_{0 \beta \gamma 0}
& =  2 \i Q_{0 [\beta \gamma]} \, .
\end{align*}
where
\begin{align*}
 Q_{0 [\beta \gamma]} & = - Q_{0 [\beta}{}^{0} \alpha_{\gamma]} - \alpha_{[\beta}  (\nabla \kappa)_{\gamma] 0}  + ( \nabla \alpha)_{[\beta \gamma]}  \, , \\
Q_{0 [\beta}{}^{0} \alpha_{\gamma]} 
 & = 
 -  \beta( \d \alpha)^{0}{}_{[\beta} \alpha_{\gamma]}   + \frac{1}{2} \beta \left( ( \nabla \kappa)_{0 [\beta}  + ( \nabla \kappa)_{[\beta 0}  \right) \alpha_{\gamma]} \alpha^0 
\end{align*}
For $\mbf{o}$-invariance, we need $\alpha(v)=0$ for all $v \in K^\perp$, i.e.\ $\alpha \wedge \kappa=0$. Then we are left with
\begin{align*}
(\wh{\nabla} \rho )_{0 \beta \gamma 0}
& =  2 \i  ( \d  \alpha )_{\beta \gamma}  = 2 \i \alpha^0 \tau_{\beta \gamma} \, ,
\end{align*}
since $\alpha_a = \alpha_0 \kappa_a$, which we know it zero. Invariance then follows.

In dimension six, $\tau_{\alpha \beta}$ is not necessarily zero. Suppose it is not. Since $\slashed{\mc{G}}_1^{0,0} \subset (\slashed{\mc{G}}_{0}^{0 \times 1})_{[-2\i:1]}$, we must have $\alpha(k)=0$, and by \eqref{eq_DD7}, \eqref{eq:Q7}, \eqref{eq:Q4}, we find
\begin{align*}
Q_{0 [\beta \gamma]} = -  Q_{0 [\beta}{}^{0} \alpha_{\gamma]} - \alpha_{[\beta} ( \nabla \kappa)_{\gamma] 0}  + ( \d \alpha )_{[\beta \gamma]} + \alpha_{0}  \tau_{\beta \gamma} \, ,
\end{align*}
and
\begin{align*}
Q_{0 [\beta}{}^{0} \alpha_{\gamma]} 
  = 
 - (\d \alpha)^{0}{}_{[\beta} \alpha_{\gamma]}  
\end{align*}
For $\mbf{o}$-invariance we must also assume $\alpha(v)=0$ for all $v \in K^\perp$. Then we are left with
\begin{align*}
(\wh{\nabla} \rho )_{0 \beta \gamma 0}  = 2 \i ( \d \alpha )_{\beta \gamma}+ 2 \i \alpha_{0} \tau_{\beta \gamma} = 4 \i \,  \alpha_{0}  \tau_{\beta \gamma}  \, ,
\end{align*}
where we have made use of the fact that $\alpha_a = \alpha_0 \kappa_a$. So for $\mbf{o}$-invariance, we must have either $\alpha_a = 0$, which we rule out, or assume in addition $\tau_{\alpha \beta} = 0$ and so $\zeta_{\alpha \beta} = 0$.
\end{itemize}
\end{proof}

\begin{rem}
The $\mbf{o}$-invariance of the subbundles $(\slashed{\mc{G}}_{-1}^{1 \times 3})_{[-2\i:1]}$ and $\slashed{\mc{G}}_0^{1,1} \cap \slashed{\mc{G}}_0^{1,2}$ in Theorem \ref{thm:gen_Rob_prop} comes as no surprise. These correspond to $N$ being preserved by the flow of any section of $K$, and the involutivity of $N$ respectively, and these geometric properties do not depend on the equivalence class $\mbf{o}$.
\end{rem}

\begin{rem}
Theorem \ref{thm:gen_Rob_prop} should be contrasted with the situation regarding generalised optical geometries, where the $\mbf{o}$-invariants are precisely the conformal invariants, namely the shear and twist, of the optical geometry of some metric $g$ in $\mbf{o}$ as pointed out in \cite{Robinson1985,Fino2020}.
\end{rem}

\subsection{Generalised Robinson geometries as $G$-structures}
Following the original definition of a generalised optical structure of \cite{Trautman1984,Trautman1985,Robinson1985,Robinson1986,Robinson1989,Trautman1999}, we express an equivalent definition of a generalised almost Robinson geometry in the following terms:
\begin{prop}\label{prop:gen-opt-G-str}
Let $\mc{M}$ be a smooth oriented $(2m+2)$-dimensional manifold. Then the following statements are equivalent.
\begin{enumerate}
\item $\mc{M}$ is endowed with a generalised Robinson structure $(K,\mbf{o})$.\label{item:gen_Rob}
\item $\mc{M}$ is endowed with a pair of distributions $K^{(1)}$ and $K^{(2m+1)}$ of rank $1$ and $2m+1$ respectively such that
\begin{align}\label{eq:filt}
K^{(1)} \subset K^{(2m+1)}
\end{align} and its associated screen bundle $K^{(2m+1)}/K^{(1)}$ is equipped with a conformal structure of Riemannian signature together with a compatible bundle complex structure.\label{item:gen_Rob-G}
\end{enumerate}
\end{prop}

\begin{proof}
Recall from \cite{Fino2020} that a generalised optical structure $(K, \mbf{o})$ is equivalent to the existence of a filtration \eqref{eq:filt}, where $K^{(1)}$ is identified with $K$, together with a conformal structure on the screen bundle $K^{(2m+1)}/K^{(1)}$. Since a generalised almost Robinson structure is in particular a generalised optical structure, it suffices to exhibit a compatible bundle complex structure on $K^{(2m+1)}/K^{(1)}$. But this follows directly from Proposition \ref{prop:equiv_Rob_conf} and the fact that the screen bundle does not depend on the choice of metric in $\mbf{o}$.
\end{proof}

A generalised almost Robinson structure on a smooth manifold $\mc{M}$ can therefore be regarded as a $G$-structure where the structure group of the frame bundle of $\mc{M}$ is reduced from $\mbf{SL}(2m+2,\R)$ (or $\mbf{GL}(2m+2, \R)$ if we drop the assumption that $\mc{M}$ is oriented) to the closed Lie subgroup $H$  that stabilises the filtration \eqref{eq:filt}, together with a conformal structure and compatible bundle complex structure on the  screen bundle. One can easily check that $H$ has dimension $(m+2)^2$.

Under the assumption of real-analyticity, the generalised  almost Robinson geometry $(\mc{M}, N, K,\mbf{o})$ is integrable as a $G$-structure if and only if there exist local coordinates $(u, v, z^\alpha, \bar{z}^{\bar{\alpha}})$ on $\mc{M}$, where $u$ and $v$ are real, $z^\alpha$ complex, and $\bar{z}^{\bar{\alpha}}= \overline{z^\alpha}$ such that
\begin{enumerate}
\item $\frac{\partial}{\partial v}$ spans $K$,
\item $\d u$ annihilates $K^\perp$,
\item $(\d u, \d z^{\alpha} )$ annihilate $N$, and
\item $\mbf{o}$ contains the Minkowski metric $g = 2 \, \d u \, \d v + 2 \, \ut{h}_{\alpha \bar{\beta}} \d z^{\alpha} \, \d \overline{z}^{\bar{\beta}} $,  where $\ut{h}_{\alpha \bar{\beta}}$ is the standard Hermitian metric on $\C^m$.
\end{enumerate}

The characterisation of integrable generalised optical structure was dealt with in \cite{Robinson1985,Fino2020}. In the case of generalised almost Robinson geometries, we have the following result.
\begin{thm}\label{thm:integrable_Rob}
Let $(\mc{M},N,K,\mbf{o})$ be a generalised almost Robinson geometry with congruence of null curves $\mc{K}$. The following statements are equivalent.
\begin{enumerate}
\item There exists a torsionfree linear connection $\nabla'$ compatible with $(N,K)$ and $\mbf{o}$, \label{item:lin_conn}
\begin{align}
\nabla'_u v  & \in \Gamma(N) \, , & \mbox{for any $v \in \Gamma(N)$, $u \in \Gamma(T \mc{M})$,} \label{eq:N_comp_gen} \\
\nabla'_u g (v , w) & \,  \propto \,  g (v , w)  \, , & \mbox{for any $g \in \mbf{o}$, $v,w \in \Gamma(K^\perp)$, $u \in \Gamma(T \mc{M})$.} \label{eq:g_comp_gen} 
\end{align}
\item For any metric $g$ in $\mbf{o}$, $(\mc{M}, g , N, K)$ is a nearly Robinson manifold whose intrinsic torsion is a section of  $\slashed{\mc{G}}_{-1}^{1} \cap \slashed{\mc{G}}_0^{1,1} \cap  \slashed{\mc{G}}_0^{1,2} \cap \slashed{\mc{G}}_0^{1,3}$. In particular, it is of Robinson--Trautman or Kundt type.\label{item:non-sh-non-sh-cong}
\end{enumerate}
Further, in the neighbourhood of any point in $\mc{M}$, there exists smooth functions $u$ and $v$ such that
\begin{itemize}
\item $\frac{\partial}{\partial v}$ spans $K$,
\item $\d u$ annihilates $K^\perp$, and
\item $\mbf{o}$ contains the metric $g = 2 \,  \d u \d v +  \ul{h}$, where $\ul{h}$ is a family of conformally flat Hermitian metrics of Gray--Hervella class $\mc{W}_4$ smoothly parametrised by $u$,
\end{itemize}
if and only if any of the conditions \eqref{item:lin_conn} and \eqref{item:non-sh-non-sh-cong} holds together with the condition
\begin{align}\label{eq:W_cond}
\left( \kappa_{[a} W_{b c] [d e} \kappa_{f]} \right)_\circ & = 0 \, ,
\end{align}
for any $1$-form $\kappa$ annihilating $K^\perp$, where $W_{a b c d}$ is the Weyl tensor of any metric in $\mbf{o}$.
\end{thm}

\begin{proof}
We show that \eqref{item:lin_conn} implies \eqref{item:non-sh-non-sh-cong}. Let $\nabla'$ be as given in \eqref{item:lin_conn}. Note that we can reexpress \eqref{eq:g_comp_gen} equivalently as
\begin{align}
\nabla'_a g_{b c} & = \beta_a g_{b c} + 2 \, \gamma_{a (b} \kappa_{c)}  \, ,
\end{align}
for some tensor fields $\beta_a$ and $\gamma_{a b}$, and where $\kappa_a$ anniliates $K^\perp$. Now, the difference between $\nabla'$ and the Levi-Civita connection $\nabla$ for some $g$ in $\mbf{o}$ can be expressed uniquely by
\begin{align*}
\nabla'_a \alpha_b & = \nabla_a \alpha_b - Q_{a b}{}^{c} \alpha_c \, , & \mbox{for any $1$-form $\alpha_a$,}
\end{align*}
where
\begin{align*}
Q_{a b c} & = \frac{1}{2} \beta_c g_{a b} + \gamma_{c (a} \kappa_{b)} - \beta_{(a} g_{b) c} - \gamma_{(a b)} \kappa_c - \kappa_{(a} \gamma_{b) c} \, .
\end{align*}
This follows from the requirement that $\nabla'$ be also torsion-free, and \eqref{eq:g_comp_gen} holds. Now choosing splitting operators $(\ell^a , \delta^a_i , k^a ) = (\ell^a , \delta^a_\alpha ,  \delta^a_{\bar{\alpha}} , k^a)$ for $g$, we find
\begin{align*}
(\nabla \kappa)^0{}_{j} & = 0 \, , &
(\nabla \kappa)_{i j} & = \frac{1}{2} \beta^0 h_{i j} \, , &
(\nabla \kappa)_{0 j}  & = \frac{1}{2} \left( \beta_j +  \gamma^0{}_j -  \gamma_j{}^0 \right) \, ,
\end{align*}
and
\begin{align*}
(\nabla \rho )^{0}{}_{\alpha \beta 0} & = 0 \, , &
(\nabla \rho )_{\alpha \beta \gamma 0} & = 0 \, , &
(\nabla \rho )_{\bar{\alpha} \beta \gamma 0} & = 2 \i \beta_{[\beta} h_{\gamma] \bar{\alpha}} \, , &
(\nabla \rho )_{0 \beta \gamma 0}  & = \gamma_{[\beta \gamma]} \, ,
\end{align*}
and it follows immediately (see e.g.\ equations \eqref{eq:S-H1_b} and \eqref{eq:S-H2_b}) that the intrinsic torsion of the almost Robinson structure of $(\mc{M}, g, N, K)$ is a section of $\slashed{\mc{G}}_{-1}^{1} \cap \slashed{\mc{G}}_0^{1,1} \cap  \slashed{\mc{G}}_0^{1,2} \cap \slashed{\mc{G}}_0^{1,3}$. We note that this is independent of the choice of metric in $\mbf{o}$ by Theorem \ref{thm:gen_Rob_prop}.

To prove that \eqref{item:non-sh-non-sh-cong} implies \eqref{item:lin_conn}, we note that since $\slashed{\mc{G}}_0^{1,1} \cap  \slashed{\mc{G}}_0^{1,2} \cap \slashed{\mc{G}}_0^{1,3}$ is a $Q$-invariant subbundle of $\slashed{\mc{G}}_{-1}^{3,0} \cap \slashed{\mc{G}}_{-1}^{1} \cap \slashed{\mc{G}}_{-1}^{2}$, we can simply take the linear connection given in Corollary \ref{cor:lin_conn}.

The final part of the proof follows from Theorem 7.6 of \cite{Fino2020} and Table \ref{tab-intors-GH}.
\end{proof}

\begin{rem}
Consider now a generalised almost Robinson geometry $(\mc{M}, N, K, \mbf{o})$ for which the Minkowski metric $\eta$ belongs to $\mbf{o}$, i.e.\ any metric in $\mbf{o}$ is conformal to a metric of the form
\begin{align*}
g & = \eta + 2 \kappa \alpha \, ,
\end{align*}
for some optical $1$-form $\kappa$ and $1$-form $\alpha$. This means that $(N,K)$ is an almost Robinson structure for Minkowski space $(\mc{M}, \eta)$.  Using the Kerr theorem of Section \ref{sec:Kerr_theorem}, one can then generate many non-flat Robinson manifolds.
\end{rem}

In the particular case where $(\mc{M}, N, K, \mbf{o})$ is of Kerr--Schild type, we recover the original Kerr--Schild metric, considered by Kerr and Schild in \cite{Kerr2009} in dimension four, and which is an exact first-order perturbation of the Minkowski metric.

\begin{exa}[The Myers--Perry metric]\label{exa:MP}
Let us write the Minkowski metric in standard coordinates $(t, x^\alpha , y^\alpha, z )_{\alpha=1,\ldots, m}$ in dimension $2m+2$,
\begin{align*}
\eta & = - (\d t )^2 + \sum_{\alpha=1}^m \left( (\d x^\alpha)^2 +  (\d y^\alpha)^2 \right) + (\d z)^2  \, ,
\end{align*}
and let
\begin{align*}
\kappa & = \d t + \sum_{\alpha=1}^m \frac{r(x^\alpha \d x^\alpha + y^\alpha \d y^\alpha)+a_\alpha(x^\alpha \d y^\alpha - y^\alpha \d x^\alpha)}{r^2 + a_\alpha^2} + \frac{z}{r} \d r \, ,
\end{align*}
and
\begin{align*}
f & = \frac{M r^2}{1- \sum_{\alpha=1}^m  \frac{a_\alpha^2 \left( (x^\alpha)^2 + (y^\alpha)^2 \right)}{(r^2+a_\alpha^2)^2}} \frac{1}{\prod_{\alpha=1}^m (r^2 + a_\alpha^2)} \, .
\end{align*}
Here, the radial coordinate is defined by
\begin{align*}
\sum_{\alpha=1}^m \frac{(x^\alpha)^2 + (y^\alpha)^2}{r^2 + a_\alpha^2} + \frac{z^2}{r^2} & = 1 \, .
\end{align*}
Then, the Kerr-Myers-Perry metric in Kerr-Schild form is given by \cite{Myers1986}
\begin{align*}
g & = \eta + f \kappa^2 \, .
\end{align*}
Being a relative of the Kerr-NUT-(A)dS metric -- see Example \ref{exa:KerrNUTAdS} -- the Myers--Perry metric admits two sets of $2^{m-1}$ Robinson structures corresponding to two optical structures.

As shown in \cite{Mason2010}, these Robinson structures are defined by the eigenspinors of a so-called \emph{conformal Killiang-Yano $2$-form} $\xi_{ab}$ that is also closed, i.e.\ $\xi_{ab}$ satisfies the overdetermined system of linear first order partial differential equations:
\begin{align*}
\nabla_a \xi_{b c} & = \frac{2}{2m+1} g_{a [b} \nabla^d \xi_{c] d} \, .
\end{align*}
Since these Robinson structures exist for $\eta$, by the Kerr theorem, they must arise from a complex submanifold of dimension $m+1$ of twistor space -- see Section \ref{sec:Kerr_theorem}. In \cite{TaghaviChabert2017}, the construction of this submanifold is given as the locus of a system of polynomials of degree two whose coefficients are determined by the prolongation of the conformal Killing-Yano equation.
\end{exa}

\begin{exa}[Fefferman--Einstein metrics and Taub--NUT--(A)dS metrics]
	Reference \cite{Taghavi-Chabert2021} shows that the Fefferman--Einstein metric and the Taub--NUT--(A)dS metric belong to the same generalised almost Robinson geometry.
\end{exa}

\begin{exa}[Metrics of supergravity]
	The metric \eqref{eq:CW_metric} together with its Robinson $3$-form \eqref{eq:CW_Rob3f}  in Example \ref{exa:SUGRA} belongs to the equivalence class of metrics of an integrable generalised almost Robinson geometry of Kerr--Schild type, and to which the Minkowski metric also belongs.
\end{exa}

\section{Generalisation to other metric signatures and odd dimensions}\label{sec:other_sign}
The setting of the present article can be easily adapted to any pseudo-Riemannian manifold $(\mc{M},g)$ of signature $(p+1,q+1)$ for any even integer $p,q$ with $p + q =2m$, where there is also a notion of almost null structure $N$, defined to be a totally null complex distribution of rank $m$. In this general case, the real index $r$ of $N$ can take any of the values \cite{Kopczy'nski1992}
\begin{align*}
r \equiv \min(p+1,q+1) \pmod{2} \, .
\end{align*}
When $pq \neq 0$, it is therefore necessary to define an almost Robinson structure as an almost null structure \emph{of real index one}. Equivalently, these can be characterised as an optical geometry whose screen bundle is endowed with a bundle complex structure compatible with the screen bundle metric. One has to be cautious in the definition of a twist-induced almost Robinson structure since the screen bundle metric is no longer positive-definite. This difference is also reflected in the pure spinor approach, which now may be of different real indices: beside the purity condition \eqref{eq:pure}, pure spinors of real index one now satisfy further algebraic conditions \cite{Cartan1967,Kopczy'nski1992}. Other than these considerations, the properties of the intrinsic torsion given in the present article will apply to different metric signatures.

Finally, one can also define an almost Robinson structure $(N,K)$ on a $(2m+3)$-dimensional smooth Lorentzian manifold $(\mc{M}, g)$ (or its conformal analogue). Here, $N$ is a totally null  complex $(m+1)$-plane distribution, i.e.\ an almost null structure, of real index one so that $K$ is the rank-one null distribution arising from the real span of $N \cap \overline{N}$.  In dimension three, they are equivalent to optical geometries. Unlike in even dimensions, the real index has to be specified here, since generically, in odd dimensions, the real index is zero \cite{Kopczy'nski1997}. Another crucial difference is that $N$ is now strictly contained in its rank-$(m+2)$ orthogonal complement $N^\perp$, which makes the algebraic classification of its intrinsic torsion significantly more involved. For instance, one may require the integrability of either $N$ or $N^\perp$, or both. Nevertheless, these are also relevant to the study of solutions to Einstein field equations in higher dimensions as was shown in \cite{Mason2010,Taghavi-Chabert2011}. The geometry of almost null structures in odd dimensions is investigated in \cite{Taghavi-Chabert2012,TaghaviChabert2017,Taghavi-Chabert2017a,Nurowski2015}. Almost Robinson structures can also be defined in signatures $(2p+1,2q+2)$ for any $p,q$ with $pq\neq0$ as an almost null structure of real index one.

\appendix

\section{Projections}\label{app:proj}
In this appendix, we define the projections from the modules of intrinsic torsion $\mbb{G}$ to its irreducible $Q_0$-modules as given in Theorem \ref{thm-main-Rob}. We choose a Robinson $3$-form  $\rho_{a b c}$ associated to an optical $1$-form $\kappa_a$, and splitting operators $(\ell^a, \delta^a_i, k^a) = (\ell^a, \delta^a_\alpha, \delta^a_{\bar{\alpha}}, k^a)$, $(\delta^i_{\alpha} , \delta^i_{\bar{\alpha}})$, with $\kappa_a = g_{a b} k^b$. We shall be using these to convert index types, with the additional convention that if $\alpha_a$ is a $1$-form, we shall write $k^a \alpha_a = \alpha^0$ and $\ell^a \alpha_a = \alpha_0$. 

As usual, the screen space symmetric bilinear form, the Hermitian form and the complex structure will be denoted $h_{i j}$, $\omega_{i j}$ and $J_i{}^j$ respectively. Now let $\Gamma_{ab}{}^{c}$ of $\V^* \otimes \g$ and $\omega_{a b} = \omega_{i j} \delta^i_a \delta^j_b$. so that  $\rho_{a b c} = 3 \kappa_{[a} \omega_{b c]}$. We streamline notation by setting
\begin{align*}
(\Gamma \cdot \kappa)_{ab} & := - \Gamma_{ab}{}^{c} \kappa_c \, , &
(\Gamma \cdot \omega)_{abc} & := 2 \, \Gamma_{a[b} \, {}^d \omega_{c]d} \, , &
(\Gamma \cdot \rho)_{abcd} & := - 3 \, \Gamma_{a[b}{}^{e} \rho_{cd]e} \, .
\end{align*}
In particular,
\begin{align*}
(\Gamma \cdot \rho)_{abcd} = 3 (\Gamma \cdot \omega)_{a[bc} \kappa_{d]} + 3 (\Gamma \cdot \kappa)_{a[b} \omega_{c d]} \, .
\end{align*}
Note also that $(\Gamma \cdot \kappa)_{ab} k^b = 0$ and $(\Gamma \cdot \omega)_{abc} k^c = - (\Gamma \cdot \kappa)_{ac} J_b \,{}^c$. Then  it is easy to check
\begin{align}
(\Gamma \cdot \rho)_{a}{}^{0}{}_{j k} & = 0 \, , &
(\Gamma \cdot \rho)_{a 0 j k} & = (\Gamma \cdot \omega)_{a j k} + (\Gamma \cdot \kappa)_{a 0} \omega_{j k} \, , \nonumber \\
(\Gamma \cdot \rho)_{a i j k}& =  3 (\Gamma \cdot \kappa)_{a[i} \omega_{j k]} \, , &
(\Gamma \cdot \rho)_{a}{}^{0}{}_{0 k} & = - (\Gamma \cdot \kappa)_{a j} J_k{}^j \, . \label{eq:G-1_alt}
\end{align}

Let us first recall the projections from the module of intrinsic torsions to its irreducible $P_0$-modules from \cite{Fino2020}:
\begin{align}\label{eq-proj-intors-sim}
\begin{aligned}
\Pi_{-2}^0 & : \V^* \otimes \g \rightarrow \mbb{G}_{-2}^0 \, , && \Gamma_{ab}{}^{c} \mapsto \Pi_{-2}^0(\Gamma)_i := (\Gamma \cdot \kappa)^{0}{}_{i}  \, , \\
\Pi_{-1} & : \V^* \otimes \g \rightarrow \mbb{G}_{-1} \, , && \Gamma_{ab}{}^{c} \mapsto \Pi_{-1}(\Gamma)_{i j} := (\Gamma \cdot \kappa)_{i j} \, , \\
\Pi_{-1}^0 & : \V^* \otimes \g \rightarrow \mbb{G}_{-1}^0 \, , && \Gamma_{ab}{}^{c} \mapsto \Pi_{-1}^0(\Gamma) :=  \Pi_{-1}(\Gamma)_{i j} h^{i j} \, , \\
\Pi_{-1}^1 & : \V^* \otimes \g \rightarrow \mbb{G}_{-1}^1 \, , && \Gamma_{ab}{}^{c} \mapsto \Pi_{-1}^1 (\Gamma)_{ij} := \Pi_{-1}(\Gamma)_{[i j]}  \, , \\
\Pi_{-1}^2 & : \V^* \otimes \g \rightarrow \mbb{G}_{-1}^2 \, , && \Gamma_{ab}{}^{c} \mapsto \Pi_{-1}^2 (\Gamma)_{ij} := \Pi_{-1}(\Gamma)_{(i j)_\circ}  \, , \\
\Pi_0^0 & : \V^* \otimes \g \rightarrow \mbb{G}_0^0 \, , && \Gamma_{ab}{}^{c} \mapsto 
\Pi_0^0 (\Gamma)_i := (\Gamma \cdot \kappa)_{0 i}  \, .
\end{aligned}
\end{align}
With reference to equations \eqref{eq:G-1_alt}, we also define alternatives to $\Pi_{-1}$ and $\Pi_0$:
\begin{align*}
{\Pi'}_{-1}  & : \V^* \otimes \g \rightarrow \mbb{G}_{-1} : \Gamma \mapsto {\Pi'}_{-1} (\Gamma)_{i j} := (\Gamma \cdot \rho)_{i j}{}^{0}{}_{0} \, , \\
{\Pi'}_{0} & : \V^* \otimes \g \rightarrow \mbb{G}_{0} : \Gamma \mapsto {\Pi'}_{0} (\Gamma)_{i} := ( \Gamma \cdot \rho)_{0 i j k}  h^{j k} \, .
\end{align*}
We note the relation ${\Pi'}_{-1} (\Gamma)_{i j} = - \Pi_{-1} (\Gamma)_{i k} J_{j}{}^k$.

In dimension six, there are two further projections from $\V^* \otimes \g$ to the self-dual and anti-self-dual parts of $\mbb{G}_{-1}^1$ --- see \cite{Fino2020}. These will not be needed as they will be subsumed in the projections to $\mbb{G}_{-1}^{1,0} \oplus \mbb{G}_{-1}^{1,1}$ and $\mbb{G}_{-1}^{1,2}$ given below.

We shall presently introduce, for each $i,j,k$, a $Q_0$-module epimorphism $\Pi_i^{j,k} : \V^* \otimes \g \rightarrow \mbb{G}_i^{j,k}$ with the properties that $\V^* \otimes \mf{q}$ lies in the kernel of $\Pi_i^{j,k}$, and $\Pi_i^{j,k}$ descends to a projection from $\mbb{G}$ to $\mbb{G}_i^{j,k}$. By construction, the kernel of $\Pi_i^{j,k} \mod \V \otimes \mf{q}$ is precisely isomorphic to the complement $(\mbb{G}_i^{j,k})^c$ of $\mbb{G}_i^{j,k}$ in $\mbb{G}$ as $Q_0$-modules, i.e.\
\begin{align}\label{eq:proj2}
\ker \Pi_i^{j,k} / \left( \V^* \otimes \mf{q} \right) & \cong  (\mbb{G}_i^{j,k} )^c \, , & \mbox{i.e.} & & \left( \ker \Pi_i^{j,k} / \left( \V^* \otimes \mf{q} \right) \right)^c & \cong \mbb{G}_i^{j,k} \, .
\end{align}

Let us define the projections
\begin{align*}
\Pi_{-1}^{3,0} & :  \V^* \otimes \g \rightarrow \mbb{G}_{-1}^{3,0} \, , && \Gamma \mapsto \Pi_{-1}^{3,0} (\Gamma)_{j k} := \frac{1}{2} (\Gamma \cdot \rho)^0{}_{0 \ell_1 \ell_2} \left( \delta^{\ell_1}_j \delta^{\ell_2}_k - J_j \, {}^{\ell_1} J_k \, {}^{\ell_2} \right) \, , \\
\Pi_0^1 & : \V^* \otimes \g \rightarrow \mbb{G}_0^1 \, , && \Gamma\mapsto \Pi_0^1 (\Gamma)_{ijk} := \frac{1}{2} (\Gamma \cdot \rho)_{i \ell_1 \ell_2 0} \left( \delta^{\ell_1}_j \delta^{\ell_2}_k - J_j \, {}^{\ell_1} J_k \, {}^{\ell_2} \right) \, , \\
\Pi_1^{0,0} & : \V^* \otimes \g \rightarrow \mbb{G}_1^{0,0} \, , && \Gamma \mapsto \Pi_1^{0,0} (\Gamma)_{jk} :=  \frac{1}{2} (\Gamma \cdot \rho)_{0 0 \ell_1 \ell_2} \left( \delta^{\ell_1}_j \delta^{\ell_2}_k - J_j \, {}^{\ell_1} J_k \, {}^{\ell_2} \right) \, .
\end{align*}
 To guide the reader, we shall note that
\begin{align*}
 \Pi_{-1}^{3,0} (\Gamma)_{\alpha \beta} & = (\Gamma \cdot \omega)^0{}_{\alpha \beta} \, , &
 \Pi_0^1 (\Gamma)_{ijk} & =  (\Gamma \cdot \omega)_{i j k} &
 \Pi_{1}^{0,0} (\Gamma)_{jk} & = (\Gamma \cdot \omega)_{0 j k} \, .
\end{align*}
We can define the remaining projections $\Pi_{i}^{j,k}  : \V^* \otimes \g \rightarrow \mbb{G}_{i}^{j,k}$ by the properties
\begin{align}
\begin{aligned}\label{eq:proj_Qmod}
 \Pi_{-1}^{1,0} (\Gamma) & := [ \Pi_{-1}^1 (\Gamma)_{\alpha \bar{\beta}} h^{\alpha \bar{\beta}} ] = \Pi_{-1}^1 (\Gamma)_{ij} \omega^{ij} \, , \\
 \Pi_{-1}^{1,1} (\Gamma)_{i j} & := \dbl \Pi_{-1}^1 (\Gamma)_{\alpha \beta} \dbr \, , &
 \Pi_{-1}^{1,2} (\Gamma)_{i j} & := [ \left( \Pi_{-1}^1 (\Gamma)_{\alpha \bar{\beta}} \right)_\circ ] \, , \\
 \Pi_{-1}^{2,0} (\Gamma)_{i j} & := [ \Pi_{-1}^2 (\Gamma)_{\alpha \bar{\beta}} ]  \, , &
\Pi_{-1}^{2,1} (\Gamma)_{i j} & :=  \dbl \Pi_{-1}^2 (\Gamma)_{\alpha \beta}  \dbr  \, , \\
\Pi_0^{0,0} (\Gamma)_i & :=  \dbl \Pi_0^0 (\Gamma)_\alpha \dbr \, , \\
\Pi_0^{1,0} (\Gamma)_i & :=  \dbl \Pi_0^1 (\Gamma)_{\bar{\alpha} \beta \gamma} h^{\beta \bar{\alpha}} \dbr = \Pi_0^1 (\Gamma)_{j k i} h^{j k} \, , &
 \Pi_0^{1,1} (\Gamma)_{i j k} & :=  \dbl \Pi_0^1 (\Gamma)_{[\alpha \beta \gamma]} \dbr \, , \\
\Pi_0^{1,2} (\Gamma)_{i j k} & := \dbl \Pi_0^1 (\Gamma)_{(\alpha \beta) \gamma} \dbr  \, , &
\Pi_0^{1,3} (\Gamma)_{i j k} & :=  \dbl  \left( \Pi_0^1 (\Gamma)_{\bar{\alpha} \beta \gamma} \right)_\circ \dbr \, .
\end{aligned}
\end{align}
where we recall that $\dbl \cdot \dbr$ and $[ \cdot ]$ denote the real spans of the enclosed quantities.

That these are indeed projections is not too difficult to check. We also define a variant of the maps $\Pi_{-1}^{1,1}$ and $\Pi_{0}^{1,0}$ by
\begin{align*}
{\Pi'}_{-1}^{1,1} & : \V^* \otimes \g \rightarrow \mbb{G}_{-1}^{1,1} \, , \Gamma \mapsto {\Pi'}_{-1}^{1,1} (\Gamma)_{i j} := \dbl \i {\Pi'}_{-1} (\Gamma)_{[\alpha \beta]} \dbr \, , \\
{\Pi'}_{0}^{0,0} & : \V^* \otimes \g \rightarrow \mbb{G}_{0}^{0,0} \, ,  \Gamma \mapsto {\Pi'}_{0}^{0,0} (\Gamma)_{i} := \frac{1}{2m-2} {\Pi'}_{0} (\Gamma)_{i j k} \omega^{j k} =  - \frac{1}{m-1} \dbl  \i (\Gamma \cdot \rho)_{0 \alpha \beta \bar{\gamma}}  h^{\beta \bar{\gamma}} \dbr \, .
\end{align*}
One can indeed verify that these satisfy $\Pi_{-1}^{1,1} (\Gamma)_{i j}  =  {\Pi'}_{-1}^{1,1} (\Gamma)_{i j}$ and $\Pi_{0}^{1,0} (\Gamma)_{i} = {\Pi'}_{0}^{1,0} (\Gamma)_{i}$. We are now in the position to introduce the following families of maps: for any $[x:y] \in \RP^1$, $[z:w] \in \CP^1$,
\begin{align*}
(\Pi_{-1}^{0 \times 1})_{[x:y]} &:  \V^* \otimes \g \rightarrow \mbb{G}_{-1}^{0,0} \oplus  \mbb{G}_{-1}^{1,0} \, ,  \Gamma \mapsto (\Pi_{-1}^{0 \times 1})_{[x:y]} (\Gamma) := x \, \Pi_{-1}^{0,0}(\Gamma) + y \, \Pi_{-1}^{1,0}(\Gamma) \, , \\
(\Pi_{-1}^{1 \times 2})_{[x:y]} &:  \V^* \otimes \g \rightarrow \mbb{G}_{-1}^{1,2} \oplus  \mbb{G}_{-1}^{2,0} \, ,  \Gamma \mapsto (\Pi_{-1}^{1 \times 2})_{[x:y]} (\Gamma) := [ x \, \Pi_{-1}^{2,0}(\Gamma)_{\alpha \bar{\beta}} -  y \, \i \, \Pi_{-1}^{1,2}(\Gamma)_{\alpha \bar{\beta}} ] \, , \\
(\Pi_{-1}^{1 \times 3})_{[z:w]} &:  \V^* \otimes \g \rightarrow \mbb{G}_{-1}^{1,1} \oplus  \mbb{G}_{-1}^{3,0} 
 \, ,  \Gamma \mapsto (\Pi_{-1}^{1 \times 3})_{[z:w]} (\Gamma) := \dbl z \, {\Pi'}_{-1}^{1,1}(\Gamma)_{\alpha \beta} + w \, {\Pi}_{-1}^{3,0}(\Gamma)_{\alpha \beta} \dbr \, , \\
(\Pi_{0}^{0 \times 1})_{[z:w]} &:  \V^* \otimes \g \rightarrow \mbb{G}_{0}^{0,0} \oplus  \mbb{G}_{0}^{1,0}  \, , \Gamma \mapsto (\Pi_{0}^{0 \times 1})_{[z:w]} (\Gamma) := \dbl z \, {\Pi'}_{0}^{0,0}(\Gamma)_{\alpha} + w \, \Pi_{0}^{1,0}(\Gamma)_{\alpha} \dbr \, .
\end{align*}

Finally, we give alternative forms of the maps defined above in terms of algebraic relations with $J_i{}^j$, $\omega_{i j}$ and $h_{i j}$:
\begin{align*}
\begin{aligned}
\Pi_{-2}^{0,0} & : \V^* \otimes \g \rightarrow \mbb{G}_{-2}^{0.0} \, , \Gamma \mapsto \Pi_{-2}^{0,0}(\Gamma)_i := \Pi_{-2}^0(\Gamma)_i \, , \\
\Pi_{-1}^{0,0} & : \V^* \otimes \g \rightarrow \mbb{G}_{-1}^{0,0} \, ,   \Gamma \mapsto \Pi_{-1}^{0,0}(\Gamma) := \Pi_{-1}^0(\Gamma) \, , \\
\Pi_{-1}^{1,0} & : \V^* \otimes \g \rightarrow \mbb{G}_{-1}^{1,0} \, ,  \Gamma \mapsto \Pi_{-1}^{1,0} (\Gamma) := \Pi_{-1}^1 (\Gamma)_{ij} \omega^{ij}  \, , \\
\Pi_{-1}^{1,1} & : \V^* \otimes \g \rightarrow \mbb{G}_{-1}^{1,1} \, ,  \Gamma \mapsto \Pi_{-1}^{1,1} (\Gamma)_{ij} := \frac{1}{2}\Pi_{-1}^1 (\Gamma)_{k\ell} \left( \delta^k_i \delta^\ell_j - J_i \, {}^k J_j \, {}^\ell \right) \, , \\
\Pi_{-1}^{1,2} & : \V^* \otimes \g \rightarrow \mbb{G}_{-1}^{1,2} \, ,  \Gamma \mapsto \Pi_{-1}^{1,2} (\Gamma)_{ij} := \frac{1}{2}\Pi_{-1}^1 (\Gamma)_{k\ell} \left( \delta^k_i \delta^\ell_j + J_i \, {}^k J_j \,{}^\ell - \frac{2}{n} \omega_{ij} \omega^{k \ell} \right) \, , \\
\Pi_{-1}^{2,0} & : \V^* \otimes \g \rightarrow \mbb{G}_{-1}^{2,0} \, ,  \Gamma \mapsto \Pi_{-1}^{2,0} (\Gamma)_{ij} := \frac{1}{2} \Pi_{-1}^2 (\Gamma)_{k\ell} \left( \delta^k_i \delta^\ell_j + J_i \,{}^k J_j \,{}^\ell \right)   \, , \\
\Pi_{-1}^{2,1} & : \V^* \otimes \g \rightarrow \mbb{G}_{-1}^{2,1} \, ,  \Gamma \mapsto \Pi_{-1}^{2,1} (\Gamma)_{ij} := \frac{1}{2} \Pi_{-1}^2 (\Gamma)_{k\ell} \left( \delta^k_i \delta^\ell_j - J_i \,{}^k J_j \,{}^\ell \right)   \, , \\
\Pi_{-1}^{3,0} & : \V^* \otimes \g \rightarrow \mbb{G}_{-1}^{3,0} \, ,  \Gamma \mapsto \Pi_{-1}^{3,0} (\Gamma)_{j k} := \frac{1}{2} (\Gamma \cdot \rho)^0{}_{0 \ell_1 \ell_2} \left( \delta^{\ell_1}_j \delta^{\ell_2}_k - J_j \, {}^{\ell_1} J_k \, {}^{\ell_2} \right)  \, ,
\end{aligned}
\end{align*}
\begin{align*}
\begin{aligned}
\Pi_0^{0,0} & : \V^* \otimes \g \rightarrow \mbb{G}_0^{0,0} \, ,  \Gamma \mapsto 
\Pi_0^{0,0} (\Gamma)_i := \Pi_0^0 (\Gamma)_i  \, , \\
\Pi_0^{1,0} & : \V^* \otimes \g \rightarrow \mbb{G}_0^{1,0} \, ,  \Gamma \mapsto \Pi_0^{1,0} (\Gamma)_k :=  \Pi_0^1 (\Gamma)_{ijk} h^{ij} \, , \\
\Pi_0^{1,1} & : \V^* \otimes \g \rightarrow \mbb{G}_0^{1,1} \, ,  \Gamma \mapsto \Pi_0^{1,1} (\Gamma)_{ijk} :=  \frac{1}{2} \Pi_0^1 (\Gamma)_{\ell m [i} \left( \delta^\ell_j \delta^m_{k]} - J_j \,{}^\ell J_{k]} \,{}^m \right) \, , \\
\Pi_0^{1,2} & : \V^* \otimes \g \rightarrow \mbb{G}_0^{1,2} \, ,  \Gamma \mapsto \Pi_0^{1,2} (\Gamma)_{ijk} :=  \frac{1}{2} \Pi_0^1 (\Gamma)_{\ell m k} \left( \delta^\ell_{(i} \delta^m_{j)} - J_{(i} \,{}^\ell J_{j)} \,{}^m \right) \, , \\
\Pi_0^{1,3} & : \V^* \otimes \g \rightarrow \mbb{G}_0^{1,3} \, ,  \\
& \Gamma \mapsto \Pi_0^{1,3} (\Gamma)_{ijk} :=  \frac{1}{2} \Pi_0^1 (\Gamma)_{\ell m [k} \left( \delta^m_{j]} \delta^\ell_i + J_{j]} \,{}^m J_i \,{}^\ell - \frac{2}{m-1} \left( h_{j]i} h^{\ell m} - \omega_{j]i} \omega^{\ell m}\right) \right) \, , \\
\Pi_1^{0,0} & : \V^* \otimes \g \rightarrow \mbb{G}_1^{0,0} \, ,  \Gamma \mapsto \Pi_1^{0,0} (\Gamma)_{jk} :=  \frac{1}{2} (\Gamma \cdot \rho)_{0 0 \ell_1 \ell_2} \left( \delta^{\ell_1}_j \delta^{\ell_2}_k - J_j \, {}^{\ell_1} J_k \, {}^{\ell_2} \right) \, .
\end{aligned}
\end{align*}
Note that for the map $\Pi_0^{1,3}$, we use the identity $\omega^{ij} ( \Gamma \cdot \omega )_{ijk} = h^{ij} ( \Gamma \cdot \omega )_{ij\ell} J_k \,{}^\ell$, or equivalently $h^{ij} ( \Gamma \cdot \omega )_{ijk} = - \omega^{ij} ( \Gamma \cdot \omega )_{ij\ell} J_k \,{}^\ell$. Other useful identities include
\begin{align*}
\Pi_{-1}^{1,1} (\Gamma)_{ki} J_j \,{}^k & = \Pi_{-1}^1 (\Gamma)_{k[i} J_{j]} \,{}^k \, , &
\Pi_{-1}^{1,2} (\Gamma)_{ki} J_j \,{}^k & = \Pi_{-1}^1 (\Gamma)_{k\ell} \left( J_{(i} \,{}^k \delta_{j)}^\ell + \frac{1}{n-2} h_{ij} \omega^{k\ell} \right) \, , \\
\Pi_{-1}^{2,0} (\Gamma)_{ki} J_j \,{}^k & = \Pi_{-1}^2 (\Gamma)_{k[i} J_{j]} \,{}^k \, , &
\Pi_{-1}^{2,1} (\Gamma)_{ki} J_j \,{}^k & = \Pi_{-1}^2 (\Gamma)_{k(i} J_{j)} \,{}^k \, , \\
\Pi_0^{1,0} (\Gamma)_j J_k \,{}^j & = \omega^{ij} ( \Gamma \cdot \omega )_{ijk}  \, .
\end{align*}
For the remaining modules, we record, for each $[x:y] \in \RP^1$, $[z,w] \in \CP^1$,
\begin{align*}
(\Pi_{-1}^{0 \times 1})_{[x:y]} &:  \V^* \otimes \g \rightarrow \mbb{G}_{-1}^{0,0} \oplus  \mbb{G}_{-1}^{1,0} \, ,  \Gamma \mapsto (\Pi_{-1}^{0 \times 1})_{[x:y]} (\Gamma) := x \, \Pi_{-1}^{0,0}(\Gamma) + y \, \Pi_{-1}^{1,0}(\Gamma) \, , \\
(\Pi_{-1}^{1 \times 2})_{[x:y]} &:  \V^* \otimes \g \rightarrow \mbb{G}_{-1}^{1,2} \oplus  \mbb{G}_{-1}^{2,0} \, ,  \Gamma \mapsto (\Pi_{-1}^{1 \times 2})_{[x:y]} (\Gamma) := x \, \Pi_{-1}^{1,2}(\Gamma)_{i k} J_j \, {}^k + y \, \Pi_{-1}^{2,0}(\Gamma)_{i j} \, , \\
(\Pi_{-1}^{1 \times 3})_{[z:w]} &:  \V^* \otimes \g \rightarrow \mbb{G}_{-1}^{1,1} \oplus  \mbb{G}_{-1}^{3,0} \, , \\
&  \Gamma \mapsto (\Pi_{-1}^{1 \times 3})_{[z:w]} (\Gamma) :=  \Re\left(z \, {\Pi'}_{-1}^{1,1}(\Gamma)_{i j} + w \, {\Pi}_{-1}^{3,0}(\Gamma)_{i j} \right) \\
& \qquad \qquad \qquad \qquad \qquad \qquad \qquad \qquad  + \Im\left(z \, {\Pi'}_{-1}^{1,1}(\Gamma)_{i k} + w \, \Pi_{-1}^{3,0}(\Gamma)_{i k} \right)  J_{j}{}^{k}   \, , \\
(\Pi_{0}^{0 \times 1})_{[z:w]} &:  \V^* \otimes \g \rightarrow \mbb{G}_{0}^{0,0} \oplus  \mbb{G}_{0}^{1,0} \, , \\
& \Gamma \mapsto (\Pi_{0}^{0 \times 1})_{[z:w]} (\Gamma) := \Re\left(z \, \Pi_{0}^{0,0}(\Gamma)_{i} + w \, \Pi_{0}^{1,0}(\Gamma)_{i} \right) \\
& \qquad \qquad \qquad \qquad \qquad \qquad \qquad \qquad + \Im\left(z \, \Pi_{0}^{0,0}(\Gamma)_{j} + w \, \Pi_{0}^{1,0}(\Gamma)_{j} \right)  J_{i}{}^{j} \, ,
\end{align*}
where $\Re(\cdot)$ and $\Im(\cdot)$ denote the real and imaginary parts respectively.

\section{Generalised almost Robinson geometry --- connections}\label{app:gen_al_Rob}
Let $(\mc{M}, N, K, \mbf{o})$ be a generalised almost Robinson geometry, and let $g$ and $\wh{g}$ be two metrics in $\mbf{o}$ related via
\begin{align*}
\wh{g}_{a b} & = g_{a b} + 2 \kappa_{(a} \alpha_{b)} \, ,
\end{align*}
for some $1$-form $\alpha_a$. Denote by $\wh{\nabla}$ and $\nabla$ their corresponding Levi-Civita connections. Then for any $1$-form $\nu_a$, we have
\begin{align*}
\wh{\nabla}_a \nu_b = \nabla_a \nu_b - Q_{a b}{}^{c} \nu_c \, ,
\end{align*}
where $Q_{a b c} = Q_{a b}{}^{d} g_{d c}$ is given by
\begin{multline*}
Q_{a b c} + Q_{a b}{}^{d} \alpha_{c} \kappa_{d} + Q_{a b}{}^{d} \kappa_{c} \alpha_{d} = - ( \nabla_c \kappa_{(a} ) \alpha_{b)}  - ( \nabla_c \alpha_{(a} ) \kappa_{b)} \\ +  ( \nabla_{(a}\kappa_{b)} ) \alpha_{c}  +  ( \nabla_{(a}\alpha_{b)} ) \kappa_{c} 
+ \alpha_{(a}  ( \nabla_{b)}\kappa_{c} )  + \kappa_{(a}  ( \nabla_{b)}\alpha_{c} ) 
\end{multline*}
Set $\beta = (1 + \alpha^0)^{-1}$. Contracting this expression with instances of $k^a$, $\delta^a_{\alpha}$, $\delta^a_{\bar{\alpha}}$ and $\ell^a$, and using the definitions \eqref{eq:S-H1_b} and \eqref{eq:S-H2_b} yields
\begin{align}
Q_{a}{}^{0 0} &  = 0 \, , \label{eq:Q1} \\
Q_{\alpha \beta}{}^{0}
& = - \beta \gamma_{(\alpha} \alpha_{\beta)} + \beta \alpha^0 \sigma_{\alpha \beta}  \, , \label{eq:Q2}  \\
Q_{\bar{\alpha} \beta}{}^{0}
   & = -  \frac{1}{2} \beta ( \nabla \kappa )^0{}_0 \alpha_{\beta}  - \frac{1}{2} \beta \gamma_\beta \alpha_{0} + \alpha^0 \beta ( \d \kappa )_{0 \beta}    +  \beta  ( \d \alpha)_{\beta}{}^0 \, .\label{eq:Q3} 
\end{align}
\begin{align}
Q^{0}{}_{\beta \gamma} 
& =  \alpha_{(\gamma}  \gamma_{\beta)} + \alpha^{0}  \tau_{\beta \gamma} \, , \label{eq:Q4} \\
Q_{\alpha \beta \gamma} 
& = - Q_{\alpha \beta}{}^{0} \alpha_{\gamma} + 2 \alpha_{(\alpha}  \tau_{\beta) \gamma}  + \alpha_{\gamma}  \sigma_{\alpha \beta} \, , \label{eq:Q5}\\
Q_{\bar{\alpha} \beta \gamma}
& = - Q_{\bar{\alpha} \beta}{}^{0} \alpha_{\gamma} - \alpha_{[\beta} ( \nabla \kappa)_{\gamma] \bar{\alpha}} + ( \nabla \kappa)_{\bar{\alpha} (\beta}  \alpha_{\gamma)} + \alpha_{\bar{\alpha}}  \tau_{\beta \gamma} \, , \label{eq:Q6} \\
Q_{0 \beta \gamma} & = - Q_{0 \beta}{}^{0} \alpha_{\gamma} - \alpha_{[\beta} (\nabla \kappa)_{\gamma] 0}  + ( \d \alpha )_{[\beta \gamma]} +  ( \nabla \kappa )_{0 (\beta} \alpha_{\gamma)}  + \alpha_{0}  \tau_{\beta \gamma} \label{eq:Q7} 
\end{align}
Now,
\begin{align*}
\wh{\nabla}_a \kappa_b & = \nabla_a \kappa_b - Q_{a b}{}^{c} \kappa_c \, , &
\wh{\nabla}_a \rho_{b c d} & = \nabla_a \rho_{b c d} - 3 \, Q_{a [b}{}^{e} \rho_{c d] e} \, , 
\end{align*}
so that
\begin{align}
 (\wh{\nabla} \kappa)_{a b} k^a \delta^b_i & =  \gamma_i \, , &
(\wh{\nabla}_{a b} )\delta^a_{(i} \delta^b_{j)_\circ} & = \sigma_{i j} -   Q_{(i j)_\circ}{}^0 \, , &
(\wh{\nabla}_{a b}) \delta^a_{[i} \delta^b_{j]} & = \tau_{i j}  \, , \label{eq_DD1}
\end{align}
\begin{align}
(\wh{\nabla} \rho)_{a b c d}) k^a \delta^b_\beta \delta^c_\gamma \ell^ d 
& = \zeta_{\beta \gamma} + 2 \i Q^{0}{}_{[\beta \gamma]} \, ,  \label{eq_DD2} \\
(\wh{\nabla} \rho)_{a b c d}) \delta^a_{[\alpha} \delta^b_{\beta]} k^c \ell^d & = 
- \i \tau_{\alpha \beta}  - \i  Q^{0}{}_{[\alpha \beta]}  \, ,  \label{eq_DD3} \\
(\wh{\nabla} \rho)_{a b c d}) \delta^a_\alpha \delta^b_\beta \delta^c_\gamma \ell^ d 
& = G_{\alpha \beta \gamma} + 2 \i Q_{\alpha [\beta \gamma]} \, ,  \label{eq_DD4}\\
(\wh{\nabla} \rho)_{a b c d}) \delta^a_{\bar{\alpha}} \delta^b_\beta \delta^c_\gamma \ell^ d 
& = G_{\bar{\alpha} \beta \gamma} + 2 \i Q_{\bar{\alpha} [\beta \gamma]} \, ,  \label{eq_DD5} \\
(\wh{\nabla} \rho)_{a b c d}) \ell^a \delta^b_\beta \delta^c_\gamma \ell^ d 
& = B_{\beta \gamma} + 2 \i Q_{0 [\beta \gamma]} \, ,  \label{eq_DD6} \\
(\wh{\nabla} \rho)_{a b c d})  \ell^a \delta^b_{\beta} \delta^c_{\gamma} \delta^d_{\bar{\alpha}}
& = 2 \i E_{[\beta} h_{\gamma] \bar{\alpha}} - 2 \i Q_{0 [\beta}{}^{0} h_{\gamma] \bar{\alpha}} \, ,  \label{eq_DD7} \\
h^{\gamma \bar{\alpha}} (\wh{\nabla}_a \rho_{b c d} )\ell^a \delta^b_{\beta} \delta^c_{\gamma} \delta^d_{\bar{\alpha}}
& = (m-1) \left( \i E_{\beta}  - 2 \i  Q_{0 \beta}{}^{0} \right) \, ,  \label{eq_DD8} \\
h^{\gamma \bar{\alpha}} (\wh{\nabla}_a \rho_{b c d}) \delta^a_{\bar{\alpha}} \delta^b_\beta \delta^c_\gamma \ell^ d 
& = - G_{\beta} +  \i h^{\gamma \bar{\alpha}} \left( Q_{\bar{\alpha} \beta \gamma} -  Q_{\bar{\alpha} \gamma \beta } \right) \, .  \label{eq_DD9}
\end{align}
In particular,
\begin{multline*}
- 2 h^{\gamma \bar{\alpha}} (\wh{\nabla}_a \rho_{b c d} )\ell^a \delta^b_{\beta} \delta^c_{\gamma} \delta^d_{\bar{\alpha}} - h^{\gamma \bar{\alpha}} (\wh{\nabla}_a \rho_{b c d}) \delta^a_{\bar{\alpha}} \delta^b_\beta \delta^c_\gamma \ell^ d  =  G_{\beta} -2(m-1) \i E_{\beta} \\
 + 4 (m-1)  \i  Q_{0 \beta}{}^0 -  \i h^{\gamma \bar{\alpha}} \left( Q_{\bar{\alpha} \beta \gamma} -  Q_{\bar{\alpha} \gamma \beta } \right) \, .
\end{multline*}

%
%
%


\printbibliography

\end{document}